\documentclass[11pt, a4paper]{amsbook}

\usepackage[lmargin=4cm, rmargin=2cm, top=2cm, bottom=2cm, includehead, includefoot,]{geometry}
\usepackage{geometry} 
\usepackage{url}
\usepackage{amssymb}
\usepackage{amsmath}
\usepackage{accents}
\usepackage{eucal}
\usepackage{amscd}
\usepackage[shortalphabetic]{amsrefs}
\usepackage{times}

\allowdisplaybreaks[1]

\newtheorem{mthm}{Main Theorem}
\newtheorem{thm}{Theorem}[chapter]
\newtheorem{lem}[thm]{Lemma}
\newtheorem{prop}[thm]{Proposition}
\newtheorem{cor}[thm]{Corollary}
\newtheorem{defn}[thm]{Definition}

\numberwithin{equation}{chapter}
\numberwithin{section}{chapter}

\providecommand{\abs}[1]{\lvert#1\rvert}

\providecommand{\norm}[1]{\lVert#1\rVert}
\DeclareMathOperator{\tr}{tr}
\DeclareMathOperator{\diam}{diam}
\newcommand{\ho}{\accentset{\circ}{h}}
\newcommand{\lo}{\accentset{\circ}{\lambda}}
\newcommand{\p}{\partial}
\newcommand{\nablap}{\accentset{\bot}{\nabla}}
\newcommand{\Rp}{\accentset{\perp}{R}}
\newcommand{\Wdot}{\accentset{\bullet}{W}}
\newcommand{\Wcirc}{\accentset{\circ}{W}}
\newcommand{\Cdot}{\accentset{\bullet}{C}}
\newcommand{\Ccirc}{\accentset{\circ}{C}}
\newcommand{\Fdot}{\accentset{\bullet}{F}}
\DeclareMathOperator{\osc}{osc}
\newcommand{\gp}{\accentset{\perp}{g}}
\DeclareMathOperator{\inj}{inj}
\DeclareMathOperator{\vol}{vol}
\newcommand{\pip}{\accentset{\bot}{\pi}}

\begin{document}
\frontmatter

\cleardoublepage
\thispagestyle{empty}
\begin{center}
\vspace*{\fill} \Huge
                        The mean curvature flow of submanifolds of high codimension
\\
\vfill\vfill\Large
                          Charles Baker
\\
\vfill\vfill
                          November 2010
\\
\vfill\vfill \normalsize
         A thesis submitted for the degree of Doctor of Philosophy\\
         of the Australian National University
\vfill

\end{center}

%\chapter*{}
\cleardoublepage
\thispagestyle{empty}
\vspace*{\fill}

\begin{center}
	\emph{For Gran and Dar}

\end{center}

\vfill\vfill\vfill

\chapter*{Declaration}\label{declaration}
\thispagestyle{empty}
The work in this thesis is my own except where otherwise stated.

\vspace{1in}

\hfill\hfill\hfill
Charles Baker
\hspace*{\fill}

\chapter*{Abstract}\label{abstract}

%\thispagestyle{empty}
%\addcontentsline{toc}{chapter}{Abstract}

A geometric evolution equation is a partial differential equation that evolves some kind of geometric object in time.  The protoype of all parabolic evolution equations is the familiar heat equation.  For this reason parabolic geometric evolution equations are also called geometric heat flows or just geometric flows.  The heat equation models the physical phenomenon 
whereby heat diffuses from regions of high temperature to regions of cooler temperature. A defining characteristic of 
this physical process, as one readily observes from our surrounds, is that it occurs smoothly: A hot cup 
of coffee left to stand will over a period of minutes smoothly equilibrate to the ambient temperature.  In the case of a geometric flow, it is some kind of geometric object that diffuses smoothly down a driving gradient.  The most natural extrinsically defined geometric heat flow is the mean curvature flow. This flow evolves regions of curves and surfaces with high curvature to regions of smaller curvature. For example, an ellipse with highly curved, pointed ends evolves to a circle, thus minimising the distribution of curvature. It is precisely this smoothing, energy-minimising characteristic that makes geometric flows powerful mathematical tools. From a pure mathematical perspective, this is a useful property because unknown and complicated objects can be smoothly deformed into well-known and easily understood objects. From an applications point of view, it is an observed natural law that 
physical systems will move towards a state that minimises some notion of energy.  As an example, crystal grains will try to arrange themselves so as to minimise the curvature of the interface between them.

The study of the mean curvature flow from the perspective of partial differential equations began with Gerhard Huisken's pioneering work in 1984.  Since that time, the mean curvature flow of hypersurfaces has been a lively area of study.  Although Huisken's seminal paper is now just over twenty-five years old, the study of the mean curvature flow of submanifolds of higher codimension has only recently started to receive attention.  The mean curvature flow of submanifolds is the main object of investigation in this thesis, and indeed, the central results we obtain can be considered as high codimension analogues of some early hypersurface theorems.  The result of Huisken's 1984 paper roughly says that convex hypersurfaces evolve under the mean curvature flow to round points in finite time.  Here we obtain the result that if the ratio of the length of the second fundamental form to the length of the mean curvature vector is bounded (by some explicit constant depending on dimension but not codimension), then the submanifold will evolve under the mean curvature flow to a round point in finite time.  We investigate evolutions in flat and curved backgrounds, and explore the singular behaviour of the flows as the first singular time is approached.

\tableofcontents
\mainmatter

\chapter{Introduction}\label{ch: introduction}

The mean curvature flow is a well-known geometric evolution equation.  The study of the mean curvature flow from the perspective of partial differential equations commenced with Huisken's seminal paper \cite{gH84} on the flow of convex hypersurfaces.  Since the appearance of that paper the mean curvature flow of hypersurfaces has been a lively area of study, and indeed continues to be so. Although this seminal paper is now just over twenty-five years old, the study of the mean curvature flow of submanifolds of higher codimension has only very recently started to receive attention.  This thesis is concerned with the mean curvature flow of submanifolds of arbitrary codimension, and the main results we obtain can be considered high codimension analogues of some early hypersurface results due to Huisken.  To give these high codimension results some context, we first briefly survey the relevant hypersurface theory.

Let $F : \Sigma^n \rightarrow N^{n+k}$ be a smooth immersion of a closed manifold $\Sigma$, and $H(p,t)$ be the mean curvature vector of $\Sigma_t(p) := F(\Sigma(p), t)$.  The mean curvature flow of an initial immersion $F_0$ is given by a time-dependent family of immersions $F : \Sigma \times [0,T) \rightarrow N^{n+k}$ that satisfy
\begin{equation}\label{e: defn MCF}
	\begin{cases}
		\frac{\p}{\p t} F(p,t) = H(p,t), \quad p \in \Sigma, \, t \geq 0 \\
		F( \cdot, 0) = F_0.
	\end{cases}
\end{equation}
The mean curvature flow equation determines a weakly parabolic quasilinear system of second order.  We refer to the initial-boundary value problem \eqref{e: defn MCF} as `MCF.'  We advise the reader that we shall sometimes refer to MCF as an equation, and at other times, as a system.  We also point out that by hypersurface or submanifold, we mean an object that has dimension greater than or equal to two.  For the entirety of this thesis the reader is to assume that $n \geq 2$. Flows of space curves have been studied before, however the techniques are not the same (the Codazzi equation is vacuous for a curve). The main theorem of \cite{gH84} asserts that the mean curvature flow evolves a convex hypersurface of Euclidean space to a round point in finite time.  Huisken's approach to this problem was inspired by Richard Hamilton's seminal work on the Ricci flow \cite{rH82}, which had appeared two years earlier.   Because the normal bundle of a hypersurface is one-dimensional, both the second fundamental form and the mean curvature can be very profitably viewed as essentially scalar-valued objects.  The second fundamental form can be treated as a scalar-valued symmetric $(1,1)$-tensor, similar to the Ricci tensor, and many of the techniques developed by Hamilton in his study of the Ricci flow can be used.  The first crucial step in \cite{gH84} is to show that convexity of the surface in preserved by the mean curvature flow, and this is achieved by Hamilton's tensor maximum principle.  After tackling the problem of hypersurfaces of Euclidean space, Huisken next went on to investigate the flow of hypersurfaces in a general Riemannian manifold, and slightly later, of hypersurfaces of the sphere.  The Riemannian case showed that negative curvature of the background hindered the flow, whilst positive curvature helped.  Although in this thesis we do not investigate the case of arbitrary Riemannian backgrounds, we mention that in \cite{gH86} convergence results similar \cite{gH84} are still true provided the initial hypersurface is sufficiently positively curved to overcome the negative curvature of the ambient space.  On the other hand, Huisken's results in \cite{gH87} are particularly relevant to some of the work in this thesis.  Since the sphere has positive curvature this helps the flow, and in this case Huisken was able to attain convergence results when the initial hypersurface satisfies a non-convex pinching condition.  The pinching condition we work with for submanifolds is very similar that of \cite{gH87}.

A feature of \cite{gH84} was that at the finite maximal time of existence, the entire hypersurface disappeared into a point at the same time.  The `roundness' of the point is made precise by magnifying the hypersurface as the singular time is approached.  This distinguishing feature is a manifestation of the convexity of the initial hypersurface.  If this condition is relaxed and the initial hypersurface is only assumed to have positive mean curvature, then in general more highly curved regions will shrink faster than less curved regions, and a singularity will develop at some point before the entire hypersurface disappears.  This naturally leads one to ask what are the possible limiting shapes of an evolving hypersurface as the (first) singular time is approached.  It is customary to break up the kinds of singularities that can form into two categories depending on the rate at which the singularity forms.  For the present discussion is suffices just to say these are called type 1 and type II singularities.  It turns out that type I singularities are much easier to analyse than type II singularities, and in the type I case Huisken was able to obtain a complete classification.  This was carried out in two papers, \cite{gH90a}, where compact blow-up limits were classified, and in \cite{gH90b}, which treated the more general complete case.  A key element of this singularity analysis was the monotonicity formula introduced in \cite{gH90a}.

Having briefly sketched the first developments in the study of the flow of hypersurfaces, we now turn to the study of the mean curvature flow of submanifolds, what is known and the results contained in this thesis.  Much of the previous work on high codimension mean curvature flow has used assumptions on the Gauss image, focussing on graphical \cites{CLT02,LiLi03,mtW02,mtW04}, symplectic or Lagrangian submanifolds \cites{SW02,CL04,mtW01,kS04,aN07}.  Another line of approach has been to make use of the fact that convex subsets of the Grassmannian are preserved \cites{TW04,mtW03,mtW05}.  In this thesis we work with conditions on the extrinsic curvature (second fundamental form), which have the advantage of being invariant under rigid motions.  Several difficulties arise in carrying out this program:  First, in high codimension the second fundamental form has a much more complicated structure than in the hypersurface case.  In particular, under MCF the second fundamental form evolves according to a reaction-diffusion system in which the reaction terms are rather complicated, whereas in the hypersurface case they are quite easily understood.  Thus it can be extremely difficult to determine whether the reaction terms are favourable for preserving a given curvature condition.  Second, there do not seem to be any useful invariant conditions on the extrinsic curvature which define convex subsets of the space of second fundamental forms.  This lack of convexity is forced by the necessity for invariance under rotation of the normal bundle.  This means that the vector bundle maximum principle formulated by Hamilton in \cite{rH86}, which states that the reaction-diffusion system will preserve an invariant convex set if the reaction terms are favourable, cannot be applied.  The latter maximum principle has been extremely effective in the Ricci flow in high dimensions \cites{BW08,BS09,sB09} where the algebraic complexity of the curvature tensor has presented similar difficulties.  For arbitrary reaction-diffusion systems, the convexity condition is necessary for a maximum principle to apply.  However, in our setting the Codazzi equation adds a constraint on the first derivatives of solutions that allows some non-convex sets to be preserved.  As we have already mentioned, a similar situation arose in \cite{gH87}, where a non-convex condition was preserved.

The content of this thesis is as follows.  In the first chapter we summarise some standard facts on the geometry of submanifolds in high codimension from a `modern' perspective.  A key aspect of this is the machinery of connections on vector bundles, which we employ extensively in deriving the evolution equations for geometric quantities.  In particular we introduce connections on tangent and normal bundles defined over both space and time, which prove very useful in deriving evolution equations and allowing simple commutation of time and space derivatives.  This connection also provides a natural interpretation of the `Uhlenbeck trick' introduced in \cite{rH86} to take into account the change in length of spatial tangent vectors under the flow.

The second chapter fills in some details in the proof of short time existence for fully nonlinear parabolic systems of even order and applies this to the mean curvature flow.  This a `standard' result that is frequently quoted in the literature, yet a complete proof, especially in the setting of equations defined on a manifold, continues to remain elusive.  In this regard, we draw attention to Lamm's Diploma Thesis, where he comprehensively proves local existence for fully nonlinear parabolic systems of even order in Euclidean space.  We reconstruct some details of the following theorem:

\begin{mthm}\label{mthm: main thm 1}
Let $E \times (0, \omega)$ be a vector bundle over $M \times (0, \omega)$, where $M$ is a smooth closed manifold, and let $U$ be a section $\Gamma(E \times (0, \omega)$.  Consider the following initial value problem:
	\begin{equation}
		\begin{cases}\label{eqn: nonlinear exist problem}
			P(U) := \p_t U - F(x,t,U, \nabla U, \ldots, \nabla^{2m}U) = 0 \text{ in } E \times (0, \omega) \\
			U(M, 0) = U_0,
		\end{cases}
	\end{equation}
with $U_0 \in C^{2m,1,\alpha}(E_{\omega})$.  The linearised operator of $P$ at $U_0$ in the direction $V$ is then given by
\begin{equation*}
\p P[U_0]V = \p_t V + (-1)^m\sum_{\abs{I} \leq 2m } A^I(x,t,U_0, \nabla U_0, \ldots, \nabla^{2m}U_0) \nabla_I V.
\end{equation*}
Suppose that the following conditions are satsified:
\begin{enumerate}
	\item  The leading coefficient $A_{b}^{a i_1 j_1 \cdots i_m j_m }$ satisfies the symmetry condition $A_{b}^{a i_1 j_1 \cdots i_m j_m } =  A_{ a }^{ b j_1 i_1 \cdots j_m i_m }$
	\item The leading coefficient satisfies the Legendre-Hadamard condition with constant $\lambda$
	\item There exists a uniform constant $\Lambda < \infty$ such that $\sum_{ \abs{I} \leq 2m } \abs{A^I}_{ \alpha; \, E_{\omega} } \leq \Lambda$
	\item $\Fdot$ is a continuous function of all its arguments
\end{enumerate}
Then there exists a unique solution $U \in C^{2m,1, \beta}(E_{\omega})$, where $\beta < \alpha$, for some short time $t_{\epsilon} > 0$ to the above initial value problem.  Furthermore, if $U_0$ and all the coefficients of the linearised operator are smooth, this solution is smooth.
\end{mthm}

Chapter 3 contains what is the main result of this thesis, which is a high codimension analogue of Huisken's original theorem on the flow of convex hypersurfaces:

\begin{mthm}\label{mthm: main thm 2}  Suppose $\Sigma_0=F_0(\Sigma^n)$ is a closed submanifold smoothly immersed in $\mathbb{R}^{n+k}$.  If $\Sigma_0$ satisfies $\abs{ H }_{ \text{min} } > 0$ and $\abs{h}^2 \leq c\abs{H}^2$, where
	\begin{equation*}
		c \leq \begin{cases}\frac{4}{3n},& \quad \text{ if }2\leq n \leq 4 \\
						\frac{1}{n-1},& \quad \text{ if }n \geq 4,
						\end{cases}
	\end{equation*}
then MCF has a unique smooth solution $F:\ \Sigma\times[0,T)\to \mathbb{R}^{n+k}$ on a finite maximal time interval, and the submanifolds $\Sigma_t$ converge uniformly to a point $q\in \mathbb{R}^{n+k}$ as $t\to T$.  A suitably normalised flow exists for all time, and the normalised submanifolds $\tilde{\Sigma}_{\tilde{t}}$ converge smoothly as $\tilde{t} \rightarrow \infty$ to a $n$-sphere in some $(n+k)$-subspace of $\mathbb{R}^{n+k}$.
\end{mthm}

As the following simple example shows, the pinching ratio in Main Theorem \ref{mthm: main thm 2} is optimal in dimensions greater than or equal to four. Consider the submanifolds $\mathbb{S}^{n-1}(\epsilon) \times \mathbb{S}^1(1)\subset \mathbb{R}^{n}\times\mathbb{R}^2$, where $\epsilon$ is a small positive number.  The second fundamental form is given by
	\begin{equation*}
		h\big|_{(\varepsilon x,y)} = \begin{pmatrix}
				\frac{1}{\epsilon} & & &\\
				& \ddots & & \\
				& & \frac{1}{\epsilon} & \\
				& & & 0
			\end{pmatrix}(x,0) + \begin{pmatrix}
						0 & & &\\
						& \ddots & & \\
						& & 0 & \\
						& & & 1
					\end{pmatrix}(0,y)
		\end{equation*}
and so they satisfy $\abs{h}^2 = \frac{1}{n-1}\left(1+\frac{\epsilon^2(n-2)}{(n-1)^2+\epsilon^2}\right)\abs{H}^2$.  These submanifolds collapse to $\mathbb{S}^1$ under the mean curvature flow and do not contract to points.  In dimensions two and three the size of the gradient and reaction terms of equation \eqref{e: h2-cH2} prevents the optimal result from being achieved.  This is similar to the situation in \cite{gH87}, where in dimension two the difficulty in controlling the gradient terms prevents the optimal result from being obtained.  We remark that contrary to the situation in \cite{gH87}, one cannot expect to obtain such a result with $c=1/(n-1)=1$ in the case $n=2$ in arbitrary codimension as the Veronese surface provides a counter-example:  This is a surface in $\mathbb{R}^5$ that satisfies $\abs{h}^2=\frac56\abs{H}^2$, but which contracts without changing shape under the mean curvature flow.  We are not aware of any such counter-examples in dimension three (there are none among minimal submanifolds of spheres \cite{CO73}).

Curvature pinching conditions similar to those in our theorem have appeared previously in a number of results for special classes of submanifolds:   In \cite{mO73} Okumura shows that if a submanifold of Euclidean space with parallel mean curvature vector and flat normal bundle satisfies $\abs{h}^2 < 1/(n-1) \abs{H}^2$, then the submanifold is a sphere. The equivalent result for hypersurfaces of the sphere with $\abs{h}^2<\frac{1}{n-1}\abs{H}^2+2$ (where the flat normal bundle condition is vacuous) was proved by Okumura in \cite{mO74}. Chen and Okumura \cite{CO73} later removed the assumption of flat normal bundle and so proved that if a submanifold of Euclidean space with parallel mean curvature vector satisfies  $\abs{h}^2 < 1/(n-1)\abs{H}^2$, then the submanifold is a sphere (or, in the case $n=2$, a minimal surface with positive intrinsic curvature in a sphere, such as the Veronese surface).  The broad structure of the proof of Main Theorem \ref{mthm: main thm 2} closely follows \cite{gH84}, which in turn, draws upon Hamilton's seminal paper on Ricci flow \cite{rH82}.

After presenting the case of a Euclidean background we progress to discuss the situation where the ambient space is a sphere of contant curvature $\bar{K}$.  We obtain the following theorem, which can likewise be considered a high codimension analogue of \cite{gH87}:

\begin{mthm}\label{mthm: main thm 3}
Suppose $\Sigma_0 = F_0(\Sigma)$ is a closed submanifold smoothly immersed in $\mathbb{S}^{n+k}$.  If $\Sigma_0$ satisfies
\begin{equation*}
	\begin{cases}
		\abs{h}^2 \leq \frac{4}{3n}\abs{H}^2 + \frac{2(n-1)}{3} \bar K, \quad n = 2,3 \\
		\abs{h}^2 \leq \frac{1}{n-1}\abs{H}^2 + 2\bar K, \quad n \geq 4,
	\end{cases}
\end{equation*}
then either
\begin{enumerate}
	\item MCF has a unique, smooth solution on a finite, maximal time interval $0 \leq t < T < \infty$ and the submanifolds $\Sigma_t$ contract to a point as $t \rightarrow T$; or
	\item MCF has a unique, smooth solution for all time $0 \leq t < \infty$ and the submanifolds $\Sigma_t$ converge to a totally geodesic submanifold $\Sigma_{\infty}$.
\end{enumerate}
\end{mthm}
The assumptions of Main Theorem 2  required that $\abs{ H }_{ \text{min} } > 0$.  In Main Theorem \ref{mthm: main thm 3} no assumption on the size of the mean curvature is made, so the initial submanifold could, for example, be minimal.  In this case the positive curvature of the background sphere still allows us to obtain convergence results.  For similar reasons to the Euclidean case, the second main theorem is also optimal in dimensions greater than and equal to four.

In the final chapter we follow Huisken's work in \cite{gH90a} and \cite{gH90b} to give a partial classification of type I singularities of the mean curvature flow in high codimension.  We pursue a slightly different blow-up argument than that used in \cite{gH90a} and \cite{gH90b}; in particular, we consider a sequence of parabolically rescaled flows rather than a continuous rescaling.  We also provide an alternate proof of the Breuning-Langer compactness theorem for immersed submanifolds of arbitrary codimension using the well-known Cheeger-Gromov compactness theorem.

\begin{mthm}\label{mthm: main thm 4}
Suppose $F_{\infty} : \Sigma_{\infty}^n \times (-\infty, 0) \rightarrow \mathbb{R}^{n+k}$ arises as the blow-up limit of the mean curvature flow $F: \Sigma^n \times [0, T) \rightarrow \mathbb{R}^{n+k}$ about a special singular point.  If $\Sigma_0$ satisfies $\abs{ H }_{ \text{min} } > 0$ and $\abs{ h }^2 \leq 4/(3n) \abs{ H }^2$, then at time $s=-1/2$, $F_{\infty}(\Sigma_{\infty})$ must be a sphere $\mathbb{S}^m(m)$ or one of the cylinders $\mathbb{S}^{m}(m) \times \mathbb{R}^{n-m}$, where $1 \leq m \leq n-1$.
\end{mthm}
We close out the last chapter by showing how a simple blow-up argument can be used instead of the convergence arguments of Section \ref{sec: The normalised flow and and convergence to the sphere} of Chapter \ref{ch: The flow of submanifolds of Euclidean space} to determine the limiting spherical shape.

The results of Chapters 2 and 4 appear in the paper `Mean curvature flow of pinched submanifolds to spheres', which is coauthored with the author's PhD supervisor, Ben Andrews.  This paper has been accepted to appear in the Journal of Differential Geometry.

\chapter{Submanifold geometry in high codimension}\label{ch: submanifold geometry in high codimension}

In order to work with the normal bundle we first discuss vector bundles, including pullback bundles and sub-bundles.  The machinery we develop is useful and new even in the codimension one case, as we work with the tangent and normal bundles as vector bundles over the space-time domain, and introduce natural
metrics and connections on these.  In particular, the connection we introduce on the `spatial' tangent bundle (as a bundle over spacetime) contains more information than the Levi-Civita connections of the metrics at each time, and this proves particularly useful in computing evolution equations for geometric quantities.

\section{Connections on vector bundles}
\subsection{Vector bundles}

We denote the space of smooth sections of a vector bundle $E$ by $\Gamma(E)$.   
If $E$ is a vector bundle over $N$, the dual bundle $E^*$ is the bundle whose fibres are the dual spaces of the fibres of $E$.  If $E_1$ and $E_2$ are vector bundles over $N$, the tensor product $E_1\otimes E_2$ is the vector bundle whose fibres are the tensor products $(E_1)_p\otimes (E_2)_p$.

\subsubsection{Metrics}

A metric $g$ on a vector bundle $E$ is a section of $E^*\otimes E^*$ which is an inner product on $E_p$ for each $p$ in $N$. A metric on $E$ defines a bundle isomorphism $\#_g$ from $E$ to $E^*$, defined by
$$
(\#_g(\xi))(\eta) = g(\xi,\eta)
$$
for all $\xi,\eta\in E_p$.  If $g$ is a metric on $E$, then there is a unique metric on $E^*$ (also denoted $g$) such that the identification $\#_g$ is a bundle isometry:  For all $\xi,\eta\in E_p$,
$$
g(\#_g(\xi),\#_g(v)) = g(\xi,\eta).
$$
If $g_i$ is a metric on $E_i$, $i=1,2$, then $g=g_1\otimes g_2\in\Gamma((E_1^*\otimes E_1^*)\otimes (E_2^*\otimes E_2^*))\simeq\Gamma((E_1\otimes E_2)^*\otimes(E_1\otimes E_2)^*)$ is the unique metric on $E_1\otimes E_2$ such that 
$g(\xi_1\otimes\eta_1,\xi_2\otimes\eta_2) = g_1(\xi_1,\xi_2)g_2(\eta_1,\eta_2)$.

\subsubsection{Connections}

A connection $\nabla$ on a vector bundle $E$ over $N$ is a map $\nabla:\ \Gamma(TN)\times\Gamma(E)\to\Gamma(E)$ which is $C^\infty(N)$-linear in the first argument and $\mathbb{R}$-linear in the second, and satisfies 
$$
\nabla_U(f\xi) = f\nabla_U\xi + (Uf)\xi
$$
for any $U\in\Gamma(TN)$, $\xi\in\Gamma(E)$, and $f\in C^\infty(N)$.  Here the notation $Uf$ means the derivative of $f$ in direction $U$.
Given a connection $\nabla$ on $E$, there is a unique connection on $E^*$ (also denoted $\nabla$) such that for all $\xi\in\Gamma(E)$, $\omega\in\Gamma(E^*)$, and $X\in \Gamma(TN)$,
\begin{equation}\label{eq:dual.connection}
X(\omega(\xi)) = (\nabla_X\omega)(\xi) + \omega(\nabla_X\xi).
\end{equation}
If $\nabla^i$ is a connection on $E_i$ for $i=1,2$, then there is a unique connection $\nabla$ on $E_1\otimes E_2$ such that 
\begin{equation}\label{eq:conn.tens.prod}
\nabla_X(\xi_1\otimes\xi_2) = (\nabla^1_X\xi_1)\otimes\xi_2 + \xi_1\otimes(\nabla^2_X\xi_2)
\end{equation}
for all $X\in\Gamma(TN)$, $\xi_i\in\Gamma(E_i)$.  In particular, for $S\in\Gamma(E_1^*\otimes E_2)$ (an $E_2$-valued tensor acting on $E_1$),  $\nabla S\in \Gamma(T_*N\otimes E_1^*\otimes E_2)$ is given by
\begin{equation}\label{eq:conn_on_tensor}
(\nabla_XS)(\xi) = \nabla^{E_2}_X(S(\xi)) - S ( \nabla^{E_1}_X\xi ).
\end{equation}
A connection $\nabla$ on $E$ is \emph{compatible} with a metric $g$ if for any $\xi,\eta\in\Gamma(E)$ and $X\in\Gamma(TN)$,
$$
Xg(\xi,\eta) = g(\nabla_X\xi,\eta) + g(\xi,\nabla_X\eta).
$$
If $\nabla$ is compatible with a metric $g$ on $E$, then the induced connection on $E^*$ is compatible with the induced metric on $E^*$.  Similarly, if $\nabla_i$ is a connection on $E_i$ compatible with a metric $g_i$ for $i=1,2$, then the metric $g_1\otimes g_2$ is compatible with the connection on $E_1\times E_2$ defined above.

Another important property of connections is that they are locally determined.
\begin{prop}\label{prop: connections locally defined}
Let $E$ be a vector bundle over $N$ and $p$ a point in $N$.  If $\xi_1$ and $\xi_2$ are two section of $E$ such that $\xi_1 = \xi_2$ on an open neighbourhood $U$ of $p$, then
$$
\nabla_X \xi_1(p) = \nabla_X \xi_2(p) 
$$
for all $X \in \Gamma(TN)$.
\end{prop}
\begin{proof}
It is obvious from the definition of a connection that the covariant derivative only depends on $X$ at the point $p$.  To show that it depends locally on $\xi$, let $\rho$ be a smooth cut-off function with support in $U$.  Then $\rho \xi_1 = \rho \xi_2$ on $U$ and hence $\nabla_X \rho \xi_1(p) = \nabla_X \rho \xi_2$.  Futhermore, the Leibniz property of a connection gives
$$
	\nabla_X \rho \xi_1(p) = (X\rho)(p)\xi_1(p) + \rho(p) \nabla_X \xi_1(p) =  \nabla_X \xi_1(p).
$$
The a same holds for $\nabla_X \rho \xi_2(p)$ too, thus $\nabla_X \xi_1(p) = \nabla_X \xi_2(p)$ as stated.
\end{proof}

\subsubsection{Curvature}

Let $E$ be a vector bundle over $N$.
If $\nabla$ is a connection on $E$, then the curvature of $\nabla$ is the section $R_\nabla\in \Gamma(T^*N\otimes T^*N\otimes E^*\otimes E)$ defined by
$$
R_\nabla(X,Y)\xi = \nabla_Y ( \nabla_X\xi ) - \nabla_X ( \nabla_Y\xi ) -\nabla_{[Y,X]}\xi.
$$

The curvature of the connection on $E^*$ given by Equation \eqref{eq:dual.connection} is characterized by the formula
$$
0 = (R(X,Y)\omega)(\xi) + \omega(R(X,Y)\xi)
$$
for all $X,Y\in\Gamma(TN)$, $\omega\in \Gamma(E^*)$ and $\xi\in\Gamma(E)$.

The curvature on a tensor product bundle (with connection defined by equation \eqref{eq:conn.tens.prod}) can be computed in terms of the curvatures of the factors by the formula
$$
R_{\nabla}(X,Y)(\xi_1\otimes\xi_2) = (R_{\nabla^1}(X,Y)\xi_1)\otimes \xi_2
+ \xi_1\otimes(R_{\nabla^2}(X,Y)\xi_2).
$$

In particular, the curvature on $E_1^*\otimes E_2$ ($E_2$-valued tensors acting on $E_1$) is given by
\begin{equation}\label{eq:curv_on_tensor}
(R(X,Y)S)(\xi) = R_{\nabla^2}(X,Y)(S(\xi)) - S(R_{\nabla^1}(X,Y)\xi).
\end{equation}

\subsection{Pullback bundles}

Let $M$ and $N$ be smooth manifolds, and let $E$ be a vector bundle over $N$ and $f$ a smooth map from $M$ to $N$.   Then $f^*E$ is the pullback bundle of $E$ over $M$, which is a vector bundle with fibre $(f^*E)_x = E_{f(x)}$.  If $\xi\in\Gamma(E)$, then we denote by $\xi_f$ the section of $f^*E$ defined by $\xi_f(x) = \xi(f(x))$ for each $x\in M$ (called the \emph{restriction} of $\xi$ to $f$).

The pull-back operation on vector bundles commutes with taking duals and tensor products, so the tensor bundles constructed from a vector bundle $E$ pull back to give the tensor bundles of the pull-back bundle $f^*E$.  In particular, if $g$ is a metric on $E$, then $g$ is a section of $E^*\otimes E^*$, and the restriction $g_f\in\Gamma(f^*(E^*\otimes E^*))\simeq \Gamma((f^*E)^*\otimes (f^*E)^*)$ defines a metric on $f^*E$.

\begin{prop}
If $\nabla$ is a connection on $E$, then there is a unique connection ${}^f\nabla$ on $f^*E$, called the \emph{pullback connection} which satisfies ${}^f\nabla_u(X_f) = \nabla_{f_*u}X$ for any $u\in TM$ and $X\in\Gamma(E)$.
\end{prop}
\begin{proof}
Suppose that $\xi$ is an arbitrary section $\xi \in \Gamma( f^*E )$ and $p \in M$.  Let $Z_i$ be a local frame for $E$ about $f(p)$.  The sections $ \{ Z_{i, f} \}$ thus form a local frame for $f^*E$ about $p$ so we can write  $\xi = \xi^i ((Z_i)_f$.  The properties of a connection and the pullback then give
\begin{align*}
	{^f}\nabla _v \xi &= {^f} \nabla_v (\xi^i (Z_i)_f ) \\
	&= v(\xi^i) (Z_i)_f + \xi^i {^f} \nabla_v (Z_i)_f \\
	&= v(\xi^i) (Z_i)_f + \xi^i \nabla_{f_*v} Z_i.
\end{align*}
A further computation shows that this is independent of the local frame used, and because connections are locally defined by Propostion \ref{prop: connections locally defined}, the pullback connection is well-defined.
\end{proof}
At first one might think that for the pullback connection to be well-defined, it would be necessary to extend the section $f_*v$ to a neighbourhood of $f(p)$ in order to operate on it locally.  The above proposition shows that this is not necessary, although in order to define the pull-back connection, we had to define it terms of a local frame.  Often in submanifold geometry the induced connection is defined in terms of a projection in $N$ onto the image of the tangent space of $M$.  This definition is frame independent, however it is necessary to extend the vector fields in order for the definition to make sense.  One can then show afterwards that the definition is independent of the extension used.

\begin{prop}
If $g$ is a metric on $E$ and $\nabla$ is a connection on $E$ compatible with $g$, then ${}^f\nabla$ is compatible with the restriction metric $g_f$.
\end{prop}

\begin{proof}
$\nabla$ is compatible with $g$ if and only if $\nabla g=0$.  We must therefore show that ${}^f\nabla g_f=0$ if $\nabla g=0$.  But this is immediate, since ${}^f\nabla_v(g_f) = \nabla_{f_*v}g =0$.
\end{proof}

\begin{prop}\label{prop:curvature}
The curvature of the pull-back connection is the pull-back of the curvature of the original connection.  Here $R_\nabla\in\Gamma(T^*N\otimes T^*N\otimes E^*\otimes E)$, so that
$$
f^*(R_\nabla)\in \Gamma(T^*M\otimes T^*M\otimes f^*(E^*\otimes E))
=\Gamma(T^*M\otimes T^*M\otimes (f^*E)^*\otimes f^*E).
$$
\end{prop}

\begin{proof}
Since curvature is tensorial, it is enough to check the formula for a basis.  Choose a local frame $\{Z_p\}_{p=1}^k$ for $E$.  Then $\{(Z_p)_f\}$ is a local frame for $f^*E$.  Choose local coordinates $\{ y^a \}$ for $N$ near $f(p)$ and $\{x^i\}$ for $M$ near $p$, and write $f^a=y^a\circ f$.  Then
\begin{align*}
R_{{}^f\nabla}(\partial_i,\partial_j)(Z_p)_f &= 
{}^f\nabla_{\partial_j}({}^f\nabla_{\partial_i}(Z_p)_f)-
(i\leftrightarrow j)\\
&={}^f\nabla_j(\nabla_{f_*\partial_i}(Z_p))-(i\leftrightarrow j)\\
&={}^f\nabla_j(\partial_if^a\nabla_a Z_p)-(i\leftrightarrow j)\\
&=(\partial_j\partial_if^a)\nabla_a Z_p + \partial_if^a{}^f\nabla_j((\nabla_a Z_p)_f) - (i\leftrightarrow j)\\
&=\partial_if^a\nabla_{f_*\partial_j}(\nabla_a Z_p)-(i\leftrightarrow j)\\
&=\partial_if^a\partial_jf^b(\nabla_b(\nabla_a Z_p)-(a\leftrightarrow b))\\
&=\partial_if^a\partial_jf^b R_\nabla(\partial_a,\partial_b)Z_p\\
&=R_\nabla(f_*\partial_i,f_*\partial_j)Z_p.
\end{align*}
\end{proof}

In the case of pulling back a tangent bundle, there is another important property:

\begin{prop}\label{prop:symmetry}
If $\nabla$ is a symmetric connection on $TN$, then the pull-back connection ${}^f\nabla$ on $f^*TN$ is symmetric, in the sense that for any $U,V\in \Gamma(TM)$, 
$$
{}^f\nabla_U(f_*V) - {}^f\nabla_V(f_*U) = f_*([U,V]).
$$
\end{prop}

\begin{proof}
Choose local coordinates ${x^i}$ for $M$ near $p$, and ${y^a}$ for $N$ near $f(p)$, and write $U=U^i\partial_i$ and $V=V^j\partial_j$.  Then
\begin{align*}
{}^f\nabla_U(f_*V)-(U\leftrightarrow V) &= {}^f\nabla_U(V^j\partial_jf^a\partial_a)-(U\leftrightarrow V)\\
&=U^i\partial_i(V^j\partial_jf^a)\partial_a + V^j(\partial_jf^a){}^f\nabla_U\partial_a-(U\leftrightarrow V)\\
&=(U^i\partial_iV^j-V^i\partial_iU^j)\partial_jf^a\partial_a + U^iV^j(\partial_i\partial_jf^a-\partial_j\partial_if^a)\partial_a\\
&\quad\null +V^jU^i\partial_jf^a\partial_if^b(\nabla_b\partial_a-\nabla_a\partial_b)\\
&=f_*([U,V]).
\end{align*}
\end{proof}

\subsection{Subbundles}

\newcommand{\nablaK}{\accentset{K}{\nabla}}
\newcommand{\nablaL}{\accentset{L}{\nabla}}

A \emph{subbundle} $K$ of a vector bundle $E$ over $M$ is a vector bundle $K$ over $M$ 
with an injective vector bundle homomorphism $\iota_K:\ K\to E$ covering the identity
map on $M$.
We consider complementary sub-bundles $K$ and $L$, so that $E_x=\iota_K(K_x)\oplus \iota_L(L_x)$, and denote by $\pi_K$ and $\pi_L$ the corresponding projections onto $K$ and $L$ (so $\pi_K\circ\iota_K=\text{Id}_K$, $\pi_L\circ\iota_L=\text{Id}_L$, $\pi_K\circ\iota_L=0$, $\pi_L\circ\iota_K=0$, and $\iota_K\circ\pi_K+\iota_L\circ\pi_L=\text{Id}_E$). If $\nabla$ is a connection on $E$, we define a connection $\nablaK$ on $K$ and a tensor $h^K\in\Gamma(T^*M\otimes K^*\otimes L)$ (the second fundamental form of $K$) by
\begin{equation}\label{eq:structure}
\nablaK_u\xi = \pi_K(\nabla_u(\iota_K\xi));\qquad h^K(u,\xi)=\pi_L(\nabla_u(\iota_K\xi));
\end{equation}
so that
\begin{equation}\label{eq:structure2}
\nabla_u(\iota_K\xi)=\iota_K(\nablaK_u\xi)+\iota_L(h^K(u,\xi))
\end{equation}
for any $u\in TM$ and $\xi\in\Gamma(K)\subset\Gamma(E)$.  The curvature $R^K$ of $\nablaK$ is related to the second fundamental form $h^K$ and the curvature of $\nabla$ via the Gauss equation:
\begin{equation}\label{eq:Gauss}
R^{K}(u,v)\xi = \pi_{K}(R_\nabla(u,v)(\iota_K\xi)) +h^{L}(u,h^K(v,\xi))-h^{L}(v,h^K(u,\xi))
\end{equation}
for all $u,v\in T_xM$ and $\xi\in\Gamma(K)$.  The other important identity relating the second fundamental form to the curvature is the Codazzi identity, which states:
\begin{equation}\label{eq:Codazzi}
\pi_{L}\!(R_\nabla(v,u)\iota_K\xi) = \nablaL_u\!(h^K(v,\xi)\!) - \nablaL_v\!(h^K(u,\xi)\!)-h^K\!(u,\nablaK_v\xi)+h^K\!(v,\nablaK_u\xi)-h^K([u,v],\xi).
\end{equation}
If we are supplied with an arbitrary symmetric connection on $TM$, then we can make sense of the covariant derivative $\nabla h^K$ of the second fundamental form $h^K$, and the Codazzi identity becomes
\begin{equation}\label{eq:Codazzi2}
\nabla_u h^K(v,\xi) - \nabla_v h^K(u,\xi) = \pi_{L}(R_\nabla(v,u)(\iota_K\xi)).
\end{equation}
An important case is where $K$ and $L$ are orthogonal with respect to a metric $g$ on $E$ compatible with $\nabla$.  Then $\nablaK$ is compatible with the induced metric $g^K$, and $h^K$ and $h^L$ are related by
\begin{equation}\label{eq:weingarten}
g^L(h^K(u,\xi),\eta) + g^K(\xi,h^{L}(u,\eta))=0
\end{equation}
for all $\xi\in\Gamma(K)$ and $\eta\in\Gamma(L)$. 

\section{The tangent and normal bundles of a time-dependent immersion}

The machinery introduced above is familiar in the following setting:  If $F:\ M^n\to N^{n+k}$ is an immersion, then $F_*:\ TM\to F^*TN$ defines the tangent sub-bundle of $F^*TN$, and its orthogonal complement is the normal bundle $NM=F_*(TM)^\perp$. If $\bar g$ is a metric on $TN$ with Levi-Civita connection $\bar\nabla$, then the metric $g^{TM}$ is the induced metric on $M$, and $\nabla^{TM}$ is its Levi-Civita connection, while $h^{TM}\in\Gamma(T^*M\otimes T^*M\otimes NM)$ is the second fundamental form, and $h^{NM}$ is minus the Weingarten map.  The Gauss identities \eqref{eq:Gauss} for $TM$ are the usual Gauss equations for a submanifold, while those for $NM$ are usually called the Ricci identities.  The Codazzi identities for the two are equivalent to each other.

Denote by $\pi$ the orthogonal projection from $F^*TN$ onto $TM$, by $\pip$ the orthogonal projection onto $NM$, and by $\iota$ the inclusion of $NM$ in $F^*TM$.  For $u$, $v \in TM$, equation \eqref{eq:structure} is exactly the usual Gauss relation:
\begin{equation*}
	{^F}\nabla_u F_* v = F_*( \nabla_u v ) + \iota h(u, v).
\end{equation*}
whilst for $\xi \in NM$ we recover the usual Weingarten relation:
\begin{equation*}
	{^F}\nabla_u F_* v = \iota ( \nablap_u \xi ) - F_*( \mathcal{W}(u, \xi) ).
\end{equation*}
At the moment, the right hand side of both of these relations is really just notation expressing the fact that we have the decomposition $F^*TN = F_*TM \oplus \iota NM$ into orthogonal sub-bundles.  We want to show, as the notation suggests, that the tangential component is the induced Levi-Civita connection on $M$, and that $h$ is a symmetric bilinear form.  Let $\alpha$ and $\beta$ be functions on $N$, then since $^F \nabla$ is a connection and $\alpha$ and $\beta$ restrict smoothly to functions on $M$,
\begin{gather*}
	{^F}\nabla_{ \alpha u } F_* v = \alpha ^F \nabla_u F_* v \\
	{^F}\nabla_{ u } \beta F_* v = (u\beta) F_*v + \beta ^F \nabla_u F_* v.
\end{gather*}
Therefore,
\begin{align*}
	{^F}\nabla_{ \alpha u } F_* v &= F_* ( \nabla_{ \alpha u } F_* v ) + \iota h( \alpha u, v ) \\
	&= \alpha F_* ( \nabla_u F_* v ) + \alpha \iota h( \alpha u, v ),
\end{align*}
and
\begin{align*}
	{^F}\nabla_{ u } \beta F_* v &= (u\beta)F_*v + \beta \big( F_*( \nabla_u F_*v ) + \iota h(u, v) \big) \\
	&= F_*( \nabla_u \beta F_*v ) + \iota h(u, \beta v).
\end{align*}
After projecting these equations onto $TM$ and $NM$ we get
\begin{align*}
	\nabla_{ \alpha u } F_*v = \alpha \nabla_u F_*v, \quad \nabla_u \beta F_v = (u\beta) F_*v + \beta \nabla_u F_*v \\
	h( \alpha u, v ) = \alpha h( u, v ), \quad h( u, \beta v ) = \beta h( u, v ).
\end{align*}
This shows that $\nabla := \pi \circ {^F}\bar{ \nabla } \circ F_*$ is indeed a connection on $M$, and that $h$ is bilinear.  Furthermore, since $^F \bar{ \nabla }$ is torsion-free and using Proposition \ref{prop:symmetry} we have
\begin{align*}
	0 &= {^F}\nabla_u F_*v + {^F}\nabla_u F_*v - [ F_*u, F_*v ] \\
	&= F_*( \nabla_u F_*v ) - F_*( \nabla_v F_*u ) - F_*( [ u, v ] ) + \iota h(u,v) - \iota h(v, u),
\end{align*}
which shows $\nabla$ is also torsion-free and $h$ is symmetric.  Finally, since $^F \nabla$ is metric-compatible, for $u, v, w \in \Gamma(TM)$,
\begin{align*}
	\nabla_w ( g( u, v ) ) &= {^F}\nabla_w ( \bar{g}( F_*u, F_*v ) ) \\
	&= \bar{g}( {^F}\nabla_w F_*u, F_*v ) + ( u \leftrightarrow v ) \\
	&= \bar{g} ( F_*( \nabla_w u ), F_*v ) + ( u \leftrightarrow v ) \\
	&= g( \nabla_w u, v ) +  g( u, \nabla_w v ),
\end{align*}
thus by uniqueness of the Levi-Civita connection, the induced connection $\nabla$ is the Levi-Civita connection on $M$.  Similar calculations show that  the Weingarten map is bilinear in both its arguments, and that $\nablap := \pip \circ {^F}\bar{ \nabla } \circ \iota$ is a metric compatible connection on the normal bundle.  Differentiating $\bar{g}( F_*u, \iota \xi ) = 0$ shows that $\gp( h(u, v), \xi ) = g( \mathcal{w}(u, \xi), v)$.  In local coordinates $\{x^i\}$ for $M$ near $p$  and $\{ y^a \}$ for $N$ near $f(p)$ the Gauss-Weingarten relations become
\begin{gather*}
	\frac{\partial^2 F^a}{\partial x^i \partial x^j} - \Gamma_{ij}^k \frac{ \p F^a }{ \p x^k } + \bar\Gamma_{cb}^a \frac{\partial F^c}{\partial x^j} \frac{\partial F^b}{\partial x^i}  = h_{ij}{^{\alpha}} \nu_{ \alpha }^a \\
	\frac{ \p \nu_{\alpha}^a }{ \p x^k } + \bar\Gamma_{cb}^a \frac{ \p F^b }{ \p x^k } \nu_{ \alpha }^c =  C_{k\alpha}^{\beta}\nu_{\beta}  - h_{kp}^{ \alpha } g^{pq} \frac{ \p F^a }{ \p x^q },
\end{gather*}
where $\Gamma_{ij}^k$ are the Christoffel symbols of the submanifold, $\Gamma_{cb}^a$ the Christoffel symbols of the ambient space, and $C_{ i \alpha }^{ \beta }$ the normal connection forms.  The Christoffel symbols of the ambient space are obviously zero if the background is flat, and the normal connection forms are zero if $M$ is a hypersurface.
 
In this thesis we want to apply the same machinery in a setting adapted to time-dependent immersions.  If $I$ is a real interval, then the tangent space $T(\Sigma\times I)$ splits into a direct product ${\mathcal H}\oplus {\mathbb{R}}\partial_t$, where ${\mathcal H}=\{u\in T(\Sigma\times I):\ dt(u)=0\}$ is the `spatial' tangent bundle.  We consider a smooth map $F:\ \Sigma^n\times I\to N^{n+k}$ which is a time-dependent immersion, i.e. for each $t\in I$, $F(.,t):\ \Sigma\to N$ is an immersion.  Then $F^*TN$ is a vector bundle over $\Sigma\times I$, which we can equip with the restriction metric $\bar g_F$ and pullback connection ${}^F\bar\nabla$ coming from a Riemannian metric $\bar g$ on $N$ and its Levi-Civita connection $\bar\nabla$.  
The map $F_*:\ {\mathcal H}\to F^*TN$ defines a sub-bundle of $F^*TN$ of rank $n$.  The orthogonal complement of $F_*({\mathcal H})$ in $F^*TN$ is a vector bundle of rank $k$ which we denote by ${\mathcal N}$ and refer to as the (spacetime) normal bundle. We denote by $\pi$ the orthogonal projection from $F^*TN$ onto ${\mathcal H}$, and by $\pip$ the orthogonal projection onto ${\mathcal N}$, and by $\iota$ the inclusion of ${\mathcal N}$ in $F^*TN$.  
 The restrictions of these bundles to each time $t$ are the usual tangent and normal bundles of the immersion $F_t$.  

The construction of the previous section gives a metric $g(u,v)=\bar g(F_*u,F_*v)$ and a connection $\nabla:=\pi\circ{}^F\bar\nabla\circ F_*$ on the bundle ${\mathcal H}$ over $\Sigma\times I$, which agrees with the Levi-Civita connection of $g$ for each fixed $t$.   
We denote by $\gp$ the metric induced on ${\mathcal N}$, given by $\gp(\xi,\eta)=\bar g(\iota\xi,\iota\eta)$.
The construction also gives a connection
$\nablap:=\pip\circ{}^F\bar\nabla\circ\iota$ on ${\mathcal N}$.  We denote by 
$h\in \Gamma({\mathcal H}^*\otimes{\mathcal H}^*\otimes{\mathcal N})$ the restriction of $h^{\mathcal H}=\pip\circ{}^F\nabla\circ F_*$ to ${\mathcal H}$ in the first argument.  Proposition \ref{prop:symmetry} implies that $h$ is a symmetric bilinear form on ${\mathcal H}$ with values in ${\mathcal N}$.  The remaining components of $h^{\mathcal H}$ are given by
\begin{align}\label{eq:h(i,t)}
h^{\mathcal H}(\partial_t,v)&= \pip({}^F\nabla_tF_*v)\notag\\
&=\pip({}^F\nabla_vF_*\partial_t + F_*([\partial_t,v])\notag\\
&=\nablap_v(\pip F_*\partial_t)+h(v,\pi F_*\partial_t)
\end{align}
where we used Proposition \ref{prop:symmetry}.  Henceforward we restrict to normal variations (with $\pi F_*\partial_t=0$), since this is the situation for the mean curvature flow.  We also define ${\mathcal W}\in\Gamma({\mathcal H}^*\otimes {\mathcal N}\otimes{\mathcal H})$ by
$
{\mathcal W}(u,\xi) = -h^{\mathcal N}(u,\xi)=-\pi({}^F\nabla_u\iota\xi)
$
for any $u\in\Gamma({\mathcal H})$ and $\xi\in\Gamma({\mathcal N})$ (we refer to this as the Weingarten map).  The Weingarten relation \eqref{eq:weingarten} gives two identities:  
\begin{align}
\gp(h(u,v),\xi) &= g(v,{\mathcal W}(u,\xi));\label{eq:weingarten2}\\
g(h^{\mathcal N}(\partial_t,\xi),v) &= -\gp(\nablap_v\pip F_*\partial_t,\xi)\label{eq:weingarten3}
\end{align}
where the latter identity used \eqref{eq:h(i,t)}.
The Gauss and Codazzi identities for ${\mathcal H}$ and ${\mathcal N}$ give the following identities for the second fundamental form:  First, if $u$ and $v$ are in ${\mathcal H}$, then the Gauss equation \eqref{eq:Gauss} for ${\mathcal H}$ amounts to the usual Gauss equation at the fixed time, i.e.
\begin{subequations}
\begin{align}
R(u,v)w &= {\mathcal W}(v,h(u,w))-{\mathcal W}(u,h(v,w))+\pi(\bar R(F_*u,F_*v)F_*w)\label{eq:spatialGauss}\\
R(u,v,w,z) &= \gp (h(u,w),h(v,z)) - \gp(h(v,w),h(u,z)) + F^*\bar R(u,v,w,z).\label{eq:spatialGauss2}
\end{align}
\end{subequations}
If $u=\partial_t$ but $v\in{\mathcal H}$, then we find:
\begin{equation}
R(\partial_t,v,w,z) ={\gp}(\nablap_w\pip F_*\partial_t, h(v,z))-{\gp}(\nablap_z\pip F_*\partial_t,h(v,w))+ F^*\bar R(\partial_t,v,w,z).\label{eq:timelikeGauss}\\
\end{equation}

The Gauss equation for the curvature $\Rp$ of ${\mathcal N}$ also splits into two parts:  If $u$ and $v$ are spatial these are simply the Ricci identities for the submanifold at a fixed time:
\begin{equation}\label{eq:Ricci}
\Rp(u,v)\xi = h(v,{\mathcal W}(u,\xi))-h(u,{\mathcal W}(v,\xi))
+\pip(\bar R(F_*u,F_*v)(\iota\xi));
\end{equation}
while if $u=\partial_t$ and $v\in{\mathcal H}$, then we have the identity
\begin{equation}\label{eq:timelikeRicci}
\Rp(\partial_t,v,\xi,\eta) = \bar R(F_*\partial_t,F_*v,\iota\xi,\iota\eta) 
- \gp(\nablap_{{\mathcal W}(v,\xi)}\pip F_*\partial_t,\eta)
+\gp(\nablap_{{\mathcal W}(v,\eta)}\pip F_*\partial_t,\xi).
\end{equation}
Finally, the Codazzi identities resolve into the tangential Codazzi identities, given by
\begin{equation}\label{eq:spatialCodazzi}
\nabla_uh(v,w) - \nabla_vh(u,w) = \pip(\bar R(F_*v,F_*u)F_*w)
\end{equation}
for all $u,v,w\in\Gamma({\mathcal H})$,
and the `timelike' part, where $u=\partial_t$ and $v,w\in\Gamma({\mathcal H})$:
\begin{equation}\label{eq:timelikeCodazzi}
\pip(\bar R(F_*v,F_*\partial_t)F_*w) = \nabla_{\partial_t}h(v,w)-\nabla_v\nabla_w(\pip F_*\partial_t)-h(w,{\mathcal W}(v,\pip F_*\partial_t)).
\end{equation}
Note that here $\nabla h\in\Gamma(T^*(\Sigma\times I)\otimes{\mathcal H}^*\otimes{\mathcal H}^*\otimes{\mathcal N})$ is defined using the connections $\nabla$ and $\nablap$ as in Equation \eqref{eq:conn_on_tensor}, that is
$\nabla_{\partial_t}h(u,v) = \nablap_{\partial_t}(h(u,v))-h(\nabla_{\partial_t}u,v)-h(u,\nabla_{\partial_t}v)$.

We remark that by construction we have $\nabla g=0$ and $\nabla\gp=0$.  In contrast to the situation in other work on evolving hypersurfaces, we have $\nabla_{\partial_t} g=0$.  That is, the connections we have constructed automatically build in the so-called `Uhlenbeck trick' \cite{rH86}*{Section 2}.

\begin{prop} The tensors
$F_*\in\Gamma({\mathcal H}^*\otimes F^*TN)$, $\iota\in\Gamma({\mathcal N}^*\otimes F^*TN)$, $\pi\in\Gamma(F^*TN\otimes{\mathcal H})$ and $\pip\in\Gamma(F^*TN\otimes{\mathcal N})$ satisfy
\begin{align}
(\nabla_U F_*)(V) &= \iota h(U,V)\label{eq:nablaF*}\\
(\nabla_U\iota)(\xi) &= -F_*{\mathcal W}(U,\xi)\label{eq:nabla.iota}\\
(\nabla_U\pi)(X) &= {\mathcal W}(U,\pip X)\label{eq:nablapiT}\\
(\nabla_U\pip)(X) &= -h(U,\pi X)\label{eq:nablapiN}
\end{align}
for all $U,V\in\Gamma({\mathcal H})$, $\xi\in\Gamma({\mathcal N})$ and $X\in\Gamma(F^*TN)$.
\end{prop}

\begin{proof}
These follow from our construction and Equation \eqref{eq:conn_on_tensor}:  For the first we have (since $F_*$ is a $F^*TN$-valued tensor acting on ${\mathcal H}$)
$$
(\nabla_U F_*)(V) = {}^F\nabla_U(F_*V)-F_*(\nabla_UV)
=F^*(\nabla_UV)+\iota h(U,V)-F_*(\nabla_UV)=\iota h(U,V),
$$
where we used the definitions of $h$ and $\nabla$.  The second identity is similar.  For the third we have:
\begin{align*}
(\nabla_U\pi)(X)&= \nabla_U(\pi X)
-\pi({}^F\nabla_UX)\\
&=\nabla_U(\pi X)
-\pi ({}^F\nabla_U(F_*\pi X+\iota \pip X))\\
&=\nabla_U(\pi X) - 
	\nabla_U(\pi X)+{\mathcal W}(U,\pip X)\\
&={\mathcal W}(U,\pip X).
\end{align*}
The fourth identity is similar to the third.
\end{proof}

We illustrate the application of the above identities in the proof of Simons' identity, which amounts to the statement that the second derivatives of the second fundamental form are totally symmetric, up to corrections involving second fundamental form and the curvature of $N$:
\begin{prop}\label{prop:Simons}
\begin{align*}
\nabla_w\nabla_zh(u,v)
-\nabla_u\nabla_vh(w,z)
&=h(v,{\mathcal W}(u,h(w,z)))-h(z,{\mathcal W}(w,h(u,v)))-h(u,{\mathcal W}(w,h(v,z)))\\
&\quad\null
+h(w,{\mathcal W}(u,h(v,z)))
+h(z,{\mathcal W}(u,h(w,v)))-h(v,{\mathcal W}(w,h(u,z)))\\
&\quad\null
-h(u,\pi \bar R(F_*v,F_*w)F_*z)
-h(w,\pi \bar R(F_*u,F_*z)F_*v)\\
&\quad\null
-h(z,\pi \bar R(F_*u,F_*w)F_*v)
-h(v,\pi \bar R(F_*u,F_*w)F_*z)\\
&\quad\null+\pip\bar R(\iota h(u,v),F_*w)F_*z-\pip\bar R(\iota h(w,z),F_*u)F_*v\\
&\quad\null
+\pip\bar R(F_*u,F_*w)\iota h(v,z)+\pip\bar R(F_*v,F_*z)\iota h(u,w)\\
&\quad\null+\pip\bar R(F_*v,F_*w)\iota h(u,z)+\pip\bar R(F_*u,F_*z)\iota h(v,w)\\
&\quad\null
+\pip \bar \nabla_{F_*u} \bar R(F_*v,F_*w)F_*z-\pip\bar \nabla_{F_*w} \bar R(F_*z,F_*u)F_*v.
\end{align*}
\end{prop}

\begin{proof}
Since the equation is tensorial, it suffices to work with $u,v,w,z\in\Gamma({\mathcal H})$ for which $\nabla u=0$, etc, at a given point.  Computing at that point we find
\begin{align*}
\nabla_w\!\nabla_zh(u,v)\!&=\nabla_w(\nabla_uh(z,v)+\pip \bar R(F_*u,F_*z)F_*v)\\
&=\nabla_u\nabla_wh(v,z)+(R(u,w)h)(v,z)+\nabla_w (\pip \bar R(F_*u,F_*z)F_*v)\\
&=\nabla_u\!\!(\!\nabla_vh(w,z)+\pip\bar R(F_*v,F_*w)F_*z\!)+(R(u,w)h)(v,z)+\nabla_w\!\! (\!\pip \bar R(F_*u,F_*z)F_*v\!)\\
&=\nabla_u\!\nabla_vh(w,z)\!+\!(R(u,w)h)(v,z)\!+\!\nabla_w \!\!(\!\pip \bar R(F_*u,F_*z)F_*v\!)\!+\!\nabla_u \!\!(\!\pip \bar R(F_*v,F_*w)F_*z\!)
\end{align*}
where we used the Codazzi identity in the first and third lines, and the definition of curvature in the second.  Since $h$ is a ${\mathcal N}$-valued tensor with arguments in ${\mathcal H}$, the second term may be computed using the identity \eqref{eq:curv_on_tensor} to give
$$
(R(u,w)h)(v,z) = \Rp(u,w)(h(v,z))-h(R(u,w)v,z)-h(v,R(u,w)z).
$$
This in turn can be expanded using the Gauss identity \eqref{eq:spatialGauss} for $R$ and the Ricci identity \eqref{eq:Ricci} for $\Rp$.  In the third term (and similarly the fourth) we apply the identity \eqref{eq:conn_on_tensor} to $\pip$:
\begin{equation*}
\nabla_w(\pip \bar R(F_*u,F_*z)F_*v)
=\nabla_w\pip(\bar R(F_*u,F_*z)F_*v) + \pip({}^F\nabla_w(\bar R(F_*u,F_*z)F_*v)).
\end{equation*}
In the first term here we apply the identity \eqref{eq:nablapiN}.  In the second we can expand further as follows:
\begin{align*}
{}^F\nabla_w(\bar R(F_*u,F_*z)F_*v) &= ({}^F\nabla_w\bar R)(F_*u,F_*z)F_*v 
+\bar R((\nabla_wF_*)u,F_*z)F_*v\\
&\quad\null+\bar R(F_*u,\nabla_wF_*(z))F_*v+\bar R(F_*u,F_*z)(\nabla_uF_*(v)).
\end{align*}
In the terms involving $\nabla F_*$ we apply \eqref{eq:nablaF*}, and we also observe that 
${}^F\nabla_w \bar R = \bar \nabla_{F_*w}\bar R$ by the definition of the connection ${}^F\nabla$.  Substituting these identities gives the required result.
\end{proof}

In subsequent computations we often work in a local orthonormal frame $\{e_i\}$ for the spatial tangent bundle ${\mathcal H}$, and a local orthonormal frame $\{\nu_\alpha\}$ for the normal bundle ${\mathcal N}$.  We use greek indices for the normal bundle, and latin ones for the tangent bundle.  When working in such orthonormal frames we sum over repeated indices whether raised or lowered.  For example
the mean curvature vector $H\in\Gamma({\mathcal N})$ may be written in the various forms 
$$
H=\tr_gh = g^{ij}h_{ij}={h_i}^i = h_{ii} = g^{ij}{h_{ij}}^\alpha\nu_\alpha=h_{ii\alpha}\nu_\alpha.
$$
Similarly, we write $|h|^2 = g^{ik}g^{jl}g^{\mathcal N}_{\alpha\beta}{h_{ij}}^\alpha{h_{kl}}^\beta=h_{ij\alpha}h_{ij\alpha}$.  The Weingarten relation \eqref{eq:weingarten2} becomes 
$$
{\mathcal W}(e_i,\nu_\alpha) = h_{iq\alpha} e_q,
$$
while the Gauss equation \eqref{eq:spatialGauss} becomes
$$
R_{ijkl} =h_{ik\alpha} h_{jl\alpha}-h_{jk\alpha} h_{il\alpha}
+\bar R_{ijkl},
$$
where we denote $\bar R_{ijkl}=\bar R(F_*e_i,F_*e_j,F_*e_k,F_*e_l)$.
The Ricci equations \eqref{eq:Ricci} give
$$
\Rp_{ij\alpha\beta} = h_{ip\alpha}h_{jp\beta}-h_{jp\alpha}h_{ip\beta}+\bar R_{ij\alpha\beta},
$$
where $\bar R_{ij\alpha\beta}=\bar R(F_*e_i,F_*e_j,\iota\nu_\alpha,\iota\nu_\beta)$, and the Codazzi identity \eqref{eq:spatialCodazzi} gives
$$
\nabla_ih_{jk}-\nabla_jh_{ik} = \bar R_{jik\alpha}\nu_\alpha.
$$
In this notation the identity from Proposition \ref{prop:Simons} takes the following form:
\begin{align*}
\nabla_k\nabla_lh_{ij}&=
\nabla_i\nabla_jh_{kl}+h_{kl\alpha}h_{ip\alpha}h_{jp}-h_{ij\alpha}h_{kp\alpha}h_{lp}\\
&\quad\null
+h_{jl\alpha}h_{ip\alpha}h_{kp}+h_{jk\alpha}h_{ip\alpha}h_{lp}-h_{il\alpha}h_{kp\alpha}h_{jp}
-h_{jl\alpha}h_{kp\alpha}h_{ip}\\
&\quad\null +h_{kl\alpha}\bar R_{i\alpha j\beta}\nu_\beta-h_{ij\alpha}\bar R_{k\alpha l\beta}\nu_\beta+\bar R_{kjlp}h_{ip}+\bar R_{kilp}h_{jp}-\bar R_{iljp}h_{kp}-\bar R_{ikjp}h_{lp}\\
&\quad\null+h_{jl\alpha}\bar R_{ik\alpha\beta}\nu_\beta+h_{ik\alpha}\bar R_{jl\alpha\beta}\nu_\beta
+h_{il\alpha}\bar R_{jk\alpha\beta}\nu_\beta+h_{jk\alpha}\bar R_{il\alpha\beta}\nu_\beta\\
&\quad\null +\bar\nabla_i\bar R_{jkl\beta}\nu_\beta - \bar\nabla_k\bar R_{lij\beta}\nu_\beta.
\end{align*}
Particularly useful is the equation obtained by taking a trace of the above identity over $k$ and $l$:  
\begin{align}
\Delta h_{ij}&=\nabla_i\nabla_jH+H\cdot h_{ip}h_{pj}-h_{ij}\cdot h_{pq}h_{pq}+2h_{jq}\cdot h_{ip}h_{pq}
-h_{iq}\cdot h_{qp}h_{pj}-h_{jq}\cdot h_{qp}h_{pi}\notag\\
&\quad\null+H_\alpha\bar R_{i\alpha j\beta}\nu_\beta-h_{ij\alpha}\bar R_{k\alpha k\beta}\nu_\beta
+\bar R_{kjkp}h_{pi}+\bar R_{kikp}h_{pj}-2\bar R_{ipjq}h_{pq}\notag\\
&\quad\null +2h_{jp\alpha}\bar R_{ip\alpha\beta}\nu_\beta+2h_{ip\alpha}\bar R_{jp\alpha\beta}\nu_\beta+\bar\nabla_i\bar R_{jkk\beta}\nu_\beta-\bar\nabla_k\bar R_{kij\beta}\nu_\beta.\label{eq:SimonsTrace}
\end{align}
Here the dots represent inner products in ${\mathcal N}$.

\section{The method of moving frames}
The Gauss, Ricci and Codazzi equations can also be quite nicely derived using Cartan's method of moving frames.  We shall only need to use this machinery once, and then only briefly, in Chapter 6, but it is nonetheless instructive to see how this can be done.  We work in the setting of a fixed immersion $F:\ M^n\to N^{n+k}$, and let $\{ e_a : 1 \leq a \leq n+k \}$ be an adapted local frame for $N$, so that $\{ e_i : 0 \leq i \leq n \}$ are tangent to $M$, and $\{ e_{ \alpha } : n+1 \leq a \leq n+k \}$ are normal to $M$.  In the Cartan formalism, the Levi-Civita connection on $N$ is given by the structure equations
\begin{gather*}
	d\omega^a = -\omega^a_b \wedge \omega^b \\ 
	\omega_b^a + \omega_a^b = 0.
\end{gather*}
The first structure equation determines a torsion-free connection, and the second guarantees the connection is metric-compatible.  The curvature of the connection is given by the structure equation
\begin{equation}
	d\omega^a_b = -\omega^a_c \wedge \omega^c_b + \bar\Omega_b^a. \label{eqn: Cartan structure 3 N}
\end{equation}
We now restrict the indices to $M$, so $\omega^{ \alpha } = 0$.  From the above structure equations we obtain
\begin{gather}
	d\omega^i = -\omega^i_j \wedge \omega^j \label{eqn: Cartan structure 1} \\ 
	\omega_j^i + \omega_i^j = 0 \label{eqn: Cartan structure 2} \\
	d\omega^i_j = -\omega^i_k \wedge \omega^k_j + \Omega_j^i  \label{eqn: Cartan structure 3} 
\end{gather}
which determines the Levi-Civita connection on $M$.  Furthermore, since $\omega^i = 0$ in the normal bundle, we also obtain
\begin{gather*}
	d\omega^{ \alpha } = -\omega^{ \alpha }_{ \beta } \wedge \omega^{ \beta } \\ 
	\omega_{ \beta }^{ \alpha } + \omega_{ \alpha }^{ \beta } = 0 \\
	d\omega^{ \alpha }_{ \gamma } = -\omega^{ \alpha }_{ \gamma } \wedge \omega^{ \gamma }_{ \beta } + \Omega_{ \beta }^{ \alpha }
\end{gather*}
which determines a metric compatible connection on the normal bundle.
Let us show how to derive only the Codazzi equation.  Since $0 = d\omega^{ \alpha } = -\omega_i^{ \alpha } \wedge \omega^j$, by Cartan's Lemma the $\omega_i^{ \alpha }$ can be expressed as a linear combination of the $\omega^j$: $\omega_i^{ \alpha } = h_{ij}{^\alpha} \omega^j$, and  also $h_{ij}{^\alpha} = h_{ji}{^\alpha}$.  In equation \eqref{eqn: Cartan structure 3 N} we restrict $a$ to $\alpha$ and $b$ to $i$ to get
\begin{equation}
	d\omega_i^{ \alpha } + \omega_j^{ \alpha } \wedge \omega_i^j + \omega_{ \beta }^{ \alpha } \wedge \omega_i^{ \beta } = \bar \Omega_i^{ \alpha }. \label{eqn: Codazzi moving frames 1}
\end{equation}
Exterior differentiation of $\omega_i^{ \alpha } = h_{il}{^{\alpha}} \omega^l$ gives $dh_{il}{^{\alpha}} \wedge \omega^l - h_{ij}{^{\alpha}}\omega_l^j \wedge \omega^l$,
and putting this together with \eqref{eqn: Codazzi moving frames 1} we obtain
$$
	( dh_{ij}{^{\alpha}} - h_{lj}{^{\alpha}}\omega_i^l - h_{il}{^{\alpha}}\omega_j^l + h_{ij}{^{\beta}}\omega_{ \beta }^{ \alpha } ) \wedge \omega^l = \bar \Omega_i^{ \alpha }.
$$
We define
$$
	h_{ijk}{^{ \alpha }} \omega^k = dh_{ij}{^{\alpha}} - h_{lj}{^{\alpha}}\omega_i^l - h_{il}{^{\alpha}}\omega_j^l + h_{ij}{^{\beta}}\omega_{ \beta }^{ \alpha },
$$
which is just the first covariant derivative of $h$, and assuming a flat background, we recover the usual Codazzi equation: $h_{ijk}{^{\alpha}} = h_{ikj}{^{\alpha}}$.

\chapter{Short-time existence theory}

The mean curvature flow equation determines a weakly parabolic quasilinear system of second order.  It is now well-known that many geometrically-defined partial differential equations possess zeroes in their principal symbol because of some kind of geometric invariance displayed by the equations. In order to assert short-time existence to the mean curvature flow we use the well-known `DeTurck trick' to first solve a related strongly parabolic equation, and we then recover a solution to the mean curvature flow from this related solution.  The DeTurck trick was first invented to solve the Ricci flow, however the method applies to many other geometric flows.  In his lecture notes \cite{rH89}, Hamilton shows how the DeTurck trick can be applied to the Ricci, mean curvature and Yang-Mills flows.  We have also seized this opportunity to fill in a few details in the proof of short-time existence for fully nonlinear parabolic systems of even order defined on a manifold.  In \cite{rH75} Hamilton gives a proof of local existence for the harmonic map heat flow (a strongly parabolic quasilinear system) using Sobolev spaces and the inverse function theorem, and we were inspired to adapt his proof to fully nonlinear systems in the H\"{o}lder space setting.  Towards the end of this task Tobias Lamm pointed out to us the he proved the short-time existence of fully nonlinear operators in Euclidean space using Schauder estimates in his Diploma Thesis \cite{tL}.  Given that we started reconstructing this theory on our own, we have still decided to include this chapter.  Since we are not claiming anything essentially new in this chapter, we have freely borrowed from Lamm's thesis to improve our own exposition.  In particular, we now use Simon's method of scaling to derive the Schauder estimates for parabolic systems, as opposed to Trudinger's method of mollification which we had originally used.  We still show how Trudinger's method can be combined with the mean value property of subsolutions to the heat equation to provide a remarkably simple proof of the Schauder estimates in the case of single equations.  We emphasise the global aspects of solving the problem more than Lamm, and we work in the setting of parabolic systems defined in sections of a vector bundle over a closed manifold.

The strategy for proving such an existence theorem is well-known: one begins with a solution to the heat equation and then uses the method of continuity and the Schauder estimates to prove existence for general linear operators.  The short-time nonlinear existence result then follows by linearising the nonlinear operator and applying the inverse function theorem.  In reconstructing the $L^2$ and linear theory, our main reference has been the Chinese text \cite{Gu}.  As we have mentioned above we use Trudinger's method of mollification to derive the interior Schauder estimate for second order parabolic equations, and we simply cite Lamm's thesis for the derivation of the Schauder estimates for even order parabolic systems in Euclidean space.  The application of the inverse function theorem to yield the nonlinear existence result was inspired by \cite{rH75}.  Our method is different to Lamm's, however in showing the solution is an appropriate H\"{o}lder space with exponent $\beta$, for $\beta < \alpha$ where $\alpha$ the H\"{o}lder exponent of the initial data, we have benefited from \cite{tL}.  The application of the DeTurck trick to the mean curvature flow first appears in \cite{rH89}, and we have simply expanded on these notes of Hamilton's, adding in a few calculations.  Combining the harmonic map heat flow with the mean curvature flow to give a simple proof of uniqueness of the mean curvature flow is also due to Hamilton; the equivalent result for the Ricci flow first appeared in \cite{rH93b}.  We adapt the Ricci flow result to the mean curvature flow, following the detailed expositions given in \cite{HA, CLN}.

\section[Short-time existence for fully nonlinear parabolic systems]{Short-time existence for fully nonlinear parabolic systems of even order}

In this section we give a proof of short-time existence for fully nonlinear parabolic systems of even order.  The nonlinear existence result that is our ultimate goal is attained by an application of the classical inverse function theorem in Banach spaces, and is in fact quite short once we have all the linear theory in place.  The linear theory plays an essential role in the nonlinear theory and most of the following is devoted to establishing the linear theory.

The setting for our study of systems of partial differential equations defined on a manifold is slightly different to that of the more familiar Euclidean case.  Here we are interested in differential operators that act on sections of a vector bundle over a manifold, and not simply functions defined on some domain of Euclidean space.  So that the reader can accustom to this setting, let us first consider linear systems of second order.  Let $E$ and $F$ be two vector bundles over $M$.  Let the indices $a,b,c, \ldots$ range from $1$ to $N$, and the indices $i,j,k, \ldots$ range from $1$ to $n$.  Suppose $\{ e_a \}$ is a local frame for the bundle $E$ over a coordinate neighbourhood of $M$ with local coordinates $\{ x^i \}$.  It may be helpful to keep in mind the specific example of the mean curvature flow, in which case the section we are interested in the the position vector of the submanifold, $N$ is dimension of the background space and $n$ the dimension of the submanifold.   A linear differential operator of second order is a map $L : \Gamma(E) \rightarrow \Gamma(F)$ which in any coordinate chart is of the form
\begin{equation*}
	L(U) := \frac{ \p }{ \p t }U^a -A_b^{ a i j }(x,t)\p_i \p_j U^b - B_b^{ a k }(x,t) \p_k U^b - C_b^a(x,t)U^b,
\end{equation*}
where $A \in \Gamma( \text{Sym}^2(T^*M) \otimes E \otimes E^* )$, $B \in \Gamma ( T^*M \otimes E \otimes E^* )$ and $C \in \Gamma(E \times E^*)$.  The notion of parabolicity is defined in terms of the principal symbol of the differential operator.  The principal symbol $\hat\sigma$ of the above linear operator $L$ in direction $\xi \in \Gamma(TM)$ is the vector bundle homomorphism
\begin{equation*}
	\hat\sigma[L](\xi) := A_b^{ a i j }\xi_i \xi_j e_a \otimes e^{* \, b}.
\end{equation*}
The system is said to be strongly (weakly) parabolic if the eigenvalues of the principal symbol are positive (non-negative).  The principal symbol encodes algebraically the analytic properties of the leading term of the differential operator.  A linear differential operator of order $2m$ is a map $L : \Gamma(E) \rightarrow \Gamma(F)$ which in any local coordinate chart is of the form
\begin{equation}\label{e: even order linear operator}
	\frac{ \p }{ \p t }U^a + (-1)^m \sum_{ I \leq 2m } A^I \p_I U,
\end{equation}
or in full
\begin{equation*}
	\frac{ \p }{ \p t }U^a + (-1)^m \big( A_b^{a \, i_1, \ldots i_{2m} }(x,t)\p_{i_1} \cdots \p_{i_{2m}} U^b + \cdots + B_b^{ a \, k }(x,t)\p_k U^b + C_b^a(x,t)U^b \big) .
\end{equation*}
The principal symbol is defined in a similar manner as before.  Let us now move on to fully nonlinear differential operators.  A fully nonlinear differential operator of order $2m$ is a map $L : \Gamma(E) \rightarrow \Gamma(F)$ which in any coordinate chart is of the form
\begin{equation}\label{eqn: fully nonlinear 1}
	\frac{ \p }{ \p t }U^a - F^a(x,t, U^{b_1}, \p U^{b_2}, \ldots, \p^{2m} U^{ b_{2m} } ),
\end{equation}
or equivalently, is globally of the form
\begin{equation}
	\frac{ \p }{ \p t }U^a - F^a(x,t, U^{b_1}, \nabla U^{b_2}, \ldots, \nabla^{2m} U^{ b_{2m} } ).
\end{equation}
We will often drop the indices relating to the section, and simply write the above equation as
\begin{equation}\label{e: fully nonlinear operator}
	\frac{ \p }{ \p t }U - F(x,t, U, \nabla U, \ldots, \nabla^{2m} U).
\end{equation}
We have chosen to categorise the equations (linear, quasilinear, etc) in terms of their form in a local chart.  This is important for operators that are not fully nonlinear, because if the connection is varying in time, the initial global appearance of the equation can be betraying.  As we shall see in the next section, this is indeed the case for the mean curvature flow.  This difference is immaterial for fully nonlinear operators, and so nothing is lost by saying such an operator looks globally of the form \eqref{e: fully nonlinear operator}.  We mention that as an alternative to passing to local coordinate descriptions, we could define the equations globally with respect to a fixed connection, which then reveals the true nature of the equations.

The notion of parabolicity for a nonlinear operator is defined in terms of its linearised operator.  The linearisation of a nonlinear operator $L$ about some fixed function $U_0$ in the direction $V$ is the linear operator given by
\begin{align*}
	\p L[U_0](V) &=  \frac{ \p }{ \p s } L(U_0 + sV) \Big|_{s=0} \\
	&= \p_t V - \Fdot^{ i_1, \ldots, i_{2m} }( x, t, U_0, \nabla U_0, \ldots, \nabla^{2m} U_0) \nabla_{ i_1 } \cdots \nabla_{ i_{2m} }V \\
	&\quad - \cdots - \Fdot^k(x, t, U_0, \nabla U_0, \ldots, \nabla^{2m} U_0) \nabla_k V - F( x, t, U_0, \nabla U_0, \ldots, \nabla^{2m} U_0 )V.
\end{align*}
A fully nonlinear operator is said to be parabolic if its linearisation evaluated at the initial time is parabolic.  The principal symbol of the linearised operator evaluated at the initial time is
\begin{equation*}
	\hat \sigma ( \p L[U_0])(\xi) =  \Fdot^{ a \,  i_1, \ldots, i_{2m} }( x, 0, U_0^{b_0}, \nabla U_0^{b_1}, \ldots, \nabla^{2m} U_0^{b_{2m}}) \xi_{ i_1 } \cdots \xi_{ i_{2m} } e_a \otimes e^{* b_{2m} },
\end{equation*}
and thus the nonlinear operator is strongly (weakly) parabolic if $\hat \sigma (\p F[U_0])(\xi) > (\geq) 0$.  Closely related to the notion of parabolicity is the Legendre-Hadamard condition.  The linear operator \eqref{e: even order linear operator} is said to satisfy the Legendre-Hadamard condition if there exists a positive constant $\lambda > 0$ such that coefficient of the leading term satisfies
\begin{equation*}
	A_b^{ a I }\xi_{ I } \eta_a \eta^{*b} = A_b^{a \, i_1, \ldots, i_{2m} }\xi_1 \cdots \xi_{2m} \eta_a \eta^{*b} > \lambda \abs{ \xi }^{2m} \abs{ \eta }^2,
\end{equation*}
for all $\xi, \eta \in \Gamma(E)$ and $\eta^* \in \Gamma(E^*)$.  We can now state the local existence theorem we wish to prove.

\begin{mthm}
Let $E \times (0, \omega)$ be a vector bundle over $M \times (0, \omega)$, where $M$ is a smooth closed manifold, and let $U$ be a section $\Gamma(E \times (0, \omega)$.  Consider the following initial value problem:
	\begin{equation}
		\begin{cases}\label{eqn: nonlinear exist problem}
			P(U) := \p_t U - F(x,t,U, \nabla U, \ldots, \nabla^{2m}U) = 0 \text{ in } E \times (0, \omega), \\
			U(M, 0) = U_0,
		\end{cases}
	\end{equation}
with $U_0 \in C^{2m,1,\alpha}(E_{\omega})$.  The linearised operator of $P$ at $U_0$ in the direction $V$ is then given by
\begin{equation*}
\p P[U_0]V = \p_t V + (-1)^m\sum_{\abs{I} \leq 2m } A^I(x,t,U_0, \nabla U_0, \ldots, \nabla^{2m}U_0) \nabla_I V.
\end{equation*}
Suppose that the following conditions are satsified:
\begin{enumerate}
	\item  The leading coefficient $A_{b}^{a i_1 j_1 \cdots i_m j_m }$ satisfies the symmetry condition $A_{b}^{a i_1 j_1 \cdots i_m j_m } =  A_{ a }^{ b j_1 i_1 \cdots j_m i_m }$
	\item The leading coefficient satisfies the Legendre-Hadamard condition with constant $\lambda$
	\item There exists a uniform constant $\Lambda < \infty$ such that $\sum_{ \abs{I} \leq 2m } \abs{A^I}_{ \alpha; \, E_{\omega} } \leq \Lambda$
	\item $\Fdot$ is a continuous function of all its arguments
\end{enumerate}
Then there exists a unique solution $U \in C^{2m,1, \beta}(E_{\omega})$, where $\beta < \alpha$, for some short time $t_{\epsilon} > 0$ to the above initial value problem.  Furthermore, if $U_0$ and all the coefficients of the linearised operator are smooth, this solution is smooth.
\end{mthm}

\subsection{Function spaces and preliminary results}

We introduce some more of our notation and the necessary function spaces.  We denote differentiation in space by $\p_i$ or $\p_x$, and differentiation in time by $\p_t$.  As an example of this notation, the spatial gradient $\sum_{\alpha=1}^N\sum_{i=1}^d \abs{ \p_i u^{\alpha} }^2$ is $\abs{ \p_x u }^2$.  The $k$th order derivative is denoted by a raised index: $\p_x^k$.  We set $\p^{k,l}$ equal to $\p_x^k + \p_t^l$; note the lack of a subscript for this combined derivative.  Usually this combined derivative will be used when referring to the $2m$th order derivative in space and the first derivative in time: $\p^{2m,1} := \p_x^{2m} + \p_t$.  We shall need to work in both parabolic H\"{o}lder and Sobolev spaces, initially on $\mathbb{R}^{d+1}$ and later on a closed manifold.  We reserve $\Omega$ for an open domain contained in $\mathbb{R}^d$, and $M$ for a closed manifold.  We shall work with the parabolic domains $P := \Omega \times (0,\omega)$ and $M_{\omega} := M \times (0,\omega)$.  For any two points $X = (x,t)$, $Y = (y,s) \in P$, the parabolic distance between them is given by
\begin{equation*}
	d(X,Y) = \max \{ \abs{ x - y }, \abs{ t - s }^{ \frac{ 1 }{ 2m } } \}.
\end{equation*}
We shall also work with the parabolic domain $P_{\delta} := \{ X =(x, t) \in P : \text{dist}(x, \Omega) > \delta \}$ and the backwards parabolic cylinders $Q_R(X_0) := \{ X \in \mathbb{R}^N \times \mathbb{R} : d(X, X_0) < R, t < t_0 \} = B_R(x_0) \times (t_0 - R^{2m}, t_0) \}$.  Let $u : P \rightarrow \mathbb{R}^N$.  For $\alpha \in (0,1)$, the H\"{o}lder semi-norm is given by
\begin{align*}
	&[u]_{\alpha; \, P } := \sup_{ X \neq Y  \in P } \frac{ \abs{  u(X) - u(Y) } }{ d(X,Y)^{ \alpha } }.
\end{align*}
The H\"{o}lder norms are given by
\begin{align*}
	&\abs{ u }_{2m, 1; \, P } := \sum_{ k=0 }^{2m} \abs{ \p_x^k u }_{ 0; \, P } + \abs{ \p_t u }_{ 0; \, P } \\
	&\abs{ u }_{2m, 1, \alpha; \, P } := \abs{ u }_{2m, 1; \, P } + [ \p^{2m, 1} u ]_{ \alpha; \, P }. \\
\end{align*}	
The set of functions
\begin{equation*}
	\{ u \in C^{2m, 1}(P) : [u]_{ 2m, 1, \alpha; \, P } < \infty \}
\end{equation*}
endowed with the norm $\abs{ u }_{ 2m, 1, \alpha; \, P }$ is called a H\"{o}lder space.  Written out in full the norm is
\begin{equation*}
	\abs{ u }_{ 2m, 1, \alpha; \, P } := \sum_{ k=0 }^{2m} \abs{ \p_x^k u }_{ 0; \, P } + \abs{ \p_t u }_{ 0; \, P }+ [ \p_x^{2m} u ]_{ \alpha; \, P } + [ \p_t u ]_{ \alpha; \, P }.
\end{equation*}
These H\"{o}lder spaces are Banach spaces.  Next we define the analogous spaces on a closed manifold.  Let $E$ be a vector bundle over a closed manifold $M$, and let $E$ and $TM$ be equipped with metrics $g$ and connections $\nabla$.  Let $d_g(X,Y)$ be the geodesic distance on $M$ measured by $g$, and let $i_g$ be the injectivity radius of the manifold $M$.  For any two points $X = (x,t)$, $Y = (y,s) \in M_{\omega}$, the parabolic distance between them is given by
\begin{equation*}
	d(X,Y) = \max \{ d_g(x,y), \abs{ t - s }^{ \frac{ 1 }{ 2m } } \}.
\end{equation*}
The definition of the H\"{o}lder space of functions $u$ on $M_{\omega}$ mimics that of functions on $\mathbb{R}^{d+1}$.  For $\alpha \in (0,1)$, we define the semi-norms
\begin{align*}
	&\abs{u}_{ 2m, 1; \, M_{\omega} } := \sum_{k=0}^{2m} \abs{ \nabla_x^k u}_{0; \, M_{\omega}} + \abs{ \p_t u }_{0; \, M_{\omega}} \\
	&\abs{u}_{2m, 1, \alpha; \, M_{\omega} } := [u]_{ 2m, 1; \, M_{\omega} } + \sup_{ X \neq Y  \in M_{\omega} } \frac{ \abs{ \nabla^{2m, 1}u(X) - \nabla^{2m, 1}u(Y) } }{ d(X,Y)^{ \alpha } }.
\end{align*}
Here we use the notation $\nabla^{2m, 1} := \nabla_x^{2m} + \p_t$.  To define H\"{o}lder spaces of sections of $E$ we need to be a little more careful: The points $X$ and $Y$ live in different vector spaces above $M$, and so parallel translation is needed to indentify the spaces in order to perform the subtraction.  As we are working on a closed manifold, geodesics always exists between any two points, however beyond the injectivity radius the geodesics may not be unique.  Let $\mathcal{P}_{ Y, X }$ denote the parallel translation along a geodesic from $Y$ to $X$.  For $U \in \Gamma(E_{\omega})$, we define the norms
\begin{align*}
	&\abs{ U }_{ 2m, 1; \, E_{\omega} } := \sum_k^{2m} \abs{ \nabla_x^{2m} U(X) }_{0; E_{\omega} } + \abs{ \p_t U(X) }_{0; \, E_{\omega} } \\
	&\abs{U}_{ 2m, 1, \alpha; \, E_{\omega} } := \abs{U}_{ 2m, 1; \, E_{\omega} } + \sup_{ \substack{ X \neq Y  \in M_{\omega} \\ d_g(x,y) < i_g} } \frac{ \abs{ \nabla^{2m, 1}U(X) - \mathcal{P}_{Y, X} \nabla^{ 2m, 1}U(Y) } }{ d(X,Y)^{ \alpha } }.
\end{align*}
The norm $\abs{ \nabla^{2m, 1}U(X) - \mathcal{P}_{Y, X} \nabla^{ 2m, 1}U(Y) }$ is measured by the bundle metric $g$, but for convenience we shall supress this dependence in our notation.  We mention in passing that  these H\"older spaces are well-defined on closed manifolds, since on a closed manifold all metrics are equivalent and the injectivity radius is always positive.  Both these conditions fail to be true on arbitrary complete manifolds.  The remaining definitions are entirely analogous to the Euclidean case and we repeat them to avoid confusion.  The H\"{o}lder norms are given by
\begin{align*}
	&\abs{ U }_{2m, 1; \, E_{\omega} } := \sum_{ k=0 }^{2m} \abs{ \nabla_x^k U }_{ 0; \, E_{\omega} } + \abs{ \p_t U }_{ 0; \, E_{\omega} } \\
	&\abs{ U }_{2m, 1, \alpha; \, E_{\omega} } := \abs{ U }_{2m, 1; \, E_{\omega} } + [ \nabla^{2m, 1} U ]_{ \alpha; \, E_{\omega} }. \\
\end{align*}	
The set of tensor fields
\begin{equation*}
	\{ U \in C^{2m, 1}(E_{\omega}) : [U]_{ 2m, 1, \alpha; \, E_{\omega} } < \infty \}
\end{equation*}
endowed with the norm $\abs{ U }_{ 2m, 1, \alpha; \, E_{\omega} }$ is again called a H\"{o}lder space and it is easily verified that it too is a Banach space.  Written out in full the norm is
\begin{equation*}
	\abs{ U }_{ 2m, 1, \alpha; \, E_{\omega} } := \sum_{ k=0 }^{2m} \abs{ \nabla_x^k U }_{ 0; \, E_{\omega} } + \abs{ \p_t U }_{ 0; \, E_{\omega} }+ [ \nabla_x^{2m} U ]_{ \alpha; \, E_{\omega} } + [ \p_t U ]_{ \alpha; \, E_{\omega} }.
\end{equation*}
Next we introduce the anisotropic Sobolev spaces we wish to work in.  The set 
\begin{equation*}
	\big\{ u : \p_x^i \p_t^j u \in L^2(P), i + 2mj \leq 2m \big\}
\end{equation*}
endowed with the norm
\begin{equation*}
	\norm{u}_{W_2^{2m,1}(P)} := \left( \iint_{P} \sum_{i+2mj \leq 2m} \abs{ \p_x^i \p_t^j u}^2 \, dx \, dt \right)^{ 1/2 }
\end{equation*}
is the Sobolev space denoted by $W_2^{2m,1}(P)$.  On a manifold the norm is given by
\begin{equation*}
	\norm{U}_{W_2^{2m,1}(E_{\omega})} := \left( \iint_{M_{\omega}} \sum_{i+2mj \leq 2m} \abs{ \nabla_{ \negthickspace x }^i \p_t^j u}^2 \, dV_g \, dt \right)^{ 1/2 }.
\end{equation*}
These spaces are also Banach spaces.  If we interchange the order of the covariant and time derivatives in our definition we obtain an equivalent norm.  Note that each time derivative counts for $2m$ space derivatives, and that we have again suppressed the dependence on the bundle metric.  We shall also need the following spaces, in which the highest order spacial derivatives is of order $m$:
\begin{equation*}
	\norm{u}_{ W_2^{m,1}(E_{\omega}) } := \bigg( \iint_{M_{\omega}} \Big( \sum_{ k \leq m} \abs{ \nabla_{ \negthickspace x }^k u}^2 + \abs{ \p_t u }^2 \Big) dV_g \, dt \bigg)^{ 1/2 }.
\end{equation*}
We also define the space
\begin{equation*}
	V(E_{\omega}) = \big\{ U \in \Wdot_2^{1,1}(E_{\omega}) : \nabla \p_t \in L^2(E_{\omega}) \big\},
\end{equation*}
and note that $V(E_{\omega})$ is dense in $\Wdot_2^{1,1}(E_{\omega})$.  

Let $\Ccirc^{ \infty }(\overline { \Omega }_T )$ be the set of all smooth functions that vanish near the spatial boundary $\{ (x,t) : x \in \p \Omega, t \in (0, \omega) \}$ of $P$, and let $\Cdot^{ \infty }(\overline{ \Omega }_{\omega})$ be the set of all smooth functions that vanish near the parabolic boundary $\{ (x,t) : x \in \p \Omega, t \in (0, T) \} \cup \{ (x,t) : x \in \overline{ \Omega }, t = 0 \}$ of $P$.  Denote by $\Wcirc_2^{2m,1}(P)$ the closure of $\Ccirc^{ \infty }(\overline{ \Omega}_T)$ in $W_2^{2m,1}(P)$, and by $\Wdot_2^{2m,1}(P)$ the closure of $\Cdot^{ \infty }(\overline{ \Omega } )$ in $W_2^{2m,1}(P)$.  We similarly define the spaces $\Wcirc_2^{m,1}(P)$ and $\Wdot_2^{m,1}(P)$.

To close out this section we recall some important results that are used in the following.  The first result is a well-known covering lemma that allows us to patch local Euclidean estimates together to give global estimate on the manifold.  For a proof we refer the reader to \cite[Corallary 4.12]{rH95} and \cite{He}.
\begin{lem}\label{lem: bounds in chart}
Let $(M,g)$ be a Riemannian manifold, $p \in M$ and $r_0 \in (0, i_g(p)/4)$.  Suppose that for each $q \geq 0$ there exist constants $A_q$ such that $\abs{ \nabla^q Rm } \leq A_q$ in $B_p(r_0)$.  Then in normal coordinates $\{x^i\}$ on $B_p(r_0)$ there exist constants $C_q = C_q(n, i_g, A_0, \ldots A_q)$ such that for each $q$ the
estimates \[ \frac12 \delta_{ij} \leq g_{ij} \leq 2  \delta_{ij} \quad \text{and} \quad \left| \frac{ \p^p g_{ij} }{ \p x^p } \right| \leq C_q \]
hold in $B_p(\min \{ A_1/\sqrt{A_0}, r_0 \})$.
\end{lem}

The second result is known as G\aa rding's inequality.  G\aa rding's inequality on $\mathbb{R}^N$ is a well-known result.  The inequality also holds on a closed manifold, where one uses patching arguments similar to those we shall use later on, to lift the Euclidean estimate on to the manifold.  A proof of G\aa rding's inequality in the Euclidean case can be found in many places; for example \cite{Gi}.

\begin{lem}[G\aa rding's inequality]
Suppose that $A_b^{ a \, IJ }$ is a smooth section of a tensor bundle $E$ over a smooth closed manifold $M$ that satisfies the Legendre-Hadamard condition.  Then there exist positive constants $\lambda_0$ and $\lambda_1$ such that the bilinear form defined by
\begin{equation*}
	\mathcal{B}(U,V) := \int_M A_b^{ a \, IJ }\nabla_I^m U^b \nabla_J^m U^b \, dV_g
\end{equation*}
satisfies the inequality
\begin{equation*}
	\mathcal{B}(U,V) \geq \lambda_0 \int_M \abs{ \nabla^m U }^2 \, dV_g - \lambda_1 \int_M \abs{ U }^2 \, dV_g.
\end{equation*}
\end{lem}

We shall also require the Poincar\'{e} inequality:
\begin{prop}[Poincar\'{e} inequality]
Let $M$ be a smooth, closed manifold.  For any $u \in \Wdot_1^1$ there exists a positive constant $C$ such that
\begin{equation}\label{e: elliptic Poincare}
	\int_M \abs{ u }^2 \, dV_g \leq C \int_M \abs{ \nabla_{ \negthickspace x } u }^2 \, dV_g.
\end{equation}
\end{prop}
For a proof this proposition we refer the reader to \cite[pg 40.]{He}.  Since the Poincar\'{e} inequality holds at each timeslice of $M_{\omega}$ we can integrate \eqref{e: elliptic Poincare} in time to get
\begin{equation*}
	\iint_{M_{\omega}} \abs{ u }^2 \, dV_g \, dt \leq C \iint_{M_{\omega}} \abs{ \nabla_{ \negthickspace x} u }^2 \, dV_g \, dt,
\end{equation*}
and then by the Kato inequality $\abs{ \nabla \abs{ \nabla u } } \leq \abs{ \nabla^2 u }$ we also obtain
\begin{equation}\label{e: parabolic Poincare inequality}
	\iint_{M_{\omega}} \abs{ u }^2 \, dV_g \, dt \leq C \iint_{M_{\omega}} \abs{ \nabla^q_{ \negthickspace x} u }^2 \, dV_g \, dt
\end{equation}
for any $q \geq 1$.

\subsection{Hilbert space theory}
We commence our existence program by studying differential operators in divergence form.  Consider the problem
\begin{equation}\
	\begin{cases}\label{e: Hilbert IVP heat}
			\frac{ \p }{ \p t }U^a + (-1)^{ \abs{J} } \sum_{ 0 \leq \abs{I},\abs{J} \leq m } \nabla_J (A_b^{a \, IJ}(X)\nabla_I U^b ) = F^a(X), \quad X \in M_{\omega} \\
		U(M,0) = U_0,
	\end{cases}
\end{equation}
If $U_0$ is is sufficiently smooth then we can consider the problem for $V := U - U_0$, so without loss of generality we can assume $U_0 = 0$.  Ultimately we are interested in smooth solutions, so for us $U_0$ will always be smooth and this transformation is always possible.  Henceforth, we will usually assume $U_0 = 0$.  For simplicity, we assume that the connection does not depend on time, so we can commute time and space derivatives without introducing derivatives of the Christoffel symbols. We now want to introduce the notion of a weak solution to the above problem, and then recast the problem in terms of bilinear form on a Hilbert space.

\begin{defn}
A section $U \in \Wcirc_2^{m,1}(\Gamma(E_{\omega}))$ is called a weak solution of the initial value problem \eqref{e: Hilbert IVP heat} if for any $\varphi \in \Ccirc^{\infty}(\Gamma(E_{\omega}))$, the equation
\begin{equation}\label{d: weak solution 1}
\begin{split}
	&\iint_{M_{\omega}} \bigg( U^a_t \varphi^a + \sum_{ \abs{I} = \abs{J} = m }A_b^{a \, IJ}\nabla_I U^b \nabla_J \varphi^a + \sum_{ \substack{ 0 \leq \abs{I}, \abs{J} \leq m \\ \abs{I} + \abs{J} \leq 2m-1 } } B_b^{ a \, IJ }\nabla_I U^b \nabla_J\varphi^a \bigg)  \, dV_g \, dt \\
	&\quad = \iint _{M_{\omega}} F^a\varphi^a \, dV_g \, dt \end{split}
\end{equation}
holds.
\end{defn}
Since $\Ccirc^{ \infty }(E_{\omega})$ is dense in $\Wcirc_2^{m,0}(M_{\omega})$, the test function can in fact be any function in $\Wcirc_2^{m,0}(E_{\omega})$.  For ease of reading will again often drop the indices running over the section $U$ and simply write
\begin{equation*}
	\iint_{M_{\omega}} \bigg( U_t \varphi + \sum_{ \abs{I} = \abs{J} = m }A^{IJ}\nabla_I^m U \nabla_J^m \varphi + \sum_{ \substack{ 0 \leq \abs{I}, \abs{J} \leq m \\ \abs{I} + \abs{J} \leq 2m-1 } } B^{IJ }\nabla_I U \nabla_J\varphi  \bigg) \, dV_g \, dt = \iint _{M_{\omega}} F\varphi \, dV_g \, dt
\end{equation*}

We have the following two characterisations of weak solutions.
\begin{prop}
A section $U \in \Wcirc_2^{m,1}(E_{\omega})$  satisfies \eqref{d: weak solution 1} if and only if $U$ satisfies
\begin{equation}\label{d: weak solution 2}
	\begin{split}
		&\iint_{M_{\omega}} U^a_t \bigg( \varphi_t^a + \sum_{ \abs{I} = \abs{J} = m }A_b^{a \, IJ}\nabla_I^m U^b \nabla_J^m \varphi_t^a + \sum_{ \substack{ 0 \leq \abs{I}, \abs{J} \leq m \\ \abs{I} + \abs{J} \leq 2m-1 } } B_b^{ a \, IJ }\nabla_I U^b \nabla_J\varphi_t^a  \bigg) \, dV_g \, dt \\
		&\quad = \iint _{M_{\omega}} F^a\varphi_t^a \, dV_g \, dt \end{split}
\end{equation}
for any $\varphi \in \Ccirc^{ \infty }(E_{\omega})$.
\end{prop}
\begin{proof}
Suppose that $U  \in \Wcirc_2^{m,1}(E_{\omega})$ satisfies \eqref{d: weak solution 1} for any $\varphi \in \Ccirc^{ \infty }(E_{\omega})$.  Because $\varphi$ is smooth, $\varphi_t$ is a valid test function, and so \eqref{d: weak solution 2} holds.  Conversely, suppose that $U  \in \Wcirc_2^{m,1}(E_{\omega})$ satisfies \eqref{d: weak solution 2} for any $\varphi \in \Ccirc^{ \infty }(M_{\omega})$.  Then since  $\varphi$ is smooth, $\int_0^t \varphi(x,s) \, ds$ is a valid test function, and choosing the test function as such in \eqref{d: weak solution 2} shows \eqref{d: weak solution 1} holds.
\end{proof}

\begin{prop}\label{p: Hilbert IVP 1}
A section $U \in \Wcirc_2^{m,1}(E_{\omega})$ satisfies \eqref{d: weak solution 1} if and only if $U$ satisfies
\begin{equation}\label{d: weak solution 3}
	\begin{split}
	&\iint_{M_{\omega}} \bigg( U^a_t \varphi_t^a + \sum_{ \abs{I} = \abs{J} = m }A_b^{a \, IJ}\nabla_I^m U^b \nabla_J^m \varphi_t^a + \sum_{ \substack{ 0 \leq \abs{I}, \abs{J} \leq m \\ \abs{I} + \abs{J} \leq 2m-1 } } B_b^{ a \, IJ }\nabla_I U^b \nabla_J\varphi_t^a  \bigg) e^{ -\theta t } \, dV_g \, dt \\
	&\quad = \iint _{M_{\omega}} F^a\varphi^a_t e^{ -\theta t } \, dV_g \, dt \end{split}
\end{equation}
for any $\varphi \in \Ccirc^{ \infty }(E_{\omega})$, where $\theta$ is a positive constant.
\end{prop}
\begin{proof}
Suppose that $U  \in \Wcirc_2^{m,1}(E_{\omega})$ satisfies \eqref{d: weak solution 1} for any $\varphi \in \Ccirc^{ \infty }(
E_{\omega})$.  Because $\varphi$ is smooth, $\varphi_t e^{ -\theta t }$ is a valid test function, and so \eqref{d: weak solution 3} holds.  Conversely, suppose that $U  \in \Wcirc_2^{m,1}(E_{\omega})$ satisfies \eqref{d: weak solution 3} for any $\varphi \in \Ccirc^{ \infty }(E_{\omega})$.  Then since  $\varphi$ is smooth, $\varphi(x,t) e^{ -\theta t } - \theta \int_0^t \varphi(x,s) e^{ \theta s } \, ds$ is a valid test function, and choosing the test function as such in \eqref{d: weak solution 3} shows \eqref{d: weak solution 2} holds, and thus \eqref{d: weak solution 1}.
\end{proof}

We shall use the Lax-Milgram lemma to prove existence and uniqueness of a weak solution to problem \eqref{d: weak solution 1}.  Our approach is similar to that of \cite{HP} and \cite{Sh}, where slightly different function spaces were used.

\begin{thm}\label{t: Lax-Milgram}
Let $H$ be a Hilbert space and $V$ an inner product space continuously embedded in $H$.  Let $\mathcal{B} : H \times V \rightarrow \mathbb{R}$ be a bilinear form with the following properties:
\begin{enumerate}
	\item For all $U \in H$ and $W \in V$, there exists a constant $C$ such that $\abs{ \mathcal{B}(U,W) } \leq C \norm{U}_H \norm{W}_V$ \\
	\item $\mathcal{B}$ is coercive, namely, there exists a $\lambda > 0$ such that $\mathcal{B}(W,W) \geq \lambda \norm{ W }_{V}^2$
\end{enumerate}
Then for any bounded, linear functional $F(W)$ in $H$, there exists a $U \in H$ such that $F(W) = \mathcal{B}(U,W)$ for each $W \in V$.  Moreover, if $W$ is dense in $H$, then $U$ is unique.
\end{thm}
For a careful proof of this result we recommend to the reader \cite[pg 118.]{Sh}.

\begin{thm}
If $F$, $U_0 \in L^2(E_{\omega})$, then the initial value problem \eqref{e: Hilbert IVP heat} admits a weak unique solution $U \in \Wdot_2^{m,1}(E_{\omega})$.
\end{thm}
\begin{proof}
Let $U \in \Wdot_2^{m,1}(E_{\omega})$, $V \in V(E_{\omega})$ and $\theta$ be some constant greater than zero that will be fixed later on.  Consider the bilinear form associated to the differential operator in problem \eqref{e: Hilbert IVP heat}
\begin{equation}\label{e: bilinear form 1}
	\mathcal{B}(U,V) := \iint_{M_{\omega}} \bigg( U_tV_t + \sum_{ \abs{I} = \abs{J} = m }A_b^{a \, IJ}\nabla_I^m U^b \nabla_J^m V^a + \sum_{ \substack{ 0 \leq \abs{I}, \abs{J} \leq m \\ \abs{I} + \abs{J} \leq 2m-1 } } B_b^{ a \, IJ }\nabla_I U^b \nabla_J V^a \bigg) e^{-\theta t}  \, dV_g \, dt.
\end{equation}
We want to show that the bilinear form $\mathcal{B}$ satisfies the conditions of the Lax-Milgram Lemma.  First, it's easy to see
\begin{equation*}
	\abs{ \mathcal{B}(U,V) } \leq \norm{ U }_{ W_2^{m,1}(E_{\omega}) } \norm{ V }_{ V(E_{\omega}) },
\end{equation*}
and so $\mathcal{B}$ is bounded. Next we show $\mathcal{B}$ is also coercive.  For convenience, write $\mathcal{B} = I_1 + I_2$, where $I_1$ and $I_2$ refer to the two summation terms of \eqref{e: bilinear form 1}.  Focussing on $I_1$, for $V \in V(E_{\omega})$ we have
\begin{align*}
	I_1 &= \iint_{M_{\omega}} \sum_{ \abs{I} = \abs{J} = m }A_b^{a \, IJ}\nabla_I V^b \nabla_J V_t^a e^{ - \theta t } \, dV_g \, dt \\
	&\geq \frac{1}{2} \iint_{M_{\omega}} \frac{\p}{\p t} \big( A_B^{ a \, IJ } \nabla_I V \nabla_J V \big) e^{ -\theta t } - \frac{1}{2} \abs{ \p_t A }_0  \iint_{M_{\omega}} \nabla_I V \nabla_J V e^{ -\theta t } \, dV_g \, dt \\
	&\geq \frac{1}{2} \iint_{M_{\omega}} \frac{\p}{\p t} \big( A_B^{ a \, IJ } \nabla_I V \nabla_J V  e^{ -\theta t } \big) + \frac{ \theta }{ 2 } \iint_{M_{\omega}} A_b^{a \, IJ} \nabla_I V \nabla_J V e^{ -\theta t } \, dV_g \, dt \\
	&\quad - \frac{1}{2} \abs{ \p_t A }_0  \iint_{M_{\omega}} \nabla_I V \nabla_J V e^{ -\theta t } \, dV_g \, dt
\end{align*}
Upon integrating the first term on the right we find both terms are non-negative: the endpoint $t = T$ from G\aa rding's inequality and $t= 0$ because $V \in V(E_{\omega})$, and we discard these terms.  We are left with
\[
	I_1 \geq \Big( \frac{ \theta \lambda }{ 2 } - \frac{1}{2} \abs{ \p_t A }_0 \Big) e^{ -\theta T}\iint_{M_{\omega}} \abs{ \nabla^m V }^2. \]
By choosing $\theta$ sufficiently large the first term on the right can be made positive.  Now we deal with $I_2$.  By using the Peter-Paul inequality on the terms of $I_2$ they are either of the order $\abs{ \nabla^m V }^2$ multiplied by an $\epsilon$, or lower order terms divided by $\epsilon$.  In the case of the former, they can again be absorbed by choosing $\theta$ sufficiently large.  In the case of all lower order terms, they can also be absorbed by using the Poincar\'e inequality and then choosing $\theta$ sufficiently large.  After all such estimation we obtain
\begin{equation*}
	\mathcal{B}(V,V) \geq \delta \norm{ V }_{ W_2^{p,1}(E_{\omega}) }
\end{equation*}
for some constant $\delta > 0$.  This shows $\mathcal{B}$ is coercive and we may now apply the Lax-Milgram Lemma.  In Theorem \ref{t: Lax-Milgram}, choose $\Wdot_2^{m,1}(E_{\omega})$ as the space $H$, $V(E_{\omega})$ as the space $W$, and $F(V) = \iint_{M_{\omega}} FV_t e^{ -\theta t } \, dV_g \, dt$.  By the Lax-Milgram Lemma, there exists a unique $U \in \Wdot_2^{m,1}\Gamma(E_{\omega})$ such that $\mathcal{B}(U,V) = F(V)$ for all $V \in V(E_{\omega})$.  Thus $U$ is the unique weak solution to problem \eqref{e: Hilbert IVP heat} by Propostion \ref{p: Hilbert IVP 1}.
\end{proof}

Next we discuss the regularity of the weak solution to problem \eqref{e: Hilbert IVP heat}.  We first need to first recall some basic facts about difference quotients.  Difference quotient approximations to weak derivatives are a common tool in PDE and proofs of the following facts can be found in many texts, for example \cite{Gi, GT83}.  Let $u : \Omega \rightarrow \mathbb{R}^N$ be function and $\Omega' \subset \subset \Omega$.  The difference quotient in direction $e_k$ is defined for all $x \in \Omega'$ by
\begin{equation*}
	D_{k,h}(x) := \frac{ u(x + he_k, t) - u(x,t) }{ h },
\end{equation*}
where $0 < \abs{h} < \text{dist}(\Omega', \Omega)$ and $k = 1, \ldots, N$.
\begin{prop}\label{prop: difference quotients 2}
\begin{enumerate}
\item Suppose $u \in W_2^p(\Omega)$ and $1 \leq p < \infty$. Then for each $\Omega' \subset \subset \Omega$ the estimate
	\begin{equation}
		\norm{ D_{h,k} }_{L^2(\Omega')} \leq C\norm{ \p u }_{L^2(\Omega)}
	\end{equation}
holds for some constant $C$ and all $0 < h < (1/2) \text{dist}(\Omega', \p \Omega)$. \\
\item Suppose that $u \in L^p(\Omega')$, $1 < p < \infty$, and that there exists a constant $C$ such that
	\begin{equation*}
		\norm{ D_{h,k}u }_{L^p(\Omega')} \leq C
	\end{equation*}
holds for all $0 < h < (1/2) \text{ dist}(\Omega', \p \Omega)$.  Then $ \p u \in L^2(\Omega')$ and
	\begin{equation*}
		\norm{ \p u }_{L^2(\Omega')} \leq C.
	\end{equation*}
\end{enumerate}
\end{prop}		

\begin{prop}[interior regularity]
Suppose that $F \in L^2(E_{\omega})$ and $U \in \Wdot_2^{m,1}(E_{\omega})$ is a weak solution to problem \eqref{e: Hilbert IVP heat}.  Then $U \in W_2^{2m,1}(E_{\omega})$ and the estimate
\begin{equation*}
	\norm{ U }_{ W_2^{2m,1}(P_{\delta}) } \leq C \big( \norm{ U }_{ W_2^{m,0}(P) } + \norm{ F }_{ L^2(P)} \big).
\end{equation*}
holds.
\end{prop}

\begin{proof}
We give the proof for the case $m = 1$.  The proof for systems of even order follows in a similar manner way, with small changes needed to incorporate the scaling of the system; in this regard see \cite{HP}.  Regularity is a local problem, so we derive the necessary regularity estimates on Euclidean space and then lift them to the manifold using patching argument.  As our starting point we therefore work with the following definition of a weak solution
\begin{equation}\label{e: Hilbert int reg 1}
\iint_{P} u^a_t \varphi^a + A_b^{a \, ij}\p_i u^b \p_j \varphi^a + B_b^{ a \, ij }\p_i u^b \p_j \varphi^a  \, dx \, dt = \iint _{P} f^a\varphi^a \, dx \, dt,
\end{equation}
which holds for all $\varphi \in \Wcirc_2^{m,0}(P)$.  Rewrite this as
\begin{equation*}
\iint_{P} u^a_t \varphi^a + A_b^{a \, ij}\p_i u^b \p_j \varphi^a = \iint _{P} g^a\varphi^a \, dx \, dt,
\end{equation*}
where $g^a = f^a - B_b^{ a \, ij }\p_i u^b \p_j \varphi^a$.  Choose $\varphi = -D_{-h,k}(\eta^2 D_{h,k} u^a )\chi_{ [0,s] }$.  This is a valid choice as we have restricted $h$ to be sufficiently small.  With this choice of $\varphi$ equation \eqref{e: Hilbert int reg 1} reads
\begin{equation}\label{e: Hilbert int reg 2}
	\begin{split}
	&\iint u^a_t ( -D_{-h,k}(\eta^2 D_{h,k} u^a ) ) + A_b^{a \, ij}\p_i u^b \p_j (-D_{-h,k}(\eta^2 D_{h,k} u^a )) \, dx \, dt  \\
	&\quad = \iint  g^a ( -D_{-h,k}(\eta^2 D_{h,k} u^a ) ) \, dx \, dt. \end{split}
\end{equation}
We focus on the term involving the time derivative.  Using the properties of difference quotients we have
\begin{align*}
	\iint u^a_t ( -D_{-h,k}(\eta^2 D_{h,k} u^a ) ) \, dx \, dt &= \iint \p_t( D_{ h,k }u^a) \eta^2 D_{h,k} u^a \, dx \, dt \\
	&= \frac{1}{2} \iint \p_t \big( \eta^2 ( D_{h,k} u^a )^2 \big) \, dx \, dt \\
	&=  \frac{1}{2} \int_{ \Omega } \big( \eta^2 ( D_{h,k} u(x,t)^a )^2 \big) \, dx \Big|_{t=0}^{t=s} \\
	&= \frac{1}{2} \int_{ \Omega } \eta^2 ( D_{h,k} u(x,s)^a )^2 \, dx.
\end{align*}
Now focus on the second term on the right of \eqref{e: Hilbert int reg 2}.  By Proposition \ref{prop: difference quotients 2}, in order to achieve the desired spatial regularity it suffices to suitably bound the $L^2$ norm of $D_{h,k} \p u$.  Using various properties of difference quotients we estimate
\begin{align*}
	&\iint_{P} A_b^{ a \, ij } \p_i u^a \p_j ( -D_{-h,k} ( \eta^2 D_{h,k} u^b ) ) \, dx \, dt \\
	&\quad = \iint_{P} A_b^{ a \, ij } \p_i u^a -D_{-h,k} \p_j ( \eta^2 D_{h,k} u^b ) \, dx \, dt \\
	&\quad = \iint_{P} D_{h,k} ( A_b^{ a \, ij } \p_i u^b ) \p_j ( \eta^2 D_{h,k} u^a ) \, dx \, dt \\
	&\quad = \iint_{P} \Big( A_b^{ a \, ij }(x + he_k, t) D_{h,k} \p_i u^b + (D_{h,k} A_b^{ a \, ij } \p_i u^b \Big) \p_j(\eta^2 D_{h,k} u^a) \, dx \, dt \\
	&\quad = \iint_{P} \Big( A_b^{ a \, ij }(x + he_k, t) D_{h,k} \p_i u^b + (D_{h,k} A_b^{ a \, ij } \p_i u^b \Big) (\eta^2 \p_j D_{h,k} u^a - 2 \eta \p_j \eta D_{h,k} u^a) \, dx \, dt \\
	&\quad = \iint_{P} \eta^2 A_b^{ a \, ij }(x + he_k, t) \p_i D_{h,k} u^b \p_j D_{h,k} u^a \, dx \, dt  \\
	&\quad \quad - 2 \eta \p_j \eta \iint_{P} A_b^{ a \, ij }(x + he_k, t) D_{h,k} \p_i u^a D_{h,k} u^b \, dx \, dt \\
	&\quad \quad + \iint_{P} ( D_{h,k} A_b^{ a \, ij } ) \p_u^a (\eta^2 \p_j D_{h,k} u^a - 2 \eta \p_j \eta D_{h,k} u^a) \, dx \, dt \\
	&\quad \geq \eta^2 \lambda_0 \iint_{ P_{\delta} } \abs{ D_{ h, k } \p u }^2 \, dx \, dt - \iint_{P} S_k^a \, dx \, dt.
\end{align*}
In going to the last line we have used G\aa rding's inequality and grouped the remaining terms into the term $S_k^a$.  By using the properties of difference quotients and the Peter-Paul inequality, $S_k^a$ as well as the term involving $g$ on the right hand side of equation \eqref{e: Hilbert int reg 2} can both be estimated by the $L^2$ norm of $\p u$ and $f$.  Recombining this estimate on the spatial derivatives with the estimate on the time derivative gives
\begin{align*}
	\sup_{0 < s < \omega} \int_{ \Omega } \eta^2 ( D_{h,k} u^a(x,s) )^2 \, dx + \iint_{P_{\delta} } \abs{ D_{h,k} \p_u }^2 \, dx \, dt \leq C \big( \norm{ \p u }_{L^2(P)} + \norm{ f }_{L^2(P) } \big).
\end{align*}
From Proposition \ref{prop: difference quotients 2} it now follow that
\begin{equation*}
	 \iint_{P_{\delta} } \abs{ \p_x^2 u }^2 \, dx \, dt \leq C \big( \norm{ \p u }^2_{L^2(P)} + \norm{ f }^2_{L^2(P) } \big).
\end{equation*}
The estimate on the time derivative can be proved in a similar fashion using the difference operator in time.  For the time derivative we obtain the estimate
\[ \iint_{P_{\delta}} (\p_t u)^2 \, dx dt \leq C\big( \norm{ \p u }_{L^2(P)}^2 + \norm{ f }_{L^2(P)}^2 \big), \]
and combining the space and time estimates complete the proof in the case $m=1$.
\end{proof}
The near-bottom boundary estimate can also derived in a similar fashion, and then the local interior and near-bottom estimates can be lifted to a closed manifold using a patching argument to give an estimate holding globally on $E_{\omega}$.  Higher regularity estimates can also be obtained using standard bootstrap arguments.  As a simple consequence of the higher regularity estimates and the parabolic Sobolev inequality (on a manifold) we have the following existence theorem in the smooth category.
\begin{cor}
Suppose that $U \in \Wdot_2^{m,1}(E_{\omega})$ is a weak solution to problem \eqref{e: Hilbert IVP heat}.  If $F \in C^{ \infty }(E_{\omega})$, then $U \in C^{ \infty }(E_{\omega})$.
\end{cor}

\subsection{Schauder theory}
In this section we derive the interior Schauder estimates in Euclidean space, and then lift these local estimates to the vector bunlde $E_{\omega}$ to obtain a global Schauder estimate holding on the bundle.  A number of methods can be used to derive the Schauder estimates; we shall present two of these.  Trudinger's method of mollification offers a simple proof of the Schauder estimates for elliptic and parabolic equations of second order.  This method extends to systems of even order, and indeed we pursued this route in an early draft of this thesis.  But perhaps an even easier and cleaner method of the deriving the Schauder estimates is Leon Simon's method of scaling \cite{Sim1}.  We shall use Simon's method to derive the estimates for systems of even order.

Trudinger's method was introduced in \cite{nT86} where he treated both equations and systems of elliptic type.  For second order equations, the method is remarkably simple, and makes use of the solid mean value inequality.  His method of mollification extends to systems of even order, where the application of the mean value inequality is replaced by an $L^2$ estimate and the Sobolev embedding theorem.  Wang \cite{xW} has used Trudinger's method of mollification to derive Schauder estimates for second order parabolic equations, where the application of the solid mean value inequality was replaced by estimates coming differentiating the fundamental solution of the heat equation.  Here we show how the mean value property of the heat equation can be used in exactly the same way as the solid mean value inequality to provide the desired estimates.  Simon's method of scaling is remarkably simple, with the transition from second order equations to high order systems made by essentially only changing notation.  Simon's method, first published in journal form in \cite{Sim1}, can also be found in his book \cite{Sim2} (which appeared some years earlier), and complete details can also be found in Simon's lecture notes on PDE \cite{Sim3}.  In \cite{Sim1} Simon's indicates how his method adapts to encompass equations and systems of parabolic type, and this is pursued in Lamm's Diploma Thesis \cite{tL}.  Before proceeding, we first recall the the H\"{o}lder space interpolation inequality, which we shall use often in the derivation of the Schauder estimates.

\begin{prop}[H\"{o}lder space interpolation inequality]
Let $\rho > 0$, $\epsilon > 0$ and $Q_{\rho} \subset \mathbb{R}^{d+1}$.  Suppose $u \in C^{2m,1,\alpha}(Q_{\rho})$.  There exists a constant $C = C(n, d, \epsilon, \alpha, m)$ such that
\begin{equation*}
	\begin{split}
		&\rho^{\alpha}[u]_{ \alpha; \, Q_{\rho} } + \rho \abs{ \p_x u }_{ 0; \, Q_{\rho} } + \ldots + \rho^{2m-1+\alpha}[ \p_x^{2m-1} u ]_{ \alpha; \, Q_{\rho} } + \rho^{2m} \abs{ \p^{2m,1} u }_{ 0; \, Q_{\rho} } \\
		&\quad \leq \epsilon\rho^{2m+\alpha} [ \p^{2m,1} u ]_{ \alpha; \, Q_{\rho} } + C\abs{ u }_{ 0; \, Q_{\rho} } \end{split}
	\end{equation*}
holds.
\end{prop}
The interpolation inequality can be established by simple contradiction arguments or directly using the mean value theorem; see, for example, \cite{tL}.

\subsubsection{Trudinger's method of mollification}

A (parabolic) mollifier (of order ${2m}$) is a fixed smooth function $\rho \in C_0^{ \infty }( \mathbb{R}^{d+1})$ with $\iint_{ \mathbb{R}^{d+1} } \rho dX = 1$.  For $\tau > 0$ we define the scaled mollifier
	\begin{equation}
		\rho_{\tau}(x,t) := \frac{1}{\tau^{d+2m}} \rho\left(\frac{x}{\tau},\frac{t}{\tau^{2m}}\right).
	\end{equation}
Let  $P \in \mathbb{R}^{n+1}$ and $u \in L^1_{\text{loc}}(P)$.  For $0 < \tau < d(X, \p P)$, the mollification of $u$ is given by
	\begin{equation*}
		u_{\tau}(x,t) := \frac{1}{\tau^{d+2m}} \iint \rho \left(\frac{x-y}{\tau},\frac{t-s}{\tau^{2m}} \right) u(y,s) \, dy \, ds
	\end{equation*}
and satisifes spt $u_{\tau} \subset P_{\tau}$, where $P_{\tau} = \{ X \in P : d(X, \p P) > \tau \}$.
\begin{prop}
We have $u_{\tau} \in C_0^{\infty}$.
\end{prop}

\begin{prop}
Let $u \in L^1_{\text{loc}}(\Omega)$.  The following estimates hold:
	\begin{align}
		&\abs{ u_{ \tau } }_{0; \, P_{\tau} } \leq \abs{u}_{ 0; \, P_{\tau} } \label{e: Holder est 1} \\
		&\abs{ \p_x^i \p_t^j u_{\tau}(x,t) }_{ 0; \, P_{\tau} } \leq C \tau^{-i-2mj} \abs{u}_{0; \, P_{\tau} }. \label{e: Holder est 2}
	\end{align}
\end{prop}
\begin{proof}
To prove \eqref{e: Holder est 1}, we have
\begin{align*}
	u_{ \tau }(x,t) &= \frac{1}{ \tau^{d+2m} } \iint \rho \left( \frac{x-y}{\tau},\frac{t-s}{ \tau^{2m}} \right) u(y,s) \, dy \, ds \\
	&\leq \abs{ u }_{0; \, P_{\tau} } \cdot \frac{1}{ \tau^{d+2m} } \iint \rho \left( \frac{x-y}{\tau},\frac{t-s}{ \tau^{2m}} \right) \, dy \, ds \\
	&= \abs{ u }_{0; \, P_{\tau} }.
\end{align*}
And for \eqref{e: Holder est 2}:
\begin{align*}
	\p_x^i \p_t^j u_{\tau}(x,t) &= \frac{1}{ \tau^{d+2m} } \iint_{P_{\tau}} \p_x^i \p_t^j \rho \left( \frac{x-y}{\tau},\frac{t-s}{ \tau^{2m}} \right) u(y,s) \, dy \, ds \\
	&\leq C \tau^{-i-2mj}\abs{ u }_{ 0; \, P_{\tau} }.
\end{align*}
\end{proof}

\begin{prop}
Let $u \in C^{\alpha}_{\text{loc}}(P)$.  The following estimates hold:
	\begin{align}
		&\abs{ u_{\tau}(x,t) - u(x,t) }_{ 0; \, P_{\tau} }\leq \tau^{\alpha} [u]_{\alpha; \, P_{\tau} } \label{e: Holder est 3} \\
		&\abs{ \p_x^i \p_t^j u_{\tau}(x,t) }_{ 0; \, P_{\tau} } \leq C \tau^{ \alpha -i-2mj } [u]_{\alpha; \, P_{\tau} }. \label{e: Holder est 4}
	\end{align}
\end{prop}

\begin{proof}
For estimate \eqref{e: Holder est 3} we have
\begin{align*}
	u_{\tau}(x,t) - u(x,t) &= \frac{1}{ \tau^{d+2m} } \iint \rho \left( \frac{x-y}{\tau},\frac{t-s}{ \tau^{2m}} \right) (u(y,s) - u(x,t)) \, dy \, ds \\
	&\leq \osc_{P_{\tau} } u \\
	&\leq \tau^{\alpha} [u]_{\alpha; P_{\tau} }.
\end{align*}
To prove the second estimate we have
\begin{align*}
		\p_x^i \p_t^j u_{\tau}(x,t) &= \frac{1}{ \tau^{d+2m} } \iint \p_x^i \p_t^j \rho \left( \frac{x-y}{\tau},\frac{t-s}{ \tau^{2m}} \right) u(y,s) \, dy \, ds \\
		&=  \frac{1}{ \tau^{d+2m} } \iint \p_x^i \p_t^j \rho \left( \frac{x-y}{\tau},\frac{t-s}{ \tau^{2m}} \right) (u(y,s) - u(x,t)) \, dy \, ds \\
		&\quad +  \frac{u(x,t)}{ \tau^{d+2m} } \iint \p_x^i \p_t^j \rho \left( \frac{x-y}{\tau},\frac{t-s}{ \tau^{2m}} \right) \, dy \, ds.
\end{align*}
The mollifier $\rho$ is has compact support on $P_{\tau}$ and so the last term vanishes by the Divergence Theorem.  Continuing, we have
	\begin{align*}
		\p_x^i \p_t^j u_{\tau}(x,t) &= \frac{1}{ \tau^{d+2m} } \iint_{P_{\tau}} \p_x^i \p_t^j \rho \left( \frac{x-y}{\tau},\frac{t-s}{ \tau^{2m}} \right) (u(y,s) - u(x,t)) \, dy \, ds \\
		&\leq C \tau^{-i-2mj}\osc_{Q_\tau} u \\
		&\leq C \tau^{ \alpha -i-2mj } [u]_{ \alpha; \, P_{\tau} }.
	\end{align*}
\end{proof}

To motivate things a little in the parabolic settting, we first briefly show how Trudinger's method works in the elliptic setting by treating the Poisson equation.  The crucial ingredient in Trudinger's method is the following norm equivalence:

\begin{lem}
Let $u \in C^{ \alpha }( \mathbb{R}^d )$, $R > 0$ and $\alpha \in (0,1)$.  There exists constant $C = C(d, \alpha)$ such that the norm equivalence
\begin{equation*}
	\frac{1}{C}[u]_{ \alpha; \, B_{R} } \leq \sup_{ 0 < \tau < R/2  } \tau^{1-\alpha} \abs{\p_x u_{\tau}}_{ 0; \, B_{ R } } \leq C[u]_{ \alpha; \, B_{R} }.
\end{equation*}
is valid.
\end{lem}
\begin{proof}
The inequality on the right follows directly from equation \eqref{e: Holder est 4} (the elliptic version) by choosing the appropriate values for the indices $i$: choosing $i = 1$ (there is no $j$ in the elliptic mollifier) gives
\begin{equation*}
	\abs{ \p_x u_{\tau}}_{ 0; \, B_{ R } } \leq C \tau^{ \alpha - 1 } [u]_{\alpha; \, B_{ R } }.
\end{equation*}
The first inequality requires a little more work.  Let $x, y \in \mathbb{R}^d$ and $\tau \in (0, R/2)$. For $ \abs{ x - y } < R$, by the triangle inequality
\begin{align*}
		\abs{u(x) - u(y)} &\leq \abs{u(x) - u_{\tau}(x)} + \abs{ u_{\tau}(x) - u_{\tau}(y) } + \abs{ u_{\tau}(y) - u(y) } \\
		&\leq 2\tau^{\alpha}[u]_{ \alpha; \, B_{ R }} + \abs{ \p_x u_{\tau} }_{ 0; \, B_{ R } }\abs{x-y}.
\end{align*}
Set $\tau = \epsilon \abs{x-y}$, where $\epsilon < 1/2$.  Factoring out and dividing by $\abs{x-y}^{\alpha}$ we find
\begin{equation*}
	(1-2\epsilon^{\alpha}) \frac{ \abs{ u(x) - u(y) } }{ \abs{ x - y }^{ \alpha } } \leq \epsilon^{\alpha - 1}\tau^{1-\alpha} \abs{ \p_x u_{\tau} }_{0; \, B_R}.
\end{equation*} 
Choosing $\epsilon < (1/2)^{-\alpha}$ and taking the supremum over $\tau \in (0, R/2)$ completes the proof.
\end{proof}

We now derive the Schauder estimate for Poisson's equation.  For simplicity we consider solutions with compact support in $\mathbb{R}^d$ (the techniques for treating the general case will be seen later on when we treat parabolic equations). Fix $\alpha \in (0,1)$ and suppose that $u \in C_0^{2, \alpha}(\mathbb{R}^d)$ solves
\begin{equation*}
	-a^{ij}(x)\p_{ij}u(x) = f(x),
\end{equation*}
where we assume $a^{ij}, f \in C^{\alpha}(\mathbb{R}^d)$ and $\lambda \abs{\xi}^2 \leq a^{ij}\xi_i\xi_j \leq \Lambda \abs{ \xi }^2$.  We proceed by the method of freezing coefficients, and accordingly fix a point $x_0 \in \mathbb{R}^d$ a rewrite the above equation equation as
\begin{align}
\notag	-a^{ij}(x_0)\p_{ij}u(x) &= (a^{ij}(x_0) - a^{ij}(x) )\p_{ij}u + f(x) \\
	&:= g(x) \label{eqn: Poisson 1}
\end{align}
By a linear coordinate transformation we can assume $a^{ij}(x_0) = \delta^{ij}$ so that equation \eqref{eqn: Poisson 1} becomes the Poisson equation.  We now mollify equation \eqref{eqn: Poisson 1} to get
\begin{equation*}
	-\Delta u_{\tau} = g_{\tau}
\end{equation*}
and then differentiate thrice with respect to $x$ to obtain
\begin{equation*}
	-\Delta \p_x^3 u_{\tau} = \p_x^3 g_{\tau}.
\end{equation*}
We choose a radius $R>0$ and work in the ball $B_R$.  Using inequality \eqref{e: Holder est 4} we can estimate
\begin{align*}
	\abs{ \p_x^3 g_{\tau} }_{0; \, B_R } &\leq C(n) \tau^{-3} \abs{ g }_{0; \, B_{R+\tau} } \\ 
	&\leq C(n) \tau^{-3} (R+\tau)^{\alpha} [ g ]_{0; \, B_{R+\tau} } \\
	&\leq C(n) \tau^{-3} (R+\tau)^{\alpha} \big( [ a ]_{\alpha; \, B_{R+\tau} } \abs{ \p_x^2 u }_{0; \, B_{R+\tau} } + [ f ]_{\alpha; \, B_{R+\tau} } \big).
\end{align*}
We now recall the solid mean value inequality for subharmonic functions: If $v$ solves $-\Delta v(x) \leq 0$ on a ball $B_R(x) \subset \mathbb{R}^d$, then $v$ satisfies
\[ v(x) \leq \frac{ C(n) }{ R^n } \int_{B_R} v(y) \, dy. \]
To apply this inequality to our situation, noting $\Delta \abs{ x }^2 = 2n$, we have
\[	-\Delta \left( \p_x^3 u_{\tau} + \frac{ \abs{ \p_x^3 g_{\tau} }_{0; \, B_R } \abs{ x }^2 }{ 2n } \right) = -\Delta \p_x^3 u_{\tau} - \abs{ \p_x^3 g_{\tau}  }_{0; \, B_R} \leq 0. \]
Thus the function $\p_x^3 u_{\tau} + \abs{ \p_x^3 g_{\tau} }_{0; \, B_R } \abs{ x }^2 / (2n)$ is subharmonic and applying the mean value inequality and estimating we obtain
\begin{align*}
	\abs{ \p_x^3 u_{\tau}(x_0) } &\leq C(n) \left( R^{-n} \left\lvert \int_{B_R} \p_y^3 u_{\tau}(y)  \, dy \right\rvert +  R^2 \abs{ \p_x^3 g_{\tau} }_{0; \, B_R } \right) \\
	&\leq C(n) \left( \frac1R \osc_{ B_R } \p_x^2 u_{\tau}(x) + \tau^{-3} R^2(R+\tau)^{\alpha} \big( [ a ]_{\alpha; B_{R+\tau} } \abs{ \p_x^2 u }_{0; \, B_{R+\tau} } + [ f ]_{\alpha; \, B_{R+\tau} } \big) \right) \\
	&\leq C(n) \left( R^{ \alpha - 1 } [ \p_x^2 u ]_{ \alpha; \, B_R } + \tau^{-3} R^2 (R+\tau)^{\alpha} \big( [ a ]_{\alpha; B_{R+\tau} } \abs{ \p_x^2 u }_{0; \, B_{R+\tau} } + [ f ]_{\alpha; \, B_{R+\tau} } \big) \right).
\end{align*}
Setting $R = N\tau$ and returning to the original coordinates we find
\[ \tau^{1-\alpha}\abs{ \p_x^3 u_{\tau}(x_0) } \leq C(n, \lambda, \Lambda, \alpha) \left( N^{ \alpha - 1 } [ \p_x^2 u ]_{ \alpha; \, B_R } + N^{2+\alpha} \big( [ a ]_{\alpha; B_{R+\tau} } \abs{ \p_x^2 u }_{0; \, B_{R+\tau} } + [ f ]_{\alpha; \, B_{R+\tau} } \big) \right). \]
Now taking the supremum over $\tau > 0$ and using the norm equivalence we obtain
\[ [ \p_x^2 u ]_{\alpha; \, \mathbb{R}^d } \leq C(n, \lambda, \Lambda, \alpha) \left( N^{ \alpha - 1 } [ \p_x^2 u ]_{ \alpha; \, \mathbb{R}^d } + N^{2+\alpha} \big( [ a ]_{\alpha; \mathbb{R}^d } \abs{ \p_x^2 u }_{0; \, \mathbb{R}^d } + [ f ]_{\alpha; \, \mathbb{R}^d } \big) \right). \]
Choosing $N$ sufficiently large and using the H\"{o}lder space interpolation inequality on the right gives the desired estimate, namely
\[ [ \p_x^2 u ]_{\alpha; \, \mathbb{R}^d } \leq C \big( \abs{ u }_{0; \, \mathbb{R}^d } + [ f ]_{\alpha; \, \mathbb{R}^d } \big), \]
where $C$ depends on $n, \lambda, \Lambda$, and $\alpha$.
Having given a feel for Trudinger's method, we move on to use this method to derive the Schauder estimates for second order parabolic equations.  The crucial equivalence of norms lemma in the parabolic setting is the following:

\begin{lem}
Let $u \in C^{ \alpha }( \mathbb{R}^{d+1} )$, $R > 0$ and $\alpha \in (0,1)$.  There exists constant C depending only on $d$ and $\alpha$ such that the norm equivalence
\begin{equation*}
		\frac{1}{C}[u]_{ \alpha; \, Q_{R} } \leq \sup_{0 < \tau < R/2 } \left\{ \tau^{1-\alpha} \abs{\p_x u_{\tau}}_{ 0; \, Q_{ R } } + \tau^{2m-\alpha}\abs{\p_t u_{\tau}}_{0; Q_{R} } \right\} \leq C[u]_{ \alpha; \, Q_{R} }.
	\end{equation*}
is valid.
\end{lem}
\begin{proof}
The second inequality follows directly from equation \eqref{e: Holder est 4} by choosing the appropriate values for the indices $i$ and $j$.  To prove the spatial part of the second inequality, choosing  $i = 1$ and $j = 0$ in estimate \eqref{e: Holder est 4} gives
\begin{equation*}
	\abs{ \p_x u_{\tau}(x,t) }_{ 0; \, Q_{ R } } \leq C \tau^{ \alpha - 1 } [u]_{\alpha; \, Q_{ R } }.
\end{equation*}
The temporal estimate follows similarly. Let $X, Y \in \mathbb{R}^{d+1}$ and $\tau \in (0, R/2)$. For $ d(X, Y) < R$, by the triangle inequality

\begin{align*}
		\abs{u(X) - u(Y)} &\leq \abs{u(X) - u_{\tau}(X)} + \abs{ u_{\tau}(Y) - u(Y) } + \abs{ u_{\tau}(x,t) - u_{\tau}(y,t) } + \abs{ u_{\tau}(y,t) - u_{\tau}(y,s) } \\
		&\leq 2\tau^{\alpha}[u]_{ \alpha; \, Q_{ R }} + \abs{x-y} \abs{ \p_x u_{\tau} }_{ 0; \, Q_{ R } } + \abs{t-s} \abs{ \p_t u_{\tau} }_{ 0; \, Q_{ R } }.
\end{align*}
Set $\tau = \epsilon d(X,Y)$, where $\epsilon < 1/2$.  Factoring out $d(X,Y)^{\alpha}$ we have
\begin{equation*}
	\abs{u(X) - u(Y)} \leq d(X,Y)^{\alpha} \left( 2\epsilon^{\alpha}[u]_{\alpha; \, Q_{ R }} + \epsilon^{\alpha-1}\tau^{1-\alpha}\abs{ \p_x u_{\tau} }_{ 0; \, Q_{ R } } + \epsilon^{\alpha-2m}\tau^{2m-\alpha}\abs{ \p_t u_{\tau} }_{ 0; \, Q_{ R } } \right).
\end{equation*}
The proposition follows by fixing $\epsilon$ sufficiently small and taking the supremum over $\tau \in (0, R/2)$.
\end{proof}
We now proceed similarly to Poisson's equation to derive the Schauder estimate for the nonhomongeneous heat equation.  Fix $\alpha \in (0,1)$ and suppose that $u \in C_0^{2, \alpha}(\mathbb{R}^{d+1})$ solves
\[ \p_t u(x,t) - a^{ij}(x,t)\p_{ij}u(x,t) = f(x,t), \]
where we assume $a^{ij}, f \in C^{\alpha}(\mathbb{R}^{d+1})$ and $\lambda \abs{ \xi }^2 \leq a^{ij}\xi_i \xi_j \leq \Lambda \abs{ \xi }^2$.  Again we freeze coefficients at a point $(x_0, t_0) \in \mathbb{R}^{d+1}$, perform a coordinate transformation and mollify the equation to get
\begin{equation}
	\p_t u_{\tau} - \Delta u_{\tau}(x,t) = g_{\tau}. \label{eqn: heat subsol 0}
\end{equation}
Given the form of the norm equivalence, the desired Schauder estimate will follow if we can establish the estimates (for the spatial component of the Schauder estimate)
\begin{align}
	\abs{ \p_x^3 u_\tau(x_0, t_0) } &\leq C \left( \frac1R \osc_{ Q_R } \p_x^2 u + R^2\abs{ \p_x^3 g_{\tau} }_{0; \, Q_R} \right) \label{eqn: heat subsol 1} \\
	\abs{ \p_t\p_x^2 u_\tau(x_0, t_0) } &\leq C \left( \frac{1}{R^2} \osc_{ Q_R } \p_x^2 u + R^2\abs{ \p_t\p_x^2 g_{\tau} }_{0; \, Q_R} \right), \label{eqn: heat subsol 3}
\end{align}
and for the temporal part
\begin{align}
	\abs{ \p_x\p_t u_\tau(x_0, t_0) } &\leq C \left( \frac1R \osc_{ Q_R } \p_t u + R^2\abs{ \p_x\p_t g_{\tau} }_{0; \, Q_R} \right) \label{eqn: time subsol 1} \\
	\abs{ \p_t^2 u_\tau(x_0, t_0) } &\leq C \left( \frac{1}{R^2} \osc_{ Q_R } \p_t u + R^2\abs{ \p_t^2 g_{\tau} }_{0; \, Q_R} \right).  \label{eqn: time subsol 2}
\end{align}
We show how to obtain the spatial estimates, as the time estimates follow in exactly the same way.  We recall the mean value property for subsolutions of the heat equation:  If $v$ is a subsolution to the heat equation on $\mathbb{R}^{d+1}$, that is if $v$ satisfies $\p_v - \Delta v \leq 0$, then $v$ satisfies
\begin{equation*}
	v(x_0, t_0) \leq \frac{1}{4r^n}\iint_{E(x_0,t_0;\,r)} v(y,s) \frac{ \abs{ x_0 - y }^2 }{ (t_0-s)^2 } \, dy ds
\end{equation*}
for each $E(x_0,t_0; r) \subset \mathbb{R}^{d+1}$.  Recall the heat ball $E(x,t;r)$ is the set given by $E(x_0,t_0;r) = \{ (y,s) \in \mathbb{R}^{d+1} : \abs{x_0-y}^2 \leq \sqrt{ -2\pi s \log[r^2/(-4\pi s)] }, s \in (t_0 - r^2/(4\pi s), t_0) \}$.  We denote the radius of the heat ball by $R_r(s) : \sqrt{ -2\pi s \log[r^2/(-4\pi s)] }$.  For further information on the mean value property of the heat equation we refer the reader to \cite{Ev} and \cite{kE04}.  Let us now show \eqref{eqn: heat subsol 1}:  Differentiate \eqref{eqn: heat subsol 0} thrice in space.  Since $\abs{ \p_x^3 g_{\tau} }_{0; \, E}\abs{x}^2/(2n)$ is independent of time we see
\begin{align*}
	\p_t \left( \p_x^3 u_{\tau} + \abs{ \p_x^3 g_{\tau} }_{0; \, E}\frac{ \abs{x}^2 }{2n} \right) - \Delta \left( \p_x^3 u_{\tau} + \abs{ \p_x^3 g_{\tau} }_{0; \, E}\frac{ \abs{x}^2 }{2n} \right) &= \p_t(\p_x^3 u_{\tau}) - \Delta (\p^3_x u_{\tau}) - \abs{ \p_x^3 g_{\tau} }_{0; \, E} \\
	&= \p_x^3 g_{\tau} - \abs{ \p_x^3 g_{\tau} }_{0; \, E} \leq 0,
\end{align*}
and hence the function $\p_x^3 u_{\tau} + \abs{ \p_x^3 g_{\tau} }_{0; \, E(x_0, t_0; r) } \abs{ x }^2 / (2n)$ is subsolution of the heat equation.  From the mean value property of subsolutions we have
\begin{equation}
\p_x^3 u_{\tau}(x_0, t_0) \leq \frac{1}{4r^n} \iint_{E(x,t;r)} \left( \p_y^3 u_{\tau}(y,s) + \abs{\p_x^3 g_{\tau}}_{0}\abs{ y }^2 \right) \frac{ \abs{ x_0 - y }^2 }{ \abs{ t_0 - s }^2 } \, dy ds. \label{eqn: heat subsol 4}
\end{equation}
By translating coordinates we can assume that $(x_0, t_0) = (0,0)$.  All the desired estimates involve evaluation the integral
\[ \frac{1}{r^n}\int_{ \frac{-r^2}{4\pi} } \frac{ R_r(s)^{\alpha} }{ s^{\beta} } \, ds, \]
where $\alpha$ and $\beta$ are given integers.  The constants can be computed explicitly, however we are only interested in the scaling behaviour with respect to the radius $r$ (and that the integral is finite).  We compute
\begin{align*}
	\frac{1}{r^n}\int_{ \frac{-r^2}{4\pi} }^0 \frac{ R_r(s)^{\alpha} }{ s^{\beta} } \, ds &= \frac{1}{r^n}\int_{ \frac{-r^2}{4\pi} }^0 \frac{ \big( -2ns \log[ r^2/ (-4\pi s ) ] \big)^{ \alpha/2} }{ -s^{\beta} } \\
	&= C(n, \alpha, \beta) r^{-n+\alpha-2\beta+2} \int_{ \frac{1}{4\pi} }^0 t^{\alpha/2 - \beta} \big( \log( 4 \pi t) \big)^{\alpha/2} \,dt \\
	&= C(n, \alpha, \beta) r^{-n+\alpha-2\beta+2}\int_0^{\infty} s^{\alpha/2} e^{-\alpha/2-\beta+1} \, ds.
\end{align*}
With further substitution this integral can be converted into the Gamma function, which is finite as long as $\alpha/2 > -1$.  Returning to \eqref{eqn: heat subsol 4}, we have
\begin{equation}
	\p_x^3 u_{\tau}(x_0, t_0)  \leq 4r^{-n} \iint_E \p_y^3 u_{\tau}(y,s) \frac{ \abs{y}^2 }{ s^2 } \, dy ds + 4\abs{ \p_x^3 g_{ \tau } }_{0; \, E} r^{-n} \iint_E \frac{ \abs{y}^4 }{ s^2 } \, dy ds. \label{eqn: heat subsol 5}
\end{equation}
We estimate the first term on the right by
\begin{align*}
	4r^{-n} \iint_E \p_y^3 u_{\tau}(y,s) \frac{ \abs{y}^2 }{ s^2 } \, dy ds &\leq Cr^{-n} \int_{ \frac{ -r^2 }{4\pi} }^0 \frac{ R_r(s)^2 }{s^2} \left( \int_{B_{R_r(s)}} \p_y^3 u_{\tau} \, dy \right) ds \\
	&\leq Cr^{-n} \int_{ \frac{ -r^2 }{4\pi} }^0 \frac{ R_r(s)^2 }{s^2} \left( \int_{\p B_{R_r(s)}} \osc \p_y^2 u \, dy \right) ds \\
	&\leq Cr^{-n} \osc_E \p_x^2 u \int_{ \frac{ -r^2 }{4\pi} }^0 \frac{ R_r(s)^{n+1} }{s^2} ds\\
	&\leq \frac{ C(n) }{ r } \osc_E \p_x^2 u.
\end{align*}
The second term on the right of \eqref{eqn: heat subsol 5} can be estimated more simply to give
\[ 4\abs{ \p_x^3 g_{ \tau } }_{0; \, E} r^{-n} \iint_E \frac{ \abs{y}^4 }{ s^2 } \, dy ds \leq C(n) r^2 \abs{ \p_x^3 g_{ \tau } }_{0; \, E}. \]
The estimates involving time derivatives can also be estimated in a similar manner.  For example, by integrating by parts in time, we have
\begin{align*}
	4r^{-n} \iint_E \p_t\p_y^2 u_{\tau}(y,s) \frac{ \abs{y}^2 }{ s^2 } \, dy ds &\leq Cr^{-n} \iint_E \osc \p_y^2 u \frac{ \abs{y}^2 }{ s^3 } \\
	&\leq \frac{ C(n) }{ r^2 } \osc_E \p_x^2 u.
\end{align*}
The derivation now continues in the exactly the same was as for the Poisson equation, using the estimates \eqref{eqn: heat subsol 1} - \eqref{eqn: time subsol 1}, the equivalence of norms lemma and the H\"{o}lder space interpolation inequality; we ultimately obtain the desired Schauder estimate:
\begin{equation}
	[\p^{2,1} u ]_{\alpha; \, \mathbb{R}^d } \leq C \big( [ f ]_{\alpha; \, \mathbb{R}^d }+ \abs{ u }_{0; \, \mathbb{R}^d } \big), \label{eqn: eqn Schauder}
\end{equation}
where $C$ depends on $n, \lambda, \Lambda$, and $\alpha$.  The method extends to more general equations and domains by using cutoff functions and Simon's absorption lemma, as we shall soon see in the case of systems.

\subsubsection{Simon's method of scaling}
As we have mentioned before, Trudinger's method extends to systems of even order, where the application of the mean value inequalities are replaced by $L^2$ estimates and the Sobolev embedding theorem.  For parabolic systems the method becomes a little computationally cumbersome, and instead we shall use the Simon's method of scaling.  For the derivation of the Schauder estimates for elliptic systems, in addition the Simon's original paper \cite{Sim1}, we highly recommend his lecture notes on PDE \cite{Sim3}.  Once one has defined the notion of a parabolic polynomial his method adapts immediately to parabolic systems.  Here we simply quote the interior and near-bottom Schauder estimates for parabolic systems of even order on Euclidean space, and refer the reader to \cite{tL} for complete proofs.  Any errors or inconsistencies are due to us.

\begin{prop}[interior Schauder estimate]
Suppose $u \in C^{2m,1,\alpha}(\bar Q_R(X_0))$ is a solution of a general linear $2m$-order parabolic system
\begin{equation}\label{eqn: interior Schauder 1}
	Lu^{a} := \p_t u^{a} + (-1)^m \sum_{ \abs{I} \leq 2m } A_{b}^{a I}(x,t)\p_I u^{b} = f^{a}.
\end{equation}
Suppose the following conditions are satisfied:
\begin{enumerate}
	\item  The leading coefficient $A_{b}^{a i_1 j_1 \cdots i_m j_m }$ satisfies the symmetry condition $A_{b}^{a i_1 j_1 \cdots i_m j_m } =  a_{ a }^{ b j_1 i_1 \cdots j_m i_m }$
	\item The leading coefficient satisfies the Legendre-Hadamard condition with constant $\lambda$
		\item There exists a uniform constant $\Lambda < \infty$ such that $\sum_{ \abs{I} \leq 2m } \abs{A^I}_{\alpha; Q_R(X_0)} \leq \Lambda$.
\end{enumerate}
Then there exists a constant $C = C(n, N, \theta, \lambda, \Lambda)$ such that the estimate
\[ [ \p^{2m,1}u]_{\alpha; \; \theta Q_R } \leq C \big( [ f ]_{\alpha; \, Q_R} + R^{-2m-\alpha} \abs{ u }_{0; \, Q_R} \big) \]
holds for each $\theta \in (0,1)$.
\end{prop}

\begin{prop}[near-bottom Schauder estimate]
Suppose $u \in C^{2m,1,\alpha}(\bar Q_R^+(X_0))$, with $u( \cdot, 0) = u_0$,  is a solution of a general linear $2m$-order parabolic system
\begin{equation}\label{eqn: interior Schauder 1}
	Lu^{a} := \p_t u^{a} + (-1)^m \sum_{ \abs{I} \leq 2m } A_{b}^{a I}(x,t)\p_I u^{b} = f^{a}.
\end{equation}
Suppose the following conditions are satisfied:
\begin{enumerate}
	\item  The leading coefficient $A_{b}^{a i_1 j_1 \cdots i_m j_m }$ satisfies the symmetry condition $A_{b}^{a i_1 j_1 \cdots i_m j_m } =  a_{ a }^{ b j_1 i_1 \cdots j_m i_m }$
	\item The leading coefficient satisfies the Legendre-Hadamard condition with constant $\lambda$
		\item There exists a uniform constant $\Lambda < \infty$ such that $\sum_{ \abs{I} \leq 2m } \abs{A^I}_{\alpha; Q_R(X_0)} \leq \Lambda$.
\end{enumerate}
Then there exists a constant $C = C(n, N, \theta, \lambda, \Lambda)$ such that the estimate
\[ [ \p^{2m,1}u ]_{\alpha; \; \theta Q^+_R} \leq C \big( [ f ]_{\alpha; \, Q^+_R} + [\p_x^{2m} u_0]_{\alpha; \, Q^+_R} +R^{-2m-\alpha} \abs{ u }_{0; \, Q^+_R} \big) \]
holds for each $\theta \in (0,1)$.
\end{prop}

The above estimates are the localised counterparts to equation \eqref{eqn: eqn Schauder}.  In order to localise the estimate, the following adsorption lemma is needed:
\begin{lem}[Simon's adsorption lemma]
Let $S$ be a real-valued monotone sub-additive function on the class of all convex subsets of $B_R(x_0) ($i.e. $S(A) \leq \sum_{i=1}^N S(A_j)$ whenever $A, A_1, \ldots, A_N$ are convex subsets with $A \subset \cup_{j=1}^N \subset B_R(x_0)$.  Suppose that $\theta_0 \in (0, 1)$, $\mu \in (0,1]$, $\gamma \geq 1$ and $l \geq 0$ are given constants.  There exists an $\epsilon_0 = \epsilon_0(l, \theta,n) > 0$ such that if
\[ \rho^lS(B_{\theta\rho}(y)) \leq \epsilon_0\rho^lS(B_{\rho}(y)) + \gamma \]
whenever $B_{\rho}(y) \subset B_R(x_0)$ and $\rho \leq \mu R$, then
\[ R^lS(B_{\theta R}(x_0)) \leq C\gamma, \]
where $C=C(n,\theta,\mu, l)$.
\end{lem}
The proof can be found in \cite{Sim1} and \cite{Sim3}.  In localising the Schauder estimate we need to apply the adsorption lemma in the case $S(A) = [ u ]_{\alpha; \, A}$.  We confirm that the lemma holds in this case, that is $S$ is monotone and sub-additive on convex subsets of $Q_R$.  Let $R >0$ a given radius and $A \subset Q_R$.  Since the H\"{o}lder constant is defined by taking the supremum over a set, monotonicity clearly holds.  To show sub-additivity, suppose $A \subset A_1 \cup A_2$, where all sets are convex.  Fix $X, Y \in A$.  If either $X,Y \in A_1$ or $X,Y \in A_2$, then
\[ \frac{ \abs{ u(X) - u(Y) } }{ d(X,Y)^{\alpha} } \leq \max \{ [u]_{\alpha; \, A_1}, [u]_{\alpha; \, A_2} \} \leq [u]_{\alpha; A_1} + [u]_{\alpha; \, A_2}. \]
If on the other hand $X \in A_1$ and $Y \in A_2$, the choose $Z \in A_1 \cap A_2$ lying on the line segment between $X$ and $Y$.  Then
\begin{align*}
	\frac{ \abs{ u(X) - u(Y) } }{ d(X,Y)^{ \alpha } } &\leq \frac{ \abs{ u(X) - u(Z) } + \abs{ u(Z) - u(Y) } }{ d(X,Z)^{\alpha} + d(Z,Y)^{\alpha} } \\
	&\leq \frac{ \abs{ u(X) - u(Z) } }{ d(X,Z)^{\alpha} }  + \frac{ \abs{ u(Z) - u(Y) } }{d(Z,Y)^{\alpha} } \\
	&\leq [u]_{\alpha; \, A} + [u]_{\alpha; \, A_2}.
\end{align*}
The general case follows by induction.

\subsubsection{Global Schauder estimate}
The above Schauder estimate holds on a small parabolic cylinder $Q_R \subset \mathbb{R}^N$.  We now want to lift these local estimates to the vector bundle $E \times (0, \omega)$ to obtain Schauder estimate globally on $E_{\omega}$.  Let $\widetilde\psi : V \times I \rightarrow \mathbb{R}^n \times \mathbb{R}_+$ be the coordinate map for a sufficiently small neighbourhood $V \times I \subset M \times (0, \omega)$.  By definition of a vector bundle, there exists a bundle trivialisation $\Psi : E_{\omega}|_{V \times I } \rightarrow V \times I \times \mathbb{R}^N$.    In fact, if $\psi$ is the coordinate map for $V$, then $\widetilde\psi = \psi \times id$. Using the bundle trivialisation and the coordinate maps we can locally identify a section $\Gamma(E \times (0, \omega))$ as a subset of $\mathbb{R}^n \times \mathbb{R}_+ \times \mathbb{R}^N $.  We  will abuse notation slighly, and for $U \in \Gamma(E \times (0, \omega))$, we shall write $(\widetilde\psi^{-1})^*U$ to mean the local trivialisation $U|_{V \times I}$ pulled back to $\mathbb{R}^n \times \mathbb{R}_+ \times \mathbb{R}^N$ via the coordinate map $\widetilde\psi$.

Next we want to control the norm of section measured with the bundle metric in terms of the Euclidean norm of the pulled-back section.
\begin{prop}\label{prop: bundle covering}
Let $U \in \Gamma(E \times (0, \omega))$ and $(V_i, \psi_i)$ be a covering of $M \times (0, \omega)$ by a finite number of normal charts of sufficiently small radius $R_0$.  Then there exists a constant $C = C(n, R_0)$ such that in each neighbourhood $V_i$ the equivalence of norms
\[ \frac1C \abs{ U \circ \widetilde\psi^{-1} }_{2m, 1, \alpha; \, \widetilde\psi_i(V_i \times I_i) } \leq \abs{ U }_{2m, 1, \alpha; \, V_i \times I_i} \leq C\abs{ U \circ \widetilde\psi^{-1} }_{2m, 1, \alpha; \, \widetilde\psi_i(V_i \times I_i) } \]
is valid.
\end{prop}
\begin{proof}
For the parts of the H\"{o}lder norm involving suprema this is easy to show, as one simple writes the covariant derivative in terms of ordinary derivatives and the Christoffel symbols and uses Lemma \ref{lem: bounds in chart}.  To deal with H\"{o}lder semi-norm, we note that parallel translation is defined in terms of solving an ordinary differential equation.  We then have control on the size of the H\"{o}lder coefficient in terms of the initial condition for the ODE in a finite number of charts, thus it too is uniformly bounded.
\end{proof}

Using the above lemma, we can now patch together the local Euclidean Schauder estimates to give the desired global Schauder estimate.
\begin{prop}[Global Schauder estimate]
Let $E \times (0, \omega)$ be a vector bundle over $M \times (0, \omega)$, where $M$ is a closed manifold.  Let $L : \Gamma(E_{\omega} ) \rightarrow \Gamma(F_{\omega})$ be linear differential operator of order $2m$.  In any local coordinate chart $L$ is of the form
\begin{equation}\label{e: even order linear operator}
	\frac{ \p }{ \p t }U^a + (-1)^m \sum_{ I \leq 2m } A^I \p_I U,
\end{equation}
or in full
\begin{equation*}
	\frac{ \p }{ \p t }U^a + (-1)^m \big( A_b^{a \, i_1, \ldots i_{2m} }(x,t)\p_{i_1} \cdots \p_{i_{2m}} U^b + \cdots + B_b^{ a \, k }(x,t)\p_k U^b + C_b^a(x,t)U^b \big),
\end{equation*}
with $U( \cdot, 0) = U_0$.  Suppose that in any coordinate chart the following conditions are satisfied:
\begin{enumerate}
	\item  The leading coefficient $A_{b}^{a i_1 j_1 \cdots i_m j_m }$ satisfies the symmetry condition $A_{b}^{a i_1 j_1 \cdots i_m j_m } =  a_{ a }^{ b j_1 i_1 \cdots j_m i_m }$
	\item The leading coefficient satisfies the Legendre-Hadamard condition with constant $\lambda$
		\item There exists a uniform constant $\Lambda < \infty$ such that $\sum_{ \abs{I} \leq 2m } \abs{A^I}_{\alpha; Q_R} \leq \Lambda$.
\end{enumerate}
Then there exists a constant $C = C(n, N, \lambda, \Lambda, M, \omega)$ such that the estimate
\[ \abs{ U }_{2m, 1, \alpha; \,  E_{\omega} } \leq C \big( \abs{ F }_{\alpha; \, E_{\omega} } +  \abs{ U_0 }_{2m, \alpha; \, E_{\omega} } + \abs{ U }_{0; \, E_{\omega} } \big). \]
\end{prop}
\begin{proof}
Because $M$ is compact, we can cover $M \times (0, \omega)$ by a finite number of coordinate patches $(V_i, \psi_i)$ of sufficiently small radii $R_i \leq R_0$ so that we can apply Proposition \ref{prop: bundle covering}.  Suppose $X, Y \in M_{\omega}$ are any two points.  If $d(X,Y) < R_0$, then we estimate
\begin{align*}
	\frac{ \abs{ \nabla^{2m,1} U(X) - \mathcal{P}_{Y,X} \nabla^{2m,1} U(Y) } }{ d(X,Y)^{\alpha} } &\leq C \sum_i [ \nabla^{2m,1} U ]_{ \alpha; \, V_i } \\
	&\leq C\sum_i \Big( \abs{ F \circ \widetilde{\psi}_i^{-1} }_{0, 0, \alpha ; \, \widetilde{\psi}_i(V_i) } + \abs{ U \circ \widetilde{\psi}_i^{-1} }_{0 ; \, \widetilde{\psi}_i(V_i) } \Big) \\
	&\leq C\sum_i \Big( \abs{ F }_{0, 0, \alpha ; \, V_i } + \abs{ U  }_{0 ; \, V_i } \Big) \\
	&\leq C \big( \abs{ F }_{0, 0, \alpha ; \, E_{\omega} } + \abs{ U }_{0 ; \, E_{\omega} } \big).
\end{align*}
On the other hand, if $d(X,Y) \geq R_0$ we estimate
\begin{align*}
	\frac{ \abs{ \nabla^{2m,1} U(X) - \mathcal{P}_{Y,X} \nabla^{2m,1} U(Y) } }{ d(X,Y)^{\alpha} } &\leq C \abs{ \nabla^{2m,1} U }_{0; \, E_{\omega}} R_0^{ \alpha } \\
	&\leq C \abs{ U }_{2m, 1, \alpha; \, E_{\omega}} \\
	&\leq C \big( \abs{ F }_{0, 0, \alpha ; \, E_{\omega} } + \abs{ U }_{0 ; \, E_{\omega} } \big).
\end{align*}
Note that we have again used the fact that we have a finite covering, as we have needed to take the supremum over the all H\"{o}lder coefficients in each chart.
\end{proof}

\subsection{Linear existence theory}
The next step in our existence program is to prove existence and uniqueness for linear operators in H\"{o}lder space.  We begin with the $2m$th order heat operator.

\begin{prop}\label{p: heat IVP}
Consider the following initial value problem:
\begin{equation}\label{e: heat IVP}
	\begin{cases}
		\p_t U + (-\Delta^m) U = F(X), \quad X \in M_{\omega} \\
		U( \cdot ,0) = U_0.
	\end{cases}
\end{equation}
Suppose that $F \in C^{0,0,\alpha}(E_{\omega})$ and $U_0 \in C^{2m,1,\alpha}(E_{\omega})$, where $\alpha \in (0,1)$.  Then problem \eqref{e: heat IVP} has a unique solution $U \in C^{2m,1,\alpha}(E_{\omega})$.
\end{prop}
\begin{proof}
As usual, we may assume without loss of generality that $U_0 = 0$. By mollification we can construct a section $F_{\epsilon} \in C^{\infty}(E_{\omega})$ such that
\begin{equation*}
	\abs{ F_{\epsilon} }_{ \alpha; \, E_{\omega} } \leq 2 \abs{ F }_{ \alpha; \, E_{\omega} }.
\end{equation*}
Now consider the approximate problem
\begin{equation}\label{e: approx problem}
	\begin{cases}
		\p_t U_{ \epsilon } + (-\Delta^m) U_{ \epsilon } = F_{ \epsilon }(X), \quad X \in M_{\omega}\\
		U( \cdot, 0) = 0.
	\end{cases}
\end{equation}
From the $L^2$ theory, there exists a unique smooth solution $U_{ \epsilon } \in C^{ \infty }(E_{\omega})$ to the above approximate problem.  A short contradiction argument (see \cite{Sim3}) shows we can estimate
\begin{equation*}
	\abs{ U_{ \epsilon } }_{ 0: \, E_{\omega} } \leq \epsilon \abs{ U_{ \epsilon } }_{ 2m, 1, \alpha; \, E_{\omega} } + c(\epsilon) \norm{ U_{ \epsilon } }_{ L^2(E_{\omega}) },
\end{equation*}
and then using the Hilbert space regularity estimates we may estimate
\begin{equation*}
	\norm{ U_{ \epsilon } }_{ L^2(E_{\omega}) } \leq C \norm{ F_{ \epsilon } }_{ L^2(E_{\omega}) } \leq C\abs{ F_{ \epsilon } }_{ 0, 0, \alpha; \, E_{\omega} }.
\end{equation*}
We point out that in the case of second order equations, using the maximum principle it is a slightly simpler matter to estimate
\begin{equation*}
	\abs{ U_{\epsilon} }_{ 0; \, E_{\omega}} \leq C\abs{ F_{\epsilon} }_{ 0; \, E_{\omega} } \leq C\abs{ F_{\epsilon} }_{\alpha; \, E_{\omega} }.
\end{equation*}
Combining this estimate with the global Schauder estimate, $\abs{ U_{ \epsilon } }_{ 2m, 1, \alpha; \, E_{\omega} } \leq C( \abs{ U_{ \epsilon } }_{ 0; \, E_{\omega} } + \abs{ F_{\epsilon } }_{\alpha; \, E_{\omega} } )$, we get
\begin{equation*}
	\abs{ U_{ \epsilon } }_{ 2m, 1, \alpha; \, E_{\omega} } \leq C\abs{ F }_{\alpha; \, E_{\omega} },
\end{equation*}
where the constant $C$ is independent of $\epsilon$.  Given that $F \in C^{\alpha}(E_{\omega})$, the left hand side is uniformly bounded.  The Arzela-Ascoli theorem now applies to give a subsequence such that $U_{ \epsilon } \rightarrow U$ uniformly in $C^{ 2m, 1 }(E_{\omega})$ as $\epsilon \rightarrow 0$, and moreover $U \in C^{ 2m, 1, \alpha }(E_{\omega})$.  Last of all, we show uniqueness by the energy method.  Suppose that $U_1$ and $U_2$ are two solutions to \eqref{e: heat IVP}, and consider the problem for $W := U_1 - U_2$, where $W$ now solves the homogeneous heat equation with zero initial condition. For $0 \leq t \leq T$ we define the energy $e(t)$ by
\[ e(t) = \int_M W(x,t)^2 \, dV_g. \]
Then
\begin{align*}
\frac{ d }{ dt } e(t) &= 2 \int_M W W_t \, dV_g \\
&= 2 \int_M W (-\Delta^m)W \, dV_g \\
&\leq 0.
\end{align*}
Thus $e(t) \leq e(0)$ for all $t \in [0,T]$, and consequently $U_1 = U_2$ and the solution is unique.
\end{proof}

With a solution to the heat operator in place we can now use the method of continuity to solve the general linear problem.

\begin{thm}[method of continuity]\label{t: method of ontinuity}
Let $B$ be a Banach space, $V$ a normed linear space, and $L_0$ and $L_1$ bounded linear operators from $B$ to $V$.  For $t \in [0,1]$ define
\begin{equation*}
	L_{\tau} := (1- \tau)L_0 + \tau L_1
\end{equation*}
and suppose there exists a constant $C$ such that the estimate
\begin{equation*}
	\norm{ u }_B \leq C \norm{ L_tu }_V
\end{equation*}
holds independent of $\tau$.  Then $L_1$ maps $B$ onto $V$ if and only if $L_0$ maps $B$ onto $V$.
\end{thm}
For a proof of the method of continuity we refer the reader to \cite[pg. 75]{GT83}.
\begin{prop}
Consider the following initial value problem:
\begin{equation}\label{e: linear IVP}
	\begin{cases}
		\p_t U^a + (-1)^m \sum_{ \abs{I} \leq 2m } A_b^{ a I }(X) \nabla_{ I } U^b = F^a(X), \quad X \in M_{\omega}  \\
		U( \cdot ,0) = U_0.
	\end{cases}
\end{equation}
Suppose that the following conditions are satisfied:
\begin{enumerate}
	\item  The coefficients $A_{b}^{a i_1 j_1 \cdots i_m j_m }$ satisfy the symmetry condition $A_{b}^{a i_1 j_1 \cdots i_m j_m } =  A_{ a }^{ b j_1 i_1 \cdots j_m i_m }$
	\item The leading coefficient satisfies the Legendre-Hadamard condition with constant $\lambda$ \\
	\item There exists a uniform constant $\Lambda < \infty$ such that $\sum_{ \abs{I} \leq 2m } \abs{A^I}_{\alpha} + \abs{F}_{ \alpha } \leq \Lambda$,
\end{enumerate}
Then problem \eqref{e: heat IVP} has a unique solution $U \in C^{2m,1,\alpha}(E_{\omega})$.
\end{prop}
\begin{proof}
As always, we may assume without loss of generality that $U_0 = 0$.  Define the operators
\begin{align*}
	&L_0 = \p_t U^a + (-\Delta^m)U^a \\
	&L_1 = \p_t U^a + (-1)^m \sum_{ \abs{I} \leq 2m } A_b^{ a I }(X) \nabla_{ I } U^b.
\end{align*}
Consider the family of equations
\begin{equation*}
	L_{\tau} := (1 - \tau)L_0  U + \tau L_1 U = F,
\end{equation*}
where $\tau$ is a parameter with $\tau \in [0, 1]$.  The operator $L_{\tau}$ satisfies the assumption of the theorem with $\lambda_{\tau}$ and $\Lambda_{\tau}$ taken as $\lambda_{\tau} = \min \{1, \lambda \}$ and $\lambda_{\tau} = \max \{1, \Lambda \}$.  Suppose that $U_{\tau}$ is a solution to \eqref{e: linear IVP}.  Then in exactly the same way as for the heat equation, using $L^2$ regularity, the same short contradiction argument and the global Schauder estimate we obtain the estimate
\begin{equation*}
	\abs{ U_{ \epsilon } }_{ 2m, 1, \alpha; \, E_{\omega} } \leq C\abs{ F }_{\alpha; \, E_{\omega} },
\end{equation*}
where $C$ is independent of $\tau$.  We may now apply the method of continuity, and since $L_0$ is solvable by Theorem \ref{p: heat IVP}, $L_1$ is also solvable.
\end{proof}

\subsection{Nonlinear existence theory}
With all the linear existence theory now in place, we are ready to prove Main Theorem \ref{mthm: main thm 1}.  We do this by an appliation of the inverse function theorem in Banach spaces.

\begin{thm}[inverse function theorem]\label{t: ift}
Let $X$ and $Y$ be Banach spaces, $P : X \rightarrow Y$ a map from $X$ to $Y$, and $U_0$ and element of $X$.  Suppose that $P$ satisfies the following:
	\begin{enumerate}
		\item $P$ is continuously differentiable at $U_0$
		\item The Fr{\'e}chet derivative of $P$ at $U_0$ is invertible.
	\end{enumerate}
Then there exists an open neighbourhood $\mathcal{U}$ of $U_0$ in $X$, and an open neighbourhood $\mathcal{V}$ of $V := P[U_0]$ in $Y$ such that $P : \mathcal{U} \rightarrow \mathcal{V}$ is an isomorphism.
\end{thm}
For a detailed proof of the inverse function theorem we recommend to the reader \cite[pg. 215]{AE}.

\begin{proof}[Proof of Main Theorem \ref{mthm: main thm 1}]
Let $X = C^{2m,1,\beta}(E_{\omega})$ and $Y = C^{\beta}(E_{\omega})$, where $\beta < \alpha$.  We consider the nonlinear operator $P$ as a map $P : X \rightarrow Y$.  To begin, linearise the nonlinear operator $P$ at the initial value $U_0$.  The linearisation of $P$ about $U_0$ in the direction $V$ is a linear system in the unknown $V$ which uniquely solvable by the Schauder theory presented in the previous section. Call $U_l$ the solution to this linear system.  From the Schauder theory we also know $U_l \in C^{2m,1,\alpha}(E_{\omega})$.  Now linearise $P$ about the solution to the linear problem $U_l$.  Next we confirm that the conditions of the inverse function theorem hold for the nonlinear operator $P$ about $U_l$.  The (G\^ateaux) derivative of $P$ at $U_l$ in the direction $V$ is given by
\begin{align}
\notag	P'(U_l)V &= \frac{ \p }{ \p s } F(U_l + sV) \Big|_{s=0} \\
	\begin{split}
		&= \p_t V - \Fdot^{ i_1, \ldots, i_{2m} }( x, t, U_l, \nabla U_l, \ldots, \nabla^{2m}_{ i_1, \ldots, i_{2m} } U_l) \nabla_{ i_1 } \cdots \nabla_{ i_{2m} }V \\
		&\quad \cdots - \Fdot^k(x, t, U_l, \nabla_k U_l, \ldots, \nabla^{2m} U_l) \nabla_k V - F( x, t, U_l, \nabla U_l, \ldots, \nabla^{2m} U_l )V.
	\end{split}\label{e: Gateaux}
\end{align}
The regularity assumptions in the statement of the theorem ensure that $P$ is continuously differentiable and Fr\'echet differentiable.  We have
\begin{align*}
	&\abs{ P(U_l +V) - P(V) - P'(U_l) }_{ 0, 0, \alpha; \, E_{\omega} } \\
	&\quad = \Big| \Big( \int_0^1 P'(U_l + sV) - P'(U_l) \, ds \Big)V \Big|_{ 0, 0, \alpha; \, E_{\omega} } \\
	&\quad \leq \norm{ P'(U_l + sV) - P'(U_l) }_{ \mathcal{L} ( C^{2m,1,\alpha}(E_{\omega}), C^{0,0,\alpha}(E_{\omega}) ) } \abs{ V }_{2m,1,\alpha; \, E_{\omega} } \\
	&\quad = o(  \abs{ V }_{2m,1,\alpha; \, E_{\omega} } ).
\end{align*}
Because $\Fdot$ is continuous in all its arguments,
\begin{equation*}
	\norm{ P'(U_l + sV) - P'(U_l) }_{ \mathcal{L}(C^{2m,1,\alpha}(E_{\omega}), C^{0,0,\alpha}(E_{\omega}) } \rightarrow 0 \quad \text{as} \quad s \rightarrow 0
\end{equation*}
and the last line above follows.  This shows that $P$ is Fr\'echet differentiable at $U_l$.  The linearisation of $P$ about $U_l$ in the direction $V$ is again a linear system in the unknown $V$ that is uniquely solvable by the Schauder theory, and thus the Fr\'echet derivative of $P$ is invertible at $U_l$.

The inverse function theorem applies and guarantees an open neighbourhood $\mathcal{U}$ of $U_l$ in $X$, and an open neighbourhood $\mathcal{V}$ of $P[U_l]$ in $Y$, such that $P : U \rightarrow V$ is an isomorphism.  For convenience, set $f_l(t) := P[U_l]$.  Define the function $f_{ \chi }(t) := \chi(t)f_l(t)$, where $\chi(t)$ is a smooth cutoff function with the properties $\chi(t) = 0$ for $t < t_{\epsilon}/2$ and $\chi(t) = 1$ for $t \geq t_{\epsilon}$, and $t_{\epsilon}$ is small number to be fixed sufficiently small.  We claim for $t \in [0, t_{\epsilon})$ where $t_{ \epsilon }$ is sufficiently small, that $f_{ \chi }$ is in $\mathcal{V}$.Beginning with the supremum estimate, if $t \geq t_{\epsilon}$, then $\abs{ f_l - f_{\chi} }_0 = 0$.  For $t \leq t_{ \epsilon }$ we use the crucial fact that since $U_l$ is the solution to the linear problem, $f_l$ satisfies $f_l(0) = 0$:
\begin{align*}
		\abs{ f_ l(t) - f_{ \chi }(t) } &= \abs{ f_ l(t) - f_{ \chi }(t) } - \abs{ f_ l(0) - f_{ \chi }(0) } \\
		&\leq [ f_ l - f_{ \chi } ]_{ \alpha } t_{ \epsilon }^{ \frac{ \alpha }{ 2m } },
	\end{align*}
and so $\abs{ f_l - f_{ \chi } }_0 \leq C t_{ \epsilon }^{ \alpha/(2m) }$ since $U_l$ is H\"{o}lder continuous.  For the H\"{o}lder estimate we consider two cases.  We may assume without loss of generality that $t > s$.  If $\abs{ t - s } < t_{\epsilon}$, then
we need to consider two further subcases: 1) $s < t_{\epsilon}$ with $0 < s < t < 2t_{\epsilon}$; and 2) $t,s \geq t_{\epsilon}$.  In the first subcase we begin estimating
\begin{align*}
	\abs{ f_ l(t) - f_{ \chi }(t) - ( f_ l(s) - f_{ \chi }(s) ) } &\leq  \abs{ f_ l(t) - f_l(s) } + \abs{ f_{\chi}(t) - f_{ \chi }(s) } \\
	&\leq \big( [ f_l ]_{ \alpha } + \abs{ \chi }_0 [ f_l ]_{\alpha} + [ \chi ]_{\alpha}\abs{ f_l}_0 \big) \abs{ t -s }^{ \frac{ \alpha }{ 2m } }.
\end{align*}
The second term on the right is easy to deal with, since $\abs{ \chi }_0 \leq 1$.  To deal with the second, we note that in this case we can estimate
\[ [ \chi ]_{\alpha}\abs{ f_l}_0 \leq \frac{ C }{ (t_{\epsilon}/2)^{\alpha/(2m)} } \cdot (2t_{\epsilon})^{\alpha/(2m)}. \]
Combining estimates we see
\[ \abs{ f_ l(t) - f_{ \chi }(t) - ( f_ l(s) - f_{ \chi }(s) ) } \leq  C \abs{ t - s }^{ \frac{ \beta }{2m} } t_{\epsilon}^{ \frac{ \alpha - \beta }{2m} } \]
where $\beta < \alpha$, and so $[ f_l - f_{ \chi } ]_{\beta} \leq Ct_{\epsilon}^{ (\alpha - \beta)/(2m)}$.  The second subcase is easy, since if $t, s \geq t_{\epsilon}$, then $\abs{ f_{\chi}(t) - f_{ \chi }(s) } = \abs{ f_ l(t) - f_l(s) }$.  To treat the second main case, namely if $\abs{ t - s } \geq t_{\epsilon}$, then
\begin{align*}
	\abs{ f_ l(t) - f_{ \chi }(t) - ( f_ l(s) - f_{ \chi }(s) ) } &\leq \abs{ f_ l(t) - f_{ \chi }(t) } + \abs{ f_ l(s) - f_{ \chi }(s) } \\
	&\leq 2 \abs{ f_l - f_{ \chi } }_0 \\
		&\leq 2C t_{ \epsilon }^{ \frac{ \alpha }{ 2m } } \\
	&\leq C\abs{ t - s }^{ \frac{ \beta }{2m} } t_{ \epsilon }^{ \frac{ \alpha - \beta }{2m} }.
\end{align*}
Therefore $\abs{ f_l - f_{ \chi } }_{\beta}$ can be made arbitrarily small on small time intervals, and so we can fix $t_{\epsilon}$ sufficiently small so that for all $t \in [0, t_{ \epsilon })$, $f_{ \chi }$ is in $V$.  By the inverse function theorem there exists a unique element $U_{\chi} \in X$ such that $P[U_{\chi}] = f_{ \chi }$, and moreover, for $t < t_{\epsilon}/2$, $P[U_{\chi}] = 0$.  Thus the element $U_{\chi}$ is the unique solution to the initial value problem \eqref{eqn: nonlinear exist problem} for some short-time $t_{\epsilon}/2$ and the proof is complete.
\end{proof}

\section{Short-time existence for the mean curvature flow}

Here we apply the nonlinear existence theory espoused in the previous section to give a proof of short time existence of the mean curvature flow.  In this section we denote the mean curvature flow, considered as a differential operator, by $M$, and the mean-curvature-DeTurck flow by $MD$.  We begin by showing that mean curvature flow is only a weakly parabolic quasilinear system, and as such we cannot immediately apply the `standard' theory.  With respect to the induced metric the Laplacian of $F$ is just
\begin{align*}
	\Delta_g F &= g^{ij} \nabla_i \nabla_j F \\
	&= g^{ij} \left( \frac{ \p^2 F }{ \p x^i \p x^j } - \Gamma_{ij}^k \frac{ \p F }{ \p x^k } \right) \\ 
	&= g^{ij}h_{ij} \\
	&= H.
\end{align*}
The mean curvature flow equation can therefore be written as
\begin{equation*}
	\frac{ \p }{ \p t } F = \Delta_g F.
\end{equation*}
The similarity is however deceptive:  The induced metric is evolving in time, and this adds extra terms to the principal symbol that result in the presence of zeroes.  The principal symbol can be computed by
\begin{align*}
	\Delta_g F^a &= g^{ij} \left( \frac{ \p^2 F^a }{ \p x^i \p x^j } - \Gamma_{ij}^k \frac{ \p F^a }{ \p x^k } \right) \\
	&= g^{ij} \left( \frac{ \p^2 F^a }{ \p x^i \p x^j } - \frac{1}{2} g^{kl} \left( \frac{ \p }{ \p x^i }g_{jl} +  \frac{ \p }{ \p x^j }g_{il} - \frac{ \p }{ \p x^l }g_{ij} \right) \frac{ \p F^a }{ \p x^k } \right) \\
	&=  g^{ij} \frac{ \p^2 F^a }{ \p x^i \p x^j } - \frac{1}{2} g^{ij} g^{kl} \frac{ \p F^a }{ \p x^k } \left( \frac{ \p^2 F^b }{ \p x^i \p x^j }\frac{ \p F^b }{ \p x^l  } + \frac{ \p F^b }{ \p x^j  } \frac{ \p^2 F^b }{ \p x^i \p x^l } + \cdots \right) \\
	&= g^{ij} \frac{ \p^2 F^a }{ \p x^i \p x^j } - g^{ij}g^{kl} \frac{ \p F^a }{ \p x^k } \frac{ \p F^b }{ \p x^l } \frac{ \p^2 F^b }{ \p x^i \p x^j } \\
	&= g^{ij} \left( \delta_b^a - g^{kl} \frac{ \p F^a }{ \p x^k } \frac{ \p F^b }{ \p x^l } \right) \frac{ \p^2 F^b }{ \p x^i \p x^j }.
\end{align*}
Observe that the term $g^{kl} \frac{ \p F^a }{ \p x^k } \frac{ \p F^b }{ \p x^l }$ is the orthogonal projection onto the tangent space of the submanifold: for any $\xi \in T\mathbb{R}^{n+k}$,
\begin{align*}
	\pi_{T\Sigma}(\xi) &= g^{kl} \big\langle \xi, F_*\p_k \big\rangle F_*\p_l \\
		&= g^{kl}\xi^{a} \frac{ \p F^{a} }{ \p x^k }\frac{ \p F^{b} }{ \p x^l } \frac{ \p }{ \p y^b }.
\end{align*}
To examine the principal symbol, without loss of generality we may assume at a point that $g_{ij} = \delta_{ij}$ and also that $\abs{\xi} = 1$, so we can choose $\xi_1 = 1$ and $\xi_i = 0$ for $i \geq 2$.  The principal symbol is thus
\begin{align*}
	\hat\sigma[M](\xi) &= \abs{\xi}^2 \big( Id - \pi_{T\Sigma}(\xi) \big) \\
	&= \abs{ \xi }^2 \pi_{N\Sigma}(\xi),
\end{align*}
which is zero if $\xi \in T\Sigma$.  Another way to see that the mean curvature flow is only weakly parabolic is to observe from the start that the equation is degenerate in tangential directions.  We have just computed that
\begin{equation*}
	 H = g^{ij} \left( \frac{ \p^2 F }{ \p x^i \p x^j } - \Gamma_{ij}^k \frac{ \p F }{ \p x^k } \right),
\end{equation*}
so the mean curvature flow can also be written as
\begin{equation*}
	\frac{ \p }{ \p t } F = \pi_{N\Sigma} \left( g^{ij} \frac{ \p^2 F }{ \p x^i \p x^j } \right).
\end{equation*}
For any $\xi \in \mathbb{R}^{n+k}$,
\begin{align*}
	\pi_{N\Sigma}(\xi) &= \xi - \pi_{T\Sigma} (\xi) \\
	&= \xi^a \frac{ \p }{ \p y^a } - g^{kl}\xi^{a} \frac{ \p F^{a} }{ \p x^k }\frac{ \p F^{b} }{ \p x^l } \frac{ \p }{ \p y^b },
\end{align*}
so again we find the mean curvature flow is given by
\begin{equation*}
	\frac{ \p F^a }{ \p t } = g^{ij} \left( \delta_b^a - g^{kl} \frac{ \p F^a }{ \p x^k } \frac{ \p F^b }{ \p x^l } \right) \frac{ \p^2 F^b }{ \p x^i \p x^j }.
\end{equation*}

The mean curvature flow is therefore not strongly parabolic and the almost standard parabolic theory cannot immediately be conjured to yield existence for a short time.  To overcome this difficulty we are going to adapt a variant of the DeTurck trick first elaborated by Hamilton \cite{rH93b} that combines the mean curvature-DeTurck flow and the harmonic map heat flow.  As the next proposition shows, the mean curvature flow is invariant under a tangential parametrisation.  This means that adding a tangential term to the mean curvature flow equation results in a solution that differs from the solution of the mean curvature flow itself only by a reparametrisation of the submanifold.  The DeTurck trick involves adding a tangential term to the mean mean curvature flow to break the geometric invariance of the equation.  The modified flow is then strongly parabolic and the almost standard parabolic theory can now be summoned to ensure short time existence.  The solution to the mean curvature flow is then recovered from the solution to the mean curvature-DeTurck flow.  Hamilton's coupling of the modified flow with the harmonic map flow serves to provide a simple proof of uniqueness.

\begin{prop}\label{prop: DeTurck trick}
Let $W$ be a time-dependent family of vector fields defined on $\Sigma \times [0, T)$.  Suppose that $F$ is a solution to
\begin{equation*}
\begin{cases}
	\frac{ \p F}{ \p t } = \Delta_g F + \nabla_W F \\
	F( \cdot, 0) = F_0.
\end{cases}
\end{equation*}
Then there exists a solution $\tilde F$ to the mean curvature flow with $\tilde F_0 = F_0$.
\end{prop}
\begin{proof}
For the moment, assume that there exists a time-dependent family of diffeomorphisms $\varphi_t : \Sigma \times [0,T) \rightarrow \Sigma$. Computing in local coordinates $\{ x^k \}$ around $\varphi_t(p)$ we calculate
\begin{align*}
	\frac{ \p \tilde F }{ \p t }(p,t) &= \frac{ \p F }{ \p t }(\varphi_t(p),t) + \nabla_k F (\varphi_t(p),t) \cdot \frac{ d \varphi_t(p)^k }{ dt } \\
	&= \Delta_g F (\varphi_t(p),t) + \left( W^k(\varphi_t(p),t) + \frac{ d\varphi_t(p)^k }{ dt } \right) \nabla_k F(\varphi_t(p),t) \\
	&= \tilde \Delta_{ \tilde g } \tilde F(p,t) + \left( W^k(\varphi_t(p),t) + \frac{ d\varphi_t(p)^k }{ dt } \right) \nabla_k F (\varphi_t(p),t).
\end{align*}
Therefore, if we can show there exists a family of diffeomorphisms solving the initial value problem
\begin{equation*}\label{e: DeTurck ODE problem}
	\begin{cases}
		\frac{ d \varphi_t(p) }{ d t } = - W(\varphi_t(p),t) \\
		\varphi_0(p) = \text{id}_{ \Sigma },
	\end{cases}
\end{equation*}
then $\tilde F$ will be the desired solution to the mean curvature flow.  In the case that $\Sigma$ is compact, standard ODE theory (for example, see \cite{Lee}) guarantees that the above ODE problem has a unique solution for as long as $W$ is defined.
\end{proof}

Let us now continue with Hamilton's argument.  Fix a background connection $\bar{ \nabla }$ on $\Sigma$.  For example, we could take the induced connection on $\Sigma$ at $t=0$.  As the vector field $W$ in the above proposition we take $W := g^{ij} ( \Gamma_{ij}^k - \bar{\Gamma}_{ij}^k )$.  Consider the mean curvature-DeTurck flow given by
\begin{align*}
	\frac{ \p F^a }{ \p t } &= \Delta_g F^a + \nabla_W F^a \\
	&= g^{ij} \left( \frac{ \p^2 F^a }{ \p x^i \p x^j } - \Gamma_{ij}^k \right) + g^{ij} ( \Gamma_{ij}^k - \bar{\Gamma}_{ij}^k ) \frac{ \p F^a }{ \p x^k } \\
	&= g^{ij} \left( \frac{ \p^2 F^a }{ \p x^i \p x^j } - \bar{\Gamma}_{ij}^k \frac{ \p F^a }{ \p x^k } \right).
\end{align*}
The principal symbol is now
\begin{equation*}
	\hat\sigma[MD](\xi) = \abs{\xi}^2 \text{id},
\end{equation*}
so the mean curvature-DeTurck flow is strongly parabolic and Main Theorem \ref{mthm: main thm 1} guarantees a unique solution to this modified flow for a least some short time.  The conditions of Main Theorem \ref{mthm: main thm 1} are easily confirmed for the mean curvature-DeTurck flow.  For example, the leading term of the linearised operator in some direction $V$ is given by
\begin{align*}
	\frac{ \p V^a }{ \p t} &= g^{ij}\frac{ \p^2 V^a }{ \p x^i \p x^j } \\
	&= g^{ij}\delta_b^a\frac{ \p^2 V^b }{ \p x^i \p x^j }.
\end{align*}
Hence $A_b^{a \, ij} = g^{ij}\delta_b^a$, and $A_b^{a \, ij} = A_a^{b \, ji}$.  As the mean curvature-DeTurck flow possesses a unique solution for some short time, the family of vector fields $W(t)$ also exist on this short time interval, and the above ODE problem has a unique solution on the same time interval.  By Proposition \ref{prop: DeTurck trick}, we recover a solution to the mean curvature flow $\tilde F$ by pulling-back the solution of the mean curvature-DeTurck flow by the diffeomorphism $\varphi_t$, that is $\tilde F(p,t) = \varphi_t^* F(p,t) = F(\varphi_t(p), t)$.

We now show uniqueness of the above solution to the mean curvature flow.  Suppose that $\tilde F$ is a solution the mean curvature flow and denote associated the induced metric by $\tilde g$.  Let $\varphi_0 : \Sigma \rightarrow \Sigma$ be a diffeomorphism.  Fix a metric $\bar g$ and associated Levi-Civita connection on the target manifold $\Sigma$, and consider the harmonic map heat flow
\begin{equation*}
	\frac{ \p }{ \p t } \varphi = \Delta_{ \tilde g, \bar g } \varphi
\end{equation*}
with respect to the domain metric $\tilde g$ and the target metric $\bar g$.  The harmonic map heat flow is a strongly parabolic quasilinear system (see, for example, \cite{rH75} or \cite{rH89}) and thus possesses a unique solution $\varphi_t$ for at least some short time.  We now define $F := \varphi_{t*} \tilde F = (\varphi_t^{-1})^*\tilde F$ and claim this is a solution to the mean curvature-DeTurck flow.  Repeating the calculation in Proposition \ref{prop: DeTurck trick} shows
\begin{equation*}
	\frac{ \p F }{ \p t } = \Delta_g F + \nabla_V F,
\end{equation*}
where $V(p) = -\varphi_*\p_t |_{ \varphi^{-1}(p) }$. Thus if we can show that $V = W$ then this establishes the claim.  This follows from following result:
\begin{prop}
Suppose $(K^n,k)$, $(M^n, g)$ and $(N^m, h)$ are manifolds, $\psi : K \rightarrow M$ a diffeomorphism and $\varphi : M \rightarrow N$ a map.  Then
\begin{equation*}
	\Delta_{\psi^*g,h}(\varphi \circ \psi)(p) = \Delta_{g,h} \varphi(\psi(p)).
\end{equation*}
\end{prop} 
The geometric meaning of this proposition is that the harmonic map Laplacian from a domain manifold to a target manifold is unchanged if we reparametrise the domain manifold.  For a proof of this proposition we refer the reader to \cite[pg. 117]{CLN} or \cite[pg. 78]{HA}.  We have $V(p) =-\varphi_*\p_t |_{ \varphi^{-1}(p) } = -\Delta_{g, \bar g}\varphi (\varphi^{-1}(p)) = \Delta_{g, \bar g} \text{id}_\Sigma$, and then adapting the above proposition to our setting $(\psi = \varphi^{-1}, h = \bar g)$ we see $\Delta_{g, \bar g} \text{id}_\Sigma = g^{ij} ( \Gamma_{ij}^k - \bar \Gamma_{ij}^k )$, and so the two vector fields $V$ and $W$ are in fact identical.

We can now finish the uniqueness argument.  Suppose that there exist two solutions $\tilde F_i$, $i=1,2$ to the mean curvature flow with initial condition $\tilde F_1( \cdot, 0) = \tilde F_2( \cdot, 0)$.  For each domain metric $\tilde g_i$ we can solve uniquely the harmonic map heat flow problem
\begin{equation*}
\begin{cases}
	\frac{ \p }{ \p t } \varphi = \Delta_{ \tilde g_i, \bar g } \varphi \\
	\varphi(\cdot, 0) = \text{id}_{\Sigma}
\end{cases}
\end{equation*}
for the functions $\varphi_i$, which then give two solutions $F_i = \varphi_*\tilde F$ to the mean curvature-DeTurck flow.  Because these two solutions satisfy the same initial condition and solutions to the mean curvature-DeTurck flow are unique, $F_1 = F_2$.  The two diffeomorphisms $\varphi_i(t)$ also solve the same ODE problem
\begin{equation*}
	\begin{cases}
		\frac{ d \varphi_{i} }{ d t } = -W(\varphi_i(p,t),t) \\
		\varphi_i(p,0) = p.
	\end{cases}
\end{equation*}
and so they too are in fact equal on their common interval of existence.  Therefore $\tilde F_1 = \varphi_1^*F_1 = \varphi_2^*F_2 = \tilde F_2$, which concludes the proof of uniqueness.

\chapter{Submanifolds of Euclidean space}\label{ch: The flow of submanifolds of Euclidean space}

Our goal in this chapter is to prove Main Theorem 2:
\begin{mthm}\label{mthm: main thm 2}
Suppose $\Sigma_0=F_0(\Sigma^n)$ is a closed submanifold smoothly immersed in $\mathbb{R}^{n+k}$.  If $\Sigma_0$ satisfies $\abs{ H }_{ \text{min} } > 0$ and $\abs{h}^2 \leq c\abs{H}^2$, where
	\begin{equation*}
		c \leq \begin{cases}\frac{4}{3n},& \quad \text{ if }2\leq n \leq 4 \\
						\frac{1}{n-1},& \quad \text{ if }n \geq 4,
						\end{cases}
	\end{equation*}
then MCF has a unique smooth solution $F:\ \Sigma\times[0,T)\to \mathbb{R}^{n+k}$ on a finite maximal time interval, and the submanifolds $\Sigma_t$ converge uniformly to a point $q\in \mathbb{R}^{n+k}$ as $t\to T$.  A suitably normalised flow exists for all time, and the normalised submanifolds $\tilde{\Sigma}_{\tilde{t}}$ converge smoothly as $\tilde{t} \rightarrow \infty$ to a $n$-sphere in some $(n+k)$-subspace of $\mathbb{R}^{n+k}$.
\end{mthm}

\section{The evolution equations in high codimension}
We begin by deriving evolution equations for various geometric quantities; of particular importance are the evolution equations for $\abs{ h }^2$ and $\abs{ H }^2$.  The mean curvature flow amounts to the prescription $F_*\partial_t=\iota H$ in the notation of the previous chapter.  For the moment we allow the background space $N$ to be an arbitrary Riemannian manifold.  The timelike Codazzi identity \eqref{eq:timelikeCodazzi} is precisely the evolution equation of the second fundamental form under the mean curvature flow:
\begin{equation}\label{eq:evol_h_in_N}
\nabla_{\partial_t} h(u,v) = \nabla_u\nabla_vH+h(v,{\mathcal W}(u,H))+\pip\left(\bar R(F_*u,\iota H)F_*v\right),
\end{equation}
or with respect to arbitrary local frames for the tangent and normal bundles
$$
\nabla_{\partial_t}h_{ij} = \nabla_i\nabla_jH+H\cdot h_{ip} h_{pj}
+H^\alpha\bar R_{i\alpha j}{}^\beta\nu_\beta.
$$
Using Simons' identity \eqref{eq:SimonsTrace}, this converts to a reaction-diffusion equation
\begin{align*}
\nabla_{\partial_t}h_{ij} &= \Delta h_{ij}+h_{ij}\cdot h_{pq}h_{pq}+h_{iq}\cdot h_{qp}h_{pj}
+h_{jq}\cdot h_{qp}h_{pi}-2h_{ip}\cdot h_{jq} h_{pq}\notag\\
&\qquad +2\bar R_{ipjq}h_{pq}-\bar R_{kjkp}h_{pi}-\bar R_{kikp}h_{pj}+h_{ij\alpha}\bar R_{k\alpha k\beta}\nu_\beta\notag\\
&\qquad -2h_{jp\alpha}\bar R_{ip\alpha\beta}\nu_\beta-2h_{ip\alpha}\bar R_{jp\alpha\beta}\nu_\beta
+\bar\nabla_k\bar R_{kij\beta}\nu_\beta-\bar\nabla_i\bar R_{jkk\beta}\nu_\beta.\label{eq:reactdiffuse}
\end{align*}
For the remainder of this chapter we are concerned only with the case $N=\mathbb{R}^{n+k}$, in which case the equation becomes
\begin{equation}\label{eq:reactdiffusEuclidean}
\nabla_{\partial_t}h_{ij} = \Delta h_{ij}+h_{ij}\cdot h_{pq}h_{pq}+h_{iq}\cdot h_{qp}h_{pj}
+h_{jq}\cdot h_{qp}h_{pi}-2h_{ip}\cdot h_{jq} h_{pq}.
\end{equation}
Taking the trace with respect to $g$ we obtain an evolution equation for the mean curvature vector:
\begin{equation}\label{eq:evolveH}
\nabla_{\partial_t}H=\Delta H + H\cdot h_{pq}h_{pq}.
\end{equation}
To derive the evolution equation for $\abs{ h }^2$, first recall that $\nabla_t g = 0$, and then at a point we compute
\begin{align*}
	\p_t \abs{ h }^2 &= \p_t \langle h, h \rangle \\
	&= 2 \langle \nabla_t h_{ij}, h_{ij} \rangle \\
	&= 2 \langle \Delta h_{ij}+h_{ij}\cdot h_{pq}h_{pq}+h_{iq}\cdot h_{qp}h_{pj}
+h_{jq}\cdot h_{qp}h_{pi}-2h_{ip}\cdot h_{jq} h_{pq}, h_{ij} \rangle.
\end{align*}
We now use $\Delta \abs{ h }^2 = 2 \langle \Delta h_{ij}, h_{ij} \rangle + 2 \abs{ \nabla h }^2$, and then noting that three of the reaction terms factor into the normal curvature we obtain
\begin{equation*}
	\frac{ \p }{ \p t} \abs{h}^2 = \Delta\abs{h}^2 - 2\abs{\nabla h}^2 + 2 \sum_{\alpha, \beta} \Big( \sum_{i,j} h_{ij\alpha}h_{ij\beta} \Big)^2 +2 \sum_{i,j,\alpha,\beta} \Big( \sum_p h_{ip\alpha}h_{jp\beta} - h_{jp\alpha}h_{ip\beta} \Big)^2 \label{eq:evol_h2}.
\end{equation*}
Similarly, using equation \eqref{eq:evolveH} the evolution for $\abs{ H }^2$ is given by
\begin{equation*}
			\frac{\partial}{\partial t}\abs{H}^2 = \Delta\abs{H}^2 - 2\abs{\nablap H}^2 + 2\sum_{i,j} \Big( \sum_{\alpha} H_{\alpha}h_{ij\alpha}\Big)^2.\label{eq:evol_H2}
\end{equation*}
The last term in \eqref{eq:evol_h2} is the length squared of the normal curvature, which we denote by $\abs{ \Rp }^2$.  For convenience we label the reaction terms of the above evolution equations as follows:
\begin{gather*}
	R_1 = \sum_{\alpha, \beta} \Big(\!\sum_{i,j} h_{ij\alpha}h_{ij\beta}\!\Big)^2 + \abs{\Rp}^2 \\
	R_2 = \sum_{i,j}\!\Big(\!\sum_{\alpha} H_{\alpha}h_{ij\alpha}\Big)^2.
\end{gather*}

The special connections we have been using are especially convenient for deriving the evolution equations in high codimension.  This will become quite evident when we come to deriving the higher derivative estimates.  Of course the special connections do not have to be used, and the methods used in the hypersurface theory can still be applied.  Let us see how some of this works in high codimension.  Since the ambient metric is fixed, the evolution of the induced metric can be computed by
\begin{align*}
	\frac{ \p }{ \p t } g_{ij} &= \p_t \big\langle \p_i F, \p_j F \big\rangle \\ 
	&= \big\langle \p_i ( H^{ \alpha } \nu_{ \alpha } ), \p_j F \big\rangle + ( i \leftrightarrow j ),
\end{align*}
then using the Weingarten relation: $\p_i \nu_{ \alpha } = C_{i \alpha }^{ \beta } \nu_{ \beta }- h_{ip\alpha}g^{pq}\p_qF$ and noting which terms are orthogonal to each other we have
\begin{align*}
	\frac{ \p }{ \p t } g_{ij} &= - \big\langle H^{\alpha}( h_{ip\alpha}g^{pq}\p_qF ), \p_j F \big\rangle + ( i \leftrightarrow j ) \\
	&= -2 H^{ \alpha } h_{ij\alpha} \\
	&= -2 H \cdot h_{ij}.
\end{align*}
To easily derive further evolution equations in this way it becomes necessary to compute in a suitably chosen evolving local frame for the normal bundle.  Since the normal bundle of a hypersurface in one-dimensional, any rotation of the normal bundle is necessarily tangential. In arbitrary codimension however, the normal vectors may `twist' inside the normal bundle giving possibly both tangential and normal motion.  Here we have
\begin{align*}
	\frac{d}{dt}\nu_{\alpha} &= \big\langle \p_t\nu_{\alpha}, \p_p F \big\rangle g^{pq}\p_q F + \big\langle \p_t\nu_{\alpha}, \nu_{\gamma} \big\rangle \nu_{\gamma}  \\
		&= -\big\langle \nu_{\alpha}, \p_t \p_p F \big\rangle g^{pq}\p_q F + \big\langle \p_t\nu_{\alpha}, \nu_{\gamma} \big\rangle \nu_{\gamma}  \\
		&= -\big\langle \nu_{\alpha}, \nabla^{\bot}_p H \big\rangle g^{pq}\p_q F + \big\langle \p_t\nu_{\alpha}, \nu_{\gamma} \big\rangle \nu_{\gamma}.
\end{align*}

Observe that mean curvature flow of the submanifold only imposes the tangential motion of the normal frame and so we are free to choose the normal motion.  A convenient choice is of course that there is no normal motion.

\begin{lem}
Let $\nu_{\alpha}(0)$, $n+1 \leq \alpha \leq n+p$, be a local orthonormal frame for the normal bundle and define the evolution of the frame by
	\begin{equation*}
		\frac{d}{dt}\nu_{\alpha}(t) = -\big\langle \nu_{\alpha}, \nabla^{\bot}_p H \big\rangle g^{pq} \p_q F.
	\end{equation*}
Then $\nu_{\alpha}(t)$ remains a local orthonormal frame for the normal bundle as long MCF has a solution.
\end{lem}
\begin{proof}
We first note that the evolution of the frame is determined by an linear system of ODE's and hence has a unique solution as long as MCF has a solution.  To show that the frame remains normal we compute
	\begin{align*}
		\frac{d}{dt}\big\langle \p_k F, \nu_{\alpha}(t) \big\rangle &= \big\langle \p_t \p_k F, \nu_{\alpha}(t) \big\rangle + \big\langle \p_k F, \p_t \nu_{\alpha}(t) \big\rangle \\
		&= \big\langle \p_k H, \nu_{\alpha}(t) \big\rangle + \big\langle \p_k F, -\big\langle \nu_{\alpha}, \nabla^{\bot}_p H \big\rangle g^{pq} \p_q F \big\rangle\\
		&= \big\langle \nabla^{\bot}_k H - F_* ( \mathcal{W}(\partial_k, H) ), \nu_{\alpha}(t) \big\rangle  -\big\langle \nu_{\alpha}, \nabla^{\bot}_p H \big\rangle g^{pq}g_{kq} \\
		&= - \big\langle F_* ( \mathcal{W}(\partial_k, H) ), \nu_{\alpha}(t) \big\rangle \\
			&= -g^{pq}\big\langle H, h(\p_k,\p_p) \big\rangle \big\langle \p_q F, \nu_{\alpha}(t) \big\rangle.
	\end{align*}
We set $Y_k = \big\langle \p_k F, \nu_{\alpha}(t)\big\rangle$ and $\Lambda_k^q =  -g^{pq}\big\langle H, h(\p_k,\p_p) \big\rangle$, and the last line above now reads
	\begin{equation}\label{e: evol normal frame}
		\frac{d}{dt}Y_k = \Lambda_k^qY_q.
	\end{equation}
Equation \eqref{e: evol normal frame} is a homogenous linear system of ODE's with initial conditions $Y_k(0) = 0$ and $Y'_k(0) = 0$, and thus its unique solution is given by $Y_k(t) = 0$ for all time as long as MCF has a solution.  From this we conclude that if the frame is initially normal then it remains so.  To show that the frame remains orthonormal we easily compute
	\begin{align*}
		\frac{d}{dt}\big\langle \nu_{\alpha}(t), \nu_{\beta}(t) \big\rangle &= \big\langle \frac{ d }{ dt }\nu_{\alpha}(t), \nu_{\beta}(t) \big\rangle + \big\langle \nu_{\alpha}(t), \frac{ d }{ dt } \nu_{\beta}(t) \big\rangle \\
		&= 0.
	\end{align*}
\end{proof}
In the coming sections the reader will note that by using the special connections we avoid needing the evolution equation for the Christoffel symbols.  In high codimension should one wish to commute the usual partial derivative in time with spatial covariant derivatives, it is also necessary to understand how the normal connection forms evolve.  By differentiating the Weingarten relation in time and using the special evolving normal frame one finds the normal connection forms evolve by
\begin{equation*}
		\frac{\partial}{\partial t}C_{k\alpha}^{\beta} = -(\nu_{\alpha} \cdot \nabla^{\bot}_p H) \, h_{kp} \cdot \nu_{\beta} + ( \nu_{\beta} \cdot \nabla^{\bot}_p H ) \, h_{kp} \cdot \nu_{\alpha}.
\end{equation*}
Note that the evolution equations for the Christoffel symbols and the normal connection forms are both of the form $h * \nabla h$.  This information is contained in the temporal Gauss and Ricci equations: they too are of the form $h * \nabla h$ (the usual spatial varieties look like $h * h$).

Another evolution equation we shall need to use on occasion is that of the volume measure.  This is derived in exactly the same manner as for a hypersurface:
\begin{align*}
	\frac{ \p }{ \p t } d\mu_{ g(t) } &= \frac{ \p }{ \p t } \sqrt{ \det g_{ij} } \\
	&= \frac{ 1 }{ 2 \sqrt{ \det g_{ij} } } \det{ g_{ij} } g^{ij} \frac{ \p }{ \p t } g_{ij} \\
	&= - \sqrt{ \det g_{ij} } g^{ij} H \cdot h_{ij} \\
	&= -\abs{H}^2 d\mu_{ g(t) }.
\end{align*}
The evolution equations in case where the background space is a sphere will be needed in the next chapter, and we delay their derivation until then.

\section{Preservation of curvature pinching}\label{l:Preservation of curvature pinching}

In this section we show that a certain curvature pinching condition is preserved by the mean curvature flow.  We will often refer to the next lemma as the Pinching Lemma.

\begin{lem}\label{l: preservation of pinching}
If a solution $F:\ \Sigma\times[0,T)\to\mathbb{R}^{n+k}$ of the mean curvature flow satisfies $\abs{h}^2 +a < c\abs{H}^2$ for some constants $ \alpha \leq \frac1n + \frac{1}{3n}$ and $a>0$ at $t = 0$, then this remains true for all $0 \leq t < T$.
\end{lem}

Note that under the conditions of Main Theorem \ref{mthm: main thm 2} (at least in the case where the inequalities hold strictly), there exist constants $c < \frac{4}{3n} $ and $a > 0$ such that the conditions of Lemma \ref{l: preservation of pinching} hold.  Thus the result implies both that $H$ remains everywhere non-zero, and that the curvature pinching is preserved.  Consider now the quantity $\mathcal Q=\abs{h}^2+a - c\abs{H}^2$, where $c$ and $a$ are positive constants.  Combining the evolution equations for $\abs{h}^2$ and $\abs{H}^2$ we get
\begin{equation}\label{e: h2-cH2}
		\frac{\partial}{\partial t}\mathcal Q = \Delta\mathcal Q - 2(\abs{\nabla h}^2 - c \abs{\nabla H}^2) + 2R_1 - 2 c R_2.
\end{equation}

By assumption this quantity is initially negative.  If there is a first point and time where ${\mathcal Q}$ becomes zero, then at this point we necessarily have $\frac{\partial Q}{\partial t}\geq 0$ and $\Delta{\mathcal Q}\leq 0$.  We will derive a contradiction by showing that the gradient tems on the right-hand side of equation \eqref{e: h2-cH2} are non-positive, whilst the reaction terms are strictly negative.  We begin by estimating the gradient terms:  

\begin{prop}
We have the estimates
	\begin{subequations}
		\begin{align}
			&\abs{\nabla h}^2 \geq \frac{3}{n+2} \abs{\nabla H}^2 \\
			&\abs{\nabla h}^2 - \frac{1}{n}\abs{\nabla H}^2 \geq \frac{2(n-1)}{3n} \abs{\nabla h}^2. \label{eqn: basic grad est 2}
		\end{align}
	\end{subequations}
\end{prop}

\begin{proof}
In exactly the same way as \cite{gH84} and \cite{rH82}, we decompose the tensor $\nabla h$ into orthogonal components $\nabla_i h_{jk} = E_{ijk} + F_{ijk}$,
where
	\begin{equation*}
		E_{ijk} = \frac{1}{n+2}(g_{ij}\nabla_k H + g_{ik}\nabla_j H + g_{jk}\nabla_i H).
	\end{equation*}
Then $\abs{\nabla h}^2 \geq\abs{E}^2 = \frac{3}{n+2} \abs{\nabla H}^2$.  The second estimate follows easily from the first.
\end{proof}

Since $c < \frac{3}{n+2}$ under the assumption of Lemma \ref{l: preservation of pinching}, the gradient terms are non-positive.  In order to estimate the reaction terms of \eqref{e: h2-cH2} it is convenient to work with the traceless part of second fundamental form $\ho=h-\frac1n H g$.  The lengths of $h$ and $\ho$ are related by $\abs{ \ho }^2 =  \abs{h}^2 - \frac{1}{n}\abs{H}^2$.  At a point where ${\mathcal Q}=0$, we certainly have $|H|\neq 0$, so we can choose
a local orthonormal frame  $\{\nu_{\alpha} : \ 1 \leq \alpha \leq k\}$ for ${\mathcal N}$ such that $\nu_1 = H/\abs{H}$.  With this choice of frame the second fundamental form takes the form
\begin{equation*}
	\begin{cases}
		\ho{_1} = h_1 - \frac{\abs{H}}{n} Id \\
		\ho_{\alpha} = h_{\alpha}, \quad \alpha > 1,
	\end{cases}
\end{equation*}
and
\begin{equation*}
	\begin{cases}
		\tr{h_1} = \abs{H} \\
		\tr{h_{\alpha}} = 0, \quad \alpha > 1.
	\end{cases}
\end{equation*}
At a point we may choose a basis for the tangent space such that $h_1$ is diagonal.  We denote the diagonal entries of $h_1$ and $\ho_1$ by $\lambda_i$ and $\lo_i$ respectively.  Additionally, we denote the norm of the $(\alpha \neq 1)$-directions of the second fundamental form by $\abs{\ho_-}^2$, that is, $\abs{\ho}^2 = \abs{\ho_1}^2 + \abs{\ho_-}^2$.  We also adopt from the following piece of notation from \cite{CdCK70}: for a matrix $A = (a_{ij})$, we denote
	\begin{equation*}
		N(A)= \tr \left(A \cdot A^t\right) = \sum_{ij} (a_{ij})^2.
	\end{equation*}
In particular, we have $ \sum_{\alpha,\beta} N(\ho_{\alpha}\ho_{\beta} - \ho_{\beta}\ho_{\alpha}) = \abs{ \Rp }^2$.

To estimate the reaction terms we work with the bases described above and separate the $(\alpha = 1)$-components from the others.  The reaction terms of \eqref{e: h2-cH2} become
\begin{gather*}
	\sum\limits_{\alpha, \beta} \Big( \sum\limits_{i,j} h_{ij\alpha}h_{ij\beta} \Big)^2 = \abs{\ho_1}^4 + \frac{2}{n}\abs{\ho_1}^2\abs{H}^2 + \frac{1}{n^2}\abs{H}^4 + 2\sum\limits_{\alpha > 1} \Big( \sum_{i,j}\ho_{ij1}\ho_{ij\alpha}\Big)^2 \\
	+ \sum\limits_{\alpha, \beta > 1} \Big( \sum\limits_{i,j} \ho_{ij\alpha}\ho_{ij\beta} \Big)^2 \\
	\abs{R^\bot}^2 = 2\sum\limits_{\alpha > 1} N(h_1\ho_{\alpha} - \ho_{\alpha}h_1) + \sum\limits_{\alpha, \beta > 1} N(\ho_{\alpha}\ho_{\beta} - \ho_{\beta}\ho_{\alpha}) \\
	\sum\limits_{i,j} \Big( \sum\limits_{\alpha} H_{\alpha}h_{ij\alpha}\!\Big)^2 = \abs{\ho_1}^2\abs{H}^2 + \frac{1}{n}\abs{H}^4.
	\end{gather*}
Writing out all the reaction terms we now have
		\begin{align}\label{eq:reaction.terms}
			2R_1-2cR_2&=2\sum\limits_{\alpha, \beta}\!\Big(\!\sum\limits_{i,j} h_{ij\alpha}h_{ij\beta}\!\Big)^2 + 2\abs{\Rp}^2 - 2c\sum_{i,j}\!\Big(\!\sum_{\alpha} H_{\alpha}h_{ij\alpha}\Big)^2\notag \\
			&= 2\abs{\ho_1}^4 - 2(c - \frac{2}{n})\abs{\ho_1}^2\abs{H}^2 - \frac{2}{n}(c - \frac{1}{n})\abs{H}^4 \\
			&\qquad\null + 4\sum_{\alpha > 1}\Big(\sum_{i,j}\ho_{ij1}\ho_{ij\alpha}\Big)^2 + 4\sum_{\alpha > 1} N(h_1\ho_{\alpha} - \ho_{\alpha}h_1) \notag\\
			&\qquad\null + 2\sum_{\alpha, \beta > 1}\Big(\sum_{i,j}\ho_{ij\alpha}\ho_{ij\beta}\Big)^2 + 2\sum_{\alpha, \beta > 1}N(\ho_{\alpha}\ho_{\beta} - \ho_{\beta}\ho_{\alpha}).\notag
		\end{align}
Now we use the fact that ${\mathcal Q}=0$ to replace $\left(c-\frac1n\right)\abs{H}^2$ by $\abs{\ho}^2+a$ in the first line of \eqref{eq:reaction.terms}, giving
	\begin{align*}
	&2\abs{\ho_1}^4 - 2(c - \frac{2}{n})\abs{\ho_1}^2\abs{H}^2 - \frac{2}{n}(c - \frac{1}{n})\abs{H}^4\\
			&\qquad= 2\abs{\ho_1}^4 - 2\abs{\ho_1}^2\left(\abs{\ho_1}^2 + \abs{\ho_-}^2+a\right) - \frac{2}{n(c-1/n)}\left(\abs{\ho_-}^2+a\right)\left(\abs{\ho_1}^2 + \abs{\ho_-}^2+a\right)  \\
			&\qquad< -\frac{2c}{c-1/n}\abs{\ho_1}^2\abs{\ho_-}^2-\frac{2}{n(c-1/n)}\abs{\ho_-}^4,
		\end{align*}
where we use the fact that all terms involving $a$ are non-positive, and we have a strictly negative term $-\frac{2a^2}{n(c-1/n)}$.
We need to control the last two lines of \eqref{eq:reaction.terms}.  In the second last line we proceed by expanding the terms and using the fact that $\ho_1$ is diagonal:
\begin{align*}
	\sum_{\alpha > 1}\Big(\sum_{i,j}\ho_{ij1}\ho_{ij\alpha}\Big)^2 &= \sum_{\alpha > 1} \Big( \sum_i\lo_i\ho_{ii\alpha} \Big)^2 \\
	&\leq \Big(\sum_i \lo_i{^2} \Big) \Big(\sum_{\substack{j \\ \alpha > 1}} (\ho_{jj\alpha})^2\Big) \\
	&= \abs{\ho_1}^2 \sum_{\substack{i \\ \alpha > 1}} (\ho_{ii\alpha})^2.
\end{align*}
Also,
\begin{align*}
	\sum_{\alpha > 1} N(h_1\ho_{\alpha} - \ho_{\alpha}h_1) &= \sum_{\substack{i \neq j \\ \alpha > 1}}(\lambda_i - \lambda_j)^2 (\ho_{ij\alpha})^2 \\
	&= \sum_{\substack{i \neq j \\ \alpha > 1}}(\lo_i - \lo_j)^2 (\ho_{ij\alpha})^2 \\
	&\leq \sum_{\substack{i \neq j \\ \alpha > 1}}2(\lo_i{^2} + \lo_j{^2})(\ho_{ij\alpha})^2 \\
	&\leq 2\abs{\ho_1}^2\sum_{\substack{i \neq j \\ \alpha > 1}}(\ho_{ij\alpha})^2 \\
	&= 2\abs{\ho_1}^2 \big(\abs{\ho_-}^2 - \sum_{\substack{i \\ \alpha > 1}} (\ho_{ii\alpha})^2\big),
\end{align*}

so
\begin{align*}
	\sum_{\alpha > 1}\Big(\sum_{i,j}\ho_{ij1}\ho_{ij\alpha}\Big)^2 + \sum_{\alpha > 1} N(h_1\ho_{\alpha} - \ho_{\alpha}h_1) &\leq 2\abs{\ho_1}^2\abs{\ho_-}^2 - \abs{\ho_1}^2\sum_{\substack{i \\ \alpha > 1}}(\ho_{ii\alpha})^2 \\
	&\leq 2\abs{\ho_1}^2\abs{\ho_-}^2.
\end{align*}

To estimate the last line we use an inequality first derived in \cite{CdCK70} for a similar purpose, and later improved \cite{LiLi92} to be independent of the codimension.  In our notation we have
	\begin{equation*}
		\sum_{\alpha,\beta > 1}\Big(\sum_{i,j}\ho_{ij\alpha}\ho_{ij\beta}\Big)^2
	+ \sum_{\alpha, \beta > 1} N(\ho_{\alpha}\ho_{\beta} - \ho_{\beta}\ho_{\alpha}) \leq \frac{3}{2} \abs{\ho_-}^4.
	\end{equation*}
\begin{proof}[Proof of Theorem \ref{l: preservation of pinching}.]
Using the above inequalities we estimate the reaction terms by
\begin{equation*} 
	2R_1-2cR_2 <  \left(6-\frac{2}{n(c-1/n)}\right)\abs{\ho_1}^2\abs{\ho_-}^2 + \left(3-\frac{2}{n(c-1/n)}\right)\abs{\ho_-}^4.
\end{equation*}

The $\abs{ \ho_1 }^2 \abs{ \ho_- }^2$ terms are nonpositive for $c \leq \frac{1}{n} + \frac{1}{3n}$ and the $\abs{ \ho_- }^4$ terms are nonpositive for $c \leq \frac{1}{n} + \frac{2}{3n}$.  The gradient terms are nonpositive for $c \leq \frac{3}{n+2}$, so the right-hand side of \eqref{e: h2-cH2} is negative for $c\leq\frac1n+\frac1{3n}$, while the left-hand side is non-negative.  This is a contradiction, so ${\mathcal Q}$ must remain negative.
\end{proof}

To apply the pinching estimate in the case where equality holds in the assumptions of Main Theorem \ref{mthm: main thm 2}, we need the following result:

\begin{prop}\label{prop:strongMP}
Suppose $\Sigma_0=F_0(\Sigma^n)$ is a submanifold satisfying the conditions of Main Theorem \ref{mthm: main thm 2}.and let $F:\ \Sigma\times[0,T)\to\mathbb{R}^{n+k}$ be the solution of MCF with initial data $F_0$.  Then for any sufficiently small $t>0$ there exists $c\leq\frac1n+\frac1{3n}$ and $a>0$ such that the conditions of Lemma \ref{l: preservation of pinching} hold for $\Sigma_t$.
\end{prop}

\begin{proof}
We assume that $\Sigma_0$ is not a totally umbillic sphere, since in that case the conditions of Lemma \ref{l: preservation of pinching} certainly apply.
Since the solution is smooth, $H$ remains non-zero on a short time interval.  On this interval we can carry out the proof of Lemma \ref{l: preservation of pinching} with $a=0$, yielding 
\[
\frac{\partial}{\partial t}\left(\abs{h}^2-c\abs{H}^2\right)\leq \Delta\left(\abs{h}^2-c\abs{H}^2\right)-2\left(1-\frac{c(n+2)}{3}\right)\abs{\nabla h}^2+\left(3-\frac{2}{n(c-1/n)}\right)\abs{\ho_-}^4.
\]
The coefficients of the last two terms are negative under the assumptions of Main Theorem \ref{mthm: main thm 2}.  By the strong maximum principle, if $\abs{h}^2-c\abs{H}^2$ does not immediately become negative, then $\nabla h\equiv 0$ and $\ho_-\equiv 0$.  The latter implies that $\Sigma_t$ lies in a $(n+1)$-subspace of $\mathbb{R}^{n+k}$, and then $\nabla h = 0$ implies that $\Sigma_t$ is a product $\mathbb{S}^p \times \mathbb{R}^{n-p} \subset \mathbb{R}^{n+k}$ (see Chapter 5), and since $\Sigma_0$ is not a sphere we have $p<n$.  But this is impossible since $\Sigma_t$ is compact.  Therefore for any small $t>0$ there exists $a>0$ such that $\abs{h}^2-c\abs{H}^2\leq -a$ on $\Sigma_t$ and Lemma \ref{l: preservation of pinching} applies.
\end{proof}

\section{Higher derivative estimates and long time existence}\label{s: Higher derivative estimates and long time existence}

Here we consider the long time behaviour of MCF and establish the existence of a solution on a finite maximal time interval determined by the blowup of the second fundamental form.
\begin{thm}\label{t: MCF exists on maximal time interval}
Under the assumptions of Main Theorem \ref{mthm: main thm 2},
	MCF has a unique solution on a finite maximal time interval $0 \leq t < T < \infty$.  Moreover, $\max_{\Sigma_t} \abs{h}^2 \rightarrow \infty$ as $t \rightarrow T$.
\end{thm}

As a first step we observe that the maximal time of existence is finite.  This follows easily from the equation for the position vector $F$: $\frac{\partial}{\partial t}\abs{F}^2 = \Delta\abs{F}^2-2n$.  The maximum principle implies $\abs{F(p,t)}^2\leq R^2-2nt$ and thus $T\leq \frac{R^2}{2n}$, where $R=\sup\{\abs{F_0(p)}:\ p\in\Sigma\}$.

Next want to prove interior-in-time higher derivative estimates for the second fundamental form.  We use Hamilton's $*$ notation: For tensors $S$ and $T$ (that is, sections of bundles constructed from ${\mathcal H}$ and ${\mathcal N}$ by taking duals and tensor products) the product $S*T$ denotes any linear combination of contractions of $S$ with $T$.

	\begin{prop}\label{p: commute time and kth covariant derivative of h}
		The evolution of the $m$-th covariant derivative of $h$ is of the form
		\begin{equation*}
			\nabla_t\nabla^m h = \Delta\nabla^m h + \sum_{i+j+k=m} \nabla^i h \ast \nabla^j h \ast \nabla^k h.
		\end{equation*}
	\end{prop}
	
	\begin{proof} 
We argue by induction on $m$. The case $m=0$ is given by the evolution equation for the second fundamental form.  Now suppose that the result holds up to $m-1$.  Differentiating the $m$-th covariant derivative of $h$ in time and using the timelike Gauss and Ricci equations to interchange derivatives we find
	\begin{align*}
		\nabla_t\nabla^m h &= \nabla\nabla_t\nabla^{m-1}h + \nabla^{m-1}h \ast h \ast \nabla h \\
		&=\nabla\big(\Delta\nabla^{m-1}h + \sum_{i+j+k=m-1} \nabla^p h \ast \nabla^q h \ast \nabla^r h\big) + \nabla^{m-1}h \ast h \ast \nabla h \\
		&= \nabla\Delta\nabla^{m-1}h + \sum_{i+j+k=m}  \nabla^i h \ast \nabla^j h \ast \nabla^k h.
	\end{align*}
The formula for commuting the Laplacian and gradient of a normal-valued tensor is given by:
	\begin{align*}
		\Delta\nabla_k T= \nabla_k\Delta T + \nabla_m\big(R(\p_k,\p_m)T\big) + \left(\big(R(\p_k,\p_m)(\nabla T\big)\right)(\partial_m).
	\end{align*}
Since $T$ and $\nabla T$ are ${\mathcal N}$-valued tensors acting on ${\mathcal H}$, equation \eqref{eq:curv_on_tensor} gives expressions for $R(\p_k,\p_m)T$ as $R\ast T+\Rp\ast T$, and similarly
$R(\p_k,\p_m)\nabla T = R\ast \nabla T+\Rp\ast \nabla T$, where $R$ and $\Rp$ are the curvature tensors on ${\mathcal H}$ and ${\mathcal N}$, which are both of the form $h\ast h$.
The terms arising in commuting the gradient and Laplacian of $\nabla^{m-1} h$ are of the form $\sum_{i+j+k=m}  \nabla^i h \ast \nabla^j h \ast \nabla^k h$, so we obtain
	\begin{equation*}
		\nabla_t\nabla^m h = \Delta\nabla^m h +  \sum_{i+j+k=m}  \nabla^i h \ast \nabla^j h \ast \nabla^k h
	\end{equation*}
as required.
	\end{proof}

	\begin{prop}\label{p: evol length of the second fundamental form squared}
	The evolution of $\abs{\nabla^mh}^2$ is of the form
		\begin{equation*}
			\frac{\p}{\p t}\abs{\nabla^m h}^2 = \Delta\abs{\nabla^m h}^2 - 2\abs{\nabla^{m+1} h}^2 + \sum_{i+j+k=m} \nabla^i h \ast \nabla^j h \ast \nabla^k h \ast \nabla^m h.
		\end{equation*}
	\end{prop}
	
	\begin{proof}
	Denoting by angle brackets the inner product on $\otimes^{m+2}{\mathcal H}^*\otimes{\mathcal N}$, which is compatible with the connection on the same bundle, we have
		\begin{align*}
			\frac{\p}{\p t}\abs{\nabla^m h}^2	 &= \frac{\p}{\p t}\left\langle\nabla^m_p h, \nabla^m_p h \right\rangle \\
			&= 2 \big\langle \nabla^m_p h, \nabla_t\nabla^m_p h \big\rangle \\
			&= 2\big\langle \nabla^m_p h, \Delta\nabla^m_p h + \sum_{i+j+k=m} \nabla^i h \ast \nabla^j h \ast \nabla^k h \big\rangle \\
					&= \Delta\abs{\nabla^m h}^2 - 2\abs{\nabla^{m+1} h}^2 + \sum_{i+j+k=m} \nabla^i h \ast \nabla^j h \ast \nabla^k h * \nabla^m h
		\end{align*}
		as required.
	\end{proof}

\begin{prop}\label{prop: interior-in-time higher derivative estimates}
		Suppose that mean curvature flow of a given submanifold $\Sigma_0$ has a solution on a time interval $t \in [0, \tau]$. If $\abs{h}^2 \leq K$ for all $t \in [0, \tau]$, then $\abs{ \nabla^m h }^2 \leq C_m\left(1+1/t^m\right)$ for all $t \in (0, \tau]$, where $C_m$ is a constant that depends on $m$, $n$ and $K$.
	\end{prop}
The strength of this estimate is that assuming only a bound on the second fundamental form (and no information about its derivatives) we can bound all higher derivatives.  The fact that these estimates
blow up as $t$ approaches zero poses no difficulty, since the short time existence result bounds all derivatives of $h$ for a short time.  While not crucial here, the interior-in-time estimates are useful in singularity analysis.	
\begin{proof}
		The proof is by induction on $m$.  We first prove the Lemma for $m=1$.  We consider the quantity $G = t\abs{ \nabla h }^2 + \abs{ h }^2$, which has a bound at $t=0$ depending only on curvature.  The strategy is now to use the good term from the evolution of $\abs{ h }^2$ to control the bad term in the evolution of $\abs{ \nabla h }^2$:  Differentiating $G$ we get
	\begin{align*}
		\frac{\p G}{\p t} &= \abs{ \nabla h }^2 + t \big( \Delta\abs{ \nabla h }^2 - 2\abs{ \nabla^2 h }^2 + h*h*\nabla h * \nabla h \big) \\
			&\qquad + \big( \Delta\abs{h}^2 - 2\abs{ \nabla h }^2 + h * h * h * h \big) \\
		&\leq \Delta G + (c_1t\abs{h}^2 - 1)\abs{ \nabla h }^2 + c_2\abs{h}^4.
	\end{align*}
For $t\leq 1/(c_1K)$ we can estimate
	\begin{equation*}
		\frac{\p}{\p t}G \leq \Delta G + c_2 K^2,
	\end{equation*}
and the maximum principle implies $\max_{x,t} G \leq K+c_2K^2t$.  Then $\abs{ \nabla h }^2 \leq G/t \leq K/t+c_2K^2$ for $t \in (0, 1/(c_1K)]$.  If $t>1/(c_1K)$ we apply the same argument on the interval $[t-1/(c_1K),t]$, yielding $\abs{\nabla h}^2 (t)\leq (c_1+c_2)K^2$.
This completes the proof for $m = 1$.  Now suppose the estimate holds up to $m-1$, and consider $G = t^m\abs{ \nabla^m h }^2 + mt^{m-1}\abs{ \nabla^{m-1} h }^2$.  Differentiating $G$ gives
	\begin{align*}
		\frac{\p}{\p t}G &= mt^{m-1}\abs{ \nabla^m h }^2 + t^m \Big\{ \Delta\abs{ \nabla^m h }^2 - 2\abs{ \nabla^{m+1} h }^2 + \sum_{i+j+k=m} \nabla^i h \ast \nabla^j h \ast \nabla^k h * \nabla^m h \Big\} \\
		&\qquad + m\Big\{(m-1)t^{m-2}\abs{ \nabla^{m-1} h }^2 + t^{m-1}\big( \Delta\abs{ \nabla^{m-1} h }^2 - 2\abs{ \nabla^m h }^2 \\
		&\qquad + \sum_{i+j+k=m-1} \nabla^i h \ast \nabla^j h \ast \nabla^k h * \nabla^{m-1} h \big).\Big\}
		\end{align*}
Noticing that in the quartic reaction terms there can only be one or two occurences of the highest order derivative, using Young's inequality we can estimate
	\begin{equation*}
		\begin{split}
		\frac{\p}{\p t}G &\leq mt^{m-1}\abs{ \nabla^m h }^2 + t^m \Big\{ \Delta\abs{ \nabla^m h }^2 + c_3\abs{ \nabla^m h }^2 + \frac{c_4}{t^m} \Big\} \\
		&\qquad + m\Big\{ (m-1)t^{m-2}\abs{ \nabla^{m-1} h }^2 + t^{m-1}\big( \Delta\abs{ \nabla^{m-1} h }^2 - 2\abs{ \nabla^{m} h }^2 + c_5\abs{ \nabla^{m-1} h }^2 + \frac{c_6}{t^{m-1}} \big) \Big\}.
		\end{split}
	\end{equation*}
We split the gradient term of order $m$ out of the second line, and then since $m$ is at least two, all other terms are bounded by the induction hypothesis for $t\leq 1$, giving
	\begin{equation*}
			\frac{\p}{\p t}G \leq \Delta G +  (c_3t - m)t^{m-1}\abs{ \nabla^m h }^2 + c_7.
	\end{equation*}
Thus $\frac{\p}{\p t}G \leq \Delta G + c_{8}$ if $t\leq \min\{1,m/c_3\}$, so by the maximum principle $\abs{ \nabla^m h }^2 \leq C/t^m$ for $t\leq \min\{1,m/c_3\}$. The same argument on later time intervals gives the result for larger $t$.
	\end{proof}

\begin{proof}[Proof of Theorem \ref{t: MCF exists on maximal time interval}]
Fix a smooth metric $\tilde g$ on $\Sigma$ with Levi-Civita connection $\tilde\nabla$.  $\tilde g$ extends to a time-independent metric on ${\mathcal H}$, and $\tilde\nabla$ extends to ${\mathcal H}$ by taking $\tilde\nabla_{\partial_t}u=0$ whenever $[\partial_t,u]=0$.  The difference $T=\nabla-\tilde\nabla$ restricts to a section of ${\mathcal H}^*\otimes{\mathcal H}^*\otimes{\mathcal H}$.
If $S$ is a section of a bundle constructed from ${\mathcal H}$, ${\mathcal N}$ and $F^*TN$, $\tilde\nabla S$ denotes the derivative of $S$ with the connection on this bundle induced by the connections $\tilde\nabla$ on ${\mathcal H}$, $\nablap$ on ${\mathcal N}$, and ${}^F\nabla$ on $F^*TN$, so that $\tilde\nabla S-\nabla S=S\ast T$.

To prove Theorem \ref{t: MCF exists on maximal time interval} we assume that $\abs{h}$ remains bounded on the interval $[0,T)$, and derive a contradiction.  This suffices to prove the Theorem, since if $\abs{h}$ is bounded on any subsequence of times approaching $T$, then Equation \eqref{eq:evol_h2} implies that $\abs{h}$ is bounded on $\Sigma\times[0,T)$.  Under this assumption
the boundedness of $\tilde\nabla_tg=-2H\cdot h$ implies that the metric $g$ remains comparable to $\tilde g$:  We have
for any non-zero vector $v\in T\Sigma$
\begin{equation*}
\left|\frac{\partial}{\partial t}\left(\frac{g(v,v)}{\tilde g(v,v)}\right)\right| = \left|\frac{\tilde\nabla_tg(v,v)}{g(v,v)}\,\frac{g(v,v)}{\tilde g(v,v)}\right|\leq 2\abs{H}\abs{h}_g \frac{g(v,v)}{\tilde g(v,v)},
\end{equation*}
so that the ratio of lengths is controlled above and below by exponential functions of time, and hence since the time interval is bounded, there exists a positive constant $c_9$ such that
\begin{equation}\label{eq:g.comparable.to.tildeg}
\frac{1}{c_9}\tilde g \leq g\leq c_9 \tilde g.
\end{equation}

Next we observe that covariant derivatives of all orders of $F$ with respect to $\tilde\nabla$ can be expressed in terms of $h$ and $T$ and their derivatives:  We prove by induction that
\begin{align}
\tilde\nabla^kF &= F_*\tilde\nabla^{k-2}T+F_*\left(\sum_{i_0+2i_1+\dots+(k-2)i_{k-3}=k-1}
T^{i_0}\ast\left(\tilde\nabla T\right)^{i_1}\ast\dots\ast\left(\tilde\nabla^{k-3}T\right)^{i_{k-3}}\right)\label{eq:kth.deriv.embedding}\\
&\quad\null+(\iota+F_*)\ast\sum_{j=1}^{k-1}\left(\sum_{\sum (n+1)i_n=k-1-j}\prod_{n=0}^{k-2-j}
\left(\tilde\nabla^nT\right)^{i_n}\right)\ast
\left(\sum_{\sum (m+1)p_m=j}\,\,\prod_{m=0}^{j-1}\left(\nabla^{m}h\right)^{p_m}\right).\notag
\end{align}
This is true for $k=2$, since 
\begin{equation}\label{eq:tildenablaF*}
\tilde\nabla_{u,v}^2F = {}^F\nabla_u (F_*v)-F_*(\tilde\nabla_uv) = 
F_*(\nabla_uv-\tilde\nabla_uv)+\iota h_{u,v}=F_*T_{u,v}+\iota h_{u,v}.
\end{equation}
To deduce the result for higher $k$ by induction, we note that equation \eqref{eq:tildenablaF*} implies a formula for the derivative of $F_*$:
$$
(\tilde\nabla F_*)(V) = F_*T(.,V)+\iota h(.,V) = F_* T\ast V + \iota h\ast V,
$$
while equation \eqref{eq:nabla.iota} gives
$$
(\tilde\nabla\iota)(\xi) = -F_*{\mathcal W}(.,\xi) = F_* h\ast\xi.
$$
The result for $k+1$ now follows by differentiating the expression \eqref{eq:kth.deriv.embedding}, and writing
$\tilde\nabla(\nabla^nh) = \nabla^{n+1}h+\nabla^nh\ast T$.   It follows that if $|\tilde\nabla^jF|_{\tilde g}$ is bounded for $j=1,\dots,k-1$, then
\begin{equation}\label{eq:bound.for.T}
|\tilde\nabla^{k-2}T|_{\tilde g}\leq C\left(1+|\tilde\nabla^kF|_{\tilde g}\right).
\end{equation}

The above observations allow us to prove $C^k$ convergence of $F$ as $t\to T$ for every $k$:  We have $\tilde\nabla_tF = \iota H$, so the boundedness of $H$ implies that $F$ remains bounded and converges uniformly as $t\to T$.  Differentiating as above, we find by induction that
\begin{equation}\label{eq:evol.kth.deriv.embedding}
\tilde\nabla_t\tilde\nabla^kF = (F_*+\iota)\ast\sum_{j=0}^{k-1}\left(
	\sum_{\sum(n+1)i_n=k-1-j}\prod_{n=0}^{k-2-j}\left(\tilde\nabla^nT\right)^{i_n}\right)\ast\left(
	\sum_{\sum (m+1)p_m=j+2}\,\,\prod_{m=0}^{j+1}\left(\nabla^mh\right)^{p_m}\right).
\end{equation}
Suppose we have established a bound on $|\tilde\nabla^j F|_{\tilde g}$ for $j\leq k-1$.  Then using the estimate \eqref{eq:bound.for.T}, the bounds on $|\nabla^nh|_g$ from Lemma \ref{prop: interior-in-time higher derivative estimates}, and the comparability of $g$ and $\tilde g$ from \eqref{eq:g.comparable.to.tildeg} we can estimate
$$
|\tilde\nabla_t\tilde\nabla^kF|_{\tilde g}\leq C\left(1+|\tilde\nabla^{k-2}T|_{\tilde g}\right)
\leq C\left(1+|\tilde\nabla^kF|_{\tilde g}\right),
$$
so that $|\tilde\nabla^kF|_{\tilde g}$ remains bounded, and $\tilde\nabla^kF$ converges uniformly as $t\to T$.    This completes the induction, proving that $F(.,t)$ converges in $C^\infty$ to a limit $F(.,T)$ which is an immersion.  

Finally, applying the short time existence result with initial data $F(.,T)$, we deduce that the solution can be continued to a larger time interval, contradicting the maximality of $T$.  This completes the proof of Theorem \ref{t: MCF exists on maximal time interval}.
\end{proof}

\section{A pinching estimate for the traceless second fundamental form}\label{s: A pinching estimate for the traceless second fundamental form}

In this section we show that the pinching actually improves along the flow.  This is the key estimate that will imply that the submanifold is evolving to a ``round" point.

\begin{thm}\label{t: pinch trace 2ff}
Under the assumptions of Main Theorem \ref{mthm: main thm 2}
there exist constants $C_0 < \infty$ and $\delta > 0$ both depending only on $\Sigma_0$ such that for all time $t \in [0, T)$ we have the estimate
	\begin{equation}
		\abs{\ho}^2 \leq C_0\abs{H}^{2-\delta}. \label{eqn: pinch trace 2ff}
	\end{equation}
\end{thm}
We wish to bound the function $f_{\sigma} = (\abs{h}^2 - 1/n\abs{H}^2)/\abs{H}^{2(1-\sigma)}$ for sufficiently small $\sigma$.  As in the hypersurface case, a distinguishing feature of mean curvature flow when compared to Ricci flow is that this result cannot be proved by a maximum principle argument alone. Somewhat more technical integral estimates and a Stampacchia iteration procedure are required.  We proceed by first deriving an evolution equation for $f_{\sigma}$.

\begin{prop}
For any $\sigma \in [0, 1/2]$ we have the evolution equation
	\begin{equation}\label{e: evol eqn f_sigma 1}
	\frac{\p}{\p t} f_{\sigma} \leq \Delta f_{\sigma} + \frac{ 4(1-\sigma) }{ \abs{H} } \big\langle \nabla_i\abs{H}, \nabla_i f_{\sigma} \big\rangle - \frac{ 2\epsilon_{\nabla} }{ \abs{H}^{2(1-\sigma)} }  \abs{ \nabla H }^2 + 2\sigma\abs{h}^2f_{\sigma}.
	\end{equation}
\end{prop}

\begin{proof}
Differentiating $f_{\sigma}$ in time and substituting in the evolutions equations for the squared lengths of the second fundamental form and mean curvature we get
\begin{equation}\label{e: evol eqn f_sigma 2}
	\begin{split}
		\p_t f_{\sigma} &= \frac{\Delta\abs{h}^2 - 2\abs{\nabla h}^2 + 2R_1 }{ (\abs{H}^2)^{1-\sigma} }  - \frac{1}{n}\frac{ ( \Delta\abs{H}^2 - 2\abs{\nabla H}^2 + 2R_2 ) }{ ( \abs{H}^2)^{1-\sigma} } \\
		&\qquad - \frac{ (1-\sigma)(\abs{h}^2 - 1/n\abs{H}^2) }{ (\abs{H}^2)^{2-\sigma} }( \Delta\abs{H}^2 - 2\abs{\nabla H}^2 + 2R_2).
	\end{split}
\end{equation}
The Laplacian of $f_{\sigma}$ is given by
	\begin{equation*}
		\begin{split}
			\Delta f_{\sigma} &= \frac{ \Delta ( \abs{h}^2 - 1/n\abs{H}^2 ) }{ ( \abs{H}^2 )^{1-\sigma} } - \frac{ 2(1-\sigma) }{ ( \abs{H}^2 )^{2-\sigma} } \big\langle \nabla_i (\abs{h}^2 - 1/n\abs{H}^2), \nabla_i\abs{H}^2 \big\rangle  \\
			&\qquad - \frac{ (1-\sigma)(\abs{h}^2 - 1/n\abs{H}^2) }{ ( \abs{H}^2 )^{2-\sigma} } \Delta\abs{H}^2 + \frac{ (2-\sigma)(1-\sigma)(\abs{h}^2 - 1/n\abs{H}^2) }{ ( \abs{H}^2 )^{3-\sigma} } \abs{\nabla\abs{H}^2}^2
		\end{split}
	\end{equation*}
Using this and the identity
	\begin{equation*}
		\begin{split}
		-\frac{ 2(1-\sigma) }{ ( \abs{H}^2 )^{2-\sigma} } \big\langle \nabla_i (\abs{h}^2 - 1/n\abs{H}^2), \nabla_i\abs{H}^2 \big\rangle &= -\frac{ 2(1-\sigma) }{ \abs{H}^2 } \big\langle \nabla_i\abs{H}^2, \nabla_i f_{\sigma} \big\rangle \\
		&\qquad - \frac{ 8(1-\sigma)^2 }{ (\abs{H}^2)^2 } f_{\sigma}\abs{H}^2\abs{ \nabla\abs{H} }^2,
	\end{split}
	\end{equation*}
equation \eqref{e: evol eqn f_sigma 2} can be manipulated into the form
\begin{equation*}
	\begin{split}
		\p_t f_{\sigma} &= \Delta f_{\sigma} + \frac{ 2(1-\sigma) }{ \abs{H}^2 } \big\langle \nabla_i\abs{H}^2, \nabla_i f_{\sigma} \big\rangle - \frac{ 2 }{ (\abs{H}^2)^{1-\sigma} } \Big( \abs{\nabla h}^2 - \frac{ \abs{h}^2 }{ \abs{H}^2 }\abs{ \nabla H }^2 \Big) + \frac{ 2\sigma R_2 f_{\sigma} }{ \abs{ H }^2 } \\
		&\qquad - \frac{ 4\sigma(1-\sigma) }{ \abs{H}^4 } f_{\sigma}\abs{H}^2\abs{ \nabla\abs{H} }^2 - \frac{ 2\sigma(\abs{h}^2 - 1/n\abs{H}^2) }{ (\abs{H}^2)^{2-\sigma} }\abs{\nabla H}^2 \\
		&\qquad + \frac{2}{ (\abs{H}^2)^{1-\sigma} }\Big( R_1 - \frac{ \abs{h}^2 }{ \abs{H}^2 }R_2 \Big).
	\end{split}
\end{equation*}
We discard the terms on the last two lines as these are non-positive under our pinching assumption.  The gradient terms on the first line may be estimated as follows:
\begin{equation*}
	-\frac{ 2 }{ (\abs{H}^2)^{1-\sigma} } \Big( \abs{\nabla h}^2 - \frac{ \abs{h}^2 }{ \abs{H}^2 } \abs{ \nabla H }^2 \Big) \leq -\frac{ 2 }{ (\abs{H}^2)^{1-\sigma} } \Big( \frac{3}{n+2} - c \Big) \abs{ \nabla H }^2,
\end{equation*}
and also $R_2 \leq \abs{h}^2\abs{H}^2$.  Importantly, observe that if $c \leq 4/(3n)$, then $\epsilon_{\nabla}:=3/(n+2) - c$ is strictly positive.
\end{proof}

The small reaction term $2\sigma\abs{h}^2f_{\sigma}$ in this evolution equation is positive and hence we cannot apply the maximum principle.  As in the hypersurface case, we exploit the negative term involving the gradient of the mean curvature by integrating a suitable form of Simons' identity: Contracting equation \eqref{eq:SimonsTrace} with the second fundamental form we obtain
\begin{equation}\label{eq:SimonsContracted}
	\frac{1}{2}\Delta\abs{ \ho }^2 = \ho_{ij} \cdot \nabla_i\nabla_j H  + \abs{ \nabla \ho }^2 + Z,
\end{equation}
where
	\begin{equation*}
		Z = -\sum_{\alpha, \beta} \Big(\!\sum_{i,j} h_{ij\alpha}h_{ij\beta}\!\Big)^2 - \abs{\Rp}^2 + \sum_{\substack{i,j,p \\ \alpha, \beta}}H_{\alpha}h_{ip\alpha}h_{ij\beta}h_{pj\beta}.
	\end{equation*}

\begin{lem}\label{l: Z pos}
	If $\Sigma^n$ is a submanifold of $\mathbb{R}^{n+k}$ that satisfies $H\neq 0$ and $\abs{h}^2 \leq c\abs{H}^2$, where
	\begin{equation*}
		c \begin{cases}
			\leq \frac{ 4 }{ 3n }, \quad n = 2, 3 \\
			< \frac{1}{n-1}, \quad \text{ if } n \geq 4,
		\end{cases}
	\end{equation*}
then there exists $\epsilon_Z>0$ such that
$Z \geq \epsilon_Z\abs{\ho}^2\abs{H}^2$.
\end{lem}
The example given in the Introduction shows the best value of $c$ that can be expected is $1/(n-1)$.  In dimensions greater than four, $4/(3n) > 1/(n-1)$ and so somewhere in the analysis the condition $c < 1/(n-1)$ had to manifest itself.  For a submanifold of Euclidean space, the condition $\abs{ h } < 1/(n-1)$ implies that the submanifold has positive intrinsic curvature.  Just as in the hypersurface case (where strict convexity implies positive intrinsic curvature), it is the positive intrinsic curvature that makes this lemma, and indeed the Main Theorem true.

\begin{proof}[Proof of Lemma \ref{l: Z pos}]
Working with the local orthonormal frames of Section \ref{l:Preservation of curvature pinching} we expand $Z$ to get
	\begin{align*}
		Z &= -\abs{\ho_1}^4 + \frac{1}{n}\abs{\ho_1}^2\abs{H}^2 + \frac{1}{n}\abs{\ho_-}^2\abs{H}^2 - 2\sum_{\alpha > 1}\!\!\!\Big(\!\sum_{i,j}\lo_i\ho_{ii\alpha}\!\Big)^2 - 2\sum_{\alpha > 1} N(h_1\ho_{\alpha} - \ho_{\alpha}h_1) \\
	& \qquad -  \sum_{\alpha, \beta > 1}\!\!\!\Big(\!\sum_{i,j}\ho_{ij\alpha}\ho_{ij\beta}\!\Big)^2 - \sum_{\alpha, \beta > 1}N(\ho_{\alpha}\ho_{\beta} - \ho_{\beta}\ho_{\alpha}) \\
	& \qquad + \sum_{ i } \abs{H}\lo_i{^3} + \sum_{\substack{\alpha > 1 \\ i}}\abs{H}\lo_i(\ho_{ii\alpha})^2 + \sum_{\substack{\alpha > 1 \\ i \neq j}}\abs{H}\lo_i(\ho_{ij\alpha})^2.
	\end{align*}
	We estimate the first summation term on line one and the two terms on line two as before, namely
	\begin{gather*}
		-2\sum_{\alpha > 1}\!\!\!\Big(\!\sum_{i,j}\lo_i\ho_{ii\alpha}\!\Big)^2 \geq -2\abs{\ho_1}^2\sum_{\substack{\alpha > 1 \\ i}}(\ho_{ii\alpha})^2 \\
		-\sum_{\alpha, \beta > 1}\!\!\!\Big(\!\sum_{i,j}\ho_{ij\alpha}\ho_{ij\beta}\!\Big)^2 - \sum_{\alpha, \beta > 1}N(\ho_{\alpha}\ho_{\beta} - \ho_{\beta}\ho_{\alpha}) \geq -\frac{3}{2}\abs{\ho_-}^4,
	\end{gather*}
however we need to work somewhat harder with the remaining summation terms.
	\begin{prop}
		For any $\eta > 0$ we have the following estimate
		\begin{multline*}
			-2\sum_{\alpha > 1} N(h_1\ho_{\alpha} - \ho_{\alpha}h_1) + \sum_{\substack{\alpha > 1 \\ i \neq j}}\abs{H}\lo_i(\ho_{ij\alpha})^2 \\
			\geq -\max\Big\{ 4, \frac{\eta}{2} \Big\} \abs{\ho_1}^2\big(\abs{\ho_-}^2 - \sum_{\substack{\alpha > 1 \\ i}}(\ho_{ii\alpha})^2\big) - \frac{1}{4\eta}\abs{H}^2\big(\abs{\ho_-}^2 - \sum_{\substack{\alpha > 1 \\ i}}(\ho_{ii\alpha})^2\big).
		\end{multline*}
	\end{prop}
	\begin{proof}
Using the Peter-Paul inequality we estimate
	\begin{align*}
			&-2\sum_{\alpha > 1} N(h_1\ho_{\alpha} - \ho_{\alpha}h_1) + \sum_{\substack{\alpha > 1 \\ i \neq j}}\abs{H}\lo_i(\ho_{ij\alpha})^2 \\
			&\qquad = -\sum_{\substack{\alpha > 1 \\ i \neq j}}\Big\{ 2(\lo_i - \lo_j)^2 - \frac{\abs{H}}{2}(\lo_i + \lo_j)\Big\}(\ho_{ij\alpha})^2 \\
			&\qquad \geq -\sum_{\substack{\alpha > 1 \\ i \neq j}}\Big\{ 2(\lo_i - \lo_j)^2 + \frac{\eta}{4}(\lo_i + \lo_j)^2 \Big\}(\ho_{ij\alpha})^2 - \frac{1}{4\eta}\abs{H}^2\big(\abs{\ho_-}^2 - \sum_{\substack{\alpha > 1 \\ i}}(\ho_{ii\alpha})^2\big) \\
			&= -\sum_{\substack{\alpha > 1 \\ i \neq j}}\Big\{ (2+\frac{\eta}{4})(\lo_i{^2} + \lo_j{^2}) + (\frac{\eta}{2} - 4)\lo_i\lo_j \Big\}(\ho_{ij\alpha})^2 \\
			&\qquad - \frac{1}{4\eta}\abs{H}^2\big(\abs{\ho_-}^2 - \sum_{\substack{\alpha > 1 \\ i}}(\ho_{ii\alpha})^2\big).
\end{align*}
If $\eta \geq 8$ we estimate
\begin{align*}
		& \qquad \geq -\sum_{\substack{\alpha > 1 \\ i \neq j}}\Big\{ (2+\frac{\eta}{4})(\lo_i{^2} + \lo_j{^2}) + (\frac{\eta}{4} - 2)(\lo_i{^2} + \lo_j{^2}) \Big\}(\ho_{ij\alpha})^2 \\
		&\qquad \qquad  - \frac{1}{4\eta}\abs{H}^2\big(\abs{\ho_-}^2 - \sum_{\substack{\alpha > 1 \\ i}}(\ho_{ii\alpha})^2\big) \\
		& \qquad = -\frac{\eta}{2}\sum_{\substack{\alpha > 1 \\ i \neq j}}(\lo_i{^2} + \lo_j{^2})(\ho_{ij\alpha})^2 - \frac{1}{4\eta}\abs{H}^2\big(\abs{\ho_-}^2 - \sum_{\substack{\alpha > 1 \\ i}}(\ho_{ii\alpha})^2\big) \\
		& \qquad \geq -\frac{\eta}{2}\abs{\ho_1}^2\big(\abs{\ho_-}^2 - \sum_{\substack{\alpha > 1 \\ i}}(\ho_{ii\alpha})^2\big) - \frac{1}{4\eta}\abs{H}^2\big(\abs{\ho_-}^2 - \sum_{\substack{\alpha > 1 \\ i}}(\ho_{ii\alpha})^2\big),
	\end{align*}
while for $\eta \leq 8$ we can similarly estimate
\begin{align*}
		& \qquad \geq -\sum_{\substack{\alpha > 1 \\ i \neq j}}\Big\{ (2+\frac{\eta}{4})(\lo_i{^2} + \lo_j{^2}) + (2 - \frac{\eta}{4} )(\lo_i{^2} + \lo_j{^2}) \Big\}(\ho_{ij\alpha})^2 \\
		&\qquad \qquad - \frac{1}{4\eta}\abs{H}^2\big(\abs{\ho_-}^2 - \sum_{\substack{\alpha > 1 \\ i}}(\ho_{ii\alpha})^2\big) \\
		& \qquad = -4\sum_{\substack{\alpha > 1 \\ i \neq j}}(\lo_i{^2} + \lo_j{^2})(\ho_{ij\alpha})^2 - \frac{1}{4\eta}\abs{H}^2\big(\abs{\ho_-}^2 - \sum_{\substack{\alpha > 1 \\ i}}(\ho_{ii\alpha})^2\big) \\
		& \qquad \geq -4\abs{\ho_1}^2\big(\abs{\ho_-}^2 - \sum_{\substack{\alpha > 1 \\ i}}(\ho_{ii\alpha})^2\big) - \frac{1}{4\eta}\abs{H}^2\big(\abs{\ho_-}^2 - \sum_{\substack{\alpha > 1 \\ i}}(\ho_{ii\alpha})^2\big).
	\end{align*}
	\end{proof}
To estimate the remaining two terms we use the following two inequalities from \cite{AdC94} and \cite{wS94}:
	\begin{gather*}
		\sum_i\abs{H}\lo_i{^3} \geq -\frac{n-2}{\sqrt{n(n-1)}}\abs{H}\abs{\ho_1}^3 \\
		\sum_{\substack{\alpha > 1 \\ i}}\abs{H}\lo_i(\ho_{ii\alpha})^2 \geq -\frac{n-2}{\sqrt{n(n-1)}}\abs{H}\abs{\ho_1}\sum_{\substack{\alpha > 1 \\ i}}(\ho_{ii\alpha})^2,
	\end{gather*}
and further estimate them using the Peter-Paul inequality to obtain
	\begin{gather*}
		\sum_i\abs{H}\lo_i{^3} \geq -\frac{\mu}{2}\abs{\ho_1}^4 - \frac{1}{2\mu}\frac{(n-2)^2}{n(n-1)}\abs{\ho_1}^2\abs{H}^2 \\
		\sum_{\substack{\alpha > 1 \\ i}}\abs{H}\lo_i(\ho_{ii\alpha})^2 \geq -\rho\abs{\ho_1}^2\sum_{\substack{\alpha > 1 \\ i}}(\ho_{ii\alpha})^2 - \frac{1}{4\rho}\frac{(n-2)^2}{n(n-1)}\abs{H}^2\sum_{\substack{\alpha > 1 \\ i}}(\ho_{ii\alpha})^2.
	\end{gather*}
Note that in dimension $n=2$ above two terms are actually zero and there is no need to further estimate them in this way.  For $n=2$ the remaining quantities can now be estimated as we have done before to give the estimate for $c < 3/4$.  For the higher dimensions, putting everything together we obtain
	\begin{align*}
		Z &\geq -\abs{\ho_1}^4 + \frac{1}{n}\abs{\ho_1}^2\abs{H}^2 + \frac{1}{n}\abs{\ho_-}^2\abs{H}^2 - 2\abs{\ho_1}^2\sum_{\substack{\alpha > 1 \\ i}}(\ho_{ii\alpha})^2 - \frac{3}{2}\abs{\ho_-}^4 \\
		& \qquad - \max\Big\{ 4, \frac{\eta}{2} \Big\} \abs{\ho_1}^2\abs{\ho_-}^2 + \max\Big\{ 4, \frac{\eta}{2} \Big\} \abs{\ho_1}^2\sum_{\substack{\alpha > 1 \\ i}}(\ho_{ii\alpha})^2 - \frac{1}{4\eta}\abs{H}^2\abs{\ho_-}^2 \\
		&\qquad + \frac{1}{4\eta}\abs{H}^2\sum_{\substack{\alpha > 1 \\ i}}(\ho_{ii\alpha})^2 -\frac{\mu}{2}\abs{\ho_1}^4 - \frac{1}{2\mu}\frac{(n-2)^2}{n(n-1)}\abs{\ho_1}^2\abs{H}^2 -\rho\abs{\ho_1}^2\sum_{\substack{\alpha > 1 \\ i}}(\ho_{ii\alpha})^2 \\
		&\qquad - \frac{1}{4\rho}\frac{(n-2)^2}{n(n-1)}\abs{H}^2\sum_{\substack{\alpha > 1 \\ i}}(\ho_{ii\alpha})^2.
	\end{align*}
	We now need to choose the optimal values of the constants $\eta$, $\mu$ and $\rho$.  First, choose $\mu$ to be equal to $n-2$ and $\rho = (n-2)/2$.  Next we want to choose $\eta$ to make the $\abs{\ho_1}^2\sum\limits_{\substack{\alpha > 1 \\ i}}(\ho_{ii\alpha})^2$ terms non-negative, that is, we would like to choose $\eta$ so that
\begin{equation*}
	\max\Big\{ 4, \frac{\eta}{2} \Big\} - 2 - \rho \geq 0.
\end{equation*}
As $\rho = (n-2)/2$, we want
\begin{equation*}
	\max\Big\{ 4, \frac{\eta}{2} \Big\} \geq \frac{ n + 2 }{ 2 },
\end{equation*}		
thus we are able to choose $\eta = n+2$.  In dimensions $3$ to $5$ this term is positive and we discard it, while for $n \geq 8$ it is identically zero.  The only mildly troublesome term that remains is
	\begin{equation*}
		\Big(\frac{1}{4\eta} - \frac{1}{4 \rho }\frac{(n-2)^2}{n(n-1)}\Big)\abs{H}^2\sum_{\substack{\alpha > 1 \\ i}}(\ho_{ii\alpha})^2.
	\end{equation*}
With our choices of $\eta$ and $\rho$ this term this term is negative for $n \geq 3$ and we estimate
	\begin{align*}
		-\Big( \frac{n-2}{2n(n-1)} - \frac{1}{ 4(n+2) } \Big) \abs{H}^2 \sum_{ \substack{ \alpha > 1 \\ i} }( \ho_{ ii\alpha } )^2 \geq -\Big( \frac{n-2}{2n(n-1)} - \frac{1}{ 4(n+2) } \Big) \abs{H}^2 \abs{\ho_-}^2.
	\end{align*}
 After substituting in our choices for $\rho$, $\mu$ and $\eta$ we have, in dimension three to five:
\begin{align*}
		Z &\geq -\abs{\ho_1}^4 + \frac{1}{n}\abs{\ho_1}^2\abs{H}^2 + \frac{1}{n}\abs{\ho_-}^2\abs{H}^2  -\frac{n-2}{2}\abs{\ho_1}^4 - 4 \abs{\ho_1}^2\abs{\ho_-}^2  \\
		& \qquad - \frac{3}{2}\abs{\ho_-}^4 - \frac{n-2}{2n(n-1)}\abs{\ho_1}^2\abs{H}^2 - \frac{n-2}{2n(n-1)}\abs{\ho_-}^2\abs{H}^2,
\end{align*}
and in dimensions six and higher:
	\begin{align*}
		Z &\geq -\abs{\ho_1}^4 + \frac{1}{n}\abs{\ho_1}^2\abs{H}^2 + \frac{1}{n}\abs{\ho_-}^2\abs{H}^2  - \frac{n-2}{2}\abs{\ho_1}^4 -\frac{n+2}{2}\abs{\ho_1}^2\abs{\ho_-}^2 \\
		& \qquad - \frac{3}{2}\abs{\ho_-}^4  - \frac{n-2}{2n(n-1)}\abs{\ho_1}^2\abs{H}^2 - \frac{n-2}{2n(n-1)}\abs{\ho_-}^2\abs{H}^2.
	\end{align*}
We now group like terms, estimate $\abs{H}^2$ from below by $\abs{\ho}{^2}/(c-1/n)$ and calculate the maximum value of $c$ permissable in each case such that the coefficients are all strictly positive.  For $n=2$ and $n=3$ the most restrictive term is the cross-term, and the best value of $c$ is given by $c= 3/4$ and $c=11/24$ respectively.  For the corresponding value of $n$ note that both $3/4$ and $11/24 \geq 4/(3n)$ and so we have simply used $4/(3n)$ in the statement of the lemma.  For $n \geq 4$ the most restrictive term is the $\abs{\ho_1}^4$ term, which is identically zero when $c = 1/(n-1)$.  Thus, for the values $c$ stated in the proposition, we have now shown there exist strictly positive constants $c_2$, $c_3$ and $c_4$ depending on $\Sigma_0$ such that
	\begin{align}
		\notag Z &\geq c_2\abs{\ho_1}^4 + c_3\abs{\ho_1}^2\abs{\ho_-}^2 + c_4\abs{\ho_-}^4 \\
			&\geq c_5\abs{\ho}^4, \label{e: Z pos}
	\end{align}
where $c_5 = \min\{c_2, c_3/2, c_4\}$.
To prove the desired estimate we note that by using Peter-Paul on various terms of $Z$ we can estimate
	\begin{equation*}
		Z \geq c_6\abs{\ho}{^2}\abs{H}^2 - c_7\abs{\ho}{^4}.
	\end{equation*}
Combining this with \eqref{e: Z pos} gives for any $a\in[0,1]$ that
	\begin{equation*}
		Z \geq a(c_6\abs{\ho}{^2}\abs{H}^2 - c_7\abs{\ho}{^4}) + (1-a)c_5\abs{\ho}^4.
	\end{equation*}
Choosing $a = c_5/(c_5 + c_7)$ gives
	\begin{equation*}
		Z \geq \frac{c_5c_6}{c_5+c_7}\abs{\ho}{^2}\abs{H}^2
	\end{equation*}
and the lemma is complete by setting $\epsilon_Z = c_5c_6/(c_5+c_7)$.
\end{proof}
Next we derive the integral estimates.
\begin{prop}\label{p: using Z pos}
For any $p \geq 2$ and $\eta > 0$ we have the estimate
	\begin{equation*}
		\int_{\Sigma} f_{\sigma}^p\abs{H}^2 \, d\mu_g \leq \frac{(p\eta + 4)}{\epsilon_Z}\int_\Sigma\frac{f_{\sigma}^{p-1}}{\abs{H}^{2(1-\sigma)}}\abs{\nabla H}^2 \, d\mu_g + \frac{p-1}{\epsilon_Z\eta}\int_{\Sigma} f_{\sigma}^{p-2}\abs{\nabla f_{\sigma}}^2 \, d\mu_g.
	\end{equation*}
\end{prop}
\begin{proof}
Using the contracted form of Simons' indentity and $\Delta\abs{H}^2 = 2\abs{H}\Delta\abs{H} + 2\abs{ \nabla \abs{H} }^2$, the Laplacian of $f_{\sigma}$ can be expressed as
	\begin{align*}
		\Delta f_{\sigma} &= \frac{2}{\abs{H}^{2(1-\sigma)}}\big\langle \ho_{ij}, \nabla_i\nabla_j H \big\rangle + \frac{2}{\abs{H}^{2(1-\sigma)}}Z - \frac{4(1-\sigma)}{\abs{H}}\big\langle \nabla_i\abs{H}, \nabla_i f_{\sigma} \big\rangle - \frac{2(1-\sigma)}{\abs{H}}f_{\sigma}\Delta\abs{H} \\
		&\qquad + \frac{ 2 }{ \abs{H}^{2(1-\sigma)}} \Big( \abs{ \nabla h }^2 - \frac{1}{n}\abs{ \nabla H }^2 \Big) -  \frac{ 2(1-\sigma)(1 - 2\sigma) }{ \abs{H}^2 } f_{\sigma} \abs{ \nabla \abs{H} }^2.
	\end{align*}
The combination of the last two terms is non-negative and we discard them. We multiply the remaining terms by $f_{\sigma}^{p-1}$ and integrate over $\Sigma$.  On the left, and in the last term on line one we use Green's first identity, and in integrating the first term on the right we use the Divergence Theorem and the Codazzi equation.  The term arising from integrating on the left is non-negative and we discard it.  Two other terms arising from the integration combine, ultimately giving
	\begin{align*}
		&2\int_{\Sigma} \frac{f_{\sigma}^{p-1}}{\abs{H}^{2(1-\sigma)}}Z \, d\mu_g \leq  2(p-1)\int_{\Sigma} \frac{f_{\sigma}^{p-2}}{\abs{H}^{2(1-\sigma)}}\big\langle \nabla_i f \cdot \ho_{ij}, \nabla_j H \big\rangle \, d\mu_g \\
		&\qquad - 4(1-\sigma)\int_{\Sigma}\frac{f_{\sigma}^{p-1}}{\abs{H}^{2(\sigma-1)+1}}\big\langle \nabla_i\abs{H} \cdot \ho_{ij}, \nabla_j H \big\rangle \, d\mu_g + \frac{2(n-1)}{n}\int_{\Sigma}\frac{f_{\sigma}^{p-1}}{\abs{H}^{2(1-\sigma)}}\abs{\nabla H}^2 \, d\mu_g \\
		&\qquad - 2(1-\sigma)(p-2)\int_{\Sigma} \frac{f_{\sigma}^{p-1}}{\abs{H}}\big\langle \nabla_i \abs{H}, \nabla_i f_{\sigma} \big\rangle \, d\mu_g + 2(1-\sigma)\int_{\Sigma} \frac{f_{\sigma}^p}{\abs{H}^2}\abs{\nabla \abs{H}}^2 \, d\mu_g.
	\end{align*}
Note the terms with an inner product do not have a sign.  Using the Cauchy-Schwarz and Young inequalities, the inequalities $f_{\sigma} \leq c\abs{H}^{2\sigma}$, $\abs{\nabla \abs{H}}^2 \leq \abs{\nabla H}^2$, $1-\sigma \leq 1$, $c \leq 1$, and $\abs{\ho}^2 = f_{\sigma}\abs{H}^{2(1-\sigma)}$ we estimate each term as follows:
	\begin{align*}
		&2(p-1)\int_{\Sigma} \frac{f_{\sigma}^{p-2}}{\abs{H}^{2(1-\sigma)}}\big\langle \nabla_i f_{\sigma} \cdot \ho_{ij}, \nabla_j H \big\rangle \, d\mu_g \\ &\qquad \leq \frac{p-1}{\eta}\int_{\Sigma} f_{\sigma}^{p-2}\abs{\nabla f_{\sigma}}^2 \, d\mu_g + (p-1)\eta \int_{\Sigma}\frac{f_{\sigma}^{p-1}}{\abs{H}^{2(1-\sigma)}} \abs{ \nabla H}^2 \, d\mu_g; \\
		&-4(1-\sigma)\int_{\Sigma}\frac{f_{\sigma}^{p-1}}{\abs{H}^{2(1-\sigma) +1}}\big\langle \nabla_i\abs{H} \cdot \ho_{ij}, \nabla_j H \big\rangle \, d\mu_g \leq 4 \int_{\Sigma}\frac{f_{\sigma}^{p-1}}{\abs{H}^{2(1-\sigma)}}\abs{\nabla H}^2 \, d\mu_g; \\
		&-2(1-\sigma)(p-2)\int_{\Sigma} \frac{f_{\sigma}^{p-1}}{\abs{H}}\big\langle \nabla_i \abs{H}, \nabla f_{\sigma} \big\rangle \, d\mu_g \\
		&\qquad \leq \frac{p-2}{\mu} \int_{\Sigma} f_{\sigma}^{p-2}\abs{\nabla f_{\sigma}}^2 \, d\mu_g + (p-2)\mu \int_{\Sigma} \frac{ f_{\sigma}^{p-1} }{ \abs{H}^{2(1-\sigma)} }\abs{\nabla H}^2 \, d\mu_g; \\
			&2(1-\sigma)\int_{\Sigma} \frac{f_{\sigma}^p}{\abs{H}^2}\abs{\nabla \abs{H}}^2 \, d\mu_g \leq 2 \int_{\Sigma} \frac{ f_{\sigma}^{p-1} }{ \abs{H}^{2(1-\sigma)} } \abs{ \nabla H }^2 \, d\mu_g.
	\end{align*}
Putting all the estimates together we obtain
	\begin{align*}
		2\int_{\Sigma} \frac{f_{\sigma}^{p-1}}{\abs{H}^{2(1-\sigma)}}Z\, d\mu_g &\leq \big( 6 + \frac{2(n-1)}{n} + (p-1)\eta + (p-2)\mu \big) \int_{\Sigma} \frac{f_{\sigma}^{p-1}}{\abs{H}^{2(1-\sigma)}}\abs{\nabla H}^2 \, d\mu_g \\
		& \qquad + \Big( \frac{p-1}{\eta} + \frac{p-2}{\mu} \Big) \int_{\Sigma} f_{\sigma}^{p-2}\abs{\nabla f_{\sigma}}^2 \, d\mu_g.
	\end{align*}
Our use for this inequality will be to show that sufficiently high $L^p$ norms of $f_{\sigma}$ are bounded.  We are not interested in finding optimal values of $p$ and consequently we are going to be a little rough with the final estimates in order to put the lemma into a convenient form.  Setting $\mu = \eta$, and using $p-2 \leq p-1 \leq p$ and Lemma \ref{l: Z pos} we get
	\begin{equation*}
		2\epsilon_Z\int_{\Sigma} f_{\sigma}^p\abs{H}^2 \leq (2p\eta + 8)\int_\Sigma\frac{f_{\sigma}^{p-1}}{\abs{H}^{2-\sigma}}\abs{\nabla H}^2 \, d\mu_g + \frac{2(p-1)}{\eta}\int_{\Sigma} f_{\sigma}^{p-2}\abs{\nabla f_{\sigma}}^2 \, d\mu_g.
	\end{equation*}
Dividing through by $2\epsilon_Z$ completes the Lemma.
\end{proof}
	
\begin{prop}\label{p: d/dt f_sigma}
For any $p \geq \max\{2, 8/(\epsilon_{\nabla} +1)\}$ we have the estimate
	\begin{equation*}
		\begin{split}
				\frac{d}{dt} \int_{\Sigma} f_{\sigma}^p \, d\mu_g  &\leq -\frac{p(p-1)}{2}\int_{\Sigma}f_{\sigma}^{p-2}\abs{\nabla f_{\sigma}}^2 \, d\mu_g \\
				&\qquad - p\epsilon_{\nabla}\int_{\Sigma}\frac{f_{\sigma}^{p-1}}{\abs{H}^{2(1-\sigma)}}\abs{\nabla H}^2 \, d\mu_g + 2p\sigma\int_{\Sigma}\abs{H}^2f_{\sigma}^p \, d\mu_g. \end{split}
	\end{equation*} 
\end{prop}
\begin{proof}
Differentiating under the integral sign and substituting in the evolution equations for $f_{\sigma}$ and the measure $d\mu_g$ gives
	\begin{align} 
		\notag\frac{d}{dt} \int_{\Sigma} f_{\sigma}^p \, d\mu_g &= \int_{\Sigma}(pf_{\sigma}^{p-1}\frac{\partial f_{\sigma}}{\partial t} - \abs{H}^2f_{\sigma}^p) \, d\mu_g \\
			 \notag&\leq  \int_{\Sigma}pf_{\sigma}^{p-1}\frac{\partial f_{\sigma}}{\partial t} \, d\mu_g \\
		\begin{split}\label{e: d/dt f_sigma 1}
				&\leq -p(p-1)\int_{\Sigma} f_{\sigma}^{p-2}\abs{\nabla f_{\sigma}}^2 \, d\mu_g + 4(1-\sigma)p\int_{\Sigma} \frac{f_{\sigma}^{p-1}}{\abs{H}}\abs{\nabla \abs{H}}\abs{\nabla f_{\sigma}} \, d\mu_g \\
						&\qquad - 2p\epsilon_{\nabla}\int_{\Sigma}\frac{f_{\sigma}^{p-1}}{\abs{H}^{2(1-\sigma)}}\abs{\nabla H}^2 \, d\mu_g + 2p\sigma\int_{\Sigma}\abs{H}^2f_{\sigma}^p \, d\mu_g.
			\end{split}
	\end{align}
We estimate the second integral by
\begin{equation*}
	\begin{split}
		&4(1-\sigma)p\int_{\Sigma} \frac{f_{\sigma}^{p-1}}{\abs{H}}\abs{\nabla \abs{H}}\abs{\nabla f_{\sigma}} \, d\mu_g \\
		&\qquad \leq \frac{2p}{\rho} \int_{\Sigma} f_{\sigma}^{p-2}\abs{\nabla f_{\sigma}}^2 \, d\mu_g + 2p\rho\int_{\Sigma} \frac{ f_{\sigma}^{p-1} }{ \abs{H}^{2(1-\sigma)} } \abs{\nabla H}^2 \, d\mu_g, \end{split}  
	\end{equation*}
and then substituting this estimate back into \eqref{e: d/dt f_sigma 1} gives
	\begin{align*}
		\frac{d}{dt}\int_{\Sigma} f_{\sigma}^p \, d\mu_g &\leq \Big(-p(p-1) + \frac{2p}{\rho}\Big)\int_{\Sigma} f_{\sigma}^{p-2}\abs{\nabla f_{\sigma}}^2 \, d\mu_g \\
		&\qquad - (2p\epsilon_{\nabla} - 2p\rho )\int_{\Sigma} \frac{f_{\sigma}^{p-1}}{\abs{H}^{2(1-\sigma)}}\abs{\nabla H}^2 \, d\mu_g + 2p\sigma\int_{\Sigma}\abs{H}^2f_{\sigma}^p \, d\mu_g \\
		&= -p(p-1) \Big( 1 - \frac{2}{\rho(p-1)} \Big) \int_{\Sigma} f_{\sigma}^{p-2}\abs{\nabla f_{\sigma}}^2 \, d\mu_g \\
			&\qquad- 2p\epsilon_{\nabla} \Big( 1 - \frac{\rho}{\epsilon_{\nabla}} \Big) \int_{\Sigma} \frac{f_{\sigma}^{p-1}}{\abs{H}^{2(1-\sigma)}}\abs{\nabla H}^2 \, d\mu_g + 2p\sigma\int_{\Sigma}\abs{H}^2f_{\sigma}^p \, d\mu_g.
	\end{align*}
We now want to choose $\rho$ so that $1 - 2/(\rho(p-1)) \geq 1/2$ and $p$ so that $1 - \rho/\epsilon_{\nabla} \geq 1/2$.  Choosing $\rho = 4/(p-1)$ and $p \geq \max\{2, 8/(\epsilon_{\nabla} + 1)\}$ gives the result.
\end{proof}

\begin{lem}\label{l: f_sigma non-increase}
There exist constants $c_8$ and $c_9$ depending only on $\Sigma_0$ such that if $p \geq c_8$ and $\sigma \leq c_9/\sqrt{p}$, then for all time $t \in [0, T)$ we have the estimate
	\begin{equation*}
		\Big( \int_{\Sigma}f_{\sigma}^p \, d\mu_g \Big)^{ \frac{ 1 }{ p } } \leq C_1,
	\end{equation*}
where $C_1$ is a uniform constant.
\end{lem}
\begin{proof}
Combining Propositions \ref{p: using Z pos} and \ref{p: d/dt f_sigma} we get
\begin{align*}
		\frac{d}{dt}\int_{\Sigma}f_{\sigma}^p \, d\mu_g &\leq -{p(p-1)} \Big( \frac{1}{2} - \frac{2\sigma}{\epsilon_Z\eta} \Big) \int_{\Sigma}f_{\sigma}^{p-2}\abs{\nabla f_{\sigma}}^2 \, d\mu_g \\
		&\qquad - \Big(p\epsilon_{\nabla} - \frac{2p\sigma(p\eta+4)}{\epsilon_Z} \Big) \int_{\Sigma}\frac{f_{\sigma}^{p-1}}{\abs{H}^{2(1-\sigma)}}\abs{\nabla H}^2 \, d\mu_g.
	\end{align*}
Suppose that
	\begin{equation*}
		\sigma \leq \frac{\epsilon_Z}{8}\sqrt{\frac{\epsilon_{\nabla}}{p}}.
	\end{equation*}
Set $\eta = 4\sigma/\epsilon_Z$, then
	\begin{equation*}
		\begin{cases}
			\frac{2\sigma_Z}{\epsilon\eta} = \frac{1}{2} \\
			\frac{2p\sigma_Z (p\eta + 4)}{\epsilon} \leq \frac{1}{4}\sqrt{p\epsilon_{\nabla}}(\frac{1}{2}\sqrt{p\epsilon_{\nabla}} + 4) \leq \frac{p\epsilon_{\nabla}}{2}.
			\end{cases}
	\end{equation*}
For the second last inequality to hold we must assume $p \geq 64/\epsilon_{\nabla}$.  We conclude that
	\begin{equation*}
		\frac{d}{dt}\int_{\Sigma}f_{\sigma}^p \, d\mu_g \leq 0.
	\end{equation*}
This implies the lemma with $c_8 = \max\{2, 8/(\epsilon_{\nabla} +1), 64/\epsilon_{\nabla}\}$, $c_9 = \epsilon_Z\sqrt{\epsilon_{\nabla}}/8$ and $C_1 = (\abs{ \Sigma_0 + 1} ) \max_{ \sigma \in [0, 1/2] } ( \max_{ \Sigma_0 } f_{ \sigma } )$.
\end{proof}

An important corollary of this lemma is the following, which states that for larger values of $p$ and smaller values of $\sigma$, powers of $H$ can be absorbed into $f_{\sigma}$.  This property is key in the final iteration argument.
\begin{cor}\label{cor: powers of H absorbed}
For  $p \geq \max\{c_7, 4n^2c_8^2\}$ and $\sigma \leq (c_8/2)/\sqrt{p})$, the estimate
	\begin{equation*}
		\int\limits_{\Sigma} H^n f_{\sigma}^p \, d\mu_g \leq \int\limits_{\Sigma}f_{\sigma'}^p \, d\mu_g.
	\end{equation*}
holds on $t \in [0, T)$.
\end{cor}
\begin{proof}
We need $\sigma' = \sigma + n/p \leq c_8/\sqrt{p}$ for sufficiently large $p$ and small $\sigma$ .  Suppose that
$p \geq \max\{c_7, 4n^2/c_8^2\}$ and $\sigma \leq (c_8/2)/\sqrt{p})$.  Then
	\begin{equation*}
		\sigma' = \sigma + \frac{n}{p} \leq \frac{c_8}{2\sqrt{p}} + \frac{1}{\sqrt{p}}\frac{n}{\sqrt{p}} \leq \frac{c_8}{\sqrt{p}}
	\end{equation*}
as required.
\end{proof}

Lemma \ref{l: f_sigma non-increase} shows that sufficiently high $L^p$ norms of $f_{\sigma}$ are bounded.  We now proceed to derive the desired sup bound on $f_{\sigma}$ by a Stampacchia iteration argument.  The argument rests on the following well-known iteration lemma.
\begin{lem}\label{l: Stampacchia}
	Let $\varphi(t)$, $k_0 \leq t < \infty$, be a non-negative and non-increasing function which satisfies
		\begin{equation*}
			\varphi(h) \leq \frac{ C }{ (h-k)^{\alpha} }\abs{\varphi(k)}^{\beta}
		\end{equation*}
for $h > k \geq k_0$, where $C$, $\alpha$, and $\beta$ are positive constants with $\beta > 1$.  Then
	\begin{equation*}
		\varphi(k_0+d) = 0,
	\end{equation*}
where
	\begin{equation*}
		d^{\alpha} = C\abs{ \varphi(k_0) }^{\beta-1}2^{\alpha\beta/(\beta-1)}.
	\end{equation*}
\end{lem}
For a proof of this lemma we refer to reader to \cite{KS00}.  Continuing with the iteration argument, set $k_0 := \max_{\sigma \in [0,1/2]} \max_{\Sigma_0} f_{\sigma}$.  For any $k \geq k_0$, define the truncated function $f_{\sigma,k} := \max \{ (f_{\sigma} - k ), 0 \}$ and the set $A(k,t) := \{ x \in \Sigma_t : f_{\sigma,k} > 0 \}$.  In exactly the same manner as Proposition \ref{p: using Z pos} we derive the following evolution equation for $f_{\sigma,k}$:
\begin{equation*}
	\begin{split}
		\frac{d}{dt} \int_{A(k,t)} f_{\sigma,k}^p \, d\mu_g &\leq -\frac{p(p-1)}{2}\int_{A(k,t)}f_{\sigma,k}^{p-2}\abs{\nabla f_{\sigma,k}}^2 \, d\mu_g - p\epsilon_{ \nabla } \int_{A(k,t)}\frac{f_{\sigma,k}^{p-1}}{H^{2(1-\sigma)}}\abs{\nabla H}^2 \, d\mu_g \\
		&\qquad + 2p\sigma \int_{A(k,t)}\abs{H}^2f_{\sigma,k}^p \, d\mu_g. \end{split}
\end{equation*}
For $p \geq 8$ we estimate
	\begin{equation*}
		\frac{ p(p-1) }{2}f_{\sigma,k}^{p-2}\abs{ \nabla f_{\sigma,k} }^2 \geq \abs{ \nabla f_{\sigma,k}^{p/2} }^2,
	\end{equation*}
then setting $v_k = f_{\sigma,k}^{p/2}$ and discarding the second term on the right we get
	\begin{equation}\label{e: Stampacchia it 1}
		\frac{d}{dt} \int_{A(k,t)} v_k^2 \, d\mu_g + \int_{A(k,t)} \abs{ \nabla v_k }^2 \, d\mu_g \leq 2p\sigma \int_{A(k,t)}\abs{H}^2f_{\sigma,k}^p \, d\mu_g.
	\end{equation}
We now make us of the Michael-Simon Sobolev inequality \cite{MS73}, which states that for any function $u \in C_0^{0,1}\negmedspace(\Sigma)$ we have
	\begin{equation*}
		\Big( \int_{\Sigma} \abs{u}^{ \frac{n}{n-1} } \, d\mu_g \Big)^{ \frac{n-1}{n} } \leq C_S \int_{\Sigma} \big( \abs{\nabla u } + \abs{ H } u \big) \, d\mu_g,
	\end{equation*}
where $C_S$ (the Sobolev constant) is a constant that depends only on $n$.  The Michael-Simon Sobolev inequality is a generalisation of the standard Sobolev inequality to functions on a submanifold.  The form of the Michael-Simon Sobolev inequality stated above corresponds to the $p=1$ case of the standard Gagliardo-Nirenberg-Sobolov inequality.  To obtain the inequality in the case $1 < p < n$ we set $v := \abs{u}^{\gamma}$, where $\gamma = p(n-1)/(n-p) > 0$, and after a use of Holder's inequality we find
	\begin{equation*}
		\Big( \int_{\Sigma} \abs{u}^{p^*} \, d\mu_g \Big)^{ \frac{1}{p^*} } \leq C_S\gamma\Big( \int_{\Sigma} \abs{\nabla u}^{p} \, d\mu_g \Big)^{ \frac{1}{p} } +  C_S\Big( \int_{\Sigma} \abs{ H }^n \, d\mu_g \Big)^{ \frac{1}{n} } \Big( \int_{\Sigma} \abs{u}^{p^*} d\mu_g \Big)^{ \frac{1}{p^*} }.
			\end{equation*}
We want to take advantage of the good gradient term on the left of \eqref{e: Stampacchia it 1}, and so we need the Sobolev inequality with $p=2$.  Squaring both sides, using $(a+b)^2 \leq 2(a^2 + b^2)$ and then setting $q = n/(n-2)$ if $n > 2$ or any number finite number if $n = 2$, we obtain
	\begin{equation}\label{e: MS sobolev ineq 1}
		\Big( \int_{\Sigma} v_k^{2q} \, d\mu_g \Big)^{ \frac{1}{q} } \leq c_{10}\Big( \int_{\Sigma} \abs{ \nabla v_k }^2 \, d\mu_g \Big) + c_{11}\Big( \int_{\Sigma} \abs{ H }^n \, d\mu_g \Big)^{2/n} \Big( \int_{\Sigma} v_k^{2q} \, d\mu_g \Big)^{ \frac{1}{q} }.
	\end{equation}
Using that by definition $f_{\sigma}$ is zero outside of $A(k,t)$ and Corollary \ref{cor: powers of H absorbed} we estimate
	\begin{equation}
		\begin{split}\label{e: Stampacchia it 4}
		\Big( \int_{A(k,t)} \abs{ H }^n \, d\mu_g \Big)^{2/n} &\leq \Big( \int_{A(k,t)} \abs{ H }^n \frac{ f_{\sigma}^p }{k^p} \, d\mu_g \Big)^{2/n} \\
		&\leq k^{-2p/n} \Big( \int_{A(k,t)} \abs{ H }^n f_{\sigma}^p \, d\mu_g \Big)^{2/n} \\
		&\leq k^{-2p/n} \Big( \int_{A(k,t)} f_{\sigma'}^p \, d\mu_g \Big)^{2/n} \\
		&\leq \Big( \frac{ (\abs{\Sigma_0} + 1)k_0}{k} \Big)^{2p/n}.
		\end{split}
	\end{equation}
Therefore we can fix a $k_1 > k_0$ sufficiently large such that for all $k \geq k_1$ the second term on the right of \eqref{e: MS sobolev ineq 1} can be absorbed into the left giving
\begin{equation}\label{e: MS Sobolev 2}
	\Big( \int_{\Sigma} v_k^{2q} \, d\mu_g \Big)^{ \frac{1}{q} } \leq c_{12}\Big( \int_{\Sigma} \abs{ \nabla v_k }^2 \, d\mu_g \Big).
\end{equation}
Combining equations \eqref{e: Stampacchia it 1} and \eqref{e: MS Sobolev 2} we obtain
	\begin{equation}\label{e: Stampacchia it 2}
		\frac{ d }{ dt }\int_{A(k,t)} v_k^2 \, d\mu_g + \Big( \int_{A(k,t)} v_k^{2q} \, d\mu_g \Big)^{1/q} \leq c_{13}\int_{A(k,t)} \abs{ H }^2 f_{\sigma}^p \, d\mu_g.
	\end{equation}
Integrating this equation from $t = 0$ until some time $\tau \in [0, T]$ we get
	\begin{equation*}
		\int_{\Sigma_{\tau}} v_k^2 \, d\mu_g - \int_{\Sigma_0} v_k^2 \, d\mu_g + \int_0^{\tau} \Big( \int_{A(k,t)} v_k^{2q} d\mu_g \Big)^{1/q} dt \leq c_{13}\int_0^{\tau} \int_{A(k,t)} \abs{ H }^2 f_{\sigma}^p \, d\mu_g \, dt.
	\end{equation*}
By definition of $v_k$, the integral evaluated at the initial time is zero.  Denote by $t_*$ the time when the first integral on the left achieves its supremum, that is $t_* = \sup_{t \in [0, T]} \int_{\Sigma_t} v_k^2 \, d\mu_g$.  We integrate \eqref{e: Stampacchia it 2} until $t_*$ and until $T$ and add the two inequalites, then discarding two unwanted terms on the left and estimating $t_*$ by $T$ on the right we obtain
	\begin{equation}\label{e: Stampacchia it 3}
		\int_{\Sigma_{t_*}} v_k^2 \, d\mu_g + \int_0^{T} \Big( \int_{A(k,t)} v_k^{2q} d\mu_g \Big)^{1/q} dt \leq c_{13}\int_0^{T} \int_{A(k,t)} \abs{ H }^2 f_{\sigma}^p \, d\mu_g \, dt.
	\end{equation}
We now need to estimate the remaining two integrals on the left.  Recalling the standard interpolation inequality for $L^p$ spaces:
	\begin{equation*}
		\norm{ \cdot }_{q_0} \leq \norm{ \cdot }_1^{\theta} \norm{\cdot}_q^{1-\theta},
	\end{equation*}
where $1 \leq q_0 \leq q$ and $0 \leq \theta \leq 1$, we interpolate with $\theta = 1/q_0$ to get
	\begin{equation*}
		\Big( \int_{A(k,t)} v_k^{2q_0} d\mu_g \Big) \leq \Big( \int_{A(k,t)} v_k^2 \, d\mu_g  \Big)^{q_0-1} \Big( \int_{A(k,t)} v_k^{2q} \, d\mu_g \Big)^{1/q}.
	\end{equation*}
Using the above interpolation inequality, and the Holder and Young inequalities we see
	\begin{align*}
		&\Big( \int_0^{T}\int_{A(k,t)} v_k^{2q_0} \, d\mu_g dt \Big)^{1/q_0} \\
		&\qquad \leq \Big( \int_0^{T} \Big( \int_{A(k,t)} v_k^2 \, d\mu_g  \Big)^{q_0-1} \Big( \int_{A(k,t)} v_k^{2q} \, d\mu_g \Big)^{1/q} dt \Big)^{1/q_0}\\
		&\qquad \leq \sup_{t \in [0,T] } \Big( \int_{A(k,t)} v_k^2 \, d\mu_g  \Big)^{ \frac{q_0-1}{q_0} } \cdot \Big( \int_0^{T} \Big( \int_{A(k,t)} v_k^{2q} \, d\mu_g \Big)^{1/q} dt\Big)^{1/q_0} \\
			&\qquad \leq \sup_{t \in [0,T]} \int_{A(k,t)} v_k^2 \, d\mu_g +   q_0^{-1}\Big( \frac{q_0}{q_0-1}\Big)^{-(q_0-1)} \int_0^{T} \Big( \int_{A(k,t)} v_k^{2q} \, d\mu_g \Big)^{1/q} dt \\
		&\qquad \leq \sup_{t \in [0,T]} \int_{A(k,t)} v_k^2 \, d\mu_g +  \int_0^{T} \Big( \int_{A(k,t)} v_k^{2q} \, d\mu_g \Big)^{1/q} dt,
	\end{align*}
and using this to estimate \eqref{e: Stampacchia it 3} from below we obtain
\begin{equation*}
	\Big( \int_0^{T}\int_{A(k,t)} v_k^{2q_0} \, d\mu_g dt \Big)^{1/q_0} \leq c_{13}\int_0^{T}\int_{A(k,t)} \abs{ H }^2 f_{\sigma}^p \, d\mu_g dt.
\end{equation*}
Set $\norm{ A(k,t) } = \int_0^T\int_{A(k,t)} d\mu_g \, dt$.  Estimating the left from below by the Holder inequality gives
\begin{equation*}
	\Big( \int_0^{T}\int_{A(k,t)} v_k^{2q_0} \, d\mu_g dt \Big)^{1/q_0} \geq \Big( \int_0^{T}\int_{A(k,t)} v_k^2 \, d\mu_g dt \Big) \cdot \norm{A(k,t)}^{ \frac{1}{q_0} - 1},
\end{equation*}
and the right from above with Holder's inequality
	\begin{equation*}
		\int_0^{T}\int_{A(k,t)} v_k^2 \, d\mu_g dt \leq c_{13}\norm{A(k,t)}^{2-1/q_0-1/r} \Big( \int_0^{T} \int_{A(k,t)} \abs{ H }^{2r}f_{\sigma}^{pr} \, d\mu_g \, dt \Big)^{1/r}.
	\end{equation*}
Choose $r$ sufficiently large so that $\gamma := 2 - 1/q_0 - 1/r > 1$, and then by the same argument as \eqref{e: Stampacchia it 4}, the second factor on the right can be bounded by a constant. For $h > k \geq k_1$ we estimate the left from below as follows:
	\begin{equation*}
		\int_0^{T}\int_{A(k,t)} v_k^2 \, d\mu_g dt \geq \abs{h-k}^p\norm{A(h,t)},
	\end{equation*}
finally obtaining
	\begin{equation*}
		\abs{h-k}^p\norm{A(h,t)} \leq c_{14}\norm{A(k,t)}^{\gamma}
	\end{equation*}
which again holds for all $h > k \geq k_1$.  From Lemma \ref{l: Stampacchia} and the defintion of $A(k,t)$ it follows that $f_{\sigma} \leq k_1 + d$, where $d^p = c_52^{p\gamma/(\gamma+1)}\norm{A(k_1,t)}^{\gamma-1}$.   Since $\int_{A(k_1,t)} \, d\mu_g \leq \abs{\Sigma_t} \leq \abs{\Sigma_0}$ and the maximal time of existence $T$ is finite, we conclude that $f_{\sigma} \leq C_0$, where $C_0$ is positive uniform constant, and the theorem follows.

\section{A gradient estimate for the mean curvature}\label{s: A gradient estimate for the mean curvature}

In this section we derive a gradient estimate for the mean curvature.  This will be used in the following section to compare the mean curvature of the submanifold at different points.
\begin{thm}\label{t: gradient estimate}
Under the assumptions of Main Theorem \ref{mthm: main thm 2},
for each $\eta > 0$ there exists a constant $C_{\eta}$ depending only on $\eta$ and $\Sigma_0$ such that the estimate
	\begin{equation}\label{eqn: grad est H}
		\abs{\nabla H}^2 \leq \eta\abs{H}^4 + C_{\eta}
	\end{equation}
holds on $\Sigma \times [0,T)$.

\end{thm}
We begin by deriving a number of evolution equations.

\begin{prop}
There exists a constant $A$ depending only on $\Sigma_0$ such that
	\begin{equation*}
		\frac{ \p }{ \p t }\abs{ \nabla H }^2 \leq \Delta\abs{ \nabla H }^2 - 2\abs{ \nabla^2 H }^2 +A\abs{H}^2\abs{\nabla h}^2.	
	\end{equation*}
\end{prop}

\begin{proof}
Differentiating $\abs{ \nabla H }^2$ in time gives
	\begin{align}
		\notag \frac{ \p }{ \p t }\abs{ \nabla H }^2 &= \frac{\p}{\p t} \left\langle\nabla H, \nabla H\right\rangle \\
		\notag &= 2\left\langle \nabla_t\nabla H, \nabla H\right\rangle \\
		\notag &= 2\gp( \nabla_k\nabla_t H + \Rp(\p_k, \p_t)H, \nabla_k H) \\
		&= 2\gp\big( \nabla_k(\Delta H + H \cdot h_{pq}h_{pq}), \nabla_k H \big) + 2\gp\big( \Rp(\p_k, \p_t)H, \nabla_k H \big). \label{eq: gradient est 1}
	\end{align}
To manipulate the last line into the desired form we need the following two formulae:
	\begin{align*}
		\Delta\abs{ H }^2 &= 2\gp( \Delta\nabla_k H, \nabla_k H) + 2\abs{ \nabla^2 H }^2 \\
		\Delta\nabla_k H &= \nabla_k\Delta H + \nabla_p\big( \Rp(\p_k, \p_p)H \big) + \Rp(\p_k, \p_p)\nablap_p H + Rc_{pk}\nablap_p H.
	\end{align*}
Substituting these into \eqref{eq: gradient est 1} and observing that the Gauss equation \eqref{eq:spatialGauss} and the Ricci equation \eqref{eq:Ricci} are of the form $R=h\ast h$ and $\Rp=h\ast h$, and that the timelike Ricci equation \eqref{eq:timelikeRicci} is of the form $\Rp( \cdot ,\partial_t)=h*\nabla h$, we find	
\begin{equation*}
		\frac{ \p }{ \p t }\abs{ \nabla H }^2 = \Delta\abs{ \nabla H }^2 - 2\abs{ \nabla^2 H }^2 + h*h*\nabla h*\nabla h.
	\end{equation*}
The proposition now follows from the Cauchy-Schwarz inequality and the Pinching Lemma.	
\end{proof}

\begin{prop}
For any $N_1$, $N_2 > 0$ we have the estimates
\begin{align}
	&\frac{ \p }{ \p t } \abs{ H }^4 \geq \Delta\abs{ H }^2 - 12\abs{H}^2\abs{ \nabla H }^2 + \frac{4}{n}\abs{H}^6 \label{e: evol H4} \\
	\begin{split}
		&\frac{ \p }{ \p t } \big( (N_1 + N_2\abs{H}^2)\abs{ \ho }^2 \big) \leq \Delta\big( (N_1 + N_2\abs{H}^2)\abs{ \ho }^2 \big) - \frac{ 4(n-1) }{ 3n }(N_2 - 1)\abs{ H }^2\abs{ \nabla h }^2 \\
		&\quad- \frac{ 4(n-1) }{ 3n }(N_1 - c_1(N_2))\abs{ \nabla h }^2 + c_2(N_1, N_2)\abs{ \ho }^2( \abs{ H }^4 + 1 ), \label{e: evol N}
	\end{split}
\end{align}
where $c_1$ and $c_2$ depend only on $\Sigma_0$, $N_1$ and $N_2$.
\end{prop}
\begin{proof}
The evolution equation for $\abs{H}^4$ is easily derived from that of $\abs{H}^2$:
\begin{equation*}
	\frac{ \p }{ \p t } \abs{ H }^4 = \Delta\abs{ H }^2 - 2\abs{ \nabla \abs{H}^2 }^2 - 4\abs{H}^2\abs{ \nabla H }^2 + 4R_2\abs{H}^2.
\end{equation*}
Equation \eqref{e: evol H4} follows from the use of $\abs{ \nabla\abs{H} }^2 \leq \abs{ \nabla H }^2$ and $R_2 \geq 1/n \abs{H}^4$.  To prove \eqref{e: evol N}, from the evolution equations for $\abs{h}^2$ and $\abs{H}^2$ we derive
\begin{align*}
	&\frac{ \p }{ \p t } \big( (N_1 + N_2\abs{H}^2)\abs{ \ho }^2 \big) \\
	&\quad= \Delta\big( (N_1 + N_2\abs{H}^2)\abs{ \ho }^2 \big) - 2N_2 \big\langle \nabla_i \abs{H}^2, \nabla_i \abs{ \ho }^2 \big\rangle - 2N_2\abs{ \ho }^2\abs{ \nabla h }^2 + 2N_2R_2\abs{ \ho }^2  \\
	&\quad- 2(N_1 + N_2\abs{ H }^2)( \abs{ \nabla h }^2 - \frac{1}{n}\abs{ \nabla H }^2 ) + 2(N_1 + N_2\abs{ H }^2 )(R_1 - \frac{1}{n} R_2).
\end{align*}
We estimate the second term on the right as follows:
\begin{align*}
	-2N_2 \big\langle \nabla_i \abs{H}^2, \nabla_i \abs{ \ho }^2 \big\rangle &\leq 8N_2\abs{h}\abs{ \ho }\abs{ \nabla H }\abs{ \nabla h } \\
	&\leq 8N_2\abs{ H }\sqrt{n}\abs{ \nabla h }^2 \sqrt{C_0}\abs{H}^{1 - \delta/2} \\
	&\leq \frac{ 4(n-1) }{3n}\abs{H}^2\abs{ \nabla h }^2 + c_1(N_2)\abs{ \nabla h }^2.
\end{align*}
Using Young's inequality, $R_2 \leq \abs{h}^2\abs{H}^2$, and $R_1 - 1/n \, R_2 \leq 2\abs{ \ho }^2\abs{ h }^2$ we estimate
\begin{equation*}
	2N_2R_2\abs{ \ho }^2  + 2(N_1 + N_2\abs{ H }^2 )(R_1 - \frac{1}{n} R_2) \leq c_2(N_1, N_2)\abs{ \ho }^2(\abs{ H }^4 + 1),
\end{equation*}
and equation \eqref{e: evol N} now follows.
\end{proof}

\begin{proof}[Proof of Theorem \ref{t: gradient estimate}]
Consider $f := \abs{ \nabla H }^2 + (N_1 + N_2\abs{H}^2)\abs{ \ho }^2$.  From the evolution equations derived above we see $f$ satisfies
\begin{align*}
	\frac{ \p }{ \p t } f &\leq \Delta f + A\abs{H}^2\abs{ \nabla h }^2 - \frac{ 4(n-1) }{ 3n }(N_2 - 1)\abs{ H }^2\abs{ \nabla h }^2 - \frac{ 4(n-1) }{ 3n }(N_1 - c_1(N_2))\abs{ \nabla h }^2 \\
	&\quad+ c_2(N_1, N_2)\abs{ \ho }^2 ( \abs{ H }^4 + 1 ).
\end{align*}
Choose $N_2$ large enough to consume the positive term arising from the evolution equation for $\abs{ \nabla H }^2$.  This leaves
\begin{equation*}
	\begin{split}
	\frac{ \p }{ \p t } f &\leq \Delta f - \frac{ 4(n-1) }{ 3n }(N_2 - 1)\abs{ H }^2\abs{ \nabla h }^2 - \frac{ 4(n-1) }{ 3n }(N_1 - c_1(N_2))\abs{ \nabla h }^2 \\
	&\qquad + c_2(N_1, N_2)\abs{ \ho }^2 ( \abs{ H }^4 + 1 ). \end{split}
\end{equation*}
Now consider $g := f - \eta\abs{H}^4$.  From the above evolution equations we have
\begin{align*}
	\frac{ \p }{ \p t } g &\leq \Delta g - \frac{ 4(n-1) }{ 3n }(N_2 - 1)\abs{ H }^2\abs{ \nabla h }^2 - \frac{ 4(n-1) }{ 3n }(N_1 - c_1(N_2))\abs{ \nabla h }^2 \\
	&\quad+ c_2(N_1, N_2)\abs{ \ho }^2 ( \abs{ H }^4 + 1 ) \big) + 12\eta \abs{H}^2\abs{ \nabla H }^2 - \frac{4\eta}{n}\abs{H}^6.
\end{align*}
By choosing $N_2$ sufficiently large the gradient term on the last line can be absorbed, and then we choose $N_1$ larger again to make the $\abs{ \nabla h }^2$ term negative.  We finally discard the negative gradient terms to get
\begin{equation*}
	\frac{ \p }{ \p t }g \leq \Delta g + c_2(N_1,N_2)\abs{ \ho }^2 ( \abs{ H }^4 + 1 )  - \frac{4\eta}{n}\abs{H}^6.
\end{equation*}
Using Theorem \ref{t: pinch trace 2ff} and Young's inequality we further estimate
\begin{equation*}
	\frac{ \p }{ \p t }g \leq \Delta g + c_3,
\end{equation*}
from which we conclude $g \leq c_4$.  The gradient estimate now follows from the definition of $g$.
\end{proof}

\section{Contraction to a point}\label{s: Contraction to a point and convergence}

In Section \ref{s: Higher derivative estimates and long time existence}
we established that MCF has a unique solution on a finite maximal time interval $0 \leq t < T$ determined by the blowup of the second fundamental form.  With the results of the previous two sections in place, we can now show that the diameter of the submanifold approaches zero as $t \rightarrow T$, or put another away, the submanifold is shrinking to a point.  This combined with Theorem \ref{t: MCF exists on maximal time interval} then completes the first part of the Main Theorem \ref{mthm: main thm 2}.

\begin{thm}\label{t: diam goes to zero}
Under the conditions of Main Theorem \ref{mthm: main thm 2},  $\diam \Sigma_t \rightarrow 0$ as $t \rightarrow T$.
\end{thm}

The proof is an adaption of Hamilton's use of Myer's Theorem in Section 15 of \cite{rH82}, however here our pinching condition gives a strictly positive lower bound on the sectional curvature of $\Sigma_t$ and we can use Bonnet's Theorem instead.  A proof of Bonnet's Theorem can be found in many places, for example \cite{pP06}.
\begin{thm}
	Let $M$ be a complete Riemannian manifold and suppose that $x\in M$ such that the sectional curvature satisfies $K \geq K_\text{min} > 0$ along all geodesics of length $\pi/\sqrt{K_\text{min}}$ from $x$.  Then $M$ is compact and $\diam M \leq \pi/\sqrt{K_\text{min}}$.
\end{thm}

We will also need the following result due to Bang-Yen Chen:
\begin{prop}\label{p: Chen's result}
For $n \geq 2$, if $\Sigma^n$ is a submanifold of $\mathbb{R}^{n+k}$, then at each point $p \in \Sigma^n$ the smallest sectional curvature $K_{\text{min}}$ satisfies
	\begin{equation*}
		K_{\text{min}}(p) \geq \frac{1}{2}\Big( \frac{1}{n-1}\abs{ H }(p)^2 - \abs{ h }(p)^2 \Big).
	\end{equation*}
\end{prop}
The proof is a consequence of careful estimation of terms appearing in the Gauss equation and can be found in \cite{byC93}*{Lemma 3.2}  Combining this with our pinching assumption we see
\begin{equation}\label{e: Kmin}
	K_{\text{min}}(p) \geq \frac{1}{2}\Big( \frac{1}{n-1} - c \Big)\abs{ H }(p)^2 = \epsilon^2\abs{ H }(p)^2 > 0.
\end{equation}

\begin{lem}\label{l: Hmin over Hmax}
	The ratio $\abs{H}_{\text{max}}/\abs{H}_{\text{min}} \rightarrow 1$ as $t \rightarrow T$.
\end{lem}

\begin{proof}
From Theorem \ref{t: gradient estimate} we know that for each $\eta > 0$ there exists a constant $C_{\eta}$ such that $\abs{\nabla H} \leq \eta\abs{H}^{2} + C_{\eta}$ on $0 \leq t < T$.  Since $\abs{H}_{\text{max}} \rightarrow \infty$ as $t \rightarrow T$, there exists a $\tau(\eta)$ such that $C_{\eta/2} \leq 1/2\eta\abs{H}^{2}_{\text{max}}$ for all $\tau \leq t < T$, and so $\abs{\nabla H} \leq \eta\abs{H}^{2}_{\text{max}}$ for all $t \geq \tau$.  
For any $\sigma\in(0,1)$ choose $\eta = \frac{\sigma(1-\sigma)\varepsilon}{\pi}$.  Let $t\in [\tau(\eta),T)$, and let $x$ be a point with $\abs{ H }(x) =\abs{H}_\text{max}$.  Then along any geodesic of length $\frac{\pi}{\varepsilon\sigma H_\text{max}}$ from $x$ we have $\abs{H}\geq \abs{H}_\text{max}-\frac{\pi}{\varepsilon\sigma \abs{H}_\text{max}}\eta\abs{H}^2_\text{max}=\sigma\abs{H}_\text{max}$, and consequently the sectional curvatures satisfy $K\geq \varepsilon^2\sigma^2\abs{H}_\text{max}^2$.  The Bonnet Theorem applies to prove that $\diam M\leq \frac{\pi}{\varepsilon\sigma H_\text{max}}$, so that $\abs{H}_\text{min}\geq \sigma\abs{H}_\text{max}$ on the entire submanifold $\Sigma_t$ for all $t \in [\tau(\eta),T)$.
\end{proof}
Since $\abs{ H }_{\text{max}} \rightarrow \infty$ as $t \rightarrow T$, the last lemma show that the same is also true for $\abs{ H }_{\text{min}}$.  Bonnet's Theorem now implies that $\diam \Sigma_t \rightarrow 0$ as $t \rightarrow T$, which completes the proof of the first part of Main Theorem \ref{mthm: main thm 2}.

\section{The normalised flow and convergence to the sphere}\label{sec: The normalised flow and and convergence to the sphere}

The second part of the Main Theorem \ref{mthm: main thm 2} deals with the asymptotic shape of the evolving submanifold as $t \rightarrow T$.  Here we shall show that a suitably normalised flow exists for all time and that the (normalised) submanifold converges to a sphere as time approaches infinity.  This clarifies the sense in which un-normalised submanifold shrinks to a `round' point.

We denote quantities pertaining to the normalised flow by a tilde.  We are going to define the normalised flow in such a way so that the size of the area of the evolving submanifold remains constant.  We do this by multiplying the solution of MCF at each time by a positive constant $\psi(t)$ so that the measure of the the normalised submanifold $\tilde{\Sigma}_t$ is equal to the measure of the initial submanifold $\Sigma_0$:
\begin{equation*}
	\tilde{F}( \cdot, t ) = \psi(t) F(\cdot, t)
\end{equation*}
such that
\begin{equation}\label{e: normalised 1}
	\int_{\tilde{ \Sigma} } d\mu_{ \tilde{g}(t) } = \abs{ \Sigma_0 }.
\end{equation}
The above rescaling is so far only a rescaling in space.
\begin{prop}\label{p: rescaling}
Suppose we rescale an immersion $F$ by $\tilde{F} := \psi \cdot F$, where $\psi$ is a positive constant.  Then various geometric quantities rescale as follows:
\begin{gather*}
	\tilde{g} = \psi^2 g \\
	\tilde{h} = \psi h \\
	\tilde{H} = \psi^{-1} H \\
	\abs{ \tilde{h} }^2 = \psi^{-2} \abs{h}^2 \\
	  d\tilde{\mu}_{ \tilde{g}(t) } = \psi^n d\mu_{ g(t) } \\
	  \tilde{ \nabla } = \nabla \\
	  \tilde{ \Delta } = \psi^{-2} \Delta.
\end{gather*}
\end{prop}
\begin{proof}
Beginning with the metric, we have
\begin{align*}
	\tilde{g}_{ij} &= \big\langle \tilde{F}_*\p_i, \tilde{F}_*\p_j \big\rangle \\
	&= \psi^2 \big\langle F_*\p_i, F_*\p_j \big\rangle \\
	&= \psi^2 g_{ij}.
\end{align*}
We also have $\tilde{g}^{-1} = \psi^{-2 } g^{-1}$.  For the second fundamental form, using Gauss' formula and noting $\tilde{ \nu }_{ \alpha } = \nu_{ \alpha }$ we have
\begin{align*}
	\tilde{h}_{ij}{^{\alpha}} &= -\Big\langle \frac{ \p \tilde{F}^{ \alpha } }{ \p x^i \p x^j }, \tilde{\nu}_{ \alpha } \Big\rangle \\
	&= - \psi \Big\langle \frac{ \p F^{ \alpha } }{ \p x^i \p x^j }, \nu_{ \alpha } \Big\rangle \\
	&= \psi h_{ij}{^{\alpha}},
\end{align*}
which is just $\tilde{h}_{ij} = \psi h_{ij}$.  The mean curvature follows from the inverse metric and the second fundamental form:
\begin{equation*}
	\tilde{H} = \tilde{g}^{ij} \tilde{h}_{ij} = \psi^{-1} H.
\end{equation*}
Similarly for the length squared of the second fundamental form we get
\begin{equation*}
	\abs{ \tilde{h} }^2 = \tilde{g}^{ik}\tilde{g}^{jl}\tilde{h}_{ij}\tilde{h}_{kl} = \psi^{-2}\abs{h}^2.
\end{equation*}
For the measure we have
\begin{equation*}
	d\tilde{\mu}_{ \tilde{g}(t) } = \sqrt{ \text{det } \tilde{g}_{ij} } = \psi^n \sqrt{ \text{det } g_{ij} } = \psi^n d\mu_{ g(t) }.
\end{equation*}
The Christoffel symbols are scale-invariant:
\begin{align*}
	\tilde{ \Gamma }_{ij}^k &= \frac{1}{2} \tilde{g}^{kl} ( \p_i \tilde{g}_{jk} + \p_j \tilde{g}_{ik} - \p_k \tilde{g}_{ij} ) \\
	&= \frac{1}{2} \psi^{-2} g^{kl} ( \psi^2 \p_i g_{jk} + \psi^2 \p_j g_{ik} - \psi^2 \p_k g_{ij} ) \\
	&= \Gamma_{ij}^j,
\end{align*}
and thus so too is the connection.  Finally, the Laplacian is given by
\begin{equation*}
	\tilde{ \Delta } = \tilde{g}^{ij}\tilde{ \nabla }_i \tilde{ \nabla }_j = \psi^{-2} g^{ij} \nabla_i \nabla_j = \psi^{-2} \Delta.
\end{equation*}
\end{proof}
We now derive the evolution equation for the normalised flow with respect to the time variable $t$.  Differentiating \eqref{e: normalised 1} with respect to $t$ we have
\begin{align*}
	\frac{d}{dt} \int_{ \Sigma }  \, d\tilde{\mu}_{\tilde{g}(t)} &= \int_{ \Sigma } \frac{d}{dt} d\tilde{\mu}_{ \tilde{g}(t) } \\
	&= 	 \int \frac{d}{dt} \big( \psi^{n} d\mu_{ g(t) } \big) \\
	&=  \int \Big( n\psi^{n-1} \frac{ d }{ dt } \psi \cdot d\mu_{ g(t) } - \psi^n \abs{H}^2 \Big) \, d\mu_{ g(t) } = 0,
\end{align*}
which implies
\begin{equation*}
	\psi^{-1} \frac{ \p \psi }{ \p t } = \frac{1}{n} \frac{ \int \abs{H}^2 \, d\mu_{ g(t) } }{ \int d\mu_{ g(t) } },
\end{equation*}
where the last line follows because the rescaling factor is a function of time and the integration is over spatial variables.  Define the average of the squared length of the mean curvature over the submanifold by
\begin{equation*}
	 \hbar := \frac{ \int \abs{H}^2 \, d\mu_{ g(t) } }{ \int d\mu_{ g(t) } }.
\end{equation*}
The evolution equation for the normalised flow with respect to the time variable $t$ is now given by
\begin{align*}
	\frac{ \p \tilde{F} }{ \p t } &= \frac{ \p \psi }{ \p t } F + \psi \frac{ \p F }{ \p t } \\
	&= \psi^2 \Big( \tilde{H} + \frac{1}{n} \tilde{\hbar}  \tilde{F} \Big).
\end{align*}
We now rescale in time to divide out the factor of $\psi^2$ in the above equation.  Note that from now on a tilde represents a rescaling in both space and time, and not only a rescaling in space as was previously the case.  We define the rescaled time variable by
\begin{equation*}
	\tilde{t}(t) := \int_0^t \psi^2(\tau) \, d\tau,
\end{equation*}
and so $\p \tilde{t}/ \p t = \psi^2$.  We now have
\begin{align*}
	\frac{ \p \tilde{F} }{ \p \tilde{t} } &= \psi^{-2} \frac{ \p \tilde{F} }{ \p t } \\
	&= \tilde{H} + \frac{1}{n} \tilde{\hbar}  \tilde{F},
\end{align*}
where this normalised flow is now defined on the time interval $0 \leq \tilde{t} < \tilde{T}$.  Next we want to show how various estimates and evolution equations for the normalised flow can be obtained from their un-normalised counterparts.  The scaling-invariant estimates are the easiest, since the the normalising factor $\psi$ simply cancels from both sides and the same estimates hold:

\begin{prop}
The following estimates hold for the normalised flow:
\begin{subequations}
	\begin{align}
		&\abs{ \tilde{h} }^2 \leq c \abs{ \tilde{H} }^2 \\
		& \frac{ \abs{ \tilde{H} }^2_{\text{min}} }{ \abs{ \tilde{H} }^2_{\text{max}} } \rightarrow 1 \text{ as } \tilde{t} \rightarrow \tilde{T} \label{eqn: normalised grad est} \\
		& \tilde{K}_{\text{min}} \geq \epsilon^2\abs{ \tilde{H} }^2 \label{eqn: Chen's est normalised}
	\end{align}
\end{subequations}
\end{prop}
The following lemma shows how evolution equations for the normalised flow can be easily obtained from their un-normalised counterparts:
\begin{lem}\label{lem: degree}
Suppose that $P$ and $Q$ depend on $g$ and $h$, and that $P$ satisfies the (un-normalised) evolution equation $\p P/\p t = \Delta P + Q$.  If $P$ has ``degree" $\alpha$, that is, $\tilde{P} = \psi^{ \alpha } P$, then $Q$ has degree $(\alpha-2)$ and $\tilde{P}$ satisfies the normalised evolution equation
\begin{equation*}
	\frac{ \p \tilde{P} }{ \p \tilde{t} } = \tilde{\Delta}\tilde{P} + \tilde{Q} + \frac{ \alpha }{ n } \tilde{ \hbar } \tilde{P}.
\end{equation*}
\end{lem}\label{l: degree}
For a proof of this lemma see Lemma 17.1 of \cite{rH82} and Lemma 9.1 of \cite{gH84}.  The evolution equation for the metric does not follow from this lemma, but is easily derived in the same way as the un-normalised equation.
\begin{prop}
The evolution equation for metric under the the normalised flow is given by
\begin{equation*}
	\frac{ \p \tilde{g}_{ij} }{ \p \tilde{t} } =  -2 \tilde{H} \cdot \tilde{h}_{ij} + \frac{2}{n} \tilde{ \hbar } \tilde{g}_{ij}.
\end{equation*}
\end{prop}
\begin{proof}
We compute
\begin{align*}
	\frac{ \p }{ \p \tilde{t} } \tilde{g}_{ij} &= \psi^{-2} \frac{ \p }{ \p t } \tilde{g}_{ij} \\
	&= \psi^{-2} ( \p_t \psi^2 g_{ij} + \psi^2 \p_t g_{ij} ) \\
	&= \psi^{-2} \Big( 2 \psi^2 \frac{1}{n} \hbar g_{ij} - \psi^2 2 H \cdot h_{ij} \Big) \\
	&= -2 H \cdot h_{ij} + \frac{2}{n} \hbar g_{ij} \\
	&= -2 \tilde{H} \cdot \tilde{h}_{ij} + \frac{2}{n} \tilde{ \hbar } \tilde{g}_{ij}.
\end{align*}
\end{proof}

Next we want to show that the mean curvature of the evolving normalised submanifold is bounded below by a constant greater than zero, and bounded above by a finite constant.  As we know of no suitable isoperimetric inequality in high codimesion, we adapt Hamilton's intrinsic arguments in \cite{rH82} to our setting.  We will need to use the following fundamental results in comparison geometry to prove these estimates:

\begin{thm}[Bishop-Gromov, G\"{u}nther volume comparison theorem]
Let $M$ be a complete Riemannian manifold and $B(r)$ a ball of radius $r$ in $M$.  Denote by $V^k(r)$ the volume of a ball of radius $r$ in the complete Riemannian manifold of constant curvature $k$.
\begin{enumerate}
	\item If the Ricci curvature of $M$ is bounded below by $Rc \geq (n-1)kg$, then
		\begin{equation*}
			vol( B (r) ) \leq V^k(r).
		\end{equation*}
	\item If the sectional curvature of $M$ is bounded above by some constant $K > 0$, then
		\begin{equation*}
		vol( B (r) ) \geq V^K(r).
		\end{equation*}
\end{enumerate}
\end{thm}
For a proof of these theorems we refer the reader to \cite{GHL04}.

\begin{lem}[Klingenberg's Lemma]
Suppose that $M$ is a compact manifold and denote the length of the shortest closed closed geodesic in $M$ by $l_{ \text{short} }$.  If the sectional curvature of $M$ is bounded above by some constant $K > 0$, then the injectivity radius of $M$ is bounded below by
\begin{equation*}
	inj_{g}(M) \geq \min \Big\{ \frac{ \pi }{ \sqrt{K} }, \frac{ 1 }{ 2 } l_{ \text{short} } \Big\}.
\end{equation*}
\end{lem}
For a proof of Klingenberg's Lemma we refer the reader to \cite{pP06}.  Since the second fundamental form of the evolving normalised submanifolds is bounded above, the length of the smallest closed geodesic must be bounded below, and therefore so too the injectivity radius:  In order for a small loop to be forming, the second fundamental form must be blowing-up, and this is not the case.  Heintze and Karcher derive an explicit lower bound for the length of the shortest closed geodesic in \cite{HK}, although it suffices for our purposes to note that this is greater than zero.

\begin{prop}
We have
	\begin{equation*}
		\abs{ \tilde{H} }_{\text{max}} \leq C_{\text{max}} < \infty
	\end{equation*}
for all time $\tilde{t} \in [0, \tilde{T})$.
\end{prop}
\begin{proof}
From equation \eqref{eqn: Chen's est normalised} the intrinsic sectional curvature of $\tilde{\Sigma}_{ \tilde{t} }$ satisifies $\tilde{K}_{ \text{min} } \geq 0$.  The Bishop-Gromov volume comparison theorem now implies $\vol( \tilde{\Sigma} ) \leq C\tilde{d}^n$, where $\tilde{d}$ the diameter.  From Bonnet's Theorem we also have $\tilde{d} \leq C/\sqrt{ \abs{ \tilde{H} }_{\text{min}} }$, and thus $\tilde{V} \leq C\abs{ \tilde{H} }^{-n/2}_{\text{min}}$.  In the normalised setting $\vol( \tilde{\Sigma} ) = \abs{ \Sigma_0 }$, so $\abs{ \tilde{H} }_{\text{min}} \leq C$ and then \eqref{eqn: normalised grad est} implies that $\abs{ \tilde{H} }_{ \text{max} } \leq C$.  
\end{proof}

\begin{prop}
There exists as constant $C_{\text{min}}$ depending only on $\Sigma_0$ such that
	\begin{equation*}
		\abs{ \tilde{H} }_{\text{min}} \geq C_{\text{min}} > 0
	\end{equation*}
holds for all time $\tilde{t} \in [0, \tilde{T})$.
\end{prop}
\begin{proof}
We work with the universal cover $\tilde{ \tilde{ \Sigma } }$ of the normalised submanifold $\tilde{ \Sigma }$.  By the G\"{u}nther volume comparison theorem, the volume of $\tilde{ \tilde{ \Sigma } }$ is some multiple its injectivity radius: $\vol(\tilde{ \tilde{ \Sigma } } ) \geq C \inj( \tilde{ \tilde{ \Sigma } } )^n$.  From the Gauss equation and the Pinching Lemma, the intrinsic sectional curvature of $\tilde{ \Sigma }$ is bounded above by some multiple of $\abs{ \tilde{H} }_{\text{max}}$, which is uniformly bounded above by the previous proposition.  Moreover, since the second fundamental form of the normalised submanifolds is also bounded above, from Klingernberg's Lemma we obtain a lower bound for the injectivity radius.  We may now estimate
\begin{equation}\label{eqn: H below 1}
		\vol( \tilde{ \tilde{ \Sigma } } ) \geq C \inj(\tilde{ \tilde{ \Sigma } } )^n \geq C \Big( \frac{ \pi }{ \sqrt{ K } } \Big)^n \geq C \abs{ \tilde{H} }_{ \text{max} }^{- \frac{ n }{ 2 } }.
\end{equation}
The evolving submanifold is not undergoing any topological change before the singularity time, so by Bonnet's Theorem the first fundamental group of $\tilde{\Sigma}$ is finite and constant in time.  We have
\begin{equation*}
	\vol(\tilde{ \tilde{ \Sigma } } ) = \abs{ \pi_1( \tilde{ \Sigma } ) } \vol{ \tilde{ \Sigma } },
\end{equation*}
and since both the first fundamental group and volume of the normalised submanifold are constant in time, $\vol(\tilde{ \tilde{ \Sigma } } )$ is also constant.  This combined with equation \eqref{eqn: H below 1} gives a lower bound on $\abs{ \tilde{H} }_{ \text{max} }$, and then equation \eqref{eqn: normalised grad est} gives the desired lower bound on  $\abs{ \tilde{H} }_{ \text{min} }$.
\end{proof}

\begin{prop}
We have
\begin{equation*}
	\int_0^T \abs{ H }^2_{ \text{max} }(t) \, dt = \infty.
\end{equation*}
\end{prop}
\begin{proof}
Follow the proof of Theorem 15.3 in \cite{rH82} with $R_{\max}$ replaced by $\abs{ H }^2_{ \text{max} }$ and use
\begin{equation*}
	\frac{ \p }{ \p t } \abs{ H }^2 \leq \Delta \abs{ H }^2 + 2c\abs{ H }_{ \text{max} }^2 \abs{ H }^2.
\end{equation*}
\end{proof}

\begin{prop}
The normalised flow exists for all time, that is, $\tilde{T} = \infty$.
\end{prop}
\begin{proof}
We have $d\tilde{t}/ dt = \psi^2$ and $\abs{ \tilde{ H } }^2 = \psi^{-2} \abs{ H }^2$, so
\begin{equation*}
	\int_0^{ \tilde{T} } \tilde{\hbar}(\tilde{ t }) \, d\tilde{ t } = \int_0^T \hbar(t) \, dt = \infty;
\end{equation*}
however $\tilde{\hbar} \leq \tilde{H}_{ \text{max} }^2 \leq C_{ \text{max} }^2$ and therefore $\tilde{T} = \infty$.
\end{proof}

The key step in the convergence argument is to show that the length of the traceless second fundamental form decays exponentially in time.  Similar to the un-normalised setting, one considers the scale-invariant quantity $\abs{ \ho }^2/ \abs{ H }$.  The reaction terms of the evolution equation for $\abs{ \ho }^2/ \abs{ H }$ are again not quite favourable enough to use the maximum principle directly, and one proceeds in a similar manner to the un-normalised setting via intergral estimates (see \cite{gH84} and below).  The Stampacchia iteration is not needed, but only the Poincar\'{e} inequality obtained from integrating Simons' identity (see Propostion \ref{p: normalised integral Pinching Improves} below).  We shall present a new argument based on the maximum principle, which simplifies the existing argument by avoiding the use of integral estimates.  Since this argument in new even for the case of hypersurfaces, we first treat the codimension one case as considered by Huisken in \cite{gH84}.

\begin{prop}
Suppose $\tilde{\Sigma}_{ \tilde{ t } }$ is an initially strictly convex hypersurface smoothly immersed in $\mathbb{R}^{n+1}$ moving by the normalised mean curvature flow.  For all time $\tilde{t} \in [0, \infty)$ we have the estimate
\begin{equation*}
	\abs{ \tilde{\nabla} \tilde{h} }^2 + \abs{ \tilde{\ho} }^2 \leq Ce^{ -\delta \tilde{t} }.
\end{equation*}
\end{prop}
\begin{proof}
The idea is to consider $f:= \epsilon \abs{ \nabla h }^2 + N \abs{ \ho }^2/\abs{ H }^2$, where $\epsilon > 0$ will be chosen small and $N$ sufficiently large.  For the moment we work in the un-normalised setting.  The evolution equation for $\abs{ \nabla h }^2$ is of the form
\begin{equation*}
	\frac{ \p }{ \p t } \abs{ \nabla h }^2 = \Delta \abs{ \nabla h }^2 - 2\abs{ \nabla^2 h }^2 + h*h*\nabla h * \nabla h,
\end{equation*}
so we obtain the estimate
\begin{equation*}
	\frac{ \p }{ \p t } \abs{ \nabla h }^2 \leq \Delta \abs{ \nabla h }^2 - 2\abs{ \nabla^2 h }^2 + c_1 \abs{ H }^2 \abs{ \nabla h }^2.
\end{equation*}
The evolution equation for $\abs{ \ho }^2 / \abs{ H }^2$ is given by
\begin{equation*}
	\frac{ \p }{\p t} \Big( \frac{ \abs{ \ho }^2 }{ \abs{ H }^2 } \Big) = \Delta \Big( \frac{ \abs{ \ho }^2 }{ \abs{ H }^2 } \Big) + \frac{ 2 }{ \abs{ H }^2 } \big\langle \nabla_i \abs{ H }^2, \nabla_i \Big( \frac{ \abs{ \ho }^2 }{ \abs{ H }^2 } \Big) \big\rangle - \frac{ 2 }{ \abs{ H }^2 } \abs{ H \cdot \nabla_i h_{kl} - \nabla_i H \cdot h_{kl} }^2
\end{equation*}
(see Lemma 5.2 of \cite{gH84} and set $\sigma = 0$).
The importance of including the gradient term $\abs{ \nabla h }^2$ in $f$ is the following: the antisymmetric part of $\abs{ \nabla h }^2$ contains curvature terms which we can use to obtain exponential convergence.  We split $\nabla^2 h$ into symmetric and anti-symmetric components, and upon discarding the the symmetric part we obtain
\begin{align*}
	\abs{ \nabla^2 h }^2 &\geq \frac{ 1 }{ 4 } \abs{ \nabla_i \nabla_j h_{kl} - \nabla_k \nabla_l h_{ij} }^2 \\
	&= \frac{ 1 }{ 4 } \abs{ R_{ikjp}h_{pl} + R_{iklp}h_{jp} }^2,
\end{align*}
where the last line follows from Simons' identity.  Some computation shows
\begin{equation*}
	\abs{ R_{ikjp}h_{pl} + R_{iklp}h_{jp} }^2 = 4 \sum_{i, j} ( \kappa_i^2 \kappa_j^4 - \kappa_i^3 \kappa_j^3 ),
\end{equation*}
then using that $\kappa_{ \text{min} } > 0$ we estimate
\begin{align}
	\notag \sum_{i, j} ( \kappa_i^2 \kappa_j^4 - \kappa_i^3 \kappa_j^3 ) &\geq \kappa_{ \text{min} }^2 \sum_{ i < j } \kappa_i \kappa_j(\kappa_i - \kappa_j)^2 \\
	&\geq n\kappa_{ \text{min} }^4 \abs{ \ho }^2 := \epsilon_1 \abs{ \ho }^2. \label{eqn: new max arg 1}
\end{align}
The next important step is to estimate the term $ \abs{ H \cdot \nabla_i h_{kl} - \nabla_i H \cdot h_{kl} }^2$ from below in terms of $\abs{ \nabla h }^2$.  It is a relatively simple matter to estimate this term from below in terms of $\abs{ \nabla H }^2$, however we want to use this good negative term to control the bad reaction term $c_1\abs{ H }^2 \abs{ \nabla h }^2$ of the evolution equation for $\abs{ \nabla h }^2$, so we need an estimate in terms of $\nabla h$.  To do this, as always let $h$ denote the second fundamental form and $B$ a totally symmetric three tensor (we have $\nabla h$ in mind).  Consider the space $\mathcal{A} := \{ h, B : \abs{ h }^2 =1, \abs{ B }^2 =1 \}$, and we also assume strict convexity of $h$.  The conditions on $h$ and $B$ imply this space is compact.  Now consider the function $G(B) =  \abs{ h_{pp} \cdot B_{ijk} - h_{ij} B_{kpp} }^2$.  We claim $G(B) \geq \delta$ for some $\delta > 0$.  Since $\mathcal{A}$ is compact, by the extreme value theorem $G$ assumes its minimum value at some element of $\mathcal{A}$.  We show by contradiction that $G \neq 0$ which proves the claim.  The anti-symmetric part of $G$ is $\abs{ B_{ipp} \cdot h_{jk} - B_{jpp} \cdot h_{ik} }^2$.  We compute at a point where $G$ obtains its minimum, and rotating coordinates so that $e_1 = \nabla H / \abs{ \nabla H }$ we have
\begin{equation*}
	\abs{ B_{ipp} \cdot h_{jk} - B_{jpp} \cdot h_{ik} }^2 = \abs{ B_{ipp} }^2 \Big( \abs{ h }^2 - \sum_{ k =1 }^n h_{1k}^2 \Big),
\end{equation*}
so $\abs{ B_{ipp} }^2 = 0$ or $ \abs{ h }^2 = \sum_{ k =1 }^n h_{1k}^2$.  The latter implies that $\abs{ h }^2 = h_{11}^2$, which contradicts the strict convexity of the hypersurface.  Therefore, if $G(B) = 0$, then $\abs{ B_{ipp} }^2 = 0$.  From the definition of $G$ it now follows that the full tensor $\abs{ B }^2 = 0$.  This contradicts $\abs{ B }^2 =1$ and the claim follows.  The term $h_{pp} \cdot B_{ijk} - h_{ij} B_{kpp}$ is a quadratic form, so for arbitrary $h$ and $B$ we obtain $G(B) \geq \delta \abs{ h }^2 \abs{ B }^2$ by scaling.  Applying this to our situation, we have $\abs{ B }^2 = \abs{ \nabla h }^2$, then estimating $\abs{ h }^2 \geq n \kappa_{\text{min}}^2$ we obtain
\begin{equation}
	\abs{ H \cdot \nabla_i h_{kl} - \nabla_i H \cdot h_{kl} }^2 \geq \delta n \kappa_{\text{min}}^2 \abs{ \nabla h }^2 := \epsilon_2 \abs{ \nabla h }^2 . \label{eqn: new max arg 2}
\end{equation}
Returning now to the evolution equation for $f$, converting to the normalised setting and using the estimates \eqref{eqn: new max arg 1} and \eqref{eqn: new max arg 2} we get
\begin{align*}
	\frac{ \p }{ \p \tilde{ t } } \tilde{ f }  &\leq \tilde{ \Delta } \tilde{ f } - \epsilon_1 \abs{ \tilde{ \ho } }^2 + c_1 \abs{ \tilde{ H } }^2 \abs{ \tilde{ \nabla } \tilde{ h } }^2 + \frac{ 2 }{ \abs{ \tilde{ H } }^2 }  \big\langle \tilde{ \nabla }_i \abs{ \tilde{ H } }^2, \tilde{ \nabla }_i \tilde{ f } \big\rangle - \frac{ 2 }{ \abs{ \tilde{ H } }^2 } \big\langle \tilde{ \nabla }_i \abs{ \tilde{ H } }^2, \tilde{ \nabla }_i ( \epsilon \abs{ \tilde{ \nabla } \tilde{ h } }^2 ) \big\rangle \\
	&\qquad - \frac{ 2 \epsilon_2 N }{ \abs{ \tilde{ H } }^2 } \abs{ \tilde{ \nabla } \tilde{ h } }^2 - \frac{ 4 \epsilon }{ n } \tilde{ \hbar } \abs{ \tilde{ \nabla } \tilde{ h } }^2.
\end{align*}
In the normalised setting the second fundamental form, and therefore all higher derivatives, are bounded above.  We can therefore estimate
\begin{equation*}
	\big\langle \tilde{ \nabla }_i \abs{ \tilde{ H } }^2, \tilde{ \nabla }_i ( \epsilon \abs{ \tilde{ \nabla } \tilde{ h } }^2 ) \big\rangle \leq 4 \abs{ \tilde{ H } } \abs{ \tilde{ \nabla } \tilde{ H } } \abs{ \tilde{ \nabla } \tilde{ h } } \abs{ \tilde{ \nabla }^2 \tilde{ h } } \leq C \abs{ \tilde{ \nabla } \tilde{ h } }^2.
\end{equation*}
Using $0 < C_{ \text{min} } \leq \abs{ \tilde{ H } }_{ \text{min} } \leq \abs{ \tilde{ H } }_{ \text{max} } \leq  C_{ \text{max} }$, we make $N$ sufficiently large to consume the bad $\abs{ \tilde{ \nabla } \tilde{ h } }^2$ terms and then we discard these terms.  Using again $C_{ \text{min} } \leq \abs{ \tilde{ H } }_{ \text{min} }$ we estimate
\begin{equation*}
		-\epsilon_1 \abs{ \tilde{ \ho } }^2 - \frac{ 4 \epsilon }{ n } \tilde{ \hbar } \abs{ \tilde{ \nabla } \tilde{ h } }^2 \leq - \delta \tilde{ f }
\end{equation*}
for some small $\delta$.  We ultimately obtain  
\begin{equation*}
	\frac{ \p }{ \p \tilde{ t } } \tilde{ f }  \leq \tilde{ \Delta } \tilde{ f } + \tilde{U}^k \tilde{ \nabla }_k \tilde{ f } - \delta \tilde{ f }.
\end{equation*}
This implies
\begin{equation*}
	\frac{ \p }{ \p \tilde{ t } } ( e^{ \delta \tilde{ t } } \tilde{ f } ) \leq \tilde{ \Delta } ( e^{ \delta \tilde{ t } } \tilde{ f } ) + U^k \tilde{ \nabla }_k ( e^{ \delta \tilde{ t } } \tilde{ f } ),
\end{equation*}
and from the maximum principle we conclude $e^{ \delta \tilde{ t } } \tilde{ f } \leq C$ and the theorem follows since $\abs{ \tilde{ H } }_{ \text{max} } \leq  C_{ \text{max} }$.
\end{proof}
Note that we obtain exponential decay of both $\abs{ \tilde{ \nabla } \tilde{ h } }^2$ and $\abs{ \tilde{ \ho } }^2$ at the same time.  Since we have pointwise control on the decay of $\abs{ \tilde{ \ho } }^2$, exponential decay of the higher derivatives can be proved by the maximum principle in a similar manner as the un-normalised estimates.  The important modification needed is that one adds in $\abs{ \tilde{ \ho } }^2$, which is exponentially decaying, rather than $\abs{ h }^2$, which is not, to generate the favourable gradient terms.  In our high codimension setting it is a simpler matter to estimate the good gradient term in the corresponding evolution equation for $\abs{ \ho }^2 / \abs{ H }^2$ in terms of $\abs{ \nabla h }^2$, and the same proof goes through provided we can estimate a lower bound for
\begin{align*}
	\abs{ \nabla^2 h }^2 &\geq \frac{ 1 }{ 4 } \abs{ \nabla_i \nabla_j h_{kl} - \nabla_k \nabla_l h_{ij} }^2 \\
	&= \frac{ 1 }{ 4 } \abs{ \Rp_{ik\alpha\beta}h_{jl\alpha}\nu_{\beta} + R_{ikjp}h_{pl} + R_{iklp}h_{jp} }^2
\end{align*}
in terms of $\abs{ \ho }^2$.  Such an estimate could hold for $c < 1/(n-1)$, although at this stage we can no longer muster the patience to attempt the index gymnastics involved.  In the absence of this calculation, we give a sketch of the original proof contained in \cite{gH84}, with the necessary adjustments made for the high codimension.  We remark that it would be nice to use the same idea in the un-normalised setting, and avoid the integral estimates.  Unfortunately, at the moment we can only make such an argument work if the submanifold is already extremely pinched.

\begin{prop}\label{p: normalised integral Pinching Improves}
There exist positive constants $C$ and $\delta$ both depending only on $\Sigma_0$ such that the estimate
	\begin{equation*}
		\int_{ \Sigma } \abs{ \tilde{ \ho } }^2 d\tilde{\mu}_{ \tilde{g}(\tilde{t}) } \leq Ce^{-\delta\tilde{t} }
	\end{equation*}
holds for all time $\tilde{t} \in [\tilde{t}_0, \infty)$, where $\tilde{t}_0$ is some sufficiently long time.
\end{prop}
\begin{proof}
Consider the function
\begin{equation*}
	\tilde{f} := \frac{ \abs{\tilde{h}}^2 }{ \abs{\tilde{H}}^2 } - \frac{1}{n},
\end{equation*}
which is scale-invariant.  The evolution equation for $\tilde{f}$ is easily obtained from equation \eqref{e: evol eqn f_sigma 1} by taking $\sigma = 0$: 
\begin{equation*}
	\frac{\p}{\p t}\tilde{f} = \tilde{\Delta} \tilde{f} - \frac{4}{ \tilde{\abs{H}} } \langle \tilde{\nabla}_i \tilde{H}, \tilde{\nabla}_i \tilde{f} \rangle - \frac{ 2\epsilon_{\nabla} }{ \tilde{\abs{H}}^2 } \abs{ \tilde{\nabla} \tilde{H} }^2.
\end{equation*}
In the same manner as Proposition \ref{p: using Z pos} we obtain the differential inequality
\begin{equation*}
	\frac{d}{dt} \int \tilde{f}^p \, d\mu_{\tilde{g}} \leq -\delta \int \tilde{f}^p\abs{ \tilde{H} }^2 \, d\mu_{\tilde{g}} + \int \tilde{f}^p ( \hbar - \abs{\tilde{H}}^2 ) \, d\mu_{\tilde{g}},
\end{equation*}
where $\delta$ is some small positive constant and the second integral on the right arises from differentiating the normalised measure.  Using estimate \eqref{eqn: normalised grad est} and $\abs{ \tilde{H} }^2_{ \text{min} } > C_{\text{min}}$ we see there exists some time $\tilde{t}_0$ such that for all $\tilde{t} \geq \tilde{t}_0$ we have
\begin{equation}\label{eqn: improved pinch normalised}
	\frac{d}{dt} \int \tilde{f}^p \, d\tilde{\mu}_{\tilde{g}} \leq -\delta \int \tilde{f}^p\abs{ \tilde{H} }^2 \, d\tilde{\mu}_{\tilde{g}}
\end{equation}
for some smaller $\delta$.  This implies
\begin{equation*}
	\int_{ \Sigma } \tilde{f}^p \, d\tilde{\mu}_{ \tilde{g}( \tilde{t} ) } \leq \int_{ \Sigma } \tilde{f}^p \, d\tilde{\mu}_{ \tilde{g}(0) } \cdot e^{ -\delta C^2_{\text{min}} \tilde{t} },
\end{equation*}
from which the proposition follows easily.
\end{proof}

\begin{prop}\label{eqn: H min mix normalised}
	We have the estimate
		\begin{equation*}
			\tilde{H}_{ \text{max} } - \tilde{H}_{ \text{min} } \leq Ce^{-\delta \tilde{t} }
		\end{equation*}
for all time $\tilde{t} \geq \tilde{t}_1$, where $\tilde{t}_1$ is some sufficiently long time.
\end{prop}
\begin{proof}
Consider $\tilde{f} := \abs{ \tilde{\nabla} \tilde{H} }^2 + N\abs{ \tilde{H} }^2\abs{ \tilde{\ho} }^2$.  This function is of degree $-4$ and from the relevant un-normalised evolution equations and Lemma \ref{lem: degree} we derive

\begin{equation*}
	\frac{\p}{\p t}\tilde{f} \leq \tilde{\Delta}\tilde{f} + c_1\abs{\tilde{H}}^2\abs{ \tilde{\nabla}\tilde{h} }^2 - 2N\langle \tilde{\nabla}_i \abs{ \tilde{H} }^2, \tilde{\nabla}_i \abs{ \tilde{\ho} }^2 \rangle - \frac{4N(n-1)}{3n}\abs{ \tilde{H} }^2\abs{ \tilde{\nabla}\tilde{h} }^2 + Ce^{-\delta\tilde{t}} - \frac{4}{n}\tilde{\hbar}\tilde{f}.
\end{equation*}
The second term on the right can be absorbed by choosing $N$ sufficiently large.  The term $2N\langle \tilde{\nabla}_i \abs{ \tilde{H} }^2, \tilde{\nabla}_i \abs{ \tilde{\ho} }^2 \rangle$ can be estimated by $C(N)\abs{ \tilde{\ho} }\abs{ \tilde{\nabla}\tilde{h} }^2$, which after some time $t_1$ will be absorbed by the negative term $\abs{ \tilde{H} }^2\abs{ \tilde{\nabla}\tilde{h} }^2$.  We use $\abs{\tilde{ H } }_{\text{min}} > C_{\text{min} }$ to estimate the last term on the right, obtaining the differential inequality
\begin{equation*}
	\frac{\p}{\p t}\tilde{f} = \tilde{\Delta}\tilde{f}+ Ce^{-\delta\tilde{t}} - \delta\tilde{f}
\end{equation*}
which holds for all time $\tilde{t} \geq \tilde{t}_1$.
We then have
\begin{equation*}
	\frac{\p}{\p \tilde{t} } (e^{\delta\tilde{t}} \tilde{f} - C\tilde{t} ) \leq \tilde{\Delta}(e^{\delta\tilde{t}} \tilde{f} - C\tilde{t} ),
\end{equation*}
and from the maximum principle conclude
\begin{align*}
	\tilde{f} &\leq C(1 + \tilde{t})e^{ -\delta\tilde{t} } \\
	&\leq Ce^{ -\delta\tilde{t} }
\end{align*}
for some $\delta$ smaller again.  The proposition now follows by integrating this estimate along geodesics and using that the diameter is bounded above.
\end{proof}

\begin{prop}\label{l: integral normalised higher deriv}
	For every $m \geq 1$ and $p \geq 2$, the estimate
		\begin{equation*}
			\int_{ \Sigma } \abs{ \tilde{\nabla}^m \tilde{h} }^p \, d\tilde{\mu}_{ \tilde{g}(\tilde{t}) } \leq C_m e^{ -\delta_m \tilde{t} }
		\end{equation*}
holds for all time $\tilde{t} \in [\tilde{t}_3, \infty)$.
\end{prop}
For a proof we refer the reader to Lemma 10.4 in \cite{gH84}.  The usual Sobolev inequality on a compact manifold now implies that $\max_{ \tilde{ \Sigma }_{ \tilde{t} } } \abs{ \tilde{\nabla}^m \tilde{h} } \leq C_m$.  With these higher derivative estimates in place we can prove the crucial pointwise bound on $\abs{ \tilde{ \ho } }$:

\begin{lem}\label{pinching improves normalised}
There exist positive constants $C$ and $\delta$ both depending only on $\Sigma_0$ such that the estimate
	\begin{equation*}
		\abs{ \tilde{ \ho } }^2 \leq Ce^{-\delta\tilde{t} }
	\end{equation*}
holds for all time $\tilde{t} \in [\tilde{t}_4, \infty)$.
\end{lem}
For a proof we refer the reader to Theorem 10.5 in \cite{gH84}.  See also \cite{rH82} for the above two results.  In particular, the reason why the Sobolev constant is uniformly bounded, and thus why can in fact use the Sobolev inequality is explained in \cite{rH82}.

\begin{prop}
The normalised submanifold $\tilde{\Sigma}_{ \tilde{t} }$ converges uniformly to a smooth limit submanifold $\tilde{\Sigma}_{\infty}$ as $\tilde{t} \rightarrow \infty$.
\end{prop}
\begin{proof}
The first step is to show $\tilde{ \Sigma }_{ \infty }$ is continuous.  As we have done in the un-normalised setting, using Lemma 14.2 of \cite{rH82} it suffices to show
\begin{equation*}
	 \int_0^{ \infty } \Big\lvert \frac{ \p \tilde{g} }{ \p \tilde{t} } \Big\rvert \, d\tilde{t} \leq C < \infty.
\end{equation*}
We estimate
\begin{align*}
	\int_0^{ \infty } \Big\lvert \frac{ \p \tilde{g} }{ \p \tilde{t} } \Big\rvert_{\tilde{g}(\tilde{t})} \, d\tilde{t} &= 2\int_0^{ \infty } \Big\lvert \tilde{H} \cdot \tilde{h}_{ij} - \frac{1}{n} \tilde{\hbar} \tilde{g}_{ij} \Big\rvert_{\tilde{g}(\tilde{t})} \, d\tilde{t} \\
	&\leq C\int_0^{ \infty } \Big\lvert \tilde{h}_{ij} - \frac{1}{n} \tilde{H} \tilde{g}_{ij} \Big\rvert_{\tilde{g}(\tilde{t})} + \frac{1}{n} \Big\lvert ( \abs{ \tilde{H} }^2 - \tilde{ \hbar } ) \tilde{g}_{ij} \Big\rvert_{\tilde{g}(\tilde{t})} \, d\tilde{t} \\
	&\leq C\int_0^{ \infty } \Big\lvert \tilde{h}_{ij} - \frac{1}{n} \tilde{H} \tilde{g}_{ij} \Big\rvert_{\tilde{g}(\tilde{t})} + \frac{1}{n} \Big\lvert ( \abs{ \tilde{H}_{ \text{max} } }^2 - \abs{ \tilde{H}_{ \text{min} } }^2 ) \tilde{g}_{ij} \Big\rvert_{\tilde{g}(\tilde{t})} \, d\tilde{t} \\
	&\leq \int_0^{ \infty } Ce^{ -\delta \tilde{t} },
\end{align*}
which is finite as desired.  In going to the last line we have used Proposition \ref{eqn: H min mix normalised} and Lemma \ref{pinching improves normalised}.  The proof that $\tilde{ \Sigma }_{\infty}$ is smooth mimics that of the un-normalised setting, where here the exponential decay of the normalised estimates guarantees that the indefinite integrals in time which arise are finite.
\end{proof}
\begin{prop}
The limit submanifold $\tilde{\Sigma}_{\infty}$ is a $n$-sphere lying in some $(n+1)$-dimensional subspace of $\mathbb{R}^{n+k}$.
\end{prop}
\begin{proof}
Lemma \ref{pinching improves normalised} implies that $\tilde{\Sigma}_{\infty}$ is totally umbilic.  By the Codazzi Theorem (see \cite[Thm. 26]{mS79} for a proof), the only closed, totally umbilic $n$-dimensional submanifold immersed in $\mathbb{R}^{n+k}$ is a $n$-sphere lying in some $(n+1)$-dimensional subspace of $\mathbb{R}^{n+k}$.
\end{proof}
The last proposition completes the proof of the second part of the Main Theorem \ref{mthm: main thm 2}.

\chapter{Submanifolds of the sphere}

In the previous chapter we studied the evolution of submanifolds of Euclidean space by the mean curvature flow.  We now want to consider the situation where the background space is a sphere of constant curvature $\bar{K}$.  Our main result is the following:

\begin{mthm}\label{main theorem sphere}
Let $\Sigma_0^n = F_0(\Sigma^n)$ be a closed submanifold smoothly immersed in $\mathbb{S}^{n+k}$.  If $\Sigma_0$ satisfies
\begin{equation*}
	\begin{cases}
		\abs{h}^2 \leq \frac{4}{3n}\abs{H}^2 + \frac{2 (n-1) }{3} \bar K, \quad n = 2,3 \\
		\abs{h}^2 \leq \frac{1}{n-1}\abs{H}^2 + 2\bar K, \quad n \geq 4,
	\end{cases}
\end{equation*}
then either
\begin{enumerate}
	\item MCF has a unique, smooth solution on a finite, maximal time interval $0 \leq t < T < \infty$ and the submanifold $\Sigma_t$ contracts to a point as $t \rightarrow T$; or
	\item MCF has a unique, smooth solution for all time $0 \leq t < \infty$ and the submanifold $\Sigma_t$ converges to a totally geodesic submanifold $\Sigma_{\infty}$.
\end{enumerate}
\end{mthm}

We highlight again that no assumption on the size $H$ is required.  The pinching condition $\abs{h}^2 < 1/(n-1)\abs{H}^2 + 2\bar K$ implies that the submanifold has positive intrinsic curvature.  A natural question to ask is whether some other geometric flow will deform all submanifolds of positive intrinsic curvature to either round points or totally geodesic submanifolds.  In the case of hypersurfaces this problem has a very nice resolution due to Andrews.  Beginning with the assumption that positive intrinsic curvature is preserved by some flow, in \cite{A02} the desired speed of the flow is found as an explicit solution of an ordinary differential equation.  He then goes on to show that this flow does indeed deform an initial hypersurface of positive intrinsic curvature to either a point or a totally geodesic hypersurface.  We point out that in the high codimension case such a theorem cannot be true (in dimension two) because of the Veronese surface.

This proof of this theorem proceeds similarly to \cite{gH87} using the high codimension techniques developed in the previous chapter.  After the relevant evolution equations are derived, we prove a version of the Pinching Lemma that holds in a sphere. The Pinching Lemma states that if the initial submanifold satisfies a certain curvature pinching, then the mean curvature flow preserves this pinching.  A stronger pinching estimate is then deduced by a Stampacchia iteration argument.  The essential content of this estimate is that in regions of large mean curvature, or after sufficiently long time, the submanifold is nearly totally umbilic.  This estimate allows us to characterise the long time shape of the evolving submanifolds, which is completed in the last sections.

\section{The evolution equations in a sphere}\label{s: The evolution equations in a sphere}

In the previous chapter we derived the evolution equation for the second fundamental form of submanifolds of arbitrary codimension in an arbitrary background space:
\begin{align*}
\nabla_{\partial_t}h_{ij} &= \Delta h_{ij}+h_{ij}\cdot h_{pq}h_{pq}+h_{iq}\cdot h_{qp}h_{pj}
+h_{jq}\cdot h_{qp}h_{pi}-2h_{ip}\cdot h_{jq} h_{pq}\notag\\
&\quad +2\bar R_{ipjq}h_{pq}-\bar R_{kjkp}h_{pi}-\bar R_{kikp}h_{pj}+h_{ij\alpha}\bar R_{k\alpha k\beta}\nu_\beta\notag\\
&\quad -2h_{jp\alpha}\bar R_{ip\alpha\beta}\nu_\beta-2h_{ip\alpha}\bar R_{jp\alpha\beta}\nu_\beta
+\bar\nabla_k\bar R_{kij\beta}\nu_\beta-\bar\nabla_i\bar R_{jkk\beta}\nu_\beta.\label{eq:reactdiffuse}
\end{align*}

In the case where the background space is a sphere the above evolution equation can be simplified significantly.   If $e_a$, $1 \leq a \leq n+k$, is an arbitrary local frame for the background sphere, then in such a frame the Riemann curvature tensor takes the form
\begin{equation}\label{e: curv tensor of sphere}
	\bar R(e_a, e_b, e_c, e_d) = \bar{K} ( \langle e_a, e_c \rangle \langle e_b, e_d \rangle - \langle e_a, e_d \rangle \langle e_b, e_c \rangle ).
\end{equation}
The derivation of the evolution equation for $\abs{ h }^2$ follows that of the Euclidean case, however extra terms are now present due to the background curvature. We will show how to deal with these extra terms.  First of all, as a sphere is a symmetric space, the first derivatives of the the background curvature are zero.  The extra ambient curvature terms that remain are
\begin{equation*}
	4 h_{ij\alpha}h_{pq\alpha}\bar{R}_{ipjq} - 4 h_{ij\alpha}h_{ip\alpha}\bar{R}_{kjkp} + 2 h_{ij\alpha}h_{ij\beta}\bar{R}_{k\alpha k \beta} - 8 h_{ij\beta}h_{ip\alpha}\bar{R}_{jp\alpha\beta}
\end{equation*}
Now, using the form of the Riemann curvature tensor of the sphere given by equation \eqref{e: curv tensor of sphere}, for example, $\bar{R}_{k \alpha k \beta} = \bar{K}(\delta_{k k}\delta_{\alpha \beta} - \delta_{k\beta}\delta_{\alpha k})$, one finds various terms are zero or cancel, ultimately leaving only
\begin{equation*}
	4\bar{K} \abs{H}^2 - 2n\bar{K}\abs{h}^2.
\end{equation*}
The evolution equation for $\abs{ h }^2$ is therefore given by
\begin{equation}\label{e: evol h2 sphere}
	\frac{\partial}{\partial t}\abs{h}^2 = \Delta\abs{h}^2 - 2\abs{\nabla h}^2 + 2R_1 + 4\bar K\abs{H}^2 - 2n\bar K\abs{h}^2,
\end{equation}
or equivalently
\begin{equation}\label{e: evol h2 sphere 1}
	\frac{\partial}{\partial t}\abs{h}^2 = \Delta\abs{h}^2 - 2\abs{\nabla h}^2 + 2R_1 + 2\bar K\abs{H}^2 - 2n\bar K\abs{ \ho }^2,
\end{equation}
The ambient curvature terms appearing in the derivation of the evolution equation for $\abs{ H }^2$ can be dealt with similarly, and we obtain
\begin{equation}\label{e: evol H2 sphere 2}
	\frac{ \p }{ \p t } \abs{H}^2 = \Delta\abs{H}^2 - 2\abs{ \nabla H }^2 + 2 R_1 + 2n\bar{K}\abs{H}^2.
\end{equation}
The contracted form of Simons' identity takes the form
\begin{equation}\label{e: contracted Simons' identity}
	\frac{1}{2} \Delta \abs{ \ho }^2 = \ho_{ij} \cdot \nabla_i\nabla_j H + \abs{ \nabla \ho }^2 + Z + n\bar K\abs{ \ho  }^2,
\end{equation}
where again
\begin{equation*}
	Z = -\sum_{ \alpha, \beta } \Big( \sum_{i,j} h_{ij\alpha} h_{ij\beta} \Big)^2 - \abs{ \Rp }^2 + \sum_{ \substack{ i,j,p \\ \alpha, \beta } } H_{\alpha} h_{ip\alpha} h_{ij\beta} h_{pj\beta}.
\end{equation*}
And finally, the basic gradient estimate
\begin{equation} \label{eqn: basic grad est 1 sphere} 
	\abs{ \nabla h }^2 \geq \frac{ 3 }{ n+2 } \abs{ \nabla H }^2
\end{equation}
carries over unchanged.

\section{Curvature pinching is preserved}\label{s: Curvature pinching is preserved}
We now prove the version of the Pinching Lemma that holds in sphere.  Whenever we make reference to the Pinching Lemma in this chapter we obviously mean the following lemma.

\begin{lem}
If a solution $F : \Sigma \times [0,T) \rightarrow \mathbb{S}^{n+k}$ of the mean curvature flow satisfies
\begin{equation}
	\begin{cases}
		\abs{h}^2 \leq \frac{4}{3n}\abs{H}^2 + \frac{n}{2}\bar K, \quad n = 2,3 \\
		\abs{h}^2 \leq \frac{1}{n-1}\abs{H}^2 + 2\bar K, \quad n \geq 4
	\end{cases}
\end{equation}
at $t = 0$, then this remains true as long as the solution exists.
\end{lem}
\begin{proof}
The proof closely follows the Euclidean case.  Here we consider $\mathcal{Q} = \abs{ h }^2 - \alpha{ H }^2 - \beta \bar{ K }$, where $\alpha$ and $\beta$ are constants.  Because we are allowing the initial submanifold to have $H = 0$, in order to compute in a local frame for the normal bundle where $\nu_1 = H / \abs{ H }$ we need to consider two cases: 1) $H = 0$ and 2) $ H \neq 0$.  For the first case, from the evolution equations for $\abs{ h }^2$ and $\abs{ H }^2$ we derive
\begin{equation}\label{e: pinch lem sphere 1}
	\frac{ \p }{ \p t } \mathcal{Q} = \Delta \mathcal{Q} - 2 \abs{ \nabla \ho }^2 + 2 R_1 - 2n\bar{K} \abs{ \ho }^2.
\end{equation}
In this case there is no need to split up directions of the second fundamental form as we did in the Euclidean case, and using the estimate of \cite{LiLi92} on all the normal directions of $R1$ we get
\begin{equation*}
	R1 = \sum_{ \alpha,\beta }\Big( \sum_{i,j} \ho_{ij\alpha}\ho_{ij\beta}\Big)^2 + \sum_{ \alpha, \beta } N(\ho_{\alpha}\ho_{\beta} - \ho_{\beta}\ho_{\alpha}) \leq \frac{3}{2} \abs{ \ho }^4.
\end{equation*}
The reaction terms of \eqref{e: pinch lem sphere 1} may therefore be estimated by
\begin{equation*}
2 R_1 - 2n\bar{K} \abs{ \ho }^2 \leq 3 \abs{ \ho }^4 - 2n\bar{K} \abs{ \ho }^2.
\end{equation*}
If $\mathcal{Q}$ doesn't stay (strictly) negative, then $\abs{ \ho }^2 = \beta \bar{K}$ and
\begin{equation*}
	3 \abs{ \ho }^4 - 2n\bar{K} \abs{ \ho }^2 < -\beta (2n - 3\beta) \bar{K}^2
\end{equation*}
which is (strictly) negative as long as $\beta < (2/3)n$.  This is a contradiction and the lemma follows in this case.  Now consider the case $H \neq 0$.  We may now work in the special local frames of the previous chapter, and the evolution equation becomes
\begin{equation}\label{eqn: pinch lem sphere 2}
	\begin{split}
		\frac{\p}{\p t} \mathcal{Q} &= \Delta \mathcal{Q} - 2( \abs{ \nabla h }^2 - \alpha\abs{ \nabla H }^2 ) \\
			&\quad+ 2R_1 - 2\alpha R_2 - 2n \bar K\abs{ \ho }^2 - 2 n ( \alpha  - 1/n ) \bar K\abs{ H }^2.
	\end{split}
\end{equation}
Arguing as in Euclidean case, if $\mathcal{Q}$ doesn't remain (strictly) negative, we may replace $\abs{ H }^2$ with $( \abs{ \ho }^2 - \beta \bar{K} ) / ( \alpha - 1/n )$, and estimating as before we get
\begin{align*}
	&2R_1 - 2\alpha R_2 - 2n \bar K\abs{ \ho }^2 - 2 n ( \alpha  - 1/n ) \bar K\abs{ H }^2 \\
		&\quad \leq 2\abs{ \ho_1 }^2 - 2( \alpha - \frac{1}{n} )\abs{ \ho_1 }^2\abs{ H }^2 + \frac{2}{n}\abs{ \ho_1 }^2\abs{ H }^2 - \frac{2}{n}(\alpha - \frac{1}{n})\abs{ H }^4 + 8\abs{ \ho_1 }^2\abs{ \ho_- }^2 + 3\abs{ \ho_- }^4 \\
		&\quad - 2n\bar K( \abs{ \ho_1 }^2 + \abs{ \ho_- }^2 ) - 2n ( \alpha - 1/n) \bar{K} \abs{ H }^2 \\
		&\quad\leq \Big( 6 - \frac{2}{ n(\alpha - \frac{1}{n}) } \Big)\abs{ \ho_1 }^2\abs{ \ho_- }^2 + \Big( 3 - \frac{2}{ n(\alpha - \frac{1}{n}) } \Big)\abs{ \ho_- }^4  \\
			&\quad + \Big( 2\beta - 4n + \frac{2\beta}{ n(\alpha - \frac{1}{n}) } \Big)\abs{ \ho_1 }^2 \bar K + 4 \Big( \frac{ \beta }{ n(\alpha - \frac{1}{n}) } - n \Big)\abs{ \ho_- }^2 \bar K \\ 
			&\quad - 2 \beta \Big( \frac{ \beta }{ n(\alpha - \frac{1}{n}) } - n \Big) \bar K^2 \\
	&\quad= \Big( 6 - \frac{2}{ n(\alpha - \frac{1}{n}) } \Big)(\abs{ \ho_1 }^2\abs{ \ho_- }^2 + \abs{ \ho_- }^4 ) + \Big( 2\beta - 4n + \frac{2\beta}{ n(\alpha - \frac{1}{n}) } \Big)\abs{ \ho_1 }^2 \bar K \\
		&\quad- 3\abs{ \ho_- }^4 + 4 \Big( \frac{ \beta }{ n(\alpha - \frac{1}{n}) } - n \Big)\abs{ \ho_- }^2 \bar K  - 2 \beta \Big( \frac{ \beta }{ n(\alpha - \frac{1}{n}) } - n \Big) \bar K^2.
\end{align*}
We have rewritten the last line as such to highlight that after choosing the coefficient of the $\abs{ \ho_1 }^2\abs{ \ho_- }^2$ term as large as we can (namely $4/(3n)$), we still have the good term $-3\abs{ \ho_- }^2$ left over.  The last line above is a quadratic form, so by requiring that its discriminant be negative we will have a strictly negative term.  The discriminant is 
\begin{equation*}
	\Delta = 8 \Big( \frac{ \beta }{ n( \alpha - \frac1n ) } - n \Big) \Big\{ 2 \Big( \frac{ \beta }{ n( \alpha - \frac1n ) } \Big) - 3\beta \Big\},
\end{equation*}
which is negative for our values of $\alpha$ and $\beta$ in dimensions two to four.  For dimensions $n \geq 4$ the best value of $\alpha$ we can expect is $1/(n-1)$, and so with this restriction, the amount of the good terms $\abs{ \ho_- }^4$ is increases to $-2(n-4) - 3$.  The discriminant is now
\begin{equation*}
	\Delta = 8 \Big( \frac{ \beta }{ n( \alpha - \frac1n ) } - n \Big) \Big\{ 2 \Big( \frac{ \beta }{ n( \alpha - \frac1n ) } \Big) - ( 2(n-4) + 3 ) \beta \Big\},
\end{equation*}
and which is strictly negative for $\beta = 2$ for all $n \geq 4$.  The most restrictive condition on the size of $\beta$ comes from the coefficient of the $\abs{ \ho_1 }^2 \bar K$ term, which gives the values of $\beta$ in the statement of the lemma.  With the chosen values of $\alpha$ and $\beta$ the right hand side of equation \label{eqn: pinch lem sphere 2} is strictly negative, which is contradiction, and so $\mathcal{Q}$ must stay strictly negative.
\end{proof}

We now want to formulate a slightly different statement of the Pinching Lemma that will be useful in later setions.  For $\epsilon > 0$, set
\begin{equation*}
	\begin{cases}
		\alpha_{\epsilon} := \frac{ 4 }{ 3n + n\epsilon }, \quad n = 2,3 \\
		\alpha_{\epsilon} := \frac{ 1 }{ n-1+\epsilon }, \quad n \geq 4
	\end{cases} \text{and} \quad
	\begin{cases}
		\beta_{\epsilon} := \frac{n}{2}(1-\epsilon), \quad n = 2,3 \\
		\beta_{\epsilon} := 2(1-\epsilon), \quad n \geq 4.
	\end{cases}
\end{equation*}
If the strict inequality $\abs{ h }^2 < \alpha \abs{ H }^2 + \beta \bar{K}$ holds everywhere on the initial submanifold, then there exists an $\epsilon > 0$ such that $\abs{ h }^2 \leq \alpha_{ \epsilon } \abs{ H }^2 + \beta_{ \epsilon } \bar{K}$ on $\Sigma_0$.  On the other hand, if equality of the pinching condition holds somewhere on the initial submanifold, that is $\abs{ h }^2 \leq \alpha \abs{ H }^2 + \beta \bar{K}$, and the pinching does not immediately improve, then the same strong maximum principle argument as in Proposition \ref{prop:strongMP} of the previous chapter shows that $\Sigma_0 = \mathbb{S}^{p} \times \mathbb{S}^{n-p}$, where $0 \leq p \leq n$.  If $p=0$, then $\Sigma_0$ is a totally umbilic sphere, in which case there exists an $\epsilon > 0$ such that $\abs{ h }^2 \leq \alpha_{ \epsilon } \abs{ H }^2 + \beta_{ \epsilon } \bar{K}$ holds.  If $p \neq 0$, the above-mentioned product of spheres all lie outside of the pinching cone being considered.  Therefore, if the equality of the pinching condition holds initially, after some short time the submanifold satisfies $\abs{ h }^2 \leq \alpha_{ \epsilon } \abs{ H }^2 + \beta_{ \epsilon } \bar{K}$ for some $\epsilon > 0$.

\section{Pinching improves along the flow}\label{s: Pinching improves along the flow}
In this section we prove an important estimate that allows us to characterise the asymptotic behaviour of the solution.  As mentioned in the introduction to this chapter, the essential content of this theorem is that in regions where the mean curvature is large, or after long enough time, the submanifold is increasingly becoming totally umbilic.  This can be interpreted by saying that the pinching improves along the flow.

\begin{thm}\label{t: pinching improves}
There exist constants $C_0 < \infty$, $\sigma_0 > 0$, and $\delta_0 > 0$ all depending only on $\Sigma_0$ such that for all time $0 \leq t < T \leq \infty$, the estimate
	\begin{equation*}
		\abs{\ho}^2 \leq C_0(\abs{H}^2 + \bar K)^{1-\sigma_0}e^{-\delta_0 t}
	\end{equation*}
holds.
\end{thm}
For technical reasons it is more convenient to work initially with the auxiliary function  $f_{\sigma} := \abs{\ho}^2/(a\abs{H}^2 + \beta_{\epsilon} \bar K)^{1-\sigma}$, where $a := 1/(n(n-1+\epsilon))$.
\begin{proof}
We begin by deriving the evolution equation for $f_{\sigma}$.

\begin{prop}
For any $\sigma \in [0, \epsilon/2]$ we have the evolution equation
	\begin{equation}\label{eqn: evol eqn f_sigma 1 sphere}
		\begin{split}
				\p_t f_{\sigma} &\leq \Delta f_{\sigma} + \frac{ 4a(1-\sigma)\abs{H} }{ a\abs{H} + \beta_{\epsilon} \bar K } \big\langle \nabla_i\abs{H}, \nabla_i f_{\sigma} \big\rangle - \frac{ 2\epsilon_{\nabla} }{ (a\abs{H}^2 + \beta_{\epsilon} \bar K)^{1-\sigma} }  \abs{ \nabla H }^2 - 2n\epsilon' \bar K f_{\sigma} \\
				&\quad + 2\sigma\abs{h}^2f_{\sigma}. \end{split}
	\end{equation}
\end{prop}

\begin{proof}
From the evolutions equations for $\abs{ h }^2$ and $\abs{ H }^2$ we get
\begin{equation}\label{eqn: evol eqn f_sigma 2 sphere}
	\begin{split}
		\p_t f_{\sigma} &= \frac{\Delta\abs{h}^2 - 2\abs{\nabla h}^2 + 2R_1 + 4 \bar K\abs{H}^2 - 2n \bar K\abs{h}^2}{(a\abs{H}^2 + \beta_{\epsilon} \bar K)^{1-\sigma}} \\
		&\quad- \frac{1}{n}\frac{ ( \Delta\abs{H}^2 - 2\abs{\nabla H}^2 + 2R_2 + 2n \bar K\abs{H}^2 ) }{ ( a\abs{H}^2 + \beta_{\epsilon} \bar K )^{1-\sigma} } \\
		&\quad- \frac{ a(1-\sigma)(\abs{h}^2 - 1/n\abs{H}^2) }{ (a\abs{H}^2 + \beta_{\epsilon} \bar K)^{2-\sigma} }( \Delta\abs{H}^2 - 2\abs{\nabla H}^2 + 2R_2 + 2n \bar K\abs{H}^2 ).
	\end{split}
\end{equation}
The Laplacian of $f_{\sigma}$ is given by
	\begin{equation*}
		\begin{split}
			\Delta f_{\sigma} &= \frac{ \Delta ( \abs{h}^2 - 1/n\abs{H}^2 ) }{ ( a\abs{H}^2 + \beta_{\epsilon} \bar K )^{1-\sigma} } - \frac{ 2a(1-\sigma) }{ ( a\abs{H}^2 + \beta_{\epsilon} \bar K )^{2-\sigma} } \big\langle \nabla_i (\abs{h}^2 - 1/n\abs{H}^2), \nabla_i\abs{H}^2 \big\rangle  \\
			&\quad- \frac{ a(1-\sigma)(\abs{h}^2 - 1/n\abs{H}^2) }{ ( a\abs{H}^2 + \beta_{\epsilon} \bar K )^{2-\sigma} } \Delta\abs{H}^2 + \frac{ a^2(2-\sigma)(1-\sigma)(\abs{h}^2 - 1/n\abs{H}^2) }{ ( a\abs{H}^2 + \beta_{\epsilon} \bar K )^{3-\sigma} } \abs{\nabla\abs{H}^2}^2.
		\end{split}
	\end{equation*}
Using this expression for the Laplacian as well as the identity
	\begin{align*}
	&-\frac{ 2a(1-\sigma) }{ ( a\abs{H}^2 + \beta_{\epsilon} \bar K )^{2-\sigma} } \big\langle \nabla_i (\abs{h}^2 - 1/n\abs{H}^2), \nabla_i\abs{H}^2 \big\rangle \\
	&\quad = -\frac{ 2a(1-\sigma) }{ a\abs{H}^2 + \beta_{\epsilon} \bar K } \big\langle \nabla_i\abs{H}^2, \nabla_i f_{\sigma} \big\rangle - \frac{ 8a^2(1-\sigma)^2 }{ (a\abs{H}^2 + \beta_{\epsilon} \bar K)^2 } f_{\sigma}\abs{H}^2\abs{ \nabla\abs{H} }^2,
	\end{align*}
equation \eqref{eqn: evol eqn f_sigma 2 sphere} can be manipulated into the form
\begin{equation}
	\begin{split}
		\frac{ \p }{ \p t } f_{\sigma} &= \Delta f_{\sigma} + \frac{ 2a(1-\sigma) }{ a\abs{H}^2 + \beta_{\epsilon} \bar K } \big\langle \nabla_i\abs{H}^2, \nabla_i f_{\sigma} \big\rangle \\
		&\quad  - \frac{ 2 }{ (a\abs{H}^2 + \beta_{\epsilon} \bar K)^{1-\sigma} } \bigg\{ \abs{\nabla h}^2 - \frac{1}{n}\abs{ \nabla H }^2 - \frac{ a\abs{ \ho }^2 }{ a\abs{H}^2 +\beta_{\epsilon} \bar K }\abs{ \nabla H }^2 \bigg\} \\ 		
		&\quad- \frac{ 4a^2\sigma(1-\sigma) }{ (a\abs{H}^2 + \beta_{\epsilon} \bar K)^2 } f_{\sigma}\abs{H}^2\abs{ \nabla\abs{H} }^2 - \frac{ 2a\sigma f_{\sigma} }{ a\abs{H}^2 + \beta_{\epsilon} \bar K }\abs{\nabla H}^2 \\
		&\quad+\frac{2}{ (a\abs{H}^2 + \beta_{\epsilon} \bar K)^{1-\sigma} }\bigg\{ R_1 - \frac{1}{n}R_2 - n\bar K\abs{ \ho }^2 - \frac{ aR_2\abs{ \ho }^2 }{ a\abs{H}^2 + \beta_{\epsilon} \bar K } - \frac{ an(1-\sigma)\bar K \abs{ \ho }^2\abs{ H }^2 }{ a\abs{H}^2 + \beta_{\epsilon} \bar K } \bigg\} \\
			&\quad + \frac{ 2a\sigma R_2f_{\sigma} }{ a\abs{H}^2 + \beta_{\epsilon} \bar K}.
	\end{split}\label{e: f_sigma 3}
\end{equation}
The gradient terms on the third line are non-positive under our pinching assumption and we discard them.  Using equation \eqref{eqn: basic grad est 1 sphere} and the Pinching Lemma we estimate the useful gradient terms on the second line as follows:
\begin{align*}
	&\frac{ -2 }{ a\abs{H}^2 + \beta_{\epsilon} \bar K } \bigg\{ \abs{ \nabla h }^2 - \frac{1}{n}\abs{ \nabla H }^2 - \frac{ a\abs{ \ho }^2 }{ a\abs{H}^2 + \beta_{\epsilon} \bar K } \abs{ \nabla H }^2 \bigg\} \\
	&\quad \leq \frac{ -2 }{ a\abs{H}^2 + \beta_{\epsilon} \bar K } \bigg\{ \frac{ 3 }{ n+2 } - \frac{1}{n} - \frac{ a\big( (\alpha_{\epsilon} - 1/n)\abs{H}^2 + \beta_{\epsilon} \bar K \big) }{ a\abs{H}^2 + \beta_{\epsilon} \bar K } \bigg\}\abs{ \nabla H }^2 \\
	&\quad \leq \frac{ -2 }{ a\abs{H}^2 + \beta_{\epsilon} \bar K } \bigg\{ \frac{ 3 }{ n+2 } - \frac{1}{n} - a \bigg\}\abs{ \nabla H }^2 \\
	&\quad := \frac{ -2\epsilon_{ \nabla }\abs{ \nabla H }^2 }{ a\abs{H}^2 + \beta_{\epsilon} \bar K },
\end{align*}
where importantly, $\epsilon_{ \nabla }$ is positive for all $n \geq 2$.  Next we estimate the reactions terms on the second last line of \eqref{e: f_sigma 3}.  Expanding these reaction terms in the special local frames and estimating $\abs{ H }^2 \geq ( \abs{ \ho }^2 - \beta_{ \epsilon } \bar{ K } ) / ( \alpha_{ \epsilon } - 1/n )$ we obtain
\begin{align*}
&\frac{2}{ (a\abs{H}^2 + \beta_{\epsilon} \bar K)^{1-\sigma} } \bigg\{ R_1 - \frac{1}{n}R_2 - n\bar K\abs{ \ho }^2 - \frac{ aR_2\abs{ \ho }^2 }{ a\abs{H}^2 + \beta_{\epsilon} \bar K } - \frac{ an(1-\sigma)\bar K \abs{ \ho }^2\abs{ H }^2 }{ a\abs{H}^2 + \beta_{\epsilon} \bar K } \bigg\} \\
	&\quad \leq \frac{2}{ (a\abs{H}^2 + \beta_{\epsilon} \bar K)^{2 -\sigma} } \bigg\{ \frac{ a }{ \alpha_{ \epsilon } - \frac{ 1 }{ n } } \Big( 3 - \frac{ 1 }{ n(\alpha_{ \epsilon } - \frac{ 1 }{ n }) } \Big) \abs{ \ho_1 }^4 \abs{ \ho_- }^2 \\
	&\quad + \frac{ a }{ \alpha_{ \epsilon } - \frac{ 1 }{ n } } \Big( 3 + \frac{ 3 }{ 2 } - \frac{ 2 }{ n(\alpha_{ \epsilon } - \frac{ 1 }{ n }) } \Big) \abs{ \ho_1 }^2 \abs{ \ho_- }^4 + \frac{ a }{ \alpha_{ \epsilon } - \frac{ 1 }{ n } } \Big( \frac{ 3 }{ 2 } - \frac{ 1 }{ n(\alpha_{ \epsilon } - \frac{ 1 }{ n }) } \Big) \abs{ \ho_1 }^4 \abs{ \ho_- }^6 \\
	&\quad + \bigg( 1 + \frac{ 1 }{ n ( \alpha_{ \epsilon } - \frac{ 1 }{ n } ) } - \frac{ an [ 2 - ( \sigma + \epsilon_2 ) ] }{ \beta_{ \epsilon } ( \alpha_{ \epsilon } - \frac{ 1 }{ n } ) } \bigg) \beta_{ \epsilon } \bar{K} \abs{ \ho_1 }^4 \\
	&\quad + \bigg( \frac{ -3a }{ \alpha_{ \epsilon } - \frac{ 1 }{ n } } + \frac{ 2a }{ n ( \alpha_{ \epsilon } - \frac{ 1 }{ n } )^2 } + \frac{ 1 }{ n ( \alpha_{ \epsilon } - \frac{ 1 }{ n } ) } + 4 - \frac{ 2an [ 2 - ( \sigma + \epsilon_2 ) ] }{ \beta_{ \epsilon } ( \alpha_{ \epsilon } - \frac{ 1 }{ n } ) } \bigg) \beta_{ \epsilon } \bar{K}\abs{ \ho_1 }^2 \abs{ \ho_- }^2 \\
	&\quad + \bigg( -\frac{ 3 }{ 2 } \frac{ a }{ \alpha_{ \epsilon } - \frac{ 1 }{ n } } + \frac{ 2a }{ n ( \alpha_{ \epsilon } - \frac{ 1 }{ n } )^2 } + \frac{ 3 }{ 2 } - \frac{ an [ 2 - ( \sigma + \epsilon_2 ) ] }{ \beta_{ \epsilon } ( \alpha_{ \epsilon } - \frac{ 1 }{ n } ) } \bigg) \beta_{ \epsilon } \bar{K} \abs{ \ho_- }^4 \\
	&\quad - \beta_{ \epsilon } \bigg( \frac{ \beta_{ \epsilon } }{ n ( \alpha_{ \epsilon } - \frac1n ) } + n ( 1 - \epsilon_1 ) - \frac{ an [ 2 - ( \sigma + \epsilon_2 ) ] }{ \alpha_{ \epsilon } - \frac1n }  \bigg) \abs{ \ho_1 }^2 \bar{K}^2 \\
	&\quad - \beta_{ \epsilon } \bigg( \frac{ a\beta_{ \epsilon } }{ n ( \alpha_{ \epsilon } - \frac1n )^2 } + n ( 1 - \epsilon_1 ) - \frac{ an [ 2 - ( \sigma + \epsilon_2 ) ] }{ \alpha_{ \epsilon } - \frac1n }  \bigg) \abs{ \ho_- }^2 \bar{K}^2 \\
	&\quad - n\beta_{ \epsilon } \epsilon_1 \bar{K}^2 \abs{ \ho }^2 - an\epsilon_2 \bar{K} \abs{ \ho }^2 \abs{ H }^2 \bigg\}.
	\end{align*}
Provided $\epsilon_1$, $\epsilon_2$ and $\sigma$ are all chosen sufficiently small all terms in the above expression can be made negative, and we discard them with the exception of the two terms on the last line.  The above terms contain two quadratic forms, which are estimated in a similar manner to the Pinching Lemma.  We have 
\begin{align*}
	&\frac{2}{ (a\abs{H}^2 + \beta_{\epsilon} \bar K)^{1-\sigma} } \bigg\{ R_1 - \frac{1}{n}R_2 - n\bar K\abs{ \ho }^2 - \frac{ aR_2\abs{ \ho }^2 }{ a\abs{H}^2 + \beta_{\epsilon} \bar K } - \frac{ an(1-\sigma)\bar K \abs{ \ho }^2\abs{ H }^2 }{ a\abs{H}^2 + \beta_{\epsilon} \bar K } \bigg\} \\
	&\quad \leq \frac{2}{ (a\abs{H}^2 + \beta_{\epsilon} \bar K)^{2 -\sigma} } \cdot -n \abs{ \ho }^2 \bar{ K } ( a \epsilon_2 \abs{ H }^2 + \epsilon_1 \beta_{ \epsilon } \bar{ K } ) \\
	&\quad \leq -2n \min\{ \epsilon_1, \epsilon_2 \} \bar{ K } f_{ \sigma } \\
	&\quad := -2n\epsilon' \bar{ K } f_{ \sigma }.
\end{align*}
Finally, we estimate the last term on the right of equation \eqref{eqn: evol eqn f_sigma 1 sphere} by $R_2 \leq \abs{h}^2\abs{H}^2$, and the proposition is complete.

\end{proof}
As in the prevous chapter, the small positive $2\sigma\abs{h}^2f_{\sigma}$ prevents us from using the maximum principle and we proceed by deriving integral estimates and an iteration procedure.  The thrust of this iteration procedure is to exploit the good negative $\abs{ \nabla H }^2$ term in \eqref{eqn: evol eqn f_sigma 1 sphere} using the contracted Simons' identity and the Divergence theorem.  In order to do this we need a lower bound on the Laplacian of $f_{\sigma}$, and as Huisken points out in \cite{gH87}, this can be achieved because the pinching condition (compare the pinching condition in \cite{gH87}) implies that the submanifold has positive intrinsic curvature.  The next estimate is the part of the argument that relies on the intrinsic curvature of the submanifold being positive.

\begin{lem}\label{l: Z pos sphere}
	Let $\Sigma_0$ be a n-dimensional submanifold immersed in a spherical background of constant curvature $\bar K$.  If $\Sigma_0$ satisifies $\abs{h}^2 \leq \alpha \abs{H}^2 + \beta \bar K$, where
	\begin{equation*}
		\alpha \begin{cases}
			\leq \frac{4}{3n}, \quad n = 2, 3 \\
			< \frac{1}{n-1}, \quad n \geq 4
		\end{cases} \text{and } \quad \beta <
			\begin{cases}
				 \frac{2(n-1)}{3}, \quad n = 2, 3 \\
				2, \quad n \geq 4,
			\end{cases}
	\end{equation*}
then there exists a positive constant $\epsilon$ depending only on $\Sigma_0$ such that the estimate
	\begin{equation*}
		Z + n\bar K\abs{ \ho  }^2\geq \epsilon\abs{\ho}^2( a \abs{H}^2 + \beta_{ \epsilon } \bar K )
	\end{equation*}
holds for all time.
\end{lem}
The proof of this lemma is similar to Lemma \ref{l: Z pos} of the previous chapter.  For the same reasons as in the Pinching Lemma, the cases $H = 0$ and $H \neq 0$ need to be examined seperately, and again the case $H=0$ is treated easily.  Let us briefly examine the case $H \neq 0$.  The computations are the same as those in Lemma \ref{l: Z pos} and one finds, in dimensions two to five
\begin{align*}
		Z + n\abs{ \ho }^2 \bar K &\geq -\abs{\ho_1}^4 + \frac{1}{n}\abs{\ho_1}^2\abs{H}^2 + \frac{1}{n}\abs{\ho_-}^2\abs{H}^2  -\frac{n-2}{2}\abs{\ho_1}^4 - 4 \abs{\ho_1}^2\abs{\ho_-}^2  \\
		& \qquad - \frac{3}{2}\abs{\ho_-}^4 - \frac{n-2}{2n(n-1)}\abs{\ho_1}^2\abs{H}^2 - \frac{n-2}{2n(n-1)}\abs{\ho_-}^2\abs{H}^2 + n( \abs{ \ho_1 }^2 + \abs{ \ho_- }^2 ) \bar K,
\end{align*}
and in dimensions six and higher
	\begin{align*}
		Z + n \abs{ \ho }^2 \bar K &\geq -\abs{\ho_1}^4 + \frac{1}{n}\abs{\ho_1}^2\abs{H}^2 + \frac{1}{n}\abs{\ho_-}^2\abs{H}^2  - \frac{n-2}{2}\abs{\ho_1}^4 -\frac{n+2}{2}\abs{\ho_1}^2\abs{\ho_-}^2 \\
		& \qquad - \frac{3}{2}\abs{\ho_-}^4  - \frac{n-2}{2n(n-1)}\abs{\ho_1}^2\abs{H}^2 - \frac{n-2}{2n(n-1)}\abs{\ho_-}^2\abs{H}^2 + n( \abs{ \ho_1 }^2 + \abs{ \ho_- }^2 ) \bar K .
	\end{align*}
The size of $\alpha$ is computed in the same way as the Euclidean case, and for $\beta$, in all dimensions $n \geq 2$ we require
\begin{equation*}
	-\frac{ \beta }{ \alpha - \frac1n } \Big( \frac1n - \frac{ n - 2 }{ 2n(n-1) } \Big)  + n \geq 0.
\end{equation*}
In dimension two and three this gives $\beta < 2(n-1)/3$, which is more restrictive than that required by the Pinching Lemma.  For $n \geq 4$ we require $\beta < 2$.  We now commence with the integral estimates.

\begin{prop}\label{p: using Z pos}
For any $\eta > 0$ we have the estimate
	\begin{equation*}
		\begin{split}
		&\int_{\Sigma_t } f_{\sigma}^p ( a \abs{H}^2 + \beta_{ \epsilon } \bar K ) \, dV_g \\
		&\quad \leq \frac{ (2p\eta + 5) }{\epsilon} \int_{ \Sigma_t } \frac{ f_{\sigma}^{p-1} }{ (a \abs{H}^2 + \beta_{ \epsilon } \bar K )^{1-\sigma} } \abs{\nabla H}^2 \, dV_g + \frac{ 2(p-1) }{ \epsilon\eta } \int_{\Sigma_t } f_{\sigma}^{p-2}\abs{\nabla f_{\sigma}}^2 \, dV_g. \end{split}
	\end{equation*}
\end{prop}
\begin{proof}
The proof this lemma follows Lemma \ref{l: Z pos} of the previous chapter.
Using the contracted Simons' identity and $\Delta \abs{ H }^2 = 2\abs{ H }\Delta \abs{ H } + 2\abs{ \nabla \abs{ H } }^2$, the Laplacian of $f_{ \sigma }$ can be written as
\begin{align*}
		\Delta f_{\sigma} &= \frac{2}{ (a \abs{H}^2 + \beta_{ \epsilon } \bar K)^{1-\sigma} } \ho_{ ij } \cdot \nabla_i\nabla_j H + \frac{ 2 }{ ( a \abs{ H }^2 + \beta_{ \epsilon} \bar K)^{ 1 - \sigma } } Z + \frac{ 2 }{ ( a \abs{ H }^2 + \beta_{ \epsilon} \bar K)^{ 1 - \sigma } } n \bar K\abs{ \ho }^2 \\
		&\quad- \frac{ 4a(1-\sigma)\abs{ H } }{ a \abs{H}^2 + \beta_{ \epsilon } \bar K }\big\langle \nabla_i\abs{H}, \nabla_i f_{\sigma} \big\rangle  - \frac{ 2a(1-\sigma) }{ a \abs{H}^2 + \beta_{ \epsilon } \bar K }f_{ \sigma }\abs{ H }\Delta\abs{H} \\
		&\quad+  \frac{ 4a^2\sigma(1-\sigma) }{ (a \abs{H}^2 + \beta_{ \epsilon } \bar K)^2 }f_{\sigma}\abs{H}^2\abs{ \nabla\abs{H} }^2 + \frac{2}{ (a \abs{H}^2 + \beta_{ \epsilon } \bar K)^{1-\sigma} } \big( \abs{ \nabla h }^2 - \frac{1}{n}\abs{ \nabla H }^2 \big) \\
		&\quad - \frac{ 2a(1-\sigma) }{ a \abs{H}^2 + \beta_{ \epsilon } \bar K }f_{ \sigma }\abs{ \nabla \abs{H} }^2.
	\end{align*}
We want to estimate $\Delta f_{ \sigma }$ from below.  The first term on the third line is non-negative and we discard it.  Working with the last two terms of line three, using the Kato-type inequality $\abs{ \nabla \abs{H} }^2 \leq \abs{ \nabla H }^2$, equation \eqref{eqn: basic grad est 1 sphere} and the Pinching Lemma we estimate
\begin{align*}
	 &\frac{2}{ (a \abs{H}^2 + \beta_{ \epsilon } \bar K)^{1-\sigma} } \big( \abs{ \nabla h }^2 - \frac{1}{n}\abs{ \nabla H }^2 \big) - \frac{ 2a(1-\sigma) }{ a \abs{H}^2 + \beta_{ \epsilon } \bar K }f_{ \sigma }\abs{ \nabla \abs{H} }^2 \\
	 &\quad\geq \frac{ 2 }{ (a \abs{H}^2 + \beta_{ \epsilon } \bar K)^{1-\sigma} } \bigg\{ \frac{3}{ n+2 } - \frac{1}{n} - \frac{ a \abs{ \ho }^2 }{ a\abs{H}^2 + \beta_{ \epsilon } \bar K } \bigg\} \abs{ \nabla H }^2,
\end{align*}
which is non-negative and we discard this term.  Having discarded these terms we are left with
\begin{align*}
		\Delta f_{\sigma} &\geq \frac{2}{ (a \abs{H}^2 + \beta_{ \epsilon } \bar K)^{1-\sigma} } \ho_{ ij } \cdot \nabla_i\nabla_j H + \frac{ 2 }{ (a \abs{H}^2 + \beta_{ \epsilon } \bar K)^{ 1 - \sigma } } Z + \frac{ 2 }{ (a \abs{H}^2 + \beta_{ \epsilon } \bar K)^{ 1 - \sigma } } n \bar K\abs{ \ho }^2 \\
		&\quad- \frac{ 4a(1-\sigma)\abs{ H } }{ a \abs{H}^2 + \beta_{ \epsilon } \bar K } \big\langle \nabla_i\abs{H}, \nabla_i f_{\sigma} \big\rangle  - \frac{ 2a(1-\sigma) }{ a \abs{H}^2 + \beta_{ \epsilon } \bar K }f_{ \sigma }\abs{ H }\Delta\abs{H}.
\end{align*}
We now multiply this equation by $f_{ \sigma }^{p-1}$ and integrate it over the submanifold.  The terms integrate as follows:
\begin{align*}
	&\int_{ \Sigma_t } f_{ \sigma }^{p-1} \Delta f_{\sigma} \, dV_g = -(p-1) \int_{ \Sigma_t } f_{\sigma}^{p-2} \abs{ \nabla f_{\sigma} }^2 \, dV_g; \\
	&2 \int_{ \Sigma_t} \frac{ f_{ \sigma }^{p-1} }{ ( a \abs{H}^2 + \beta_{ \epsilon } \bar K )^{ 1-\sigma } } \big\langle \ho_{ij}, \nabla_i\nabla_j \big\rangle \, dV_g = -2(p-1) \int_{ \Sigma_t} \frac{ f_{ \sigma }^{p-2} }{ ( a \abs{H}^2 + \beta_{ \epsilon } \bar K )^{1-\sigma} } \big\langle \nabla_i f_{ \sigma } \ho_{ij}, \nabla_j H \big\rangle \, dV_g \\
	&\quad- 2\frac{ (n-1) }{n} \int_{ \Sigma_t } \frac{ f_{\sigma}^{p-1} }{ (a \abs{H}^2 + \beta_{ \epsilon } \bar K)^{1-\sigma} } \abs{ \nabla H }^2 \, dV_g \\
	&\quad + 4(1-\sigma)\int_{ \Sigma_t } \frac{ f_{\sigma}^{p-1} \abs{H} }{ (a \abs{H}^2 + \beta_{ \epsilon } \bar K)^{2-\sigma} } \big\langle \ho_{ij} \nabla_i \abs{H}, \nabla_j \abs{H} \big\rangle \, dV_g; \\
	&-2(1-\sigma) \int_{ \Sigma_t } \frac{ f_{\sigma}^p \abs{H} \Delta\abs{H} }{ (a \abs{H}^2 + \beta_{ \epsilon } \bar K) } \, dV_g =2(1-\sigma) \int_{ \Sigma_t } \frac{ f_{ \sigma }^{p-1} \abs{H} }{ a \abs{H}^2 + \beta_{ \epsilon } \bar K } \big\langle \nabla_i f_{\sigma}, \nabla_i \abs{H} \big\rangle \, dV_g \\
	&\quad+ 2(1-\sigma) \int_{ \Sigma_t } \frac{ f_{\sigma} }{ a \abs{H}^2 + \beta_{ \epsilon } \bar K } \abs{ \nabla \abs{H} }^2 \, dV_g - 4(1-\sigma) \int_{ \Sigma_t } \frac{ f_{\sigma} \abs{H}^2 }{ ( a \abs{H}^2 + \beta_{ \epsilon } \bar K )^2 } \abs{ \nabla \abs{H} }^2 \, dV_g.
\end{align*}	
In performing the integration we have made use of Green's First Identity, the Codazzi equation and the Divergence Theorem.  We discard two terms that have the appropriate sign, noting that the terms with and inner product do not have a sign.  After some factoring and rearranging, we estimate the terms with an inner product using the Cauchy-Schwarz and Kato inequalities to obtain
\begin{align*}
	&2\int_{ \Sigma_t } \frac{ f_{\sigma}^{p-1} }{ (a \abs{H}^2 + \beta_{ \epsilon } \bar K)^{ 1 - \sigma } } Z \, dV_g + 2n \bar K \int_{ \Sigma_t } \frac{ f_{\sigma}^{p-1} }{ (a \abs{H}^2 + \beta_{ \epsilon } \bar K)^{ 1 - \sigma } } \abs{ \ho }^2 \, dV_g \\
	&\quad \leq 2(p-1) \int_{ \Sigma_t} \frac{ f_{ \sigma }^{p-2} }{ ( a \abs{H}^2 + \beta_{ \epsilon } \bar K )^{1-\sigma} } \abs{ \nabla f_{\sigma} } \abs{ \ho } \abs{ \nabla H } \, dV_g \\
	&\quad + 2\frac{ (n-1) }{n} \int_{ \Sigma_t } \frac{ f_{\sigma}^{p-1} }{ (a \abs{H}^2 + \beta_{ \epsilon } \bar K)^{1-\sigma} } \abs{ \nabla H }^2 \, dV_g \\
	&\quad + 4 \int_{ \Sigma_t} \frac{ f_{ \sigma }^{p-1} }{ ( a \abs{H}^2 + \beta_{ \epsilon } \bar K )^{2-\sigma} } \abs{H} \abs{ \ho } \abs{ \nabla H }^2 \, dV_g \\
	&\quad + 4(1-\sigma)(p-2) \int_{ \Sigma_t} \frac{ f_{ \sigma }^{p-1} }{ ( a \abs{H}^2 + \beta_{ \epsilon } \bar K ) } \abs{H} \abs{ \nabla H } \abs{ \nabla f_{\sigma} } \, dV_g \\
	&\quad + 4 \int_{ \Sigma_t } \frac{ f_{\sigma}^p }{ ( a \abs{H}^2 + \beta_{ \epsilon } \bar K )^2 } \abs{H}^2 \abs{ \nabla H }^2 \, dV_g.
\end{align*}
Using the Peter-Paul inequality, as well as the inequalities $\abs{ \ho }^2 \leq f_{\sigma} ( a \abs{H}^2 + \beta_{ \epsilon } \bar K )^{1-\sigma}$,  $f_{\sigma} \leq ( a \abs{H}^2 + \beta_{ \epsilon } \bar K )^{\sigma}$ and $(1 - \sigma) \leq 1$, we estimate each term on the right as follows:
\begin{align*}
	&2(p-1) \int_{ \Sigma_t } \frac{ f_{\sigma}^{p-2} }{ ( a \abs{H}^2 + \beta_{ \epsilon } \bar K )^{1-\sigma} } \abs{ \nabla f_{\sigma} } \abs{ \ho } \abs{ \nabla H } \, dV_g \\
	&\quad \leq \frac{ (1-\sigma) }{ \eta } \int_{ \Sigma_t } f_{\sigma}^{p-2} \abs{ \nabla f_{\sigma} }^2 \, dV_g + (p-1)\eta \int_{ \Sigma_t } \frac{ f_{\sigma}^{p-1} }{ ( a \abs{H}^2 + \beta_{ \epsilon } \bar K )^{1-\sigma} } \abs{ \nabla H }^2 \, dV_g; \\
	&4 \int_{ \Sigma_t } \frac{ f_{\sigma}^{p-1} }{ ( a \abs{H}^2 + \beta_{ \epsilon } \bar K )^{2-\sigma} } \abs{H} \abs{ \ho } \abs{ \nabla H }^2 \, dV_g \leq 4 \int_{ \Sigma_t } \frac{ f_{\sigma}^{p-1} }{ ( a \abs{H}^2 + \beta_{ \epsilon } \bar K )^{1-\sigma} } \abs{ \nabla H }^2 \, dV_g; \\
	&4(p-2) \int_{ \Sigma_t } \frac{ f_{\sigma}^{p-1} }{ ( a \abs{H}^2 + \beta_{ \epsilon } \bar K ) } \abs{H} \abs{ \nabla H } \abs{ \nabla f_{\sigma} } \, dV_g \leq \frac{2}{\eta} (p-2) \int_{\Sigma_t} f_{\sigma}^{p-2} \abs{ \nabla f_{\sigma} }^2 \, dV_g \\
	&\quad + 2(p-2)\eta \int_{ \Sigma_t } \frac{ f_{\sigma}^{p-1} }{ ( a \abs{H}^2 + \beta_{ \epsilon } \bar K )^{1-\sigma} } \, dV_g; \\
	&4 \int_{ \Sigma_t } \frac{ f_{\sigma}^p }{ (a \abs{H}^2 + \beta_{ \epsilon } \bar K)^2 } \abs{H}^2 \abs{ \nabla H }^2 \, dV_g \leq 4 \int_{ \Sigma_t } \frac{ f_{\sigma}^{p-1} }{ (a \abs{H}^2 + \beta_{ \epsilon } \bar K)^{1-\sigma} } \abs{ \nabla H }^2 \, dV_g.
\end{align*}
We use Lemma \ref{l: Z pos} to estimate the two terms on the left:
\begin{equation*}
	2\int_{ \Sigma_t } \frac{ f_{\sigma}^{p-1} }{ (a \abs{H}^2 + \beta_{ \epsilon } \bar K)^{ 1 - \sigma } } ( Z + n \bar K\abs{ \ho }^2 ) \, dV_g \geq 2\epsilon \int_{ \Sigma_t } f_{\sigma}^p ( a \abs{H}^2 + \beta_{ \epsilon } \bar K ) \, dV_g.
\end{equation*}
Putting everything together with a little rough estimation of the coefficients to coax them into a more convenient form we obtain
\begin{equation*}
	\begin{split}
&2\epsilon \int_{ \Sigma_t } f_{\sigma}^p ( a \abs{H}^2 + \beta_{ \epsilon } \bar K ) \, dV_g \\
&\quad \leq (3p\eta + 10 ) \int_{ \Sigma_t } f_{\sigma}^p ( a \abs{H}^2 + \beta_{ \epsilon } \bar K ) \, dV_g = \frac{ 3(p-1) }{\eta} \int_{ \Sigma_t } f_{\sigma}^{p-2} \abs{ \nabla f_{\sigma} }^2 \, dV_g. \end{split}
\end{equation*}
Dividing through by $2\epsilon$ completes the proposition.
\end{proof}
The next step is to show that sufficiently high $L^p$ norms of $f_{\sigma}$ are bounded, and in fact decay exponentially in time.
\begin{prop}\label{p: d/dt f_sigma}
For any $p \geq 8/(\epsilon_{\nabla} +1)$ we have the estimate
	\begin{equation*}
		\begin{split}
		\frac{d}{dt} \int_{\Sigma} f_{\sigma}^p \, dV_g &\leq -\frac{p(p-1)}{2}\int_{\Sigma}f_{\sigma}^{p-2}\abs{\nabla f_{\sigma}}^2 \, dV_g - 2p\epsilon_{\nabla}\int_{\Sigma}\frac{f_{\sigma}^{p-1}}{\abs{H}^{2-\sigma}}\abs{\nabla H}^2 \, dV_g \\ 
		&\quad- 2 n \epsilon \bar Kp \int_{\Sigma} f_{\sigma}^p \, dV_g +2\sigma p \int_{\Sigma}f_{\sigma}^p (a\abs{H}^2+ \beta_{\epsilon} \bar K)\, dV_g.
		\end{split}
	\end{equation*} 
\end{prop}
\begin{proof}
We differentiate under the integral sign and substitute in the evolution equations for $f_{\sigma}$ and the measure $dV_g$ to get
	\begin{align} 
		\notag &\frac{d}{dt} \int_{\Sigma} f_{\sigma}^p \, dV_g  \\ 
		&\quad = \int_{\Sigma}(pf_{\sigma}^{p-1}\frac{\partial f_{\sigma}}{\partial t} - \abs{H}^2f_{\sigma}^p) \, dV_g \notag \\
			 \notag&\quad \leq  \int_{\Sigma}pf_{\sigma}^{p-1}\frac{\partial f_{\sigma}}{\partial t} \, dV_g \\
		\begin{split}\label{e: d/dt f_sigma 1}
				&\quad \leq -p(p-1)\int_{\Sigma} f_{\sigma}^{p-2}\abs{\nabla f_{\sigma}}^2 \, dV_g + 4(1-\sigma)p\int_{\Sigma} \frac{f_{\sigma}^{p-1}}{ a \abs{H}^2 + \beta_{ \epsilon } \bar K } \abs{H} \abs{\nabla \abs{H}} \abs{\nabla f_{\sigma}} \, dV_g \\
						&\quad \quad- 2p\epsilon_{\nabla}\int_{\Sigma}\frac{f_{\sigma}^{p-1}}{ ( a \abs{H}^2 + \beta_{ \epsilon } \bar K )^{1-\sigma} } \abs{\nabla H}^2 \, dV_g - 2 n \epsilon \bar Kp \int_{\Sigma} f_{\sigma}^p \, dV_g + 2\sigma p \int_{\Sigma}\abs{h}^2f_{\sigma}^p \, dV_g.
			\end{split}
	\end{align}
We estimate the second integral by
	\begin{equation*}
		\begin{split}
		&4(1-\sigma)p\int_{\Sigma} \frac{f_{\sigma}^{p-1}}{ a \abs{H}^2 + \beta_{ \epsilon } \bar K } \abs{H}\abs{\nabla \abs{H}}\abs{\nabla f_{\sigma}} \, dV_g \\
		&\quad \leq \frac{2p}{\mu} \int_{\Sigma} f_{\sigma}^{p-2}\abs{\nabla f_{\sigma}}^2 \, dV_g + 2p\mu \int_{\Sigma} \frac{ f_{\sigma}^{p-1} }{ (a \abs{H}^2 + \beta_{ \epsilon } \bar K)^{1-\sigma} } \abs{\nabla H}^2 \, dV_g, \end{split}
	\end{equation*}
then substituting this estimate back into \eqref{e: d/dt f_sigma 1} gives
	\begin{align*}
		&\frac{d}{dt}\int_{\Sigma} f_{\sigma}^p \, dV_g \\
		&\quad \leq \Big(-p(p-1) + \frac{2p}{\mu}\Big)\int_{\Sigma} f_{\sigma}^{p-2}\abs{\nabla f_{\sigma}}^2 \, dV_g \\
		&\quad \quad - ( 2p\epsilon_{\nabla} - 2p\mu )\int_{\Sigma} \frac{f_{\sigma}^{p-1}}{ ( a \abs{H}^2 + \beta_{ \epsilon } \bar K )^{1-\sigma} }\abs{\nabla H}^2 \, dV_g \\
		&\quad \quad+ 2 n \epsilon \bar Kp \int_{\Sigma} f_{\sigma}^p \, dV_g + 2\sigma p \int_{\Sigma}\abs{h}^2f_{\sigma}^p \, dV_g \\
		&\quad = -p(p-1) \Big( 1 - \frac{2}{\mu(p-1)} \Big) \int_{\Sigma} f_{\sigma}^{p-2}\abs{\nabla f_{\sigma}}^2 \, dV_g \\
		&\quad \quad - 2p\epsilon_{\nabla} \Big( 1 - \frac{\mu}{\epsilon_{\nabla}} \Big) \int_{\Sigma} \frac{f_{\sigma}^{p-1}}{ (a \abs{H}^2 + \beta_{ \epsilon } \bar K)^{1-\sigma} }\abs{\nabla H}^2 \, dV_g \\
		&\quad \quad - 2 n \epsilon \bar Kp \int_{\Sigma} f_{\sigma}^p \, dV_g +2\sigma p \int_{\Sigma}\abs{h}^2f_{\sigma}^p \, dV_g.
	\end{align*}
We want to choose $\mu$ so that $1 - 2/(\mu(p-1)) \geq 1/2$ and $p$ so that $1 - \mu/\epsilon_{\nabla} \geq 1/2$.  We therefore choose $\mu = 4/(p-1)$ and $p \geq \max\{2, 8/(\epsilon_{\nabla} + 1)\}$.  In the last term we estimate $\abs{h}^2 \leq ( a \abs{H}^2 + \beta_{ \epsilon } \bar K )$, and the proposition is complete.
\end{proof}

\begin{lem}\label{l: high Lp bounded}
There exist constants $C_2$ and $C_3$ both depending only on $\Sigma_0$ such that if $p \geq C_2$ and $\sigma \leq C_3/\sqrt{p}$, then for all time $t \in [0, \infty)$ we have the estimate
	\begin{equation*}
		\Big( \int_{\Sigma_t } f_{\sigma}^p \, dV_g \Big)^{1/p} \leq C_1 e^{-\delta_1 t}.
	\end{equation*}
\end{lem}
\begin{proof}
Combining Propositions \ref{p: using Z pos} and \ref{p: d/dt f_sigma} we get
\begin{align*}
		&\frac{d}{dt}\int_{\Sigma}f_{\sigma}^p \, dV_g \\
		&\quad \leq -{p(p-1)} \Big( \frac{1}{2} - \frac{4\sigma}{\epsilon\eta} \Big) \int_{\Sigma}f_{\sigma}^{p-2}\abs{\nabla f_{\sigma}}^2 \, dV_g \\
		&\quad \quad- 2\Big(p\epsilon_{\nabla} - \frac{p\sigma(2p\eta+5)}{\epsilon} \Big) \int_{\Sigma}\frac{f_{\sigma}^{p-1}}{ (a \abs{H}^2 + \beta_{ \epsilon } \bar K)^{1-\sigma} }\abs{\nabla H}^2 \, dV_g  \\
		&\quad \quad - 2 n \epsilon \bar Kp \int_{\Sigma} f_{\sigma}^p \, dV_g.
	\end{align*}
Recall we are already assuming that $p \geq \max\{2, 8c/(\epsilon_{\nabla} + 1)\}$. Now suppose that
	\begin{equation*}
		\sigma \leq \frac{\epsilon}{8}\sqrt{\frac{\epsilon_{\nabla}}{p}}.
	\end{equation*}
Set $\eta = 4c\sigma/\epsilon$, then
	\begin{equation*}
		\begin{cases}
			\frac{4\sigma}{\epsilon\eta} = \frac{1}{2} \\
			\frac{p\sigma(2p\eta + 5)}{\epsilon} \leq \frac{1}{16}\sqrt{p\epsilon_{\nabla}}(\sqrt{p\epsilon_{\nabla}} + 5) \leq \frac{p\epsilon_{\nabla}}{8} < p\epsilon_{\nabla}.
			\end{cases}
	\end{equation*}
We require $p \geq 25/\epsilon_{\nabla}$ for the last inequality to hold.  With these assumptions on $p$ and $\sigma$ we have
	\begin{equation*}
		\frac{d}{dt}\int_{\Sigma}f_{\sigma}^p \, dV_g \leq - 2 n \epsilon \bar Kp \int_{\Sigma} f_{\sigma}^p \, dV_g,
	\end{equation*}
and thus
\begin{equation*}
		\frac{d}{dt}\int_{\Sigma}f_{\sigma}^p \, dV_g \leq - \int_{\Sigma}f_{\sigma}^p \, dV_g \Big\lvert_{t=0} e^{-2 n \epsilon \bar Kpt}.
\end{equation*}
This implies the lemma with $C_1 = (\abs{\Sigma_0} + 1) \max_{ \sigma \in [0, 1/2] } ( \max \Sigma_0 f_{\sigma} )$, $\delta_1 \leq 2 n \epsilon \bar Kp$, \\$C_2 := \max \{ 8/(\epsilon_{\nabla} +1), 25/\epsilon_{\nabla} \}$ and $C_3 := \epsilon\sqrt{ \epsilon_{\nabla} }/8$.
\end{proof}

Lemma \ref{l: high Lp bounded} shows that for $\sigma$ sufficiently small, sufficiently high $L^p$ norms of $f_{\sigma}$ are bounded and exponentially decaying in time. We can now proceed as in \cite{gH87} via a Stampacchia iteration procedure to uniformly bound $g_{\sigma} := f_{\sigma}e^{(\delta_1/2) t}$, from which the theorem easily follows.  We point out that during the course of this argument, $\sigma$ is fixed sufficiently small once and for all.
\end{proof}

\section{A gradient estimate for the mean curvature}\label{s: A gradient estimate for the mean curvature}
Here we establish a gradient estimate for the mean curvature.  This estimate is required in the following section to compare the mean curvature at different points of the submanifold.

\begin{thm}\label{t: grad est}
For each $\eta > 0$ there exists a constant $C_{\eta}$ depending only on $\eta$ such that for all time the estimate
	\begin{equation*}
		\abs{ \nabla H }^2 \leq ( \eta \abs{H}^4 + C_{\eta} ) e^{-\delta_0 t/2}
	\end{equation*}
holds.
\end{thm}
We begin by deriving an evolution equation for $\abs{ \nabla H }^2$.
\begin{prop}
We have the evolution equation
\begin{equation}
	\begin{split}\label{e: evol nabla H}
	\frac{ \p }{ \p t }\abs{ \nabla H }^2 &= \Delta\abs{ \nabla H }^2 - 2\abs{ \nabla^2 H }^2 + 2\gp \big( \Rp(\p_k, \p_t)H - \nabla_p\big( \Rp(\p_k, \p_p)H \big) + \Rp(\p_k, \p_p)\nabla_p H \\
	&\quad + Rc_{pk}\nabla_p H + H \cdot h_{pq}h_{pq} + n \bar K H, \nabla_k H \big).
	\end{split}
\end{equation}
\end{prop}
\begin{proof}
We compute
	\begin{align}
		\notag \frac{ \p }{ \p t }\abs{ \nabla H }^2 &= \frac{\p}{\p t} \left\langle\nabla_k H, \nabla_k H\right\rangle \\
		\notag &= 2\gp (  \nabla_t\nabla_k H, \nabla_k H ) \\
		\notag &= 2\gp\big( \nabla_k\nabla_t H + \Rp(\p_k, \p_t)H, \nabla_k H \big) \\
		&= 2\gp\big( \nabla_k(\Delta H + H \cdot h_{pq}h_{pq} + n \bar K H ) + \Rp(\p_k, \p_t)H, \nabla_k H \big). \label{e: gradient est 1}
	\end{align}
The proposition now follows after the use of following two identities:
	\begin{align*}
		\Delta\abs{ \nabla H }^2 &= 2\gp( \Delta\nabla_k H, \nabla_k H) + 2\abs{ \nabla^2 H }^2 \\
		\Delta\nabla_k H &= \nabla_k\Delta H + \nabla_p\big( \Rp(\p_k, \p_p)H \big) + \Rp(\p_k, \p_p)\nabla_p H + Rc_{pk}\nabla_p H.
	\end{align*}
\end{proof}
\begin{cor}
There exist constants $A$ and $B$, depending only on $\Sigma_0$, such that we have the estimate
	\begin{equation*}
		\frac{ \p }{ \p t }\abs{ \nabla H }^2 \leq \Delta\abs{ \nabla H }^2 + A\abs{ H }^2\abs{ \nabla h }^2 + B\abs{ \nabla h }^2.
	\end{equation*}
\end{cor}
\begin{proof}
We estimate the reaction terms of \eqref{e: evol nabla H}.  Using the spacelike Gauss and Ricci equations, all the reaction terms except the first one look like $h^{*2} * \nabla H^{*2}$ and $\nabla H^{*2}$.  We need to use the timelike Ricci equation to estimate the first reaction term, however simple estimation of the ambient curvature term in the timelike Ricci equation gives rise to a term that looks like $H^{*2} * \nabla H$.  A closer inspection of this term shows that in fact it is zero:
\begin{align*}
	\gp\big( \bar{R}(F_*\p_k, F_*\p_t)H, \nablap_k H \big) &= \gp\big( \bar{R}(F_*\p_k, H)H, \nablap_k H \big) \\
	&= \langle F_*\p_k, H \rangle \langle H, \nablap_k H  \rangle -  \langle F_*\p_k, \nablap_k H \rangle \langle H, H \rangle \\
	&= 0.
\end{align*}
All the reaction terms now look like $h^{*2} * \nabla H^{*2}$ and $\nabla H^{*2}$, which we can estimate by $A\abs{ H }^2\abs{ \nabla h }^2$ and  $B\abs{ \nabla h }^2$ using the Pinching Lemma the Cauchy-Schwarz inequality.
\end{proof}
We need two more estimate to complete the proof.
\begin{prop}
We have the estimates
\begin{align}
	&\frac{ \p }{ \p t } \abs{ H }^4 \geq \Delta\abs{ H }^2 - 12\abs{H}^2\abs{ \nabla H }^2 + \frac{4}{n}\abs{H}^6 \label{e: evol H4} \\
	\begin{split}
		&\frac{ \p }{ \p t } \big( (N_1 + N_2\abs{H}^2)\abs{ \ho }^2 \big) \leq \Delta\big( (N_1 + N_2\abs{H}^2)\abs{ \ho }^2 \big) - \frac{ 4(n-1) }{ 3n }(N_2 - 1)\abs{ H }^2\abs{ \nabla h }^2 \\
		&\quad- \frac{ 4(n-1) }{ 3n }(N_1 - c_1(N_2))\abs{ \nabla h }^2 - c_2(N_1, N_2)\abs{ \ho }^2( \abs{ H }^4 + 1 ). \label{e: evol N}
	\end{split}
\end{align}
\end{prop}
\begin{proof}
The evolution equation for $\abs{H}^4$ is easily derived from that of $\abs{H}^2$:
\begin{equation*}
	\frac{ \p }{ \p t } \abs{ H }^4 = \Delta\abs{ H }^2 - 2\abs{ \nabla \abs{H}^2 }^2 - 4\abs{H}^2\abs{ \nabla H }^2 + 4R_2\abs{H}^2 + 4n \bar K\abs{H}^4.
\end{equation*}
We discard the last term and the proposition follows from the use of $\abs{ \nabla\abs{H} }^2 \leq \abs{ \nabla H }^2$ and \\ $R_2 \geq 1/n \abs{H}^4$.  To prove \eqref{e: evol N}, from the evolution equations for $\abs{h}^2$ and $\abs{H}^2$ we derive
\begin{align*}
	&\frac{ \p }{ \p t } \big( (N_1 + N_2\abs{H}^2)\abs{ \ho }^2 \big) \\
	&\quad= \Delta\big( (N_1 + N_2\abs{H}^2)\abs{ \ho }^2 \big) - 2N_2 \big\langle \nabla_i \abs{H}^2, \nabla_i \abs{ \ho }^2 \big\rangle - 2N_2\abs{ \ho }^2\abs{ \nabla h }^2 + 2N_2R_2\abs{ \ho }^2  \\
	&\quad- 2(N_1 + N_2\abs{ H }^2)( \abs{ \nabla h }^2 - \frac{1}{n}\abs{ \nabla H }^2 ) + 2(N_1 + N_2\abs{ H }^2 )(R_1 - \frac{1}{n} R_2) - 2n \bar KN_1\abs{ \ho }^2.
\end{align*}
We estimate the second term on the right as follows:
\begin{align*}
	-2N_2 \big\langle \nabla_i \abs{H}^2, \nabla_i \abs{ \ho }^2 \big\rangle &\leq 8N_2\abs{h}\abs{ \ho }\abs{ \nabla H }\abs{ \nabla h } \\
	&\leq 8N_2\abs{ H }\sqrt{n}\abs{ \nabla h }^2 C_0( \abs{H}^2 + \bar K )^{(1-\sigma)/2} \\
	&\leq \frac{ 4(n-1) }{3n}\abs{H}^2\abs{ \nabla h }^2 + c_1(N_2)\abs{ \nabla h }^2.
\end{align*}
Using Young's inequality, $R_2 \leq \abs{h}^2\abs{H}^2$, and $R_1 - 1/n \, R_2 \leq 2\abs{ \ho }^2\abs{ h }^2$ we estimate
\begin{equation*}
	2N_2R_2\abs{ \ho }^2  + 2(N_1 + N_2\abs{ H }^2 )(R_1 - \frac{1}{n} R_2) \leq c_2(N_1, N_2)\abs{ \ho }^2(\abs{ H }^4 + 1).
\end{equation*}
The constants depend on more that just $N_1$ and $N_2$, however we only highlight the dependence on $N$ as this is relevant in the following proof.  We discard the last term on the right, and equation \eqref{e: evol N} now follows.
\end{proof}
\begin{proof}[Proof of Theorem \ref{t: grad est}]
Consider $f := \abs{ \nabla H }^2 + (N_1 + N_2\abs{H}^2)\abs{ \ho }^2$.  From the evolution equations derived above we see $f$ satisfies
\begin{align*}
	\frac{ \p }{ \p t } f &\leq \Delta f + A\abs{H}^2\abs{ \nabla h }^2 + B\abs{ \nabla h }^2 + \frac{ 4(n-1) }{ 3n }(N_2 - 1)\abs{ H }^2\abs{ \nabla h }^2 \\
	&\quad + \frac{ 4(n-1) }{ 3n }(N_1 - c_1(N_2))\abs{ \nabla h }^2 + c_2(N_1, N_2)\abs{ \ho }^2 ( \abs{ H }^4 + 1 ).
\end{align*}
We choose $N_1$ and $N_2$ large enough to consume the positive terms arising from the evolution equation for $\abs{ \nabla H }^2$.  This leaves
\begin{equation*}
	\begin{split}
	\frac{ \p }{ \p t } f &\leq \Delta f + \frac{ 4(n-1) }{ 3n }(N_2 - 1)\abs{ H }^2\abs{ \nabla h }^2 - \frac{ 4(n-1) }{ 3n }(N_1 - c_3(N_2))\abs{ \nabla h }^2 \\
	&\quad + c_4(N_1, N_2)\abs{ \ho }^2 ( \abs{ H }^4 + 1 ). \end{split}
\end{equation*}
Now consider $g := e^{(\delta_0/2) t} f - \eta\abs{H}^4$.  From the above evolution equations we have
\begin{align*}
	&\frac{ \p }{ \p t } ( e^{(\delta_0/2) t} f - \eta\abs{H}^4 ) \\
	&\quad \leq \frac{ \delta_0 }{ 2 }e^{ (\delta_0/2) t } \big( \abs{ \nabla H }^2 + (N_1 + N_2\abs{H}^2)\abs{ \ho }^2 \big) \\
	&\quad \quad + e^{(\delta_0/2) t}  \big( \Delta f + \frac{ 4(n-1) }{ 3n }(N_2 - 1)\abs{ H }^2\abs{ \nabla h }^2 - \frac{ 4(n-1) }{ 3n }(N_1 - c_3(N_2))\abs{ \nabla h }^2 \\
	&\quad \quad + c_4(N_1, N_2)\abs{ \ho }^2 ( \abs{ H }^4 + 1 ) \big) - \eta \big( \Delta\abs{ H }^2 - 12 \abs{H}^2\abs{ \nabla H }^2 + \frac{4}{n}\abs{H}^6 \big).
\end{align*}
The terms on the first line can be absorbed into those on the second line by suitable estimation.  By choosing $N_2$ sufficiently large the gradient term on the last line can be absorbed, and then we choose $N_1$ larger again to make the $\abs{ \nabla h }^2$ term negative.  We finally discard the negative gradient terms to get
\begin{equation*}
	\frac{ \p }{ \p t }g \leq \Delta g + c_5e^{(\delta_0/2) t}\abs{ \ho }^2\abs{H}^4 - \frac{4\eta}{n}\abs{H}^6.
\end{equation*}
Using Theorem \ref{t: pinching improves} then Young's inequality we obtain
\begin{equation*}
	\frac{ \p }{ \p t }g \leq \Delta g + c_6e^{-(\delta_0/2) t}
\end{equation*}
from which we conclude $g \leq c_7$.  The gradient estimate now follows from the definition of $g$.
\end{proof}

\section{Asymptotic behaviour of the solution}\label{s: Asymptotic behaviour of the solution}
In this final section we study the long time behaviour of the solution.  Two limit profiles are possible, determined by whether or not the mean curvature blows up.  We first examine the case where the mean curvature becomes unbounded.  We do this by using the gradient estimate and Bonnet's Theorem to compare the submanifold at different points.  In the case of a spherical background, the Chen's estimate combined with our pinching condition gives
\begin{equation}
	K_{\text{min}}(x) \geq \frac{1}{2}\Big( \frac{1}{n-1} - \alpha_{\epsilon} \Big)\abs{ H }^2(x) + \frac{ (2 - \beta_{\epsilon})\bar{K} }{ 2 }. \label{e: Kmin pos}
\end{equation}

\begin{thm}
If $\abs{H}_{ \text{max} } \rightarrow \infty$ as $t \rightarrow T$, then $T$ must be finite and $\diam \Sigma_t \rightarrow 0$ as $t \rightarrow T$.
\end{thm}
\begin{proof}
From Theorem \ref{t: grad est}, we know that for any $\eta > 0$ there exists a constant $C_{\eta}$ such that $\abs{\nabla H} \leq \eta\abs{H}^{2} + C_{\eta}$ on $0 \leq t < T$.  We highlight that at this stage, $T$ could be infinite. Since by assumption $\abs{H}_{\text{max}} \rightarrow \infty$ as $t \rightarrow T$, there exists a $\tau(\eta)$ such that $C_{\eta/2} \leq 1/2\eta\abs{H}^{2}_{\text{max}}$ for all $\tau \leq t < T$.  Thus $\abs{\nabla H} \leq \eta\abs{H}^{2}_{\text{max}}$ for all $t \geq \tau$.  Fix some $\sigma \in(0,1)$ and set $\eta = \frac{\sigma(1-\sigma)\varepsilon}{\pi}$.  Let $t\in [\tau(\eta),T)$, and $x$ be a point with $\abs{H}(x)=\abs{H}_\text{max}$.  Along any geodesic of length $\frac{\pi}{\varepsilon\sigma H_\text{max}}$ from $x$, we have $\abs{H}\geq \abs{H}_\text{max}-\frac{\pi}{\varepsilon\sigma \abs{H}_\text{max}}\eta\abs{H}^2_\text{max}=\sigma\abs{H}_\text{max}$, and consequently the sectional curvatures satisfy $K\geq \varepsilon^2\sigma^2\abs{H}_\text{max}^2$. From Bonnet's Theorem it follows that $\diam \Sigma \leq \frac{\pi}{\varepsilon\sigma H_\text{max}}$, from which we conclude that $\abs{H}_\text{min}\geq \sigma\abs{H}_\text{max}$ on the whole of $\Sigma_t$ for $t\in [\tau(\eta),T)$.

The previous line shows that by choosing $\tau$ sufficiently large, $\abs{H}_{\text{min} }$ can be made arbitrarily large.  It follows from Theorem \ref{t: pinching improves} that after some sufficiently large time the submanifold is as pinched as we like (and in particular can be made to satisfy $\abs{h}^2 < 1/(n-1)\abs{H}^2$ in dimensions $n \geq 4$ and $\abs{h}^2 < 4/(3n)\abs{H}^2$ in dimensions $2 \leq n \leq 4$).  We now show that once the submanifolds are pinched as such, the maximal time of existence must be finite.  Define $Q=\abs{H}^2 - a\abs{h}^2 - b(t)$, where $a=\frac{3n}{4}$ and $b$ is some time-dependent function.  Because $\abs{ H }_{ \text{min} } > 0$ and the submanifolds are as pinched as we like, for some sufficiently large time $\tau$ we can choose a $b(\tau) = b_{\tau} >0$ such that $Q\geq 0$ for $t = \tau$.  The evolution equation for $\mathcal{Q}$ is
\begin{align*}
	\frac{\p}{\p t}Q &= \Delta Q - 2(\abs{\nabla H}^2 - a\abs{\nabla h}^2) + 2R_2 - 2aR_1 + 2(n-a)\bar{K}\abs{ \ho }^2 + 2an\bar{K}\abs{H}^2 - b'(t) \\
	&\geq \Delta Q - 2(\abs{\nabla H}^2 - a\abs{\nabla h}^2) + 2R_2 - 2aR_1 - b'(t).
\end{align*}
Estimating the reaction terms as before we obtain
\begin{align*}
&2R_2-2aR_1-b'(t) \\
	&\quad= \sum_{i,j} \Big( \sum_{\alpha} H_{\alpha}h_{ij\alpha} \Big)^2 - 2a\sum_{\alpha, \beta} \Big( \sum_{i,j} h_{ij\alpha}h_{ij\beta} \Big)^2 - 2a\abs{\Rp}^2 - b'(t) \\
	&\quad \geq  2\abs{\ho_1}^2(a\abs{\ho_1}^2 + a\abs{\ho_-}^2 + b) + \frac{2}{n(1-a/n)}(a\abs{\ho_-}^2 + b)(a\abs{\ho_1}^2 + a\abs{\ho_-}^2 + b) \\
	&\quad - 2a\abs{\ho_1}^4 - 8a\abs{\ho_1}^2\abs{\ho_-}^2 - 3a\abs{\ho_-}^4 - b'(t).
\end{align*}
Equating coefficients, we find $Q \geq 0$ is preserved if
$\frac{db}{dt} \leq \frac{8b^2}{n}$. We can therefore take
\begin{equation*}
	b(t) = \frac{nb_0}{n-8b_0(t - \tau)}.
\end{equation*}
This is unbounded as $t \rightarrow \tau +\frac{n}{8b_0}$, so we must have $T \leq \tau +\frac{n}{8b_0}$.

\end{proof}

Let us now consider the case where the mean curvature stays bounded for all time.
\begin{thm}
If $\abs{H}_{ \text{max} }$ remains bounded, then the flow exists for all time and $\Sigma_t$ converges to a totally geodesic submanifold $\Sigma_{\infty}$.
\end{thm}
\begin{proof}
Since $\abs{H}_{\text{max}}$ is bounded, from Theorem \ref{t: pinching improves} and Theorem \ref{t: grad est} we have the estimates
\begin{equation}\label{e: H bounded est}
	\abs{ \ho }^2 \leq C_0e^{-\delta_0 t}
\end{equation}
and
\begin{equation}\label{e: grad H bounded est}
	\abs{ \nabla H }^2 \leq Ce^{-(\delta_0/2) t}.
\end{equation}
From \eqref{e: Kmin pos} we know that smallest sectional curvature is positive, so from Bonnet's Theorem it follows that the diameter of $\Sigma_t$ is bounded.  Using this fact, integrating the second estimate of \eqref{e:  grad H bounded est} along geodesics gives
\begin{equation}\label{e: Hmin - Hmax bounded}
	\abs{H}_{ \text{max} } - \abs{H}_{ \text{min} } \leq Ce^{-(\delta_0/2) t}.
\end{equation}
Now observe that if the time of existence is infinite, then $\abs{H}_{ \text{min} }$ must remain zero:  From \eqref{e: H bounded est} it follows that after sufficiently a long time the submanifolds are again as pinched as we like, and if $\abs{H}_{ \text{min} } > 0$, then the same argument just given in the previous case would show that $T$ must be finite.  Therefore we must have $\abs{H}_{ \text{min} } = 0$.  From equation \eqref{e: Hmin - Hmax bounded} it now follows that $\abs{H}_{ \text{max} } \leq Ce^{-(\delta_0/2) t}$.  Thus $\abs{H}^2$ decays exponentially and consequently 
\begin{equation*}
	\abs{ h }^2 \leq C_0e^{-\delta_0 t}.
\end{equation*}
We now have all the necessary estimates in place to repeat the convergence arguments of the previous chapter to obtain smooth exponential convergence of the submanifolds to a totally geodesic submanifold.
\end{proof}

\chapter{A partial classification of type I singularities}

In Chapter \ref{ch: The flow of submanifolds of Euclidean space} we show that if a submanifold satisfies a suitable pinching condition, then the mean curvature flow evolves the submanifold to round point in finite time. In this chapter we relax the pinching of the initial submanifold and seek to understand the asymptotic shape of the evolving submanifold as we approach the maximal time of existence.  We still assume that $\abs{ H }$ is everywhere positive initially, however having relaxed the pinching assumption, we no longer necessarily expect the entire submanifold to disappear at the maximal time.  In the case of mean-convex hypersurfaces, a classification of type I singularities was achieved by Huisken in \cite{gH90a} and \cite{gH90b}.  A key ingredient in this analysis was Huisken's monontoncity formula, introduced in \cite{gH90a}, which also holds in arbitrary codimension.  The singularities classified by Husiken in \cite{gH90a} are a special kind of type I singularity called a `special' type I singularity.  The more general kind of singularity is naturally called a `general' type I singularity and in order to have a complete understanding of type I singularity formation, it is desirable to be able to treat general singularities (definitions of the various kinds of singularities follow).  In the case of embedded hypersurfaces, the classification of general type I singularities is due to Stone \cite{aS94}.

Here we follow \cite{gH90a} and \cite{gH90b} to give a partial classification of special type I singularities of the mean curvature flow in high codimension. Instead of using the continuous rescaling argument used in \cite{gH90a}, we proceed slightly differently by considering a sequence of parabolically rescaled flows.  Huisken's original argument is recast in terms of rescaled flows in \cite{kE04} and also \cite{mtW04b}.  We cannot make Stone's argument to classify general type I singularities work in high codimension, essentially because a pointwise curvature condition does not seem enough to conclude that embeddedness is preserved.  We point out that even in the codimension one case, the classification of general type I singularities of immersed hypersurfaces is an outstanding problem.  For an excellent account of singularity analysis in the mean curvature flow of hypersurfaces (as well as a wonderful introduction to the mean curvature flow) we recommend to the reader the recent book \cite{cM10} by Mantegazza.

A similar classification of type I singularities of the mean curvature flow in high codimension has previously been obtained by Smoczyk \cite{kS05}.  In \cite{kS05} Smoczyk classifies blow-up limits of the the mean curvature flow that have flat normal normal bundle.  Although this curvature condition is much more restrictive than the pinching condition we have been working with, Smoczyk's classification includes additional submanifolds that do not feature in our classification, namely, products of Euclidean space with an Abresch-Langer curve, which also appear in the hypersurface classification.  These spaces do not appear in our classification as they do not satisfy the pinching assumption.  It's worthwhile to point out that the condition of having flat normal bundle is not preserved by the mean curvature flow.

With regards to this last chapter, we wish to express our gratitude to Patrick Breuning for sending us a draft of his PhD thesis \cite{Br}, in which he improves upon and extends Langer's compactness theorem \cite{La} to submanifolds of arbitrary codimension.  We would also like to thank Andrew Stone for friendly correspondence and for sending us a copy of his PhD thesis.

\section{The blow-up argument}
We shall need some basic concepts from measure theory in this chapter.  We remind the reader that $\Sigma$ is a fixed manifold, and that $\Sigma_t := F_t(\Sigma)$ refers to the immersed submanifold.   For a function $f \in C^{\infty}(\mathbb{R}^{n+k} \times \mathbb{R})$ defined on the ambient space, we follow standard abuse of notation and denote
\[ \int_{\Sigma} f(F(p)) \, d\mu_g(p) = \int_{\Sigma} f(p)\, d\mu_g(p). \]
Integration over the manifold $\Sigma$ with respect to $d\mu_g$ and integration over the image $F(\Sigma)$ in $\mathbb{R}^{n+k}$ are linked by the area formula.  We denote the pushforward measure by $\mu = F(\mu_g)$, where for $U \in \mathbb{R}^{n+k}$ an open set, $\mu(U) := \mu_g(F^{-1}(U))$.  In order for the pushforward measure $\mu$ to be a Radon measure, the immersion $F$ must be a proper immersion.  Recall an immersion is proper if the inverse image of a compact set is also compact.  The area formula relates the induced measure on $\Sigma$ to the Hausdorff measure on $\mathbb{R}^{n+k}$ restricted to the image $F(\Sigma)$ of the immersion.  We denote Hausdorff measure on the ambient space by $d\mathcal{H}^{n+k}$ or simply by $d\mathcal{H}$.  For a $\mu_g$-measurable function $f : \Sigma \rightarrow \mathbb{R}$, by the area formula we have
\[ \int_{\Sigma} f(p) \, d\mu_g(p) = \int_{\mathbb{R}^{n+k} } \left( \sum_{p \in F^{-1}\{x\} } f(p) \right) \, d\mathcal{H}^{n+k}(x). \]
Choosing $f = \chi_{[F^{-1}(F(\Sigma))]}$ gives
\[ \int_{\Sigma} d\mu_g(p) = \int_{F(\Sigma)} \left( \sum_{p \in F^{-1}\{x\} } \right) \, d\mathcal{H}^{n+k}(x). \]
Thus, denoting by $\theta$ the multiplicity function, we have $\mu = \mathcal{H} \llcorner \theta$.  In particular, if $F : \Sigma^n \rightarrow \mathbb{R}^{n+k}$ is a properly embedded submanifold, then $\theta \equiv 1$ and
\[ \int_{\Sigma} d\mu_g(p) = \int_{F(\Sigma)} d\mathcal{H}^{n+k}(x). \]

 For more details on Hausdorff measure and the area formula we refer the reader to \cite{EG}.  One final piece of notation before getting underway, for a point $p \in \Sigma$, we put $\lim_{t\rightarrow T} F(p) := \hat p \in \mathbb{R}^{n+k}$.

In order to study the asymptotic shape of the evolving submanifold $\Sigma_t$ around a singular point as the first singular time is approached, we progressively `magnify' the solution around this point by considering a sequence of rescaled flows.  The limit of such rescaled flows is called a blow-up limit.  Our first task is to show how to obtain such a limit.  In order to obtain a smooth blow-up, we assume that the submanifold is developing a so-called type I singularity.  This imposes a natural maximum rate at which the singularity can develop, which then enables us to rescale at a rate that keeps the maximum curvature of the rescaled solution bounded.

A submanifold is said to be developing a type I singularity at $T$ if there exists a constant $C_0 \geq 1$ such that
\begin{equation*}
	\max_{p \in \Sigma} \abs{ h }^2(p,t) \leq \frac{ C_0 }{ 2(T - t) }.
\end{equation*}
The blow-up rate of any singularity also satisfies the lower bound
\begin{equation*}
	\max_{ p \in \Sigma } \abs{ h }^2(p,t) \geq \frac{ 1 }{ 2(T-t) }
\end{equation*}
(see \cite{gH84} or \cite{cM10}), and so in the case of a type I singularity
\begin{equation*}
	\frac{1}{2(T-t)} \leq \max_{p \in \Sigma} \abs{ h }^2(p,t) \leq \frac{ C_0 }{ 2(T - t) }.
\end{equation*}

Let $q \in \Sigma$ be a fixed point and assume that the type I condition holds.  We want to rescale the solution around the point $\hat{q} \in \mathbb{R}^{n+k}$ by remaining time.  Let $(t_k)_{ k \in \mathbb{N} }$ be any sequence of times such that $t_k \rightarrow T$ as $k \rightarrow \infty$.  For example, we could take $t_k := T - 1/k$.  To rescale by remaining time we set the scale $\lambda_k = 1/\sqrt{2(T - t_k)}$.  We then define a sequence of parabolically rescaled flows
\begin{equation}\label{eqn: rescaled flows}
	F_k(p, s):= \lambda_k \big( F(p, T + s/\lambda_k^2)  - \hat q \big).
\end{equation}
Then for each $k$, $F_k : \Sigma \times [-\lambda_k^2T, 0)$ is a solution to the mean curvature flow (in the time variable $s$) that exists on the time interval $s \in [-\lambda_k^2 T, 0)$.  Under our parabolic rescaling the second fundamental form rescales like $\abs{ h }^2_{ \lambda_k } = \abs{ h }^2/ \lambda_k^2$, so using the type I hypothesis
\begin{align*}
	\abs{ h }^2_{\lambda_k}(p,s) &= \frac{ \abs{ h }^2(p, T + s/\lambda_k^2) }{ \lambda_k^2 } \\
	&\leq 2(T - t_k) \cdot \frac{ C_0 }{ 2(T - T - s/\lambda_k^2) } \\
	&= \frac{ - C_0 }{ 2s }
\end{align*}
which holds on $s \in [-\lambda_k^2 T, 0)$.  Consequently, on the time intervals $I_k := (-\lambda_k^2T, 1/k)$, the rescaled flows have bounded second fundamental form.  We would now like to apply a compactness theorem for immersed submanifolds in order to obtain a limit flow.  The compactness theorem usually quoted in this context is \cite{La}.  The result presented in \cite{La} is for a sequence of two-surfaces of Euclidean three-space with $L^p$-bounded second fundamental form and a global area bound, whereas we need to apply the result to a sequence of $n$-dimensional submanifolds of codimension $k$ in the presence of bounds on all higher derivatives of the second fundamental form and only a local area bound.  Very recently we learnt that in his PhD thesis Patrick Breuning has extended Langer's result to submanifolds of arbitrary codimension in the presence of a local area bound \cite{Br}.  We record Breuning's compactness theorem as follows:
\begin{thm}[Breuning-Langer compactness theorem for immersed submanifolds]
Let  $F_k : M_k^n \rightarrow \mathbb{R}^N$ be a sequence of proper immersions, where $M_k$ is a $n$-manifold without boundary and $0 \in F_k(M_k)$.  Assume the following conditions are satisfied:
\begin{enumerate}
 	\item Uniform curvature derivative bounds: \\
	For each $k \in \mathbb{N}$,  for every $m \in \mathbb{N}$ there exists a constant $C_m(R)$ depending on $m$ and $R$ such that $\abs{ \nabla_k^m h_k}_{ F_k } \leq C_m$.
	\item Local area bound: \\
	For every $R > 0$ there exists a constant $C_R$ depending on R such that $\mu^k(B_R) \leq C_R$.
\end{enumerate}
Then there exists a proper immersion $F_{\infty} : M_{\infty} \rightarrow \mathbb{R}^{n+k}$, where $M_{\infty}$ is again a $n$-manifold without boundary, such that after passing to a subsequence there exists a sequence of diffeomorphisms $\phi_k : U_k \rightarrow (F_k)^{-1}(B_k) \subset M_k$, where $U_k \subset M_{\infty}$ are open sets with $U_k \subset\subset U_{k+1}$ and $M_{\infty} = \bigcup_{j=1}^{\infty} U_j$ such that $\phi_k^*F_k |_{U_j}$ converges in $C^{\infty}( U_j, \mathbb{R}^N)$ to $F_{\infty} |_{U_j}$.
\end{thm}
This is the essentially the statement of the Breuning's theorem in his thesis; we have simply changed some notation to conform with our own.  Note Breuning states the local area bound in terms of the pushforward measure.  Before we learnt of Breuning's compactness theorem we did not know whether Langer's theorem did in fact hold in arbitrary codimension and we produced the following compactness theorem for immersed submanifolds in arbitrary codimension using the well-known compactness theorem of Cheeger and Gromov for abstract manifolds.  We refer the reader to \cite{HA} for an introduction to Cheeger-Gromov convergence, and to \cite{rH95, pP06} for proofs of the Cheeger-Gromov compactness theorem.  We have been influenced by the treatment in \cite{HA}.

We consider the following notion of convergence of a sequence of immersed submanifolds: For each $k \in \mathbb{N}$, let $M_k$ be a complete smooth manifold, $F_k : M_k \rightarrow \mathbb{R}^{N}$ a smooth immersion and $p_k \in M_k$ a basepoint. We say that $(M_k , F_k )$ converges to $(M_{\infty}, F_{\infty})$ on compact sets of $\mathbb{R}^N \times \mathbb{R}$ if there exists an exhaustion 
$\{U_k \}_{k \in \mathbb{N} }$ of $M_{\infty}$ and a sequence of smooth diffeomorphisms $\phi_k : U_k \rightarrow V_k \subset M_k$ satisfying: 
\begin{enumerate}
	\item For every compact $K \subset M_{\infty}$, $\phi_k^*F_k |_K$ converges in $C^{\infty}( K, \mathbb{R}^N)$ to $F_{\infty} |_K$
	\item For any compact $A \subset \mathbb{R}^N$ there is some $k_0 \in \mathbb{N}$ such that $(\phi_k^* F_k)(U_k) \cap A = F_k(M_k) \cap A$ for all $k \geq k_0$
\end{enumerate}
We remark that the Langer-Breuning compactness theorem in the form we have stated it is not quite satisfactory for our purposes, as it does not address the second criterion of the above definition of convergence.
 
\begin{thm}\label{thm: my compactness thm}
Suppose $(F_k , M_k )_{k \in \mathbb{N} }$ is a sequence of proper immersions $F_k : M_k \rightarrow \mathbb{R}^N$ of smooth complete $n$-dimensional manifolds $M_k$ that satisfy the following conditions:
\begin{enumerate}
	\item Uniform curvature derivative bounds: \\
	For each $k \in \mathbb{N}$,  for every $m \in \mathbb{N}$ there exists a constant $C_m(R)$ depending on $m$ and $R$ (and independent of $k$) such that $\abs{ \nabla_k^m h_k}_{ F_k } \leq C_m$
	\item The sequence $( F_k )_{k \in \mathbb{N} }$ does not disappear at infinity: \\
	There exists a radius $R > 0$ such that $B_R(0) \cap F_k(M_k) \neq \emptyset$
for all $k \in N$. 
\item Local area bound: \\
	For every $R > 0$ there exists a constant $C_R$ depending on $R$ (and independent of $k$), such that
	\begin{equation*} \int_{F^{-1}_k(B_R)} d\mu_g^k \leq C_R. \end{equation*}

\end{enumerate} 
Then there exists a subsequence of $(F_k, M_k)_{ k \in \mathbb{N} }$ which converges on compact sets of $\mathbb{R}^N \times \mathbb{R}$ to a complete proper immersion $(M_{\infty}, F_{\infty})$ that also satisfies the the same local area bound.
\end{thm}
Before we prove this theorem, we mention two issues that need to be dealt with in the proof.  First of all, the limit produced by applying the Cheeger-Gromov compactness theorem is an abstract limit that a priori loses all knowledge of the background space.  Second, the limit may be disconnected (e.g. a lengthening cylinder), and so the metric produced by the Cheeger-Gromov compactness theorem only sees the connected component of which is is part.  We therefore need to take care to capture all connected components of the limit.

\begin{proof}
Let $g_k$ denote the metric induced by $F_k$ . We want to use the Cheeger-Gromov compactness theorem to extract a convergent sequence of manifolds $(M_k, g_k )_{ k \in \mathbb{N} }$.  The second assumption of the theorem guarantees that there is at least one sequence of points $(p_k)_{ k \in \mathbb{N} } \in M_k$ whose image lies some ball of finite radius.  As we have already mentioned in Section 7 of Chapter 4, bounds on all higher derivatives of the second fundamental form imply a lower injectivity radius bound.  We may apply the Cheeger-Gromov compactness theorem, and upon passing to a subsequence, we obtain a complete pointed limit manifold $(M_{\infty}, g_{\infty}, p_{\infty})$, an exhaustion $\{U_k \}_{k \in \mathbb{N}}$ of $M_{\infty}$, and a sequence of diffeomorphisms $(\phi_k : U_k \rightarrow V_k \subset M_k )_{k \in \mathbb{N}}$ such that $\phi_k^*g_k$ converges smoothly to $g_{\infty}$ on each compact set $K \subset M$.  By induction, similar to the proof of the higher derivative estimates in Chapter 4, it follows that all higher derivatives of $\phi_k^*F_k$ are uniformly bounded with respect to $g_{\infty}$ on compact sets of $M_{\infty}$.  Passing to a futher subsequence, we obtain smooth convergence of $\phi_k^*F_k$ to a limit immersion $F_{\infty}$ on compact sets of $M_{\infty}$.  At this stage we have shown the first condition in our definition of convergence on compact sets is satisfied.

We now need to show the second condition of our definition is also satisified.  The fact that the area bound holds on the limit is a simple consequence of the $C^1$-convergence of the metrics.  The remaining argument is accomplished by induction and a Cantor diagonal sequence argument.  We begin by looking inside a ball $\bar B_1(0)$ in the ambient space.  Suppose that there exists no $k_0 \in \mathbb{N}$ such that $\phi_k^*F_k(U_k) \cap \bar B_1(0) = F_k(M_k) \cap \bar B_1(0)$ for all $k \geq k_0$, then we can pass to a subsequence such that there exists $\tilde{p}_k \in M_k$ such that for all $k$, $\tilde{p}_k \notin V_k$, whilst $F_k(\tilde{p}_k) \in \bar B_1(0)$.  Passing to a further subsequence, we can assume the sequence of pointed manifolds $(M_k, g_k, \tilde{p}_k)$ converges to a limit $(\tilde{M}_{\infty}, \tilde{g}_{\infty}, \tilde{p}_{\infty})$, so that there is an exhaustion $\{ \tilde U_k \}_{k \in \mathbb{N} }$ and diffeomorphisms $\tilde \phi_k : \tilde U_k \rightarrow \tilde V_k \subset M_k$ with $\tilde \phi_k(\tilde p) = (\tilde p_k)$ such that $\tilde \phi_k^*g_k$ converges to $\tilde g$ smoothly on compact sets in $\tilde M$.  As before, by the Arzela-Ascoli theorem, passing to another subsequence, we can assume that $\tilde \phi_k^*F_k$ converges smoothly on compact subsets to a limit immersion $\tilde F_{\infty}$.  Now we replace $M_{\infty}$ with $M_{\infty} \sqcup \tilde M_{\infty}$ and repeat the process again.  All of these components intersect with $\bar B_1(0)$, and have area inside $B_2(0)$ bounded below, so by the local area bound this process must stop after finitely may steps, and we have produced a manifold $M$ with finitely many connected components with both parts of the compactness theorem holding on $\bar B_1(0)$.

We complete the proof by induction on the size of the balls in the ambient space:  If we have subsequence for which both parts of the theorem hold on $\bar B_n(0)$, then we add in more components if there are points in $\bar B_{n+1}(0)$ that are in $F_k(M_k)$ but not in $\phi_k^*F_k(U_k)$.  By the same argument, after adding in finitely many components we produce a subsequence and a limit immersion satisfying both parts of the compactness theorem on $\bar B_{n+1}(0)$.
\end{proof}

We can use the above compactness theorem for immersed submanifolds to obtain a compactness theorem for mean curvature flows.  The proof follows Hamilton's compactness theorem for Ricci flows.  One applies the compactness theorem for immersed submanifolds at the initial time, and then using the higher derivative bounds and the Arzela-Ascoli theorem combined with a diagonal sequence argument, one obtains a properly immersed limit solution to the mean curvature flow that satisfies the first condition of our definition of convergence on compact sets of $\mathbb{R}^N \times \mathbb{R}$.  That the second convergence criterion is also satisfied again follows quickly from the curvature bounds.  In particular, we can apply the compactness theorem to our sequence of rescaled flows \eqref{eqn: rescaled flows} (where we have assumed the type I hypothesis).  By analogy with Hamilton's compactness theorem for Ricci flows (\cite{rH95}), we only require a bound on the second fundamental form itself (and not any higher derivatives), as the type 1 assumption ensures that all higher derivatives are indeed bounded above.  The missing essential ingredient is the local area bound, which we shall address in the next section, it being a consequence of Huisken's monotonicity formula.

\begin{thm}[Compactness theorem for mean curvature flows]\label{thm: compactness for flows}
Suppose that $(F_k , M_k )_{k \in \mathbb{N} }$ is a sequence of proper time-dependent immersions of smooth complete $n$-dimensional manifolds $M_k$ that satisfy the mean curvature flow on the time interval $I = [t_0, T)$.  Assume the following conditions are satisfied:
\begin{enumerate}
	\item Uniform curvature derivative bounds: \\
	For each $k \in \mathbb{N}$, there exists a uniform constant $C_0$ such that $\abs{ h_k}_{ F_k } \leq C_0$ on $M_k \times I$
	\item The sequence doesn't (initially) disappear at infinity: \\
	There exists a time $t_0$ and radius $R > 0$ such that $B_R(0) \cap F_k(M_k, t_0) \neq \emptyset$
for all $k \in \mathbb{N}$. 
\item Initial local area bound: \\
	For every $R > 0$ there exists a constant $C_R$ depending on $R$ (and independent of $k$), such that
	\begin{equation*} \int_{F_k^{-1}(\cdot, \, t_0)(B_R(0))} d\mu_{g_{t_0}}^k \leq C_R. \end{equation*}
\end{enumerate} 
Then there exists a subsequence $(F_k, M_k)_{ k \in \mathbb{N} }$ which converges on compact sets of $\mathbb{R}^N \times \mathbb{R}$ to a complete proper time-dependent immersion $(M_{\infty}, F_{\infty})$ that is also a solution to the mean curvature flow on the time interval $I$.
\end{thm}

\subsection{Huisken's monotonicity formula}
For a fixed point $(x_0, t_0) \in \mathbb{R}^{n+k} \times \mathbb{R}$ we define the backwards heat kernel centred at $(x_0, t_0)$ by
\begin{equation*}
	\rho_{ x_0, t_0 }(x,t) := \frac{ 1 }{ (4\pi(t_0 - t) )^{n/2} } \text{exp} \Big( \frac{ -\abs{ x - x_0 }^2 }{ 4(t_0 - t) } \Big),
\end{equation*}
which is well-defined on $\mathbb{R}^{n+k} \times (-\infty, t_0)$.  The centre of our backward heat kernel will most often be $(\hat{p}, T) \in \mathbb{R}^{n+k} \times \mathbb{R}$.  Note the backwards heat kernel is defined on the ambient space and so we are adhering to the abuse of notation mentioned at the beginning of this chapter.  Huisken's montonicity formula, which holds in arbitrary codimension, is the following:
\begin{thm}[Huisken's monotonicity formula]
Let $F: \Sigma \times [0,T) \rightarrow \mathbb{R}^{n+k}$ be a solution of the mean curvature flow.  For any fixed point $p \in \Sigma$, the formula
\begin{equation*}
	\frac{d}{dt} \int_{\Sigma} \rho_{ \hat{p}, T } \, d\mu_{g_t} = - \int_{\Sigma} \rho_{ \hat{p}, T } \Big| H + \frac{ F^{\bot} }{ 2(T - t) } \Big| \, d\mu_{g_t} \leq 0
\end{equation*}
holds for all time $0 \leq t  < T$.
\end{thm}
For each pair of times $0 < t_1 < t_2 < T$, the monotonicity formula implies that 
\begin{equation*}
	\int_{ \Sigma } \rho_{ \hat{p}, T } \, d\mu{g_{ t_2 }} \leq \int_{ \Sigma } \rho_{ \hat{p}, T } \, d\mu{g_{ t_1 }}
\end{equation*}
and being the limit of a monotone sequence of decreasing functions, the limit
\begin{equation*}
	\lim_{ t \rightarrow T } \int_{\Sigma} \rho_{ \hat{p}, T } \, d\mu_{g_t}
\end{equation*}
certainly exists and is finite.  We shall also use the notation
\begin{equation*}
	\theta(p,t) := \int_{\Sigma} \rho_{ \hat{p}, T } \, d\mu_{g_t}
\end{equation*}
and
\begin{equation*}
	\Theta(p) := \lim_{ t \rightarrow T } \theta(p,t).
\end{equation*}
Since $\Theta$ is the limit of a monotone sequence of continuous functions, it follows that $\Theta$ is upper-semicontinuous.  We refer to $\theta$ as the heat density and $\Theta$ as the limit heat density.  An important property of the monotonicity formula is that it is invariant under parabolic rescalings.  By the definition of our parabolic rescaling, for each $k$ we have
\begin{align*}
	\int_{\Sigma} \rho_{ \hat{p}, T } \, d\mu_{g_t} &=  \frac{ 1 }{ ( 4\pi(T - t) )^{n/2} } \int_{\Sigma} e^{ - \frac{ \abs{x-\hat{p}}^2 }{ 4(T-t) } } \, d\mu_{g_t} \\
	&= \frac{ 1 }{ (-4\pi s)^{n/2} } \int_{ \Sigma } e^{ -\frac{ \abs{y}^2}{-4s} } \, d\mu_{g_s}^{(\hat{p}, T), \lambda_k} \\
	&= \int_{ \Sigma } \rho \, d\mu_{g_s}^{(\hat{p}, T), \lambda_k}.
\end{align*}
Recalling that $t = T + s/\lambda_k^2$, for each fixed $s \in [-\lambda_k^2 T, 0)$ and all $k$ we have
\[ \int_{\Sigma} \rho_{\hat{p},T} \, d\mu_{g_t} =  \int_{\Sigma} \rho \, d\mu_{g_{s}}^{(\hat{p}, T), \lambda_{k}}, \]
and consequently
\begin{equation}\label{eqn: rescaled mono 1}
	\lim_{ t \rightarrow T } \int_{\Sigma} \rho_{ \hat{p}, T } \, d\mu_{g_t} = \lim_{ k \rightarrow \infty } \int_{\Sigma } \rho \, d\mu_s^{\lambda_k}.
\end{equation}
When it is (reasonably) clear which point we are rescaling around, we will often omit the notation $(\hat{p}, T)$ above the measure as we have just done to reduce clutter.  An important application of the monotonicity formula is that it provides the local area bound (independent of k) necessary to apply the compactness theorem for mean curvature flows.  It suffices to obtain the area bound on bounded subintervals $I_l := [-\lambda_l^2T, 1/l] \subset [-\lambda_l^2T, 0)$, as the final argument will be completed by a diagonal sequence argument sending $l$ to infinity.  Let us fix a point $p \in \Sigma$ and some $k_0 >> 0$ sufficiently large.  With these choices of $p$ and $k_0$, then for all $s \in I_{k_0}$ and every $k > k_0$ monotonicity formula gives the estimate
\begin{equation*}
	\int_{\Sigma } \rho \, d\mu_ {g_{ s }}^{ \lambda_k } \leq \int_{ \Sigma } \rho_{ \hat{p}, T } d\mu_{ g_{t_0} } \leq \frac{ \mu_{ g_{t_0} }(\Sigma) }{ (4\pi T)^ \frac{ n }{ 2 } }.
\end{equation*}
We then compute
\begin{align*}
	\int_{F_k^{-1}(B_R)} \, d\mu_{g_s}^{\lambda_k} &= \int_{F_k^{-1}(B_R)} \chi_{B_R} \, d\mu_{g_s}^{\lambda_k} \\
	&\leq \int_{F_k^{-1}(B_R)} \chi_{B_R} e^{ \frac{ R^2 - \abs{y}^2 }{ -4s } }  \, d\mu_{g_s}^{\lambda_k} \\
	&\leq e^{ \frac{ k_0R^2 }{ 4 } } \int_{F_k^{-1}(B_R)} e^{ \frac{- \abs{y}^2 }{ -4s } }  \, d\mu_{g_s}^{\lambda_k} \\
	&\leq e^{ \frac{ k_0R^2 }{ 4 } } (4\pi \lambda_{k_0}^2 T )^{n/2} \int_{ \Sigma } \frac{ 1 }{ (-4\pi s)^{n/2} } e^{ \frac{- \abs{y}^2 }{ -4s } }  \, d\mu_{g_s}^{\lambda_k} \\
	&\leq e^{ \frac{ k_0R^2 }{ 4 } } \lambda_{k_0}^n \mu_{g_{t_0} }(\Sigma),
\end{align*}
and thus
\[ \int_{F_k^{-1}(B_R)} \, d\mu_{g_s}^{\lambda_k} \leq C_R(\Sigma_0, T, I). \]

We can now apply Theorem \ref{thm: compactness for flows} to our sequence of rescaled flows $F_k : \Sigma \times [-\lambda_l^2T, 1/l] \rightarrow \mathbb{R}^{n+k}$ defined by
\[ F_k(p, s) =  \lambda_k \big( F(p, T + s/\lambda_k^2)  - \hat q \big). \]  We highlight that here $\Sigma$ is fixed, and by assumption closed, however $\Sigma_{\infty}$ is complete and not necessarily compact.  The existence of the limit flow $(F_{\infty}, \Sigma_{\infty})$ on the time interval $(-\infty, 0)$ follows by diagonal sequence argument letting $l \rightarrow \infty$.

Another consequence of the monotonicty formula is the following important result, which enables us to pass the limit through the integral in the rescaled heat densities.  The result is due independently to Ilmanen \cite{tI95} and Stone \cite{aS94}, who proved it slightly different contexts. Ilmanen proved it in the setting of Brakke flows, while Stone proved it in the context of Huisken's original continuous rescaling argument.  We recast their proof in our setting.

\begin{prop}\label{p: tightness of measures}
Let $F_k : \Sigma \times [-\lambda_k^2T, 0) \rightarrow \mathbb{R}^{N}$ be a sequence of proper mean curvature flows of a closed manifold $\Sigma$ that subconverges on compact sets of $\mathbb{R}^N \times \mathbb{R}$ to a proper mean curvature flow $F_{\infty} : \Sigma_{\infty} \times (\infty, 0) \rightarrow \mathbb{R}^{N}$, where $\Sigma_{\infty}$ is a complete manifold.  Assume that for all $R > 0$ the initial submanifold satisfies the area bound
\[ \int_{F^{-1}_0(B_R)} \, d\mu_{g_{t_0}} \leq AR^m. \]
Then for any given $\epsilon > 0$ and any fixed point $p \in \Sigma$, there exists a sufficiently large radius radius $R = R(\epsilon, \Sigma_0, s)$ such that for each fixed $s \in [ -\lambda_{k_0}^2 T, 0)$ and all $k > k_0$ we have
\begin{equation*}
	\int_{ \Sigma \setminus F^{-1}_k(B_{R}) } \rho \, d\mu_{g_s}^{(\hat{p}, T), \lambda_k} \leq \epsilon.
\end{equation*}
\end{prop}
\begin{proof}
By localising the mononicity formula (see \cite{tI95} or \cite{kE04}) and using the initial area bound we see for each fixed $s \in [-\lambda_k^2T, 0)$ and all $k > k_0$ that
\[ \int_{F^{-1}_k(B_R)} d\mu_{g_s}^k \leq C(A, m) R^m \]
for every $R > R_0$.  For every $R >  R_0$ and each fixed $s \in [ -\lambda_k^2 T, 0)$ we estimate
\begin{align*}
	\int_{ \Sigma  \setminus F_k^{-1}(B_{R}) } \Phi \, d\mu_{g_s}^{ (\hat{p}, T), \lambda_k} &\leq \frac{ C }{ (-s)^{ \frac{ n }{ 2 } } } \sum_{j=1}^{\infty} \int_{ F_k^{-1} (B_{ R^{j+1} } \setminus B_{R^j} ) } e^{ -R^{2j} / (-4s) } \, d\mu_{g_s}^{ (\hat{p}, T), \lambda_k } \\
	&\leq \frac{ C }{ (-s)^{ \frac{ n }{ 2 } } } \sum_{j=1}^{\infty} R^{ m(j+1) } e^{ -R^{2j}/(-4s) }.
\end{align*}
For each fixed $s \in [ -\lambda_k^2 T, 0)$, the term on the right can be made as small as we like by choosing $R$ sufficiently large, so for any given $\epsilon$ we can fix $R$ sufficiently large so that the desired estimate holds for all $R \geq R_1$.
\end{proof}

The proposition is, by definition, the statement that the family of weighted measures $\rho \, d\mu_{g_s}^{\lambda_k}$ is tight for each fixed $s$.  By Prohorov's Theorem we immediately obtain the following important corollary:
\begin{equation*}
	\lim_{k \rightarrow \infty} \int_{\Sigma } \rho \, d\mu_{g_s}^{\lambda_k} = \int_{ \Sigma_{\infty} } \rho \, d\mu_{g_s}^{\lambda_{\infty} } < \infty.
\end{equation*}
Let us dwell for a second on why this result is important:  The limit manifold $\Sigma_{\infty}$ we obtain from the compactness theorem is complete, and not necessarily compact.  Certainly if $\Sigma_{\infty}$ contains a compact component, then this component is diffeomorphic to $\Sigma$ by definition of the convergence.  However, if $\Sigma_{\infty}$ is only complete, as it often will be, then the integral
\[ \int_{\Sigma_{\infty}} \rho \, d\mu_{g_s}^{\lambda_{\infty}} \]
could very well be infinite.  The fact that the weighted family of measure is tight ensures that the measure `does not escape to infinity' in the limit.  We remark that the $C^1$-convergence of $F_k$ and $g_k$ obtained from the compactness theorem implies that $\mu_k \rightarrow \mu$, that is the pushforward measures converge weak-${*}$ in $\mathbb{R}^{n+k}$.

\section{A partial classification of special type I singularities}
In order to probe the shape of the evolving submanifold as the first  singular time is approached, we want to rescale the monotonicity formula around the singular point $\hat{p}$.  A point $p \in \Sigma$ is called a general singular point if there exists a sequence of points $p_k \rightarrow p$ and times $t_k \rightarrow T$ such that for some constant $\delta > 0$,
\begin{equation*}
	\abs{ h }^2(p_k, t_k) \geq \frac{ \delta }{ T - t_k }.
\end{equation*}
A point $p \in \Sigma$ is called a special singular point if there exists a sequence times $t_k \rightarrow T$ such that for some constant $\delta > 0$,
\begin{equation*}
	\abs{ h }^2(p, t_k) \geq \frac{ \delta }{ T - t_k }.
\end{equation*}
This distinction between singular points is not made in \cite{gH90a}, and the points studied in \cite{gH90a} are actually special singular points (see Defintion 2.1 of \cite{gH90a}).  The analysis to cope with moving points was subsequently contributed by Stone in \cite{aS94}.  We now give a partial classification of special type I singularities in high codimension.

\begin{prop}
Let $\Sigma : \times [0, T) \rightarrow \mathbb{R}^{n+k}$ be a solution of the mean curvature flow.  If the evolving submanifold exhibits a special type 1 singularity as $t \rightarrow T$, then there exists a sequence of rescaled flows $F_k(\Sigma)$ that subconverges to a limit flow $F_{\infty}(\Sigma_{\infty})$ on compact set of $\mathbb{R}^{n+k} \times \mathbb{R}$ as $k \rightarrow \infty$.  Moreover, $F_{\infty} : \Sigma_{\infty} \times (-\infty, 0) \rightarrow \mathbb{R}^{n+k}$ satisfies $ H = - 1/(2s)F^{\bot}$ and is not a plane.
\end{prop}
\begin{proof}
The existence of the limit flow, which exists for $s \in (\infty, 0)$, was shown in preceeding section.  It remains to show the last two assertions of the proposition.  Suppose the special type I singulariy is forming at some point $(\hat{p}, T) \in \mathbb{R}^{n+k} \times \mathbb{R}$, so by definition there exists a sequence times $t_k \rightarrow T$ such that for some constant $\delta > 0$, we have $\abs{ h }^2(p, t_k) \geq \frac{ \delta }{ T - t_k }$.
Rescaling Huisken's monontonicity formula at each scale $\lambda_k = 1/\sqrt{2(T-t_k)}$ about the single fixed point $\hat{p}$ gives
\begin{equation*}
	\frac{ d }{ ds } \int_{ \Sigma } \rho \, d\mu_{g_s}^{\lambda_k} = - \int_{ \Sigma } \rho \Big| H_{ \lambda_k } + \frac{ 1 }{ 2s } F_{ \lambda_k }^{ \bot } \Big| \, d\mu_{g_s}^{ \lambda_k },
\end{equation*}
which holds for all $k$ and $s \in [-\lambda_k^2T, 0)$.  For any fixed $s_0 \in [-\lambda_k^2T, 0)$ and $\sigma > 0$ we integrate this from $s_0 -  \sigma$ to $s_0$ and rearrange a little to get
\begin{equation*}
			 \int_{s_0 - \sigma }^{ s_0 } \int_{ \Sigma } \rho \Big| H_{ \lambda_k } + \frac{ F_{ \lambda_k }^{ \bot } }{ 2s } \Big| \, d\mu_{g_s}^{ \lambda_k } = \int_{ \Sigma } \rho \, d\mu_{g_{ s_0 - \sigma } }^{ \lambda_k } - \int_{ \Sigma } \rho \, d\mu_{g_{ s_0 }}^{ \lambda_k }.
\end{equation*}
We take the limit as $k \rightarrow \infty$, and by equation \eqref{eqn: rescaled mono 1} and Proposition \ref{p: tightness of measures} we have
\begin{equation*}
	\int_{ \Sigma_{\infty} } \rho \, d\mu_{g_{ s_0 - \sigma } }^{ \lambda_{ \infty } } = \lim_{ t \rightarrow T } \int_{\Sigma} \rho_{ ( \hat{p}, T ) } \, d\mu_{g_t} = \int_{ \Sigma_{\infty} } \rho \, d\mu_{g_{s_0} }^{ \lambda_{ \infty } } < \infty.
\end{equation*}
We then conclude, using Proposition \ref{p: tightness of measures} again, that
\begin{align*}
	\lim_{ k \rightarrow \infty } \int_{s_0 - \sigma }^{ s_0 } \int_{ \Sigma } \rho \Big| H_{ \lambda_k } + \frac{ F_{ \lambda_k }^{ \bot } }{ 2s } \Big| \, d\mu_{g_s}^{ \lambda_k } &= \int_{ s_0 - \sigma }^{ s_0 } \int_{ \Sigma_{\infty} } \rho \Big| H_{ \lambda_{ \infty } } +\frac{ 1 }{ 2s } F_{ \lambda_{ \infty } }^{ \bot } \Big| \, d\mu_{g_s}^{ \lambda_{ \infty } } \\
	&= 0,
\end{align*}
and therefore $ H_{ \lambda_{ \infty } } = - 1/(2s) F_{ \lambda_{ \infty } }^{\bot}$ on $s \in [s_0 - \sigma, s_0]$.  Finally, for every scale $\lambda_k$, at the fixed point $p$ at time $s_k = \lambda_k^2( T - t_k ) = -1/2$ the rescaled second fundamental form satifies the lower bound
\begin{align*}
	\abs{ h }_{ \lambda_k }^2(p,s_k) &= \frac{ \abs{ h }^2(p, t_k) }{ \lambda_k^2 } \\
	&\geq 2(T - t_k) \cdot \frac{ \delta }{ T - t_k } \\
	&= 2\delta.
\end{align*}
Thus the the limit flow also satisifies $\abs{ h }_{ \lambda_{ \infty } }^2(p, -1/2 ) \geq 2\delta$ and consequently it is not flat.
\end{proof}

We have just shown that the blow-up limit of a type I singularity is self-similar. In order to give a partial classification of these solutions, in addition to assuming that $\Sigma_0$ satisifes $\abs{ H }_{ \text{min} } > 0$, we also assume it satisfies the pinching condition $\abs{h}^2 \leq 4/(3n) \abs{ H }^2$.  The pinching condition allows us to eventually reduce the problem to that of classifying hypersurfaces of a $\mathbb{R}^{n+1}$.  This classification result was also used in the application of the strong maximum principle in Chapter \ref{ch: The flow of submanifolds of Euclidean space}
and for completeness we give a proof, adopting the proof in \cite{CdCK70} to the case of a flat background.  We mention that this classification first appeared in \cite{bL69}, where different techniques were used.

\begin{prop}\label{prop: covariant const class}
Let $F : M^n \rightarrow \mathbb{R}^{n+1}$ be an immersion of a closed manifold.  If $F(M)$ satisfies $\nabla h = 0$, then $F(M)$ is of the form $\mathbb{S}^{p} \times \mathbb{R}^{n-p}$, where $0 \leq p \leq n$.
\end{prop}
\begin{proof}
The proof is a very nice application of the method of moving frames and Frobenius' Theorem.  Recall from Chapter \ref{ch: submanifold geometry in high codimension} that the structure equations of $\mathbb{R}^{n+1}$ restricted to the hypersurface $F(M)$ are
\begin{gather}
	d\omega^i = -\omega_j^i \wedge \omega^j \label{e: structure eqn 3} \\
	\omega_j^i = - \omega_i^j \label{e: structure eqn 1} \\
	d\omega_j^i = - \omega_k^i \wedge \omega_j^k - \omega_{n+1}^i \wedge \omega_j^{n+1} \label{e: structure eqn 2},
\end{gather}
and that the first covariant derivative of $h$ is 
\begin{equation}\label{e: covariant const 1}
	h_{ijk}\omega^k = dh_{ij} - h_{il}\omega_j^l - h_{lj}\omega_i^l.
\end{equation}
Choose a local frame $\{ e_1, \ldots, e_n, \nu \}$ for $M$ that diagonalises the second fundamental form. So $h_{ij} = 0$ for ${i \neq j}$.  If the principal curvatures are all zero, in which case $M = \mathbb{R}$, the lemma is clearly true, and so from now on we assume that at least one of the principle curvatures is non-zero.  If $h_{ijk} = 0$, then setting $i = j$ in the above equation and using that $h_{ij} = 0$ for ${i \neq j}$ we get
\begin{equation*}
	0 = dh_{ii} - 2 h_{il}\omega_i^l,
\end{equation*}
and since $h_{il}$ is symmetric in $i$ and $l$, and $\omega_i^l$ is antisymmentric (from \eqref{e: structure eqn 1}), we have $dh_{ii} = 0$, so $h_{ii}$ is constant.  Since $h_{ijk} = 0$, $dh_{ij}$ must be zero, and with these two conditions equation \eqref{e: covariant const 1} becomes
\begin{equation*}
	0 = h_{il}\omega_j^l - h_{lj}\omega_i^l = (h_i - h_j)\omega_j^i,
\end{equation*}
which shows that $\omega_j^i = 0$ whenever $h_{ii} \neq h_{jj}$.  Thus if $h_{ii} \neq h_{jj}$, equation \eqref{e: structure eqn 2} reads
\begin{equation*}
	0 = d\omega_j^i = - \omega_k^i \wedge \omega_j^k - \omega_{n+1}^i \wedge \omega_j^{n+1}.
\end{equation*}
The term $- \omega_k^i \wedge \omega_j^k$ must be zero because $\omega_k^i \neq 0$ and $\omega_j^k \neq 0$ would imply $h_{ii} = h_{jj} = h_{kk}$, which contradicts our assuption that $h_{ii} \neq h_{jj}$.  Therefore,
\begin{align*}
	0 &= -\omega_{n+1}^i \wedge \omega_j^{n+1} \\
	&= h_{ik}h_{jl} \omega^k \wedge \omega^l \\
	&=  h_{ii}h_{jj} \omega^i \wedge \omega^j.
\end{align*}
We conclude that if $h_{ii} \neq h_{jj}$, then either $h_{ii}$ or $h_{jj}$ is zero, but not both.  By reordering the indices of the frame if necessary, for each $0 \leq p \leq n$ and a constant $\kappa \neq 0$ we have now shown
\begin{equation}\label{e: covariant const 2}
	\begin{cases}
		\kappa_i = \ldots \kappa_p = \kappa \\
		\kappa_{p+1} = \kappa_{n} = 0 \\
		\omega_j^i = 0 \text{ for } 1 \leq i \leq p \text{ and } p + 1 \leq j \leq n.
	\end{cases}
\end{equation}
Now define two distributions by $\omega^1 = \ldots = \omega^p = 0$ and $\omega^{p+1} = \ldots = \omega^n = 0$.  Frobenius's Theorem states that a distibution $\omega^k = 0$, $1 \leq k \leq n$, is integrable if and only if $d\omega^k = 0$ for every $k$.  From the structure equation \eqref{e: structure eqn 3} we have $d\omega^i = -\omega^i_j \wedge \omega^j$, and so by the third equation of \eqref{e: covariant const 2} both the distributions just defined are integrable.  We therefore obtain a local decomposition at every point of $M$ given by $\mathbb{S}^{p} \times \mathbb{R}^{n-p}$, where $0 \leq p \leq n$.
\end{proof}
	
Let us now commence with classification in the compact case.

\begin{thm}
Suppose $F_{\infty} : \Sigma_{\infty}^n \times (-\infty, 0) \rightarrow \mathbb{R}^{n+k}$ arises as the blow-up limit of the mean curvature flow $F: \Sigma^n \times [0, T) \rightarrow \mathbb{R}^{n+k}$ about a special singular point.  Additionally, suppose that $\Sigma_0$ satisfies $\abs{ H }_{ \text{min} } > 0$ and $\abs{ h }^2 \leq 4/(3n) \abs{ H }^2$.  If $F_{\infty}(\Sigma_{\infty})$ is compact, then at time $s = -1/2$, $F_{\infty}(\Sigma_{\infty})$ must be a sphere $\mathbb{S}^m(m)$ or one of the cylinders $\mathbb{S}^{m}(m) \times \mathbb{R}^{n-m}$, where $1 \leq m \leq n-1$.
\end{thm}
\begin{proof}
Choose a local frame $\{ e_i : 1 \leq i \leq n \}$ for $\Sigma$.  We advise the reader that in the following we are making the indentification $e_j = F_*e_j$.  Take inner product of $ H = - F^{\bot}$ with $H$ and differentiate in the ambient space in the chosen local frame to get
\[ 2 \langle \nabla_j  H, H \rangle = -\langle \nabla_j F, H \rangle - \langle F, \nabla_j H \rangle. \]
Using the Gauss relation and $\nabla_j F = e_j$ we continue to compute
\[ \langle \nablap_j H, H \rangle = \big\langle H, \langle F, e_p \rangle h_{ip} \big\rangle \]
and therefore $\nablap_j H = \langle F, e_p \rangle h_{jp}$.  A further differentiation gives
\begin{align}
	\notag \nabla_i \nabla_j H &= \langle e_i, e_p \rangle h_{jp} + \langle F, h_{ip} \rangle h_{jp} + \langle F, e_p \rangle \nabla_p h_{ij} \\
	\notag &=	\langle e_i, e_p \rangle h_{jp} - \langle H, h_{ip} \rangle h_{jp} + \langle F, e_p \rangle \nabla_p h_{ij} \\
	&= h_{ij} - H \cdot h_{ip} h_{jp} + \langle F, e_p \rangle \nabla_p h_{ij}. \label{e: compact 2}
\end{align}
Contracting \eqref{e: compact 2} with $g_{ij}$ gives
\begin{equation*}
	\Delta H = H - H \cdot h_{ip} h_{ip} + \langle F, e_p \rangle \nabla_p H,
\end{equation*}
and after taking the inner product with $H$ we obtain
\begin{equation}\label{e: compact H2}
	\Delta \abs{H}^2 = 2 \abs{H}^2 - 2 \sum_{ i,j } ( H \cdot h_{ij} )^2 + \langle F, e_p \rangle \nabla_p H \cdot H + 2 \abs{ \nabla H }^2.
\end{equation}
On the other hand, contracting \eqref{e: compact 2} with $g_{ij}$ we get
\begin{equation}\label{e: compact 3}
	h_{ij} \cdot \nabla_i \nabla_j H = \abs{h}^2 - H \cdot h_{ip} h_{ij} \cdot h_{jp} + \langle F, e_p \rangle \nabla_p h_{ij} \cdot h_{ij}.
\end{equation}
Now recall Simons' indentity: $\Delta \abs{h}^2 = 2 h_{ij} \cdot \nabla_i \nabla_j H + 2 \abs{ \nabla h }^2 + 2Z$.  Combining Simons' indentity and \eqref{e: compact 3} gives
\begin{equation}\label{e: compact h2}
	\Delta \abs{ h }^2 = 2 \abs{ h }^2 + 2\langle F, e_p \rangle \nabla_p h_{ij} \cdot h_{ij} + 2 \abs{ \nabla h }^2 - 2 \sum_{ \alpha, \beta } \Big( \sum_{ i, j } h_{ij\alpha} h_{ij\beta} \Big) - \abs{ \Rp }^2.
\end{equation}
Note that the term $H \cdot h_{ip} h_{ij} \cdot h_{jp}$ cancels.  The idea now is to examine the scaling-invariant quantitiy $\abs{ h }^2 / \abs{ H }^2$, and to do so, we first establish $\abs{ H } \neq 0$ in order to perform the division.  The strong elliptic minimum principle applied to equation \eqref{e: compact H2} shows that either $\abs{ H } \equiv 0$ or $\abs{ H } > 0$ everywhere.  Since $F_{\infty}(\Sigma_{\infty})$ is assumed to be compact, it must be that $\abs{ H } > 0$ everywhere.  Using equations \eqref{e: compact H2} and \eqref{e: compact h2} we compute $\Delta( \abs{ h }^2 / \abs{ H }^2 )$ and obtain
\begin{equation}\label{e: compact h2 / H2}
	\begin{split}
		0 &= \Delta \left( \frac{ \abs{ h }^2 }{ \abs{ H }^2 } \right) - \frac{ 2 }{ \abs{ H }^2 } \big( \abs{ \nabla h }^2 - \frac{ \abs{ h }^2 }{ \abs{ H }^2 } \abs{ \nabla H }^2 \big) + \frac{ 2 }{ \abs{ H }^2 } \big( R_1 - \frac{ \abs{ h }^2 }{ \abs{ H }^2 } R_2 \big) \\		&\qquad + \frac{ 2 }{ \abs{ H }^2 } \nabla_i \abs{ H }^2 \nabla_i \left( \frac{ \abs{ h }^2 }{ \abs{ H }^2 } \right) - \langle F, e_i \rangle \nabla_i \left( \frac{ \abs{ h }^2 }{ \abs{ H }^2 } \right).
	 \end{split}
\end{equation}
Since $F_{\infty}(\Sigma_{\infty})$ is assumed to be compact, the function $\abs{ h }^2 / \abs{ H }^2 $ attains a maximum somewhere in $\Sigma$.  At a maximum $\nabla_i ( \abs{ h }^2 / \abs{ H }^2 ) = 0$ and $\Delta ( \abs{ h }^2 / \abs{ H }^2 ) \leq 0$, and so at a maximum we have
\begin{equation*}
	0 = \Delta \left( \frac{ \abs{ h }^2 }{ \abs{ H }^2 } \right) - \frac{ 2 }{ \abs{ H }^2 } \big( \abs{ \nabla h }^2 - \frac{ \abs{ h }^2 }{ \abs{ H }^2 } \abs{ \nabla H }^2 \big) + \frac{ 2 }{ \abs{ H }^2 } \big( R_1 - \frac{ \abs{ h }^2 }{ \abs{ H }^2 } R_2 \big).
\end{equation*}
Moreover, from the basic gradient estimate \eqref{eqn: basic grad est 2} and the Pinching Lemma we can estimate
\begin{equation}
	0 \leq \Delta \left( \frac{ \abs{ h }^2 }{ \abs{ H }^2 } \right) - c_1(n) \abs{ \nabla h }^2 - c_2(n) \abs { \ho_1 }^2\abs{ \ho_- }^2 - c_3(n) \abs{ \ho_- }^4,
\end{equation}
where $c_1$, $c_2$ and $c_3$ are positive constants that depend only on $n$. We conclude from the strong elliptic maximum principle that $\abs{ h }^2 / \abs{ H }^2$ must be equal to a constant and $\abs{ \nabla h }^2 = \abs{ \ho_- }^2 = 0$.  This implies that $F_{\infty}(\Sigma_{\infty})$ is a hypersurface of some $(n+1)$-subspace of $\mathbb{R}^{n+k}$ with covariant constant second fundamental form, and since was assumed to be compact, from Proposition \ref{prop: covariant const class} it must be a $n$-sphere.
\end{proof}
If $F_{\infty}(\Sigma_{\infty})$ is no longer compact then we cannot apply the maximum principle as we have just done.  In this more general case, following \cite{gH90b}, we multiply equation \eqref{e: compact h2 / H2} by $\nabla_i e^{ - \abs{ x }^2 / 2 } $ and integrate by parts.  The following theorem includes the previous one as a special case.
\begin{mthm}
Suppose $F_{\infty} : \Sigma_{\infty}^n \times (-\infty, 0) \rightarrow \mathbb{R}^{n+k}$ arises as the blow-up limit of the mean curvature flow $F: \Sigma^n \times [0, T) \rightarrow \mathbb{R}^{n+k}$ about a special singular point.  If $\Sigma_0$ satisfies $\abs{ H }_{ \text{min} } > 0$ and $\abs{ h }^2 \leq 4/(3n) \abs{ H }^2$, then at time $s=-1/2$, $F_{\infty}(\Sigma_{\infty})$ must be a sphere $\mathbb{S}^m(m)$ or one of the cylinders $\mathbb{S}^{m}(m) \times \mathbb{R}^{n-m}$, where $1 \leq m \leq n-1$.
\end{mthm}
\begin{proof}
We multiply equation \eqref{e: compact h2 / H2} by $\nabla_i e^{ - \abs{ x }^2 / 2 } $ and integrate the term involving the Laplacian by parts to achieve
\begin{equation*}
 	\begin{split}
		0 &= -\int_{ \Sigma_{\infty} } \Big| \nabla \left( \frac{ \abs{ h }^2 }{ \abs{ H }^2 } \right) \Big|^2 e^{ \frac{ - \abs{ x }^2 }{ 2 } } \, d\mu_{g} - 2 \int_{ \Sigma_{\infty} } \frac{ \abs{ h }^2 }{ \abs{ H }^2 } \big( \abs{ \nabla h }^2 - \frac{ \abs{ h }^2 }{ \abs{ H }^2 } \abs{ \nabla H }^2 ) e^{ \frac{ - \abs{ x }^2 }{ 2 } } \, d\mu_{g} \\
		&\qquad + 2 \int_{ \Sigma_{\infty} } \frac{ \abs{ h }^2 }{ \abs{ H }^2 } ( R_1 - \frac{ \abs{ h }^2 }{ \abs{ H }^2 } R_2 ) e^{ \frac{ - \abs{ x }^2 }{ 2 } } \, d\mu_{g}. \end{split}
\end{equation*}
The above equation again implies that $\abs{ h }^2 / \abs{ H }^2$ must be equal to a constant and $\abs{ \nabla h }^2 = \abs{ \ho_- }^2 = 0$ and the theorem follows.

\section{General type I singularities}

As we mentioned in the introduction to this chapter, because the mean curvature flow in high codimension does not preserve embeddedness we are not able to extend Stone's hypersurface argument to high codimension.  Let us explore a little why this is the case.  Stone's result for hypersurfaces is the following:

\begin{prop}\label{prop: Stone}
Let $F : \Sigma^n \times [0, T) \rightarrow \mathbb{R}^{n+1}$ be a solution of the mean curvature flow.  Suppose that $\Sigma_0$ is embedded and satisfies $\abs{ H }_{\min} > 0$.  If the evolving submanifold develops a type I singularity at some point $p \in \Sigma$ as $t \rightarrow T$, then $p$ is a special singular point.
\end{prop}
Stone's analysis shows that it is in fact enough to understand special singular points.  We follow closely \cite{aS94}, adapting his proof from the continuous rescaling setting to that of rescaled flows.  Stone's argument requires the classification of special type I singularities for hypersurfaces obtained by Huisken in \cite{gH90a} and \cite{gH90b}:

\begin{thm}\label{thm: Huisken classification}
Let $F_{ \infty} (\Sigma_{\infty}^n) \subset \mathbb{R}^{n+1}$ be a hypersurface that arises as a blow-up limit of the mean curvature flow.  If $\Sigma_0$ is embedded and satisfies $H \geq 0$, then $\Sigma_{ \infty }$ must be a hyperplane, the sphere $\mathbb{S}^m(m)$ or one of the cylinders $\mathbb{S}^{m}(m) \times \mathbb{R}^{n-m}$, where $1 \leq m \leq n-1$.
\end{thm}
Our equivalent theorem for submanifolds is Main Theorem \ref{mthm: main thm 4}.  We also need to know that embeddedness of hypersurfaces is preserved by the mean curvature flow, and that the blow-up limit is also embedded.  A proof of the former follows the next proposition, whilst for embeddedness of the limit we refer the reader to \cite{cM10}.

\begin{proof}[Proof of Proposition \ref{prop: Stone}]
Suppose that $\Sigma_t$ is developing a general type I singularity at some point $(p, T)$.  By definition, there exists a sequence of points $p_k \rightarrow p$ and times $t_k \rightarrow T$ such that for some constant $\delta > 0$,
\begin{equation*}
	\abs{ h }^2(p_k, t_k) \geq \frac{ \delta }{ T - t_k }.
\end{equation*}
As before, we want rescale the monotonicity formula, but now we need to rescale about the moving point $\hat{p}_k$.  Rescaling the monotonicity formula about the moving points $\hat{p}_k$ gives
\begin{equation*}
	\frac{ d }{ ds } \int_{ \Sigma } \rho \, d\mu_{g_s}^{ (\hat{p}_k, T), \lambda_k} = - \int_{ \Sigma } \rho \Big| H_{ \lambda_k } + \frac{ 1 }{ 2s } F_{ \lambda_k }^{ \bot } \Big| \, d\mu_{g_s}^{ (\hat{p}_k, T), \lambda_k },
\end{equation*}
which holds for each $k$ and $s \in [-\lambda_k^2T, 0)$.  For any fixed $s_0 \in [-\lambda_k^2T, 0)$ and $\sigma > 0$ we integrate this from $s_0 -  \sigma$ to $s_0$ and rearrange a little to get
\begin{equation}\label{eqn: Stone's Thm 1}
			 \int_{s_0 - \sigma }^{ s_0 } \int_{ \Sigma } \rho \Big| H_{ \lambda_k } + \frac{ F_{ \lambda_k }^{ \bot } }{ 2s } \Big| \, d\mu_{g_s}^{ (\hat{p}_k, T), \lambda_k } = \int_{ \Sigma } \rho \, d\mu_{g_{ s_0 - \sigma } }^{ (\hat{p}_k, T), \lambda_k } - \int_{ \Sigma } \rho \, d\mu_{g_{ s_0 } }^{ (\hat{p}_k, T), \lambda_k }.
\end{equation}
The difficulty now is that in general, $\lim_{ k \rightarrow \infty } \theta(p_k, t_k) \neq \Theta(p)$.  The proof is now by contradiction.  If $p$ is a general singular point but is not a special singular point, then by definition there exists some function $\epsilon(t)$ with $\epsilon(t) \rightarrow 0$ as $t \rightarrow T$ such that
\begin{equation*}
	\abs{ h }^2(p,t) \leq \frac{ \epsilon(t) }{ 2( T - t ) }
\end{equation*}
for all time $t \in [0, T)$.  This implies that any blow-up about the single fixed point $\hat{p}$ would satisfy $\abs{ h }^2 = 0$.  From Theorem \ref{thm: Huisken classification} we know that a blow-up around a special singular point is one of $n+1$ different hypersurfaces.  Furthermore, the heat density function evaluated on these hypersurfaces takes on $n+1$ distinct values, of which $1$ is the smallest, which corresponds to a unit multiplicity plane.  Full details of these calculations can be found in the Appendix of \cite{aS94}.  Crucially, since $\Sigma_{ \infty }$ is also embedded, it can only be a unit multiplicity plane, and not a plane of higher mulitplicity.  Since $\Theta$ is upper-semicontinuous, it is actually continuous at $p$, and therefore $\Theta = 1$ in a whole neighbourhood of $p$.  Dini's Theorem on the monotone convergence of functions now implies for $k$ sufficiently large, that $\theta(p_k, t_k) \rightarrow \Theta(p)$ uniformly.  This is the point at which the argument breaks down in high codimenion: since embeddedness of the initial submanifold is not preserved, the blow-up limit may be a plane of higher multiplicity, and thus  $\Theta(p)$ could be any integer.  Therefore, we cannot conclude that $\Theta$ is continuous at $p$, and Dini's Theorem is no longer applicable.

We complete Stone's argument: Returning now to equation \eqref{eqn: Stone's Thm 1}, for every fixed $s_0$ and every fixed point $\hat{p}_k$ the monontonicty formula implies
\begin{equation*}
-\int_{ \Sigma } \rho \, d\mu_{g_{ s_0 }}^{ (\hat{p}_k, T), \lambda_k } \leq -\int_{ \Sigma } \rho \, d\mu_{g_{s_0}}^{ (\hat{p}_k, T), \lambda_l }
\end{equation*}
for all $l > k$.  Estimating as such, for all $l > k$ we have
\begin{equation*}
			 \int_{s_0 - \sigma }^{ s_0 } \int_{ \Sigma } \rho \Big| H_{ \lambda_k } + \frac{ F_{ \lambda_k }^{ \bot } }{ 2s } \Big| \, d\mu_{g_s}^{ (\hat{p}_k, T), \lambda_k } \leq \int_{ \Sigma } \rho \, d\mu_{g_{ s_0 - \sigma }}^{ (\hat{p}_k, T), \lambda_k } - \int_{ \Sigma } \rho \, d\mu_{g_{ s_0 }}^{ (\hat{p}_k, T), \lambda_l }.
\end{equation*}
Sending $l \rightarrow \infty$ and using Proposition \ref{p: tightness of measures} we obtain
\begin{equation*}
	\int_{s_0 - \sigma }^{ s_0 } \int_{ \Sigma } \rho \Big| H_{ \lambda_k } + \frac{ F_{ \lambda_k }^{ \bot } }{ 2s } \Big| \, d\mu_{g_s}^{ (\hat{p}_k, T), \lambda_k } \leq \int_{ \Sigma } \rho \, d\mu_{g_{ s_0 - \sigma } }^{ (\hat{p}_k, T), \lambda_k } - \Theta(p_k).
\end{equation*}
By Dini's Theorem, given any $\epsilon > 0$, there exists a $k_0$ such that for all $k > k_0$ we have
\begin{equation*}
	\int_{s_0 - \sigma }^{ s_0 } \int_{ \Sigma } \rho \Big| H_{ \lambda_k } + \frac{ F_{ \lambda_k }^{ \bot } }{ 2s } \Big| \, d\mu_{g_s}^{ (\hat{p}_k, T), \lambda_k } \leq \epsilon
\end{equation*}
and thus
\begin{equation*}
	\lim_{ k \rightarrow \infty } \int_{s_0 - \sigma }^{ s_0 } \int_{ \Sigma } \rho \Big| H_{ \lambda_k } + \frac{ F_{ \lambda_k }^{ \bot } }{ 2s } \Big| \, d\mu_{g_s}^{ (\hat{p}_k, T), \lambda_k } = 0.
\end{equation*}
Using the blow-up procedure of the previous section we obtain a limit flow on $(-\infty, 0)$, and by Proposition \ref{p: tightness of measures} the limit solution satisfies $H_{ \lambda_{ \infty } } = -1/(2s)F_{ \lambda_{ \infty } }$ and is again not flat.  This is a contradicton, since by Proposition \ref{p: tightness of measures},

\begin{equation*}
	\lim_{ k \rightarrow \infty } \theta(p_k, t_k) = \Theta(p) = 1,
\end{equation*}
which implies the limit solution is a plane and hence flat.
\end{proof}

In order to extend Stone's argument to submanifolds, we must conclude that the blow-up limit is a unit multiplicity plane.  As in the case of hypersurfaces, it would be enough to show that the limit is embedded.  In high codimension embeddedness is not in general preserved by the mean curvature flow, and unfortunately for us, a pointwise pinching condition alone does not seem enough to guarantee the preservation of embeddedness.  We give a proof of that the mean curvature flow preserves the embeddedness of hypersurfaces to highlight the problem the high codimension introduces.

\begin{prop}\label{prop: preservation of embeddedness}
Let $\Sigma^n$ be a closed manifold, and $F : \Sigma^n \times [0, T) \rightarrow \mathbb{R}^{n+1}$ a solution of the mean curvature flow.  If $\Sigma_0$ is embedded, then it remains embedded for as long as the flow is defined.
\end{prop}
\begin{proof}
We follow \cite[pg. 25]{cM10} initially, but give an alternate argument to show that the distance squared $d^2$ between two points is non-decreasing in time.  Let $F : \Sigma^n \times [0,T]$ be a closed hypersurface, initially embedded, moving by the mean curvature flow, and suppose for a contradiction that $T$ is the first time at which the hypersurface fails to be embedded.  The set $S$ of pairs of points $(x,y)$, $x \neq y$, such that $F(x,T) = F(y,T)$ is a nonempty closed set disjoint from the diagonal in $\Sigma \times \Sigma$, otherwise $\Sigma_T$ fails to be an immersion at some point of $\Sigma$.  We may therefore remove a small open neighbourhood $B_{ \epsilon }( \Delta )$ from around the diagonal such that $ \overline{B_{ \epsilon }( \Delta ) } \cap S = \emptyset$.  We consider the quantity
\begin{equation*}
	\delta = \inf_{ t \in [0,t] } \inf_{ (p,q) \in \p B_{ \epsilon }(\Delta) } \abs{ F(y,T) -F(x,T) },
\end{equation*}
and note that $\delta$ is positive, since $ \overline{ B_{ \epsilon }( \Delta ) } \cap S = \emptyset$ and $\p B_{ \epsilon }(\Delta)$  is compact.  Next we claim that the square of the minimum of the distance function
\begin{equation*}
	d^2(t) = \min_{ (x,y) \in \Sigma \times \Sigma \setminus B_{ \epsilon }(\Delta) } \abs{ F(y,T) -F(x,T) }^2
\end{equation*}
is bounded below by $\min \{ d^2(0), \delta \} > 0$ on $[0,T]$.  This contradicts the fact that $S$ is nonempty and contained in $\Sigma \times \Sigma \setminus B_{ \epsilon }(\Delta)$.  To this end, if at some time $d^2(t) < \delta$, then this must occur at points not belonging to $\p B_{ \epsilon }(\Delta)$, that is at points $(x,y) \in \Sigma \times \Sigma \setminus B_{ \epsilon }(\Delta)$.  We now want to show that $d^2$ is non-decreasing on $ \Sigma \times \Sigma \setminus B_{ \epsilon }(\Delta) \times [0,T]$, which proves the claim and the theorem.  We compute the first, second and time derivates of $d^2$ in some choice of local coordinates $\{ x^i \}$ near $x$ and $\{ y^i \}$ near $y$.  The first derivatives are
\begin{gather*}
	\p_{ y^j } d^2 = 2 \langle y - x, \p_{y^j} \rangle \\
	\p_{ x^j } d^2 = -2 \langle y - x, \p_{x^j} \rangle;
\end{gather*}
the second derivatives
\begin{gather*}
	\p_{ y^i } \p_{ y^j } d^2 = 2 g_{ij}^y - 2 \langle y - x, h_{ij}^y \nu_y \rangle \\
	\p_{ x^i } \p_{ x^j } d^2 = 2 g_{ij}^x + 2 \langle y - x, h_{ij}^x \nu_x \rangle  \\
	\p_{ x^i } \p_{ y^j } d^2 = -2 \langle \p_{x^i}, \p_{y^j} \rangle;
\end{gather*}
and last of all the time derivative is
\begin{equation*}
	\p_t d^2 = 2 \langle y - x, -H_y \nu_y + H_x \nu_x \rangle.
\end{equation*}
Importantly, observe that at a minimum of the distance function the tangent planes at $x$ and $y$ are parallel to each other.  We may therefore choose local coordinates such that  $\{ x^i \}$ and $\{ y^i \}$ are parallel for each $i$.  We now compute
 \begin{align*}
 	\frac{ \p d^2 }{ \p t } &- \left( g^{ij}_x \frac{ \p^2 d^2 }{ \p x^i \p x^j } + g^{ij}_y \frac{ \p^2 d^2 }{ \p y^i \p y^j } + 2g_x^{ik} g_y^{jl} \langle \p_{ x^k }, \p_{ y^l } \rangle\frac{ \p^2 d^2 }{ \p x^i \p y^j } \right) \\
	&= 2 \langle y - x, -H_y \nu_y + H_x \nu_x \rangle - 2n - 2 \langle y - x, H_x \nu_x \rangle - 2n + 2 \langle y - x, H_y \nu_y \rangle + 4n \\
	&=0.
\end{align*}
We conclude by the maximum principle that $d^2$ is non-decreasing in time.
\end{proof}
In the above proof it was crucial that we were able to choose parallel orthonormal frames at the points $x$ and $y$ : without the good $4n$ contribution from the cross-derivative terms the proof does not work.  In high codimension this is not possible to do in general since the tangent planes could easily be orthogonal to each other at a point of minimum distance, which results in zero contribution from the cross-terms.
  
\section{Hamilton's blow-up procedure}
In the above singularity analysis the type 1 assumption was essential to obtain a smooth blow-up limit.  It is tempting to think that the above blow-up analysis could be simplified by using a Hamilton blow-up argument, in which case even for a type II singularity we could obtain a smooth blow-up limit.  If one performs a type II Hamilton blow-up and uses this in combination with the monotonicity formula in a similar fashion to what we have done above, then one again obtains a smooth limit solution to the mean curvature flow that satisfies $\abs{ H } = - F$.  The subtle problem with doing this is that the point of maximum curvature may not actually lie in this limit.  As an example, blowing-up the grim reaper in such a fashion would in fact result in a cylindrical limit.

Next we want to give a more successful application of a Hamilton blow-up to give a short proof of the limiting spherical shape of the evolving submanifolds considered in Main Theorem 2.  The interested reader may like compare the following with the corresponding argument in the Ricci flow, which can found, for example, in \cite{pT06}.   Here the Codazzi equation performs the same role as the contracted second Bianchi indentity, and the Codazzi Theorem that of Schur's Theorem.  For a proof of the Codazzi Theorem we refer the reader to \cite[Thm. 26]{mS79}. To begin, pick any sequence of times $(t_k)_{k \in \mathbb{N}}$ such that $t_k \rightarrow T$ as $k \rightarrow \infty$.  The Pinching Lemma implies that $\abs{ h }^2$ and $\abs{ H }^2$ have equivalent blow-up rates, so we can in fact rescale by $\abs{ H }^2$.  Then, since $\Sigma$ is assumed to be closed, we can pick a sequence of points $(p_k)_{ k \in \mathbb{N} }$ defined by
\begin{equation*}
	\abs{H}(p_k, t_k)  = \max_{ p \in \Sigma }\abs{ H }( p, t_k ).
\end{equation*}
For notational convenience, set $\lambda_k := \abs{H}(p_k, t_k)$.  We now define a sequence of rescaled and translated flows by
\begin{equation*}
	F_k(q,s) = \lambda_k \big( F(q, t_k + s/\lambda_k^2) - F(p_k,t_k) \big),
\end{equation*}
where for each $k$, $F_k : \Sigma \times [\lambda_k^2T, 0] \rightarrow \mathbb{R}^{n+k}$ is a solution of the mean curvature flow (in the time variable $s$).  The second fundamental form of the rescaled flows is uniformly bounded above independent of $k$ and we can apply the compactness theorem for mean curvature flows to obtain a smooth limit solution of the mean curvature flow $F_{\infty} : \Sigma_{\infty} \times (-\infty, 0] \rightarrow \mathbb{R}^{n+k}$.  Futhermore, at $s=0$ the limit solution satisfies $\abs{ H }^2_{ \lambda_k } = 1$ by construction, so the limit is not flat.  By definition of the rescaling, the second fundamental form rescales as $\abs{ h }_{\lambda_k}^2 = \abs{h}^2/ \lambda_k^2$, and so estimate \eqref{eqn: pinch trace 2ff} of Chapter \ref{ch: The flow of submanifolds of Euclidean space} rescales as
\begin{equation*}
	\abs{ \ho }^2_{ \lambda_k } \leq C_0 \lambda_k^{ -\delta }\abs{ H }^2_{ \lambda_k }.
\end{equation*}
The limit therefore satisfies
\begin{equation}\label{e: blow up 1}
	\abs{ \ho }^2_{ \lambda_{ \infty } } = 0,
\end{equation}
and thus $F_{\infty}(\Sigma_{\infty})$ is a totally umbilic submanifold.  By the Codazzi Theorem, $F_{\infty}(\Sigma_{\infty})$ must be plane or a $n$-sphere lying in a $(n+1)$-dimensional affine subspace of $\mathbb{R}^{n+k}$. We know that $\abs{ H }_{ \lambda_{ \infty } } = 1$, and so $F_{\infty}(\Sigma_{\infty})$ is not a plane.
\end{proof}

\end{document}